\documentclass[a4paper,reqno]{amsart}
\usepackage{amssymb}
\usepackage{amsmath}
\usepackage{amsthm}
\usepackage[dvipsnames]{xcolor}
\usepackage[shortlabels]{enumitem}
\usepackage{braket,comment,soul,cancel}
\usepackage{hyperref}
\usepackage{pdfpages,faktor}
\usepackage{graphicx,mathabx,bbm}
\usepackage{caption}
\usepackage{subcaption}
\usepackage{mathrsfs} 
\usepackage{tikz-cd} 
\usepackage{bm}
\usepackage{enumitem}
\usepackage{mathtools}
\usepackage{pgfplots}
\usepackage{import}
\usepackage{xifthen}
\usepackage{transparent}
\usepackage{xargs}

\usepackage[colorinlistoftodos,prependcaption,textsize=tiny]{todonotes}
\newcommandx{\unsure}[2][1=]{\todo[linecolor=red,backgroundcolor=red!25,bordercolor=red,#1]{#2}}
\newcommandx{\change}[2][1=]{\todo[linecolor=blue,backgroundcolor=blue!25,bordercolor=blue,#1]{#2}}
\newcommandx{\info}[2][1=]{\todo[linecolor=OliveGreen,backgroundcolor=OliveGreen!25,bordercolor=OliveGreen,#1]{#2}}
\newcommandx{\improvement}[2][1=]{\todo[linecolor=Plum,backgroundcolor=Plum!25,bordercolor=Plum,#1]{#2}}

\makeatletter
\DeclareFontFamily{U}{tipa}{}
\DeclareFontShape{U}{tipa}{m}{n}{<->tipa10}{}
\newcommand{\arc@char}{{\usefont{U}{tipa}{m}{n}\symbol{62}}}%

\newcommand{\arc}[1]{\mathpalette\arc@arc{#1}}

\newcommand{\arc@arc}[2]{%
  \sbox0{$\m@th#1#2$}%
  \vbox{
    \hbox{\resizebox{\wd0}{\height}{\arc@char}}
    \nointerlineskip
    \box0
  }%
}
\makeatother

\numberwithin{equation}{section}
\theoremstyle{plain}
\newtheorem{theorem}{Theorem}[section]

\newtheorem{prop}[theorem]{Proposition}
\newtheorem{lem}[theorem]{Lemma}
\newtheorem{cor}[theorem]{Corollary}

\newtheorem{example}[theorem]{Example}
\newtheorem{question}[theorem]{Question}
\newtheorem*{question*}{Question}
\newtheorem{rmk}[theorem]{Remark}
\newtheorem{conj}[theorem]{Conjecture}

\theoremstyle{definition}
\newtheorem{defn}[theorem]{Definition}
\newtheorem{remark}[theorem]{Remark}

\theoremstyle{definition}
\newtheorem{thmx}{Theorem}

\newcommand{\R}{\mathbb{R}}

\newcommand{\C}{\mathbb{C}}

\newcommand{\Z}{\mathbb{Z}}

\newcommand{\N}{\mathbb{N}}
\newcommand{\D}{\mathbb{D}}

\newcommand{\cF}{\mathcal{F}}
\newcommand{\cG}{\mathcal{G}}
\newcommand{\cH}{\mathcal{H}}

\newcommand{\cP}{\mathcal{P}}

\newcommand{\cT}{\mathcal{T}}

\newcommand{\pcf}{\textrm{post-critically \ finite} }

\renewcommand{\epsilon}{\varepsilon}

\DeclareMathOperator{\pc}{pcf}

\DeclareMathOperator{\QS}{QS}

\DeclareMathOperator{\PL}{PL}
\DeclareMathOperator{\PPSL}{PPSL}
\DeclareMathOperator{\PDyad}{PDyad}

\DeclareMathOperator{\Mob}{Mob}
\DeclareMathOperator{\bdd}{bdd}

\DeclareMathOperator{\Conf}{Conf}
\DeclareMathOperator{\QC}{QC}

\DeclareMathOperator{\Mol}{Mol}

\DeclareMathOperator{\Aut}{Aut}
\DeclareMathOperator{\Per}{Per}
\DeclareMathOperator{\Int}{Int}

\DeclareMathOperator{\diam}{diam}

\DeclareMathOperator{\rat}{Rat}
\DeclareMathOperator{\kle}{Klein}
\DeclareMathOperator{\sch}{Sch}

\DeclareMathOperator{\Teich}{Teich}
\DeclareMathOperator{\Ber}{Ber}
\DeclareMathOperator{\PC}{PostCrit}

\pgfplotsset{compat=1.18} 

\numberwithin{figure}{section}

\title{Universality of the Basilica}

\begin{author}[Y.~Luo]{Yusheng Luo}
\address{Department of Mathematics, Cornell University, 212 Garden Ave, Ithaca, NY 14853, USA}
\email{yl3769@cornell.edu, yusheng.s.luo@gmail.com}
\thanks{Y.L. was partially supported by NSF Grant DMS-2349929.}
\end{author}

\begin{author}[M.~Mj]{Mahan Mj}
\address{School of Mathematics, Tata Institute of Fundamental Research, 1 Homi Bhabha Road, Mumbai 400005, India}
\email{mahan@math.tifr.res.in, mahan.mj@gmail.com}
\thanks{M.M. was partially supported by  the Department of Atomic Energy, Government of India, under project no.12-R\&D-TFR-5.01-0500, and an endowment of the Infosys Foundation.}
\end{author}

\begin{author}[S.~Mukherjee]{Sabyasachi Mukherjee}
\address{School of Mathematics, Tata Institute of Fundamental Research, 1 Homi Bhabha Road, Mumbai 400005, India}
\email{sabya@math.tifr.res.in, mukherjee.sabya86@gmail.com}
\thanks{S.M. was partially supported by the Department of Atomic Energy, Government of India, under project no.12-R\&D-TFR-5.01-0500, an endowment of the Infosys Foundation, and SERB research project grant MTR/2022/000248.}
\end{author}

\begin{document}

\begin{abstract}
We establish universality of  the fat Basilica Julia set $J(z^2-\frac34)$  in conformal dynamics in the following sense: $J(z^2-\frac34)$ is quasiconformally equivalent to
the fat Basilica Julia set of any polynomial as well as to the limit set of any geometrically finite closed surface Bers boundary group. We thus obtain the first example of a connected rational Julia set, not homeomorphic to the circle or the sphere, that is quasiconformally equivalent to a Kleinian limit set. It follows that any geometrically finite Bers boundary limit set is conformally removable.
Other consequences of this  universality result include  quasi-symmetric uniformization of polynomial fat Basilicas by round Basilicas, and the existence of infinitely many non-commensurable uniformly quasi-symmetric surface subgroups of the Basilica quasi-symmetry group. We apply our techniques to cuspidal Basilica Julia sets arising from Schwarz reflections and cubic polynomials, yielding further universality classes. We also show that the standard Basilica Julia set $J(z^2-1)$ is the archbasilica in the David hierarchy.
\end{abstract}

\date{\today}

\maketitle

\setcounter{tocdepth}{1}
\tableofcontents

\section{Introduction}\label{sec-intro}
In this paper, we study relationships between two kinds of conformal dynamical systems on the Riemann sphere:  Kleinian groups and rational maps. Our investigation is motivated, in part, by the following question:
\begin{itemize}
    \item Consider the collection of  limit sets of finitely generated Kleinian groups and Julia sets of rational maps. When are two elements of this collection quasiconformally equivalent?
\end{itemize}

Two markedly different phenomena arise as evidenced by our main results (see \S~\ref{subsec:rvu} for more details):
\begin{enumerate}
    \item In the {\em rigid regime}, any quasiconformal homeomorphism can be promoted to a conformal one, and the two dynamical systems are `commensurable'. 
    Example: Sierpinski carpet Julia sets or limit sets.
    \item In the {\em universality regime},  any two that are homeomorphic are quasiconformally so (under an appropriate type-preserving condition).\\
    Examples:
    \begin{itemize}[leftmargin=8mm]
        \item quasi-circle Julia sets or limit sets of conformal dynamical systems including Kleinian groups (cf. \cite{Ber60,McM84}),
        \item cauliflower Julia sets or limit sets of certain conformal dynamical systems including Schwarz reflections as in \cite{McM25},
        \item Basilica Julia sets or limit sets of Bers boundary groups/Schwarz reflections as in the setting of this paper.
    \end{itemize}
\end{enumerate}


Our first two theorems provide positive evidence towards  this universality phenomenon.
\begin{thmx}\label{thm:qcclassfn-ltsets}
    Let $G$ be a geometrically finite $B$-group, i.e. a geometrically finite Kleinian group with a simply connected, totally invariant component $\Delta_\infty$ of the domain of discontinuity.
    Let $\Sigma = \Delta_\infty / G$ so that $\Sigma$ is a two-dimensional orbifold with negative Euler characteristic.

    Suppose that $\Sigma$ is compact.
    Then exactly one of the following holds.
    \begin{enumerate}
        \item $G$ is quasi-Fuchsian, and the limit set $\Lambda(G)$ is quasiconformally equivalent to the round circle $\mathbb{S}^1 = J(z^2)$; or
        \item $G$ is a Bers boundary group, and the limit set $\Lambda(G)$ is quasiconformally equivalent to the fat Basilica Julia set $J(Q)$, where $Q(z) = z^2-\frac{3}{4}$.
    \end{enumerate}

    On the other hand, suppose that $\Sigma$ is not compact.
    Then exactly one of the following holds.
    \begin{enumerate}
        \item $G$ is quasi-Fuchsian, and the limit set $\Lambda(G)$ is quasiconformally equivalent to the round circle $\mathbb{S}^1 = J(z^2)$; or
        \item $G$ is a Bers boundary group, and the limit set $\Lambda(G)$ is {\bf not} quasiconformally equivalent to the Julia set of any rational map.
    \end{enumerate}
\end{thmx}

We remark that there are infinitely many quasiconformal classes of limit sets of geometrically finite Bers boundary groups for non-compact surfaces (see Figure~\ref{fig:persistantBasilica}, Theorem~\ref{thm:infKleinianRational} and \S~\ref{subsec:beyondclassC} for more details).

\begin{figure}[ht]
\captionsetup{width=0.96\linewidth}
  \centering
  \includegraphics[width=0.44\textwidth]{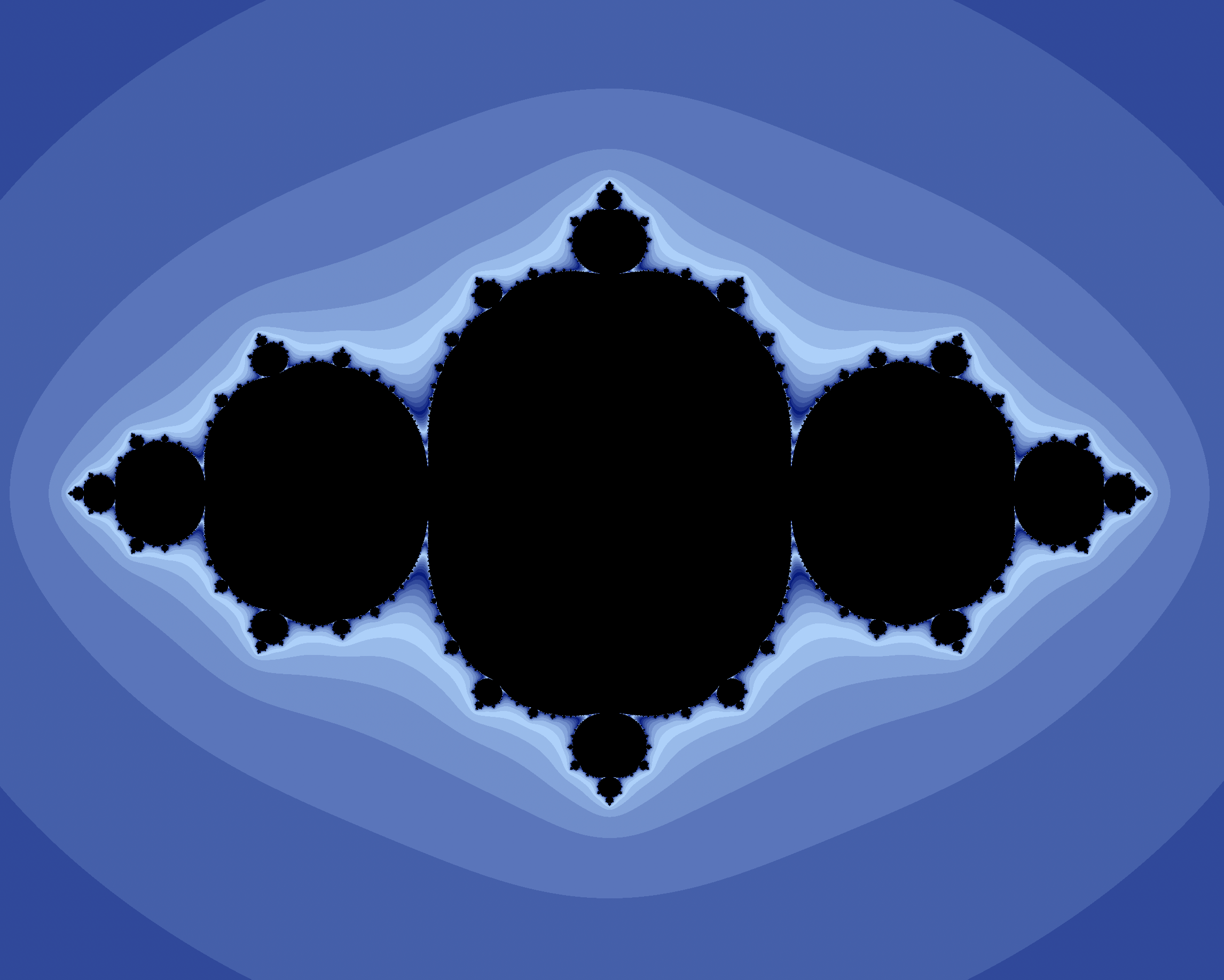}
  \includegraphics[width=0.366\textwidth, angle=90]{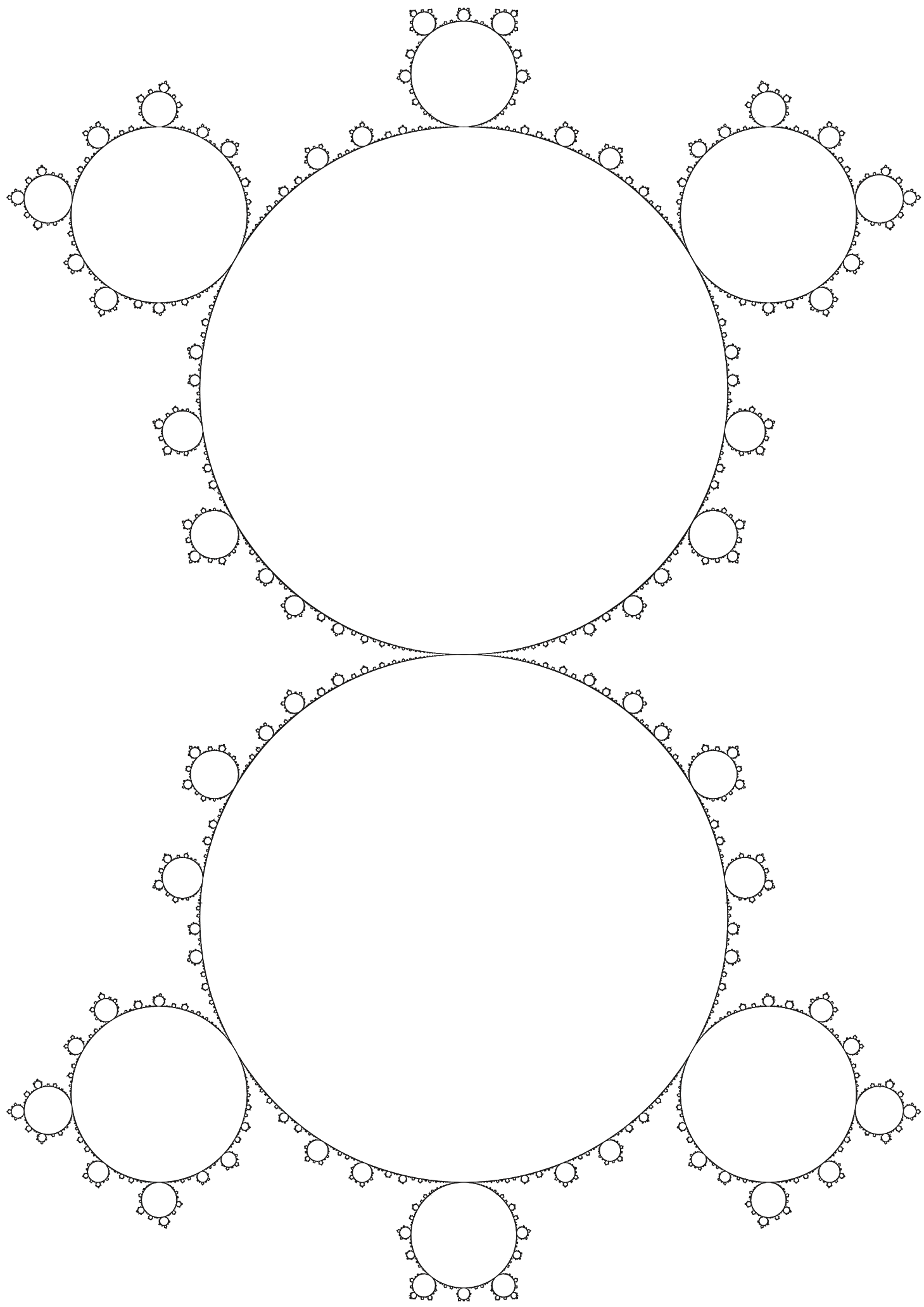}
  \caption{A fat Basilica Julia set (for the polynomial $Q(z) = z^2-\frac{3}{4}$) on the left and a Basilica limit set of a genus $2$ closed surface group on the right. These two sets are quasiconformally homeomorphic.}
  \label{fig:Basilica}
\end{figure}

More generally, a closed set $K \subseteq \C$ is called a {\em Basilica} if it is homeomorphic to $J(Q)$, where $Q(z) = z^2-\frac{3}{4}$ (see Figure~\ref{fig:Basilica}).
A Basilica $K$ is {\em fat} if
\begin{enumerate}
    \item\label{fatbasi:2intro} each bounded component of $\C - K$ is a quasi-disk; and
    \item\label{fatbasi:3intro} if any two bounded components of $\C - K$ touch, they touch tangentially.
\end{enumerate}
(See \S~\ref{subsubsec:topmodelpoly} for a precise definition of two Jordan domains touching tangentially.)
We remark that any geometrically finite Bers boundary group has a fat Basilica limit set. Conversely, any limit set of a finitely generated Kleinian group that is homeomorphic to a Basilica comes from a Bers boundary group.

Our next theorem is a universality statement for fat Basilica Julia sets.

\begin{thmx}\label{thm-basilica-stdpolymodel}
    Let $f$ be a rational map with a fat Basilica Julia set $J(f)$. Then $J(f)$ is quasiconformally equivalent to $J(Q)$ if and only if the unique non-Jordan disk Fatou component is attracting/super-attracting.

    In particular, any fat Basilica Julia set of a polynomial is quasiconformally equivalent to $J(Q)$.
\end{thmx}

\noindent (See \S~\ref{proof_three_main_thm_subsec} for proofs of Theorems~\ref{thm:qcclassfn-ltsets} and~\ref{thm-basilica-stdpolymodel}.)
\smallskip

    Let $\mathfrak{C}$ denote the class consisting of
    \begin{itemize}
        \item Kleinian groups $G$ with Basilica limit set, and
        \item rational maps $R$ with fat Basilica  Julia set,
    \end{itemize} 
    such that the induced dynamics on the ideal boundary $\partial^I U_\infty$ of the unique non-Jordan disk component $U_\infty \subseteq \widehat{\C} - K$ is hyperbolic. This is equivalent to requiring that $U_\infty/G$ is compact or $R\vert_{U_\infty}$ is (super-)attracting.
    It will be convenient to treat the two kinds of dynamical systems above on a common footing in the spirit of the Sullivan dictionary. 
    Then Theorem~\ref{thm:qcclassfn-ltsets} and Theorem~\ref{thm-basilica-stdpolymodel} show the universality of $J(Q)$ in the class $\mathfrak{C}$ (c.f. Conjecture~\ref{conj:qcuniversalityBasilica} and Conjecture~\ref{conj:confdim1univer} for universality beyond class $\mathfrak{C}$).
    
    \begin{cor}[Quasiconformal Universality]\label{cor:quasiconformaluniversality}
        Let $K_i, i\in\{1,2\}$, be either a fat Basilica limit set or a fat Basilica Julia set for a conformal dynamical system in $\mathfrak{C}$. Then $K_1, K_2$ are quasiconformally equivalent.
    \end{cor}

\begin{remark}
    We remark that Theorem~\ref{thm-basilica-stdpolymodel} extends immediately to any $p/q-$rabbit quadratic polynomial or fat $p/q-$rabbit with appropriate modifications. We focus on the case of $Q$ for concreteness, and because it is the only Julia set in this family that is quasiconformally equivalent to the limit set of a Kleinian group.

    In class $\mathfrak{C}$, it is convenient to use the quadratic polynomial $Q(z) = z^2-\frac{3}{4}$ as the model. The important feature of $Q$ that we rely on is  {\em transitivity} of the dynamics: all bounded Fatou components
    lie in the same grand orbit of $Q$. The same is true for every contact point between two bounded Fatou components. 

    Our theorems fit naturally in a broader framework of classifying dynamical fractal sets up to quasiconformal equivalence. See \eqref{eqn:reformmain}, \S~\ref{subsec:beyondclassC}, and \S~\ref{subsec:rvu} for further discussion on Basilica limit sets beyond the class $\mathfrak{C}$.
    \end{remark}

\subsection{Corollaries of quasiconformal universality}
Before discussing the background and related results, let us list some immediate consequences of Theorem~\ref{thm:qcclassfn-ltsets} and Theorem~\ref{thm-basilica-stdpolymodel}.

\subsection*{Quasiconformal equivalence of Julia set and limit set}

It was conjectured in \cite{LLMM23b} that  if a Julia set $J$ and a limit set $\Lambda$ are 
\begin{enumerate}
    \item connected,
    \item not homeomorphic to the circle or the 2-sphere,
\end{enumerate}  then they are quasiconformally nonequivalent. As an immediate corollary of Corollary~\ref{cor:quasiconformaluniversality}, we give a negative answer to the above conjecture.
\begin{cor}\label{cor-basilica}
    There exists a connected limit set of a Kleinian group, not homeomorphic to a circle or a $2-$sphere, that is quasiconformally equivalent to the Julia set of a rational map.
\end{cor}

We remark that it is proved in \cite{LN24} that a gasket Julia set and a gasket limit set can be locally quasiconformally homeomorphic. We expect that globally, a gasket Julia set is not quasiconformally homeomorphic to a gasket limit set, see \S~\ref{subsubsec:gasketJL} for more discussions.

\subsection*{Quasiconformal uniformization}
A Basilica $K$ is called {\em round} if the boundary of each bounded component of $\C - K$ is a round circle.
Since the limit set of any geometrically finite Bers boundary group is quasiconformally equivalent to a round Basilica, we immediately obtain the following quasiconformal uniformization result for the fat Basilica (c.f. \cite{LN24, Nta25} for fat gaskets).
\begin{cor}\label{cor-totgeobdd}
    Let $g$ be a polynomial with a fat Basilica Julia set $J(g)$.
    Then $J(g)$ can be quasiconformally uniformized to a round Basilica, i.e., there exists a quasiconformal map $\Phi: \C \longrightarrow \C$ so that $\Phi(J(g))$ is a round Basilica.
\end{cor}
\noindent The above corollary has been obtained independently by Dimitrios Ntalampekos using purely analytic tools.

\subsection*{Uniform quasi-symmetry groups}
It is known that the quasi-symmetry group of the Basilica Julia set $J(z^2-1)$ is big (see \cite{LM18, BF25}). The same construction also works for the fat Basilica Julia set.
A subgroup  $H \subseteq \QS(J)$ is {\em uniformly quasi-symmetric} if there exists a constant $K$ so that the quasi-symmetry constant of any element $h \in H$ is bounded by $K$.
Theorem~\ref{thm:qcclassfn-ltsets} implies that there is an abundance of  subgroups of $\QS(J)$ that are uniformly quasi-symmetric. We should point out here, for context,  that Sullivan  \cite{Sul78} proved that a group of uniformly quasiconformal automorphisms of the unit disk or the Riemann sphere $\widehat{\C}$ is quasiconformally conjugate to a group of M\"obius automorphisms.
Markovic \cite{Mar06} extended this to a 
group of uniformly quasi-symmetric homeomorphisms of the unit circle. 
\begin{cor}\label{cor-surfgpinbasilica}
    Let $\Sigma$ be a non-rigid, compact two-dimensional orbifold with negative Euler characteristic. 
    There exists a faithful representation of $\pi_1(\Sigma)$ into the quasi-symmetry group $\QS(J(Q))$, so that the image is a uniformly quasi-symmetric subgroup.
    In particular, $\QS(J(Q))$ contains infinitely many uniformly quasi-symmetric surface subgroups.
\end{cor}
Here, non-rigid means that the surface $\Sigma$ has a non-trivial Teichm{\"u}ller space; i.e., $\Sigma$ is quasiconformally deformable. This means that $\Sigma$ is not a genus $0$ surface with three marked points.

We remark that there are infinitely many non-commensurable uniformly quasi-symmetric subgroups in $\QS(J(Q))$. Indeed, consider a genus $g\geq 2$ closed surface $\Sigma_g$ with a separating curve $\mathcal{C}_g$ that divides $\Sigma_g$ into a genus $1$ and a genus $g-1$ surface with boundary. Then by pinching the curve $\mathcal{C}_g$, $(\Sigma_g, \mathcal{C}_g)$ induces a Bers boundary group, and thus by Theorem~\ref{thm:qcclassfn-ltsets}, a uniformly quasi-symmetric subgroup $H_g\subseteq \QS(J(Q))$. It is easy to verify explicitly that $H_g, H_{g'}$ are not commensurable in $\QS(J(Q))$ if $g \neq g'$.

\subsection*{Conformal removability}
A subset $K \subseteq \widehat{\C}$ is called {\em conformally removable} if \emph{any} homeomorphism $\Phi: \widehat{\C} \longrightarrow \widehat{\C}$ which is conformal on $\widehat{\C} - K$ is a M\"obius map.
Certain Julia sets and limits sets are known to be conformally removable by the work in \cite{Jon95, Kah98, JS00, LMMN, LN24a}. 
As a corollary to Theorems~\ref{thm:qcclassfn-ltsets},~\ref{thm-basilica-stdpolymodel}, we have a new list of conformally removable limit sets. 
\begin{cor}\label{cor-confremovableltset}
    Let $\Lambda$ be the limit set of a finitely generated Kleinian group that is homeomorphic to a Basilica.
    Then $\Lambda$ is conformally removable.
\end{cor}

We remark that a specific class of Basilica limit sets (coming from reflection groups) is shown to be conformally removable in \cite[Theorem 9.1]{LMMN}.

Some dynamical precursors to Corollary~\ref{cor-confremovableltset} in the context of Kleinian groups are worth mentioning here. In \cite{BCM12}, Brock-Canary-Minsky showed that if $\rho, \rho'$ are Kleinian surface groups such that their actions on the Riemann sphere are topologically  conjugate, then $\rho, \rho'$ are quasiconformally conjugate. In \cite[\S 4.5]{Mj14} it was shown that it suffices to assume that the actions of the Kleinian surface groups on their limit sets are  topologically  conjugate. The key difference with Corollary~\ref{cor-confremovableltset} lies in the fact that these theorems are dynamical in their nature, and the group action on the limit set is an  essential part of the hypothesis.

To prove the corollary,
we note that if $G$ corresponds to a compact surface, then $\Lambda(G)$ is conformally removable as it is the quasiconformal image of $J(Q)$, and we know $J(Q)$ is conformally removable (see \cite[Theorem C]{LMMN}).
If $G$ corresponds to a non-compact surface, then $\Lambda(G)$ is the image of $J(z^2-1)$ under a David homeomorphism (see Theorem~\ref{thm:davidhi}). Since $J(z^2-1)$ is a John domain (see \cite{CJY94}) and the David image of a John domain is conformally removable (see \cite[\S~2.4.]{LMMN}), the corollary follows.


\subsection{Basilica limit sets beyond class $\mathfrak{C}$}\label{subsec:beyondclassC}
Motivated by the results of \cite{McM25}, our discussion can be placed in a broader context of {\em meta-Teichm\"uller theory} for fractal sets: the classification of topological objects up to quasiconformal equivalence.

Let $\mathfrak{B}$ denote the space of all Basilicas $K \subseteq \widehat{\C}$ subject to the equivalence relation $[K_1] = [K_2]$ if there exists a quasiconformal map $\phi: \widehat{\C} \longrightarrow \widehat{\C}$ so that $\phi(K_1) = K_2$.

Each Basilica Julia set or limit set determines a point in $\mathfrak{B}$.
We say that a point in $\mathfrak{B}$ is {\em rational} or {\em Kleinian} if it arises as the Julia set of a geometrically finite rational map or the limit set of a geometrically finite Kleinian group, respectively.
Denote by $\mathfrak{U}_{\rat}, \mathfrak{U}_{\kle} \subseteq \mathfrak{B}$ the set of rational and Kleinian points, respectively.

From this perspective, Theorem~\ref{thm:qcclassfn-ltsets} implies that 
\begin{equation}\label{eqn:reformmain}
    \mathfrak{U}_{\rat}\cap \mathfrak{U}_{\kle} = \{[J(Q)]\} \subseteq \mathfrak{B},
\end{equation}
i.e., $[J(Q)]$ is the only point that is both rational and Kleinian.
On the other hand, both sets $\mathfrak{U}_{\rat}, \mathfrak{U}_{\kle} \subseteq \mathfrak{B}$ are infinite.
\begin{thmx}\label{thm:infKleinianRational}
    The set $\mathfrak{B}$ contains infinitely many rational points and infinitely many Kleinian points.
\end{thmx}

Let us mention the main ideas of the proof in the Kleinian case.
One can associate a {\em bi-colored contact tree} to a Basilica limit set: the vertices correspond to components of $\Omega(G)-\Delta_\infty$ (where $\Delta_\infty$ is the unique non-Jordan disk component of $\Omega(G)$), and there is an edge between two vertices if the corresponding components touch. The vertex is colored white if the stabilizer of the corresponding component contains a persistent (not accidental) parabolic, and is colored black otherwise (see Figure~\ref{fig:persistantBasilica}).
We show that any quasi-symmetry between Basilica limit sets induces a color preserving isomorphism between the associated bi-colored contact trees, concluding that the  bi-colored contact tree is a quasiconformal invariant. We prove the infinitude of Kleinian points in $\mathfrak{B}$ by constructing infinitely many Kleinian Basilica limit sets with non-isomorphic bi-colored contact trees. The case of rational points is similar. 
We refer the reader to \S~\ref{subsubsec:bicolct} and \S~\ref{subsubsec:bicolctrat} for further discussion and a complete proof of Theorem~\ref{thm:infKleinianRational}.

We propose in Conjecture~\ref{conj:comqcinv} that the bi-colored contact tree is a complete quasiconformal invariant.
This is a special case of Conjecture~\ref{conj:qcuniversalityBasilica} below.

\begin{figure}[ht]
\captionsetup{width=0.96\linewidth}
  \centering
  \includegraphics[width=0.55\textwidth]{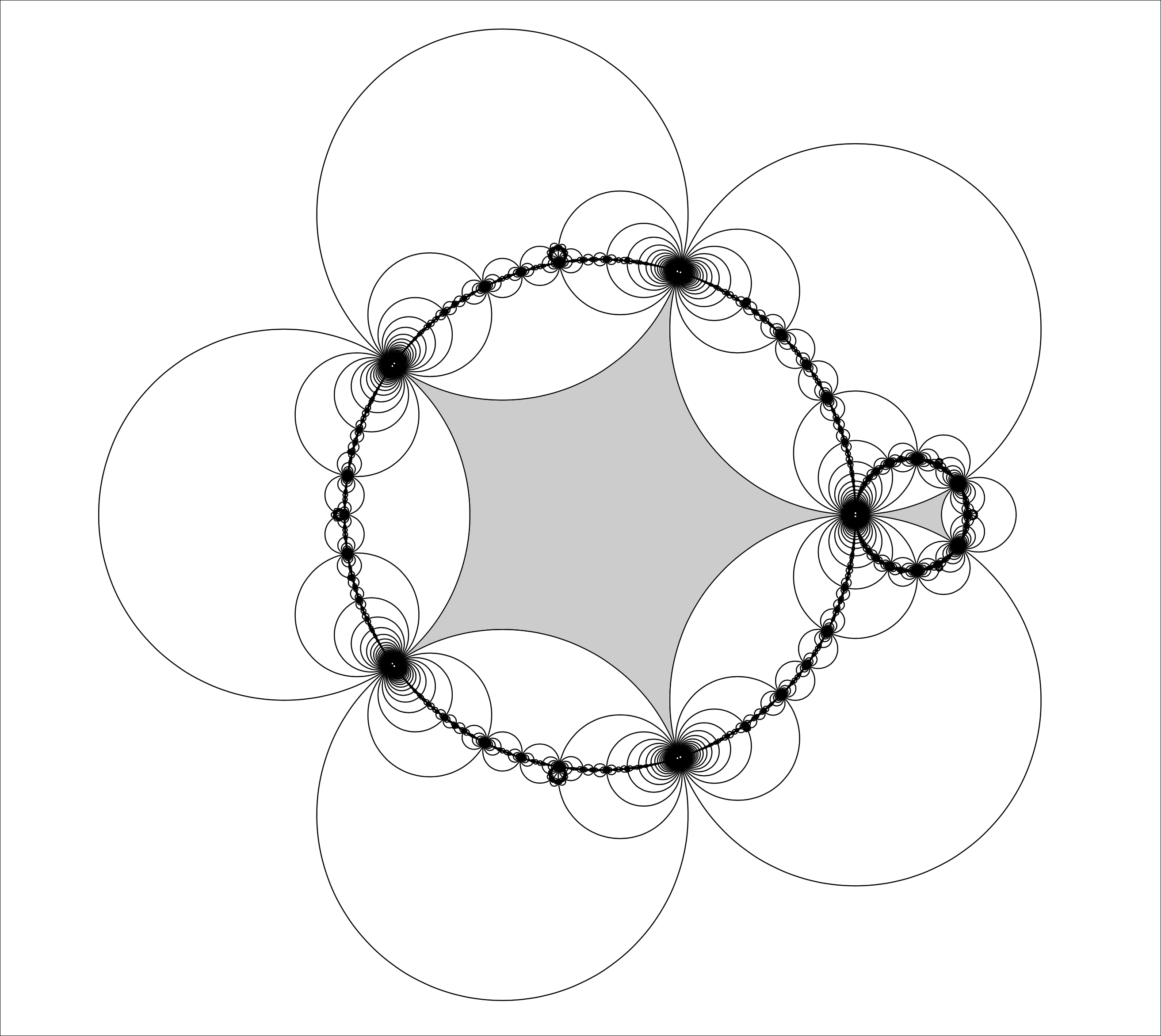}
  \caption{A Basilica limit set with persistently parabolic elements. Note that as we zoom in at a persistently parabolic point, the limit set converges to a line in the Hausdorff topology. On the other hand, at any non-contact point of the fat Basilica limit set in Figure~\ref{fig:Basilica}, we can choose a specific sequence of zoom-ins so that the limit set converges to a closed set that is not topologically a line. This shows that these two Basilica limit sets are quasiconformally different (see Proposition~\ref{prop:qsinv} for more details).}
  \label{fig:persistantBasilica}
\end{figure}

\subsection*{Quasiconformal universality for Basilica limit sets}
Let $K$ be a Basilica Julia set or a limit set, and let $U$ be a component of $\widehat{\C}- K$. 
    For a map, we say that a point $x \in \partial^I U$ on the ideal boundary of $U$ is {\em parabolic} if it is in the grand orbit of a parabolic point under the induced map on the ideal boundaries of the Fatou components.
    Similarly, for a group, $x\in \partial^I U$ is {\em parabolic} if it is the fixed point of a parabolic element for the induced Fuchsian group $G_U=\phi\ \mathrm{Stab}_G(U)\ \phi^{-1}$, where $\mathrm{Stab}_G(U)$ is the stabilizer of $U$ in $G$, and $\phi:U\to\D$ is the Riemann map of $U$.
    We say that a point on the ideal boundary is {\em hyperbolic} otherwise.
     
    Let $K_i,\ i=1, 2$, be either a Basilica Julia set or a limit set. We say that a homeomorphism $\Phi: \widehat{\C} \longrightarrow \widehat{\C}$ with $\Phi(K_1) = K_2$ is {\em type-preserving} if for each component $U$ of $\widehat{\C} - K_1$, the induced map on the ideal boundary $\partial^I U$ sends parabolic points {\em onto} parabolic points (i.e., it induces a bijection between parabolic points). 
    It is {\em weakly type-preserving} if it sends parabolic points to parabolic points (but hyperbolics can go to parabolics). 

    Note that for a fat Basilica Julia set, the contact point between two bounded Fatou components is parabolic on the ideal boundary of each component. The same is true for any Basilica limit set. 
    Thus, Corollary~\ref{cor:quasiconformaluniversality} is a special case of the following conjecture.
    \begin{conj}\label{conj:qcuniversalityBasilica}
        Let $K_i, i \in\{1,2\}$, be the Basilica limit set of a geometrically finite conformal dynamical system (such as the limit set of a geometrically finite Kleinian group or the Julia set of a geometric finite rational map).
    
        Suppose that $K_1, K_2$ are homeomorphic via some type-preserving map. Then $K_1$ is quasi-symmetric to $K_2$.
    \end{conj}

\noindent (See \cite{McS98} for a general definition of conformal dynamical systems.)

\begin{rmk}
    We believe that with appropriate combinatorial modifications, the methods introduced in this paper are sufficient to prove the conjecture. See \S~\ref{subsec:zoo} for some additional cases of Conjecture~\ref{conj:qcuniversalityBasilica}.
\end{rmk}

\subsection*{David hierarchy}
Let $[K_1], [K_2] \in \mathfrak{B}$. We write $[K_1] \succeq [K_2]$ if there exists a David homeomorphism $\phi: \widehat{\C} \longrightarrow \widehat{\C}$ (also known as homeomorphisms of \emph{exponentially integrable distortion}, see \cite[Chapter~20]{AIM09}) so that $\phi(K_1) = K_2$.
We expect that the rational and Kleinian points in $\mathfrak{B}$ admit a natural David hierarchy, and we show that the standard Basilica Julia set $J(z^2-1)$ is the archbasilica in this David hierarchy in the  sense that it dominates all others.
\begin{thmx}\label{thm:davidhi}
    Let $Q_{\pc}(z) = z^2-1$ be the quadratic \pcf polynomial with Basilica Julia set. 
    Let $[K] \in \mathfrak{B}$ be Kleinian or rational.
    Then 
    $$
    [J(Q_{\pc})] \succeq [K].
    $$
    Moreover, $[K] = [J(Q_{\pc})]$ if $K$ is the Basilica Julia set of a geometrically finite rational map without parabolic cycles.
\end{thmx}
\noindent See \S~\ref{sec:davidHierarchy} for a proof of Theorem~\ref{thm:davidhi}.
The above theorem also motivates the following conjecture for general Basilica Julia sets and limit sets.
\begin{conj}[David hierarchy]\label{conj:davidhier}
    Let $K_i,\ i \in\{1,2\}$, be the Basilica limit set of a geometrically finite conformal dynamical system.
    Suppose that there is a weakly type-preserving map from $K_1$ to $K_2$. Then there exists a David homeomorphism $\phi: \widehat{\C} \longrightarrow \widehat{\C}$ with $\phi(K_1)=K_2$.
\end{conj}
\begin{remark}
1)    We remark that unlike quasiconformal maps, the composition of two David maps and the inverse of a David map are not necessarily David. Thus, compared to Conjecture~\ref{conj:qcuniversalityBasilica}, obtaining a complete picture of the David hierarchy is fraught with additional technical challenges (see Remark~\ref{david_hierarchy_rem} for more details). 
    We also point out that Conjecture~\ref{conj:davidhier} would imply that $\succeq$ is a partial ordering on $\mathfrak{U}_{\rat} \cup \mathfrak{U}_{\kle}~\subseteq~\mathfrak{B}$.
\smallskip

\noindent 2) We point out that the Basilica Julia set of a geometrically finite rational map with a parabolic cycle (of multiplicity $3$) can also be quasiconformally equivalent to $J(Q_{\pc})$. 
    
\end{remark}

\subsection{A zoo of Basilica limit sets}\label{subsec:zoo}
We now illustrate some additional limit sets of conformal dynamical systems that fit into Conjecture~\ref{conj:qcuniversalityBasilica}.
\subsection*{Schwarz reflection Basilicas}
Schwarz reflections provide a class of anti-conformal dynamical systems which often arise as matings of antiholomorphic polynomials and reflection groups. 
They have been extensively explored in  recent years (see \cite{LM25a,LM25b} for surveys of this mating phenomenon both in the holomorphic and antiholomorphic worlds).

The Schwarz reflection of the deltoid is the simplest example of this mating framework: it combines the antiholomorphic polynomial $\bar z^2$ with the ideal triangle reflection group (see \cite{LLMM23}).
More generally, for each $d \geq 2$, matings of $\bar z^d$ and ideal $(d+1)-$gon reflection groups can be realized as Schwarz reflections. These Schwarz reflections arise from univalent restrictions of the following family of degree $d+1$ rational maps:
$$ 
\Sigma_d^* := \left\{ f(z)= z+\frac{a_1}{z} + \cdots -\frac{1}{d z^d} : f\vert_{\widehat{\C}-\overline{\D}} \textrm{ is univalent}\right\}.
$$
Further, there exists a canonical bijection between $\Sigma_d^*$ and the Bers compactification of the Teichm{\"u}ller space of ideal $(d+1)-$gon reflection groups, such that the Schwarz reflection coming from $f\in\Sigma_d^*$ is the mating of $\overline{z}^d$ with the corresponding reflection group (see \cite{LMM22}).  Abusing notation, we will identify the space $\Sigma_d^*$ of rational maps with the corresponding family of Schwarz reflection maps. With this identification, $\Sigma_d^*$ can be regarded as a Bers compactification of the Teichm{\"u}ller space of ideal $(d+1)-$gon reflection groups.
Let us also denote $\Sigma^* := \bigsqcup_{d\geq 2} \Sigma_d^*$.

When $S \in \Int{\Sigma^*}$, \cite[Theorem 1.2]{McM25} implies that the limit set $\Lambda(S)$ is quasiconformally equivalent to the cauliflower Julia set $J(z^2+\frac14)$.
When $S \in \partial \Sigma^*=\bigsqcup_{d\geq 2} \partial\Sigma_d^*$, the limit set is a Basilica. Let us use $\mathfrak{U}_{\partial \Sigma^*} \subseteq \mathfrak{B}$ to denote the set of quasiconformal classes of Basilicas arising as the limit sets of Schwarz reflections in $\partial\Sigma^*$.

Our theory applies to the Schwarz reflection setting, and we have an analog of Theorem~\ref{thm:qcclassfn-ltsets} and Theorem~\ref{thm-basilica-stdpolymodel} for Schwarz reflections. To model the limit sets of certain Schwarz reflections on $\partial\Sigma^*$, we consider the polynomial
$$
R(z) = az^4+\frac{1-4a}{3}z^3+\frac{2+a}{3}.
$$ 
Here, $a=\frac{1}{12}+\frac{i\sqrt{2}}{24}$ is chosen so that the critical point $1-\frac{1}{4a}$ is mapped to the parabolic fixed point $1$.
Its Julia set is a {\em cuspidal Basilica}, see Figure~\ref{fig:CuspidalBasilica}.

\begin{thmx}\label{thm:schwarz}
The set $\mathfrak{U}_{\partial \Sigma^*} \subseteq \mathfrak{B}$ is infinite.
Moreover,
    \begin{align*}
    \mathfrak{U}_{\partial \Sigma^*} \cap \mathfrak{U}_{\rat} &= \{[J(R)]\}; \text{ and }\\
    \mathfrak{U}_{\partial \Sigma^*} \cap \mathfrak{U}_{\kle} &= \emptyset.
    \end{align*}
\end{thmx}

\noindent See Appendix~\ref{sec:schwarz} for a proof of Theorem~\ref{thm:schwarz}. We point out that a key ingredient in the proof is a quasiconformal universality result stated in Proposition~\ref{prop:Rmodel}.

\begin{figure}[ht]
\captionsetup{width=0.96\linewidth}
  \centering
  \includegraphics[width=0.545\textwidth]{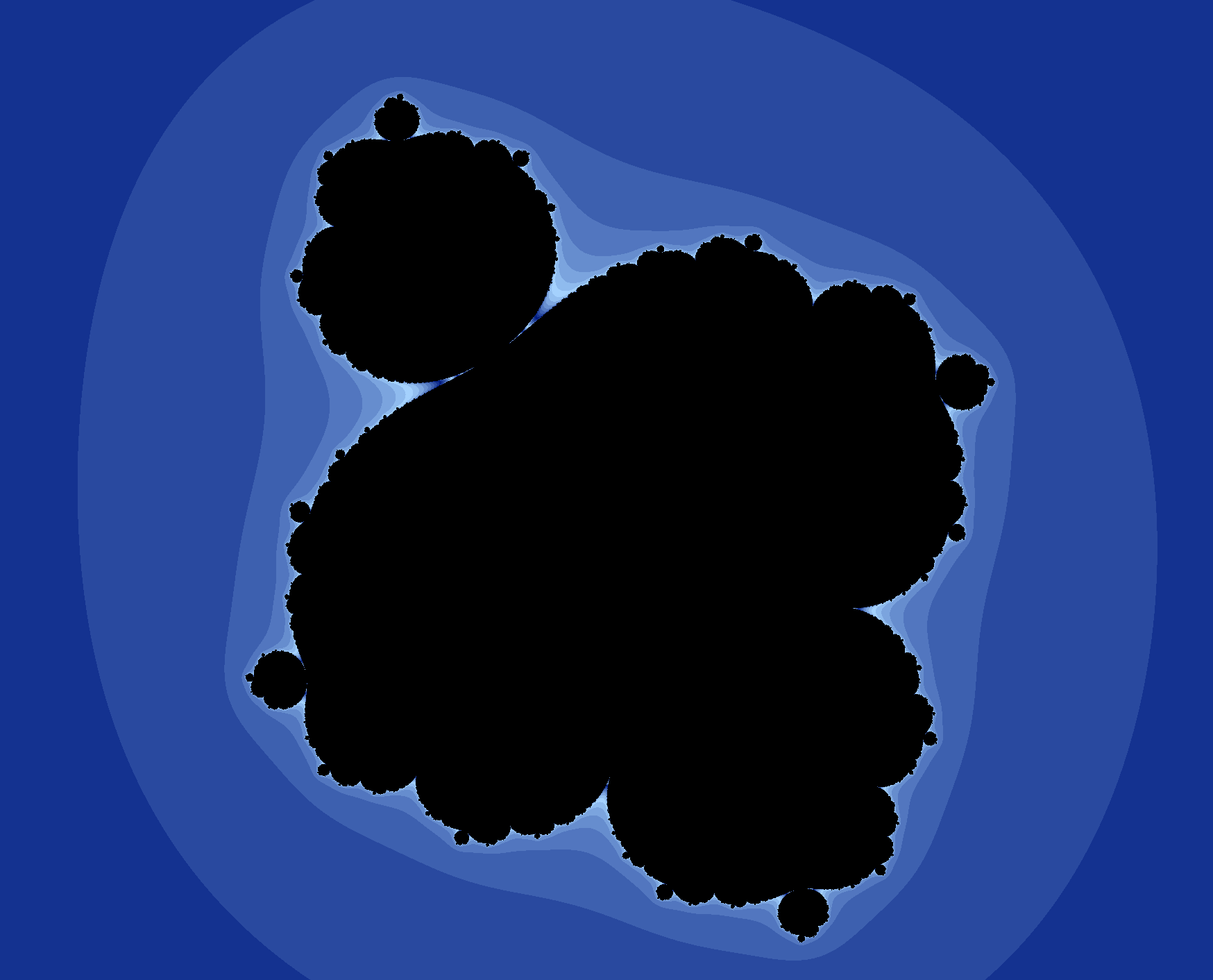}
\includegraphics[width=0.44\textwidth, angle=90]{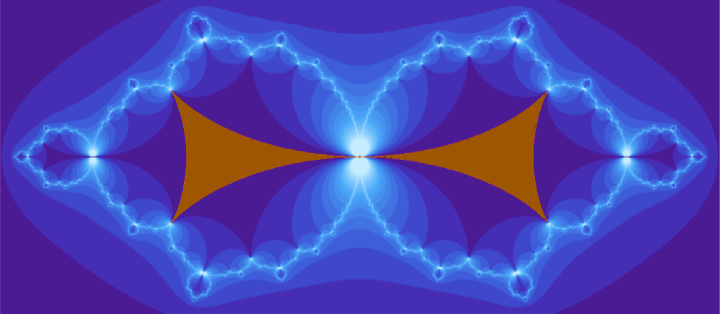}
  \caption{Left: The cuspidal Basilica Julia set for the polynomial $R(z)$ is shown. The large black bounded Fatou component has the parabolic fixed point $1$ and the simple critical point $1-\frac{1}{4a}$ on its boundary. Right: The limit set of a Schwarz reflection in $\Sigma_3^*$ arising from the rational map $f(z)=z+\frac{2}{3z}-\frac{1}{3z^3}$ is displayed. These two sets are quasiconformally homeomorphic.}
  \label{fig:CuspidalBasilica}
\end{figure}

\begin{remark}
    We remark that any Schwarz reflection $S \in  \partial \Sigma_d^*$ is a mating of $\bar z^d$ with a Bers boundary group of an ideal polygon reflection group, which is necessarily geometrically finite in the reflection setting. 

There are other constructions of Schwarz reflections (or their holomorphic counterparts, B-involutions) whose limit sets are Basilicas. For example, one may consider the mating of a Fuchsian group with a geometrically finite polynomial (or a geometrically finite rational map with a connected, simply connected parabolic basin) having a Basilica limit set (cf. \cite{LMM24,LLM24}), or the mating of a parabolic Blaschke product with a Bers boundary group \cite{LMMN}. In fact, using quasiconformal surgery, one can show that both  $\mathfrak{U}_{\textrm{anti-rat}}$ (Basilica Julia sets of geometrically finite anti-rational maps) and $\mathfrak{U}_{\textrm{reflect}}$ (Basilica limit sets of Kleinian reflection groups) naturally embed in the space $\mathfrak{U}_{\sch}$ of {\em Schwarzian Basilicas}. However, Theorem~\ref{thm:schwarz} shows that Schwarzian Basilicas arising from $\partial\Sigma^*$ only intersects $\mathfrak{U}_{\rat}$ in a singleton, different from $\left[J(Q)\right]$.
    
    In the general setting, Conjecture~\ref{conj:qcuniversalityBasilica} provides a conjectural quasiconformal classification of all such Basilicas. However,     it is more subtle to state the type-preserving condition (of Conjecture~\ref{conj:qcuniversalityBasilica}) explicitly in this case. 
\end{remark}

\subsection*{Cubic polynomials}
Let $\Per_1(0) := \{z^3+\frac{3a}{2}z^2:a\in \C\}$ be Milnor's $\Per_1(0)$ curve, and let $\mathcal{H} \subseteq \Per_1(0)$ be the {\em main hyperbolic component}, i.e., the hyperbolic component that contains $z^3$ (see Figure~\ref{fig:cubicpara}, cf. \cite{Mil92}).
Note that in this family, $0$ is a fixed critical point and we have one free critical point $-a$.
The topology of the boundary $\partial \mathcal{H}$ is studied extensively in \cite{Roe07}, where it is shown that $\partial \mathcal{H}$ is a Jordan curve.
Any geometrically finite polynomial $f_a \in \partial \mathcal{H}$ has a Basilica Julia set, and is of exactly one of the following two types (see Figure~\ref{fig:cubicBasilica}):
\begin{itemize}[leftmargin=8mm]
    \item (parabolic type) the free critical point $-a$ is in the parabolic basin, in which case $a \in \partial \mathcal{H}$ corresponds to a cusp;
    \item (Julia type) the free critical point $-a$ is on the Julia set.
\end{itemize}
These geometrically finite parameters on $\partial \mathcal{H}$ are the landing points of rational internal rays in $\mathcal{H}$ with angle periodic or strictly pre-periodic under the map $\sigma_2(t) = 2t$ (see \cite{Roe07} for more details).
Our methods yield a classification of the geometrically finite polynomials on $\partial\cH$.
\begin{thmx}\label{thm:cubicclass}
    Let $f_a$ be a geometrically finite polynomial on $\partial \mathcal{H}$. Then exactly one of the following holds.
    \begin{enumerate}[leftmargin=8mm]
        \item $f_a$ is of parabolic type, and $[J(f_a)] = [J(z^3+2iz^2)]$.
        \item $f_a$ is of Julia type, and $[J(f_a)] = [J(z^2-1)]$.
    \end{enumerate}
\end{thmx}
\noindent See Appendix~\ref{sec:cubic_poly} for a proof of Theorem~\ref{thm:cubicclass}.

\begin{figure}[ht]
\captionsetup{width=0.96\linewidth}
  \centering
  \includegraphics[width=0.471\textwidth]{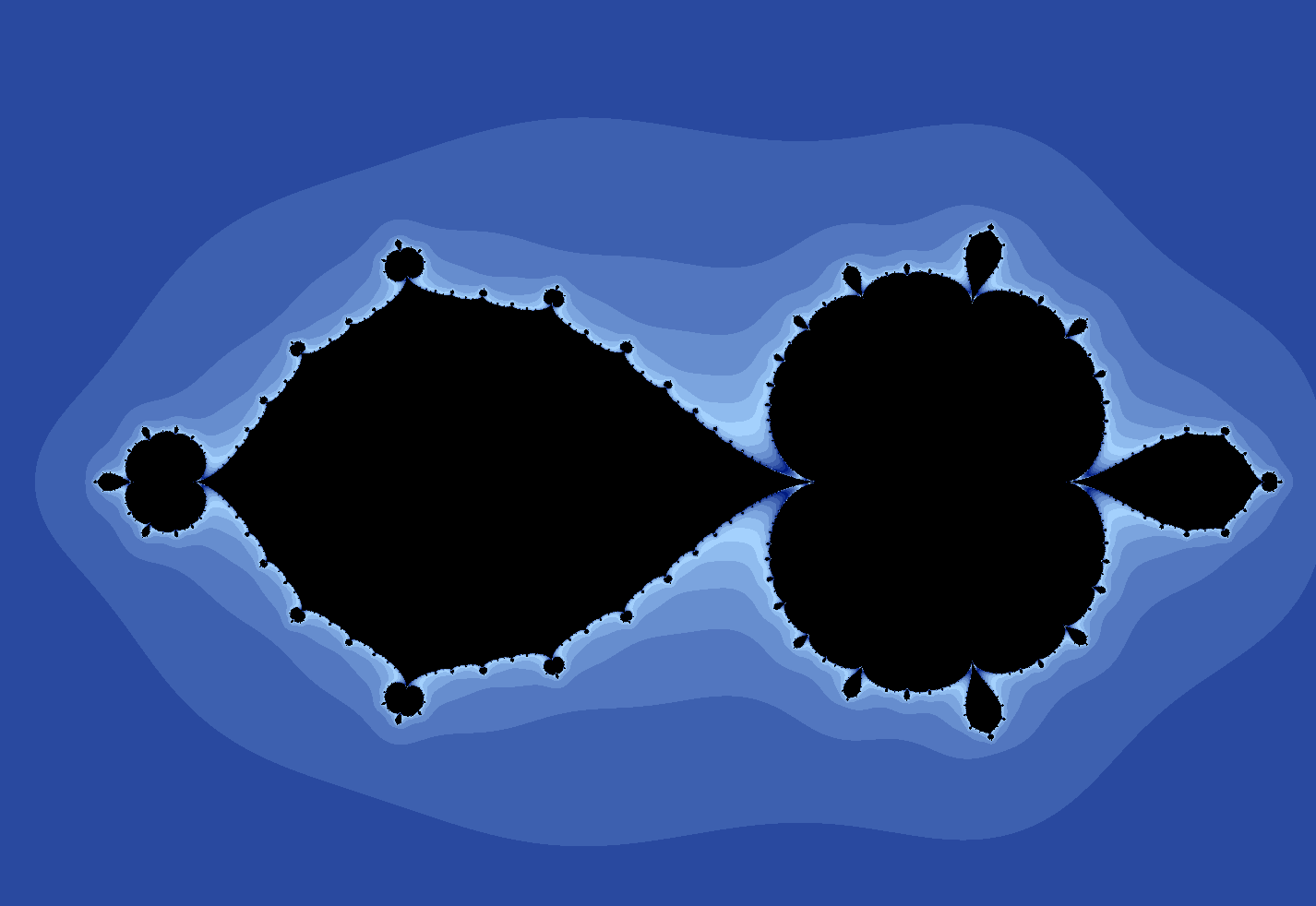}
  \includegraphics[width=0.44\textwidth]{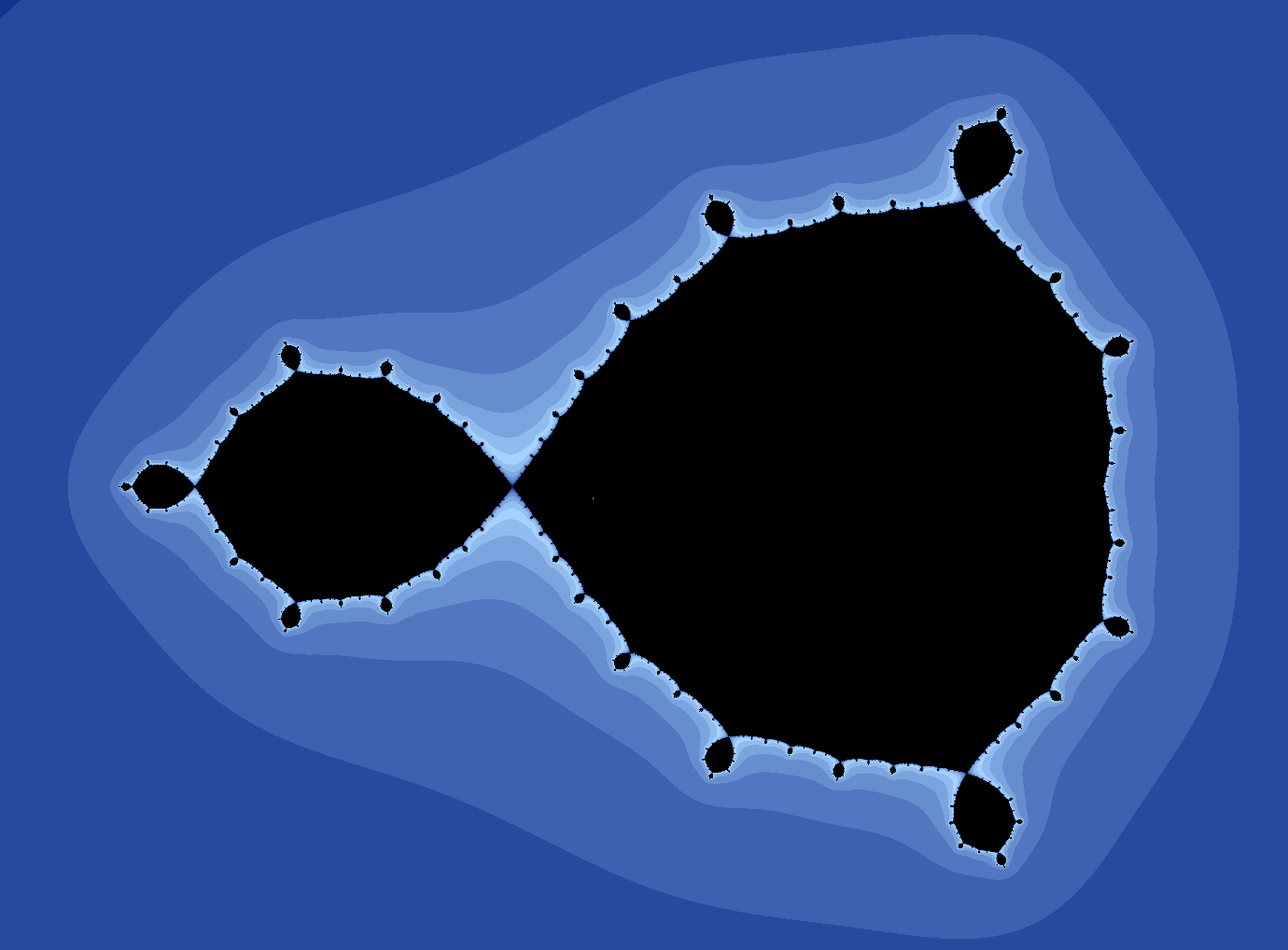}
  \caption{Two cubic Basilica Julia sets on the boundary of the main hyperbolic component $\mathcal{H}$ of Milnor's $\Per_1(0)$ curve.  The left polynomial $z^3+2iz^2$ has a parabolic fixed point and hence a critical point in the parabolic basin. The right polynomial $z^3+\frac{3}{2}z^2$ has a strictly pre-periodic Julia critical point as a contact point, and its Julia set is quasiconformally equivalent to the standard Basilica $J(Q_{\pc})$. Any geometrically finite polynomial on $\partial\cH$ has a Basilica Julia set, and is quasiconformal to one of these two models.}
  \label{fig:cubicBasilica}
\end{figure}

\begin{figure}[ht]
\captionsetup{width=0.96\linewidth}
  \centering
  \includegraphics[width=0.45\textwidth, angle=90]{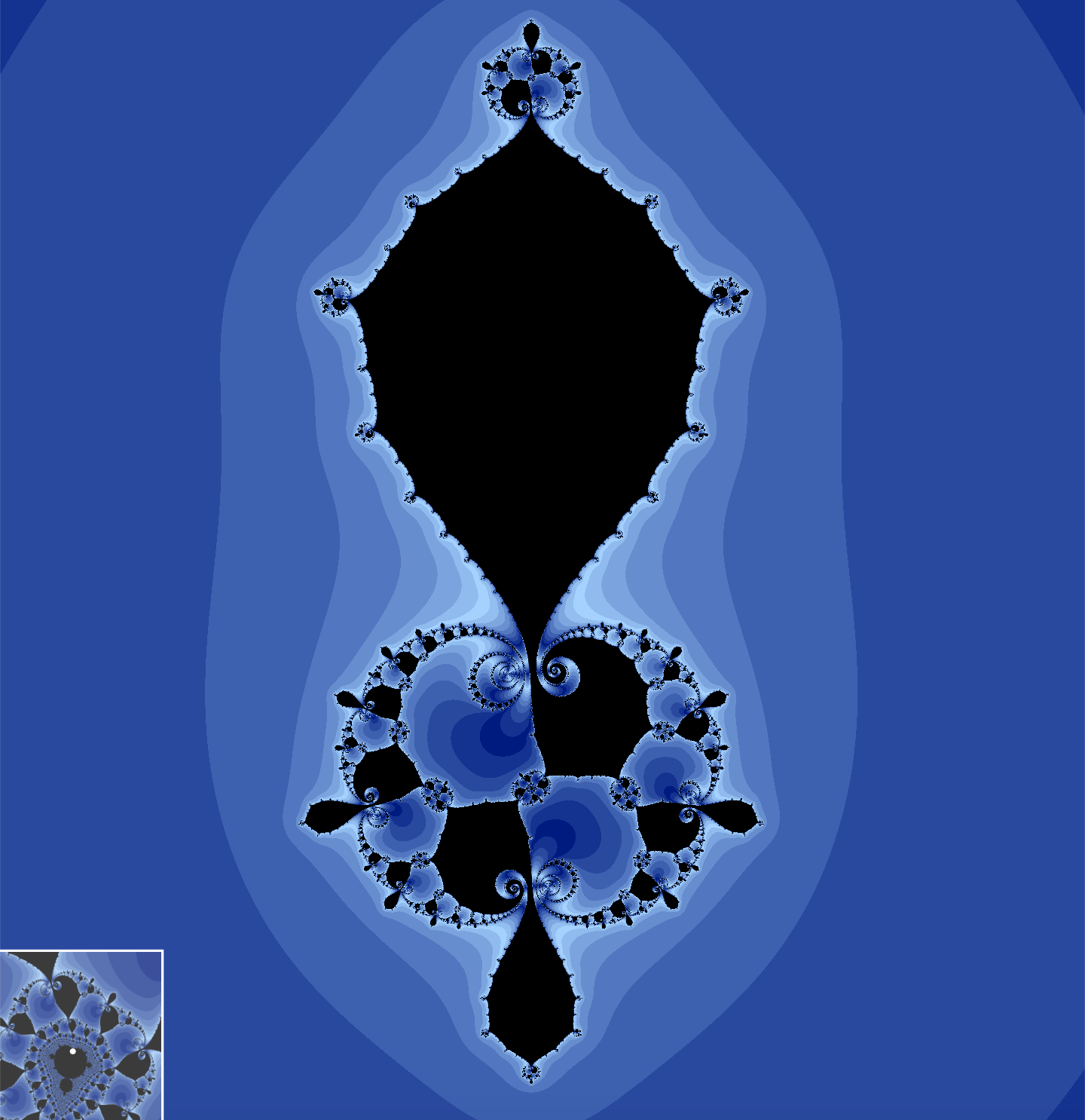}
  \includegraphics[width=0.45\textwidth, angle=90]{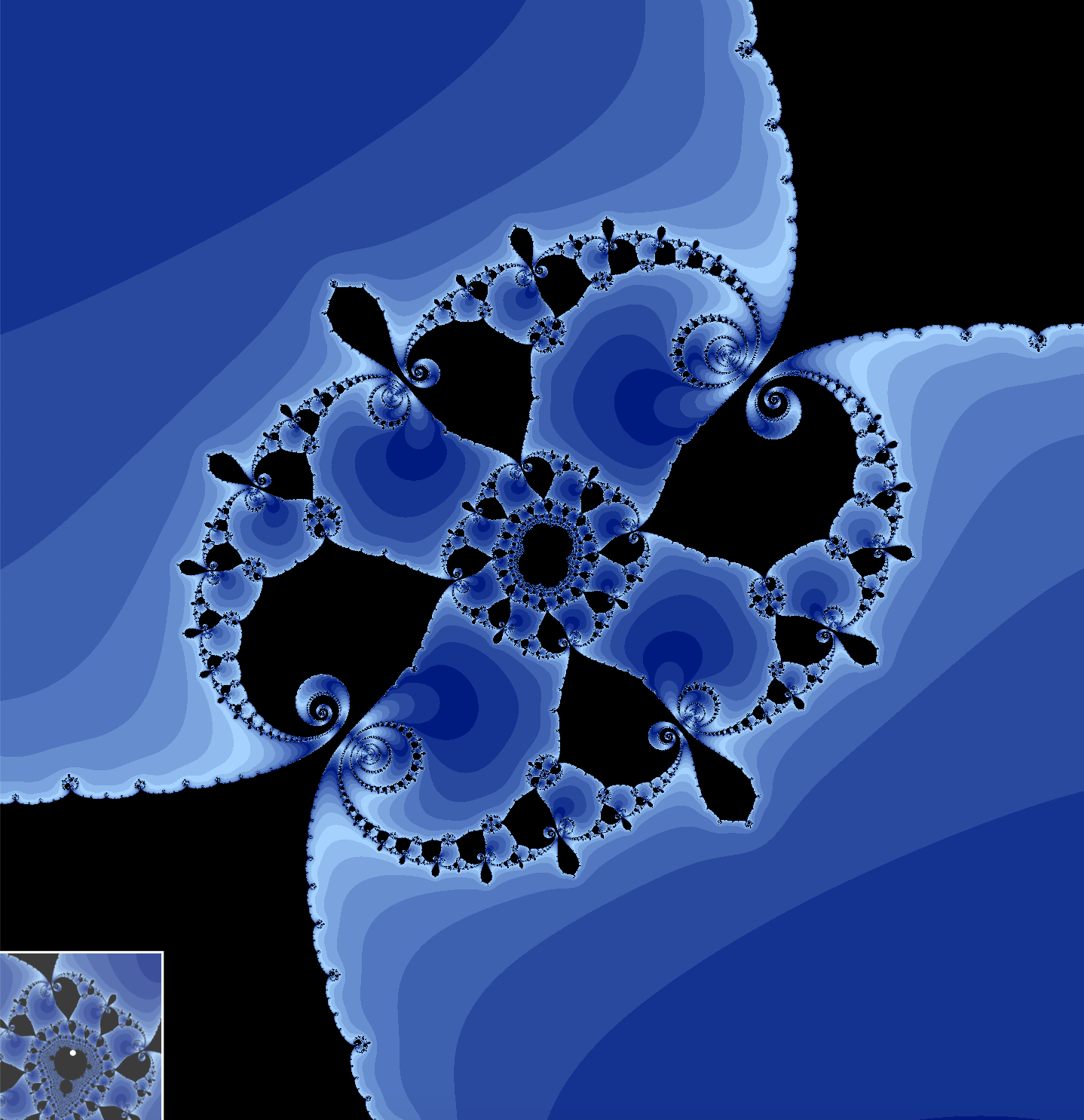}
  \caption{A geometrically finite Basilica Julia set on $\partial \mathcal{H}$ and its magnification (zoom) on the right. It follows from Theorem~\ref{thm:cubicclass} that it is quasiconformally equivalent to $J(z^3+2iz^2)$ in Figure~\ref{fig:cubicBasilica} left. The quasiconformal homeomorphism sends the small cauliflower Fatou component in the center of the zoomed figure to the most prominent cauliflower Fatou component in Figure ~\ref{fig:cubicBasilica} left. This illustrates the flexibility of the quasiconformal map obtained in Theorem~\ref{thm:cubicclass}.}
  \label{fig:wildcubicBasilica}
\end{figure}

\begin{figure}[ht]
\captionsetup{width=0.96\linewidth}
  \centering
  \includegraphics[width=0.45\textwidth]{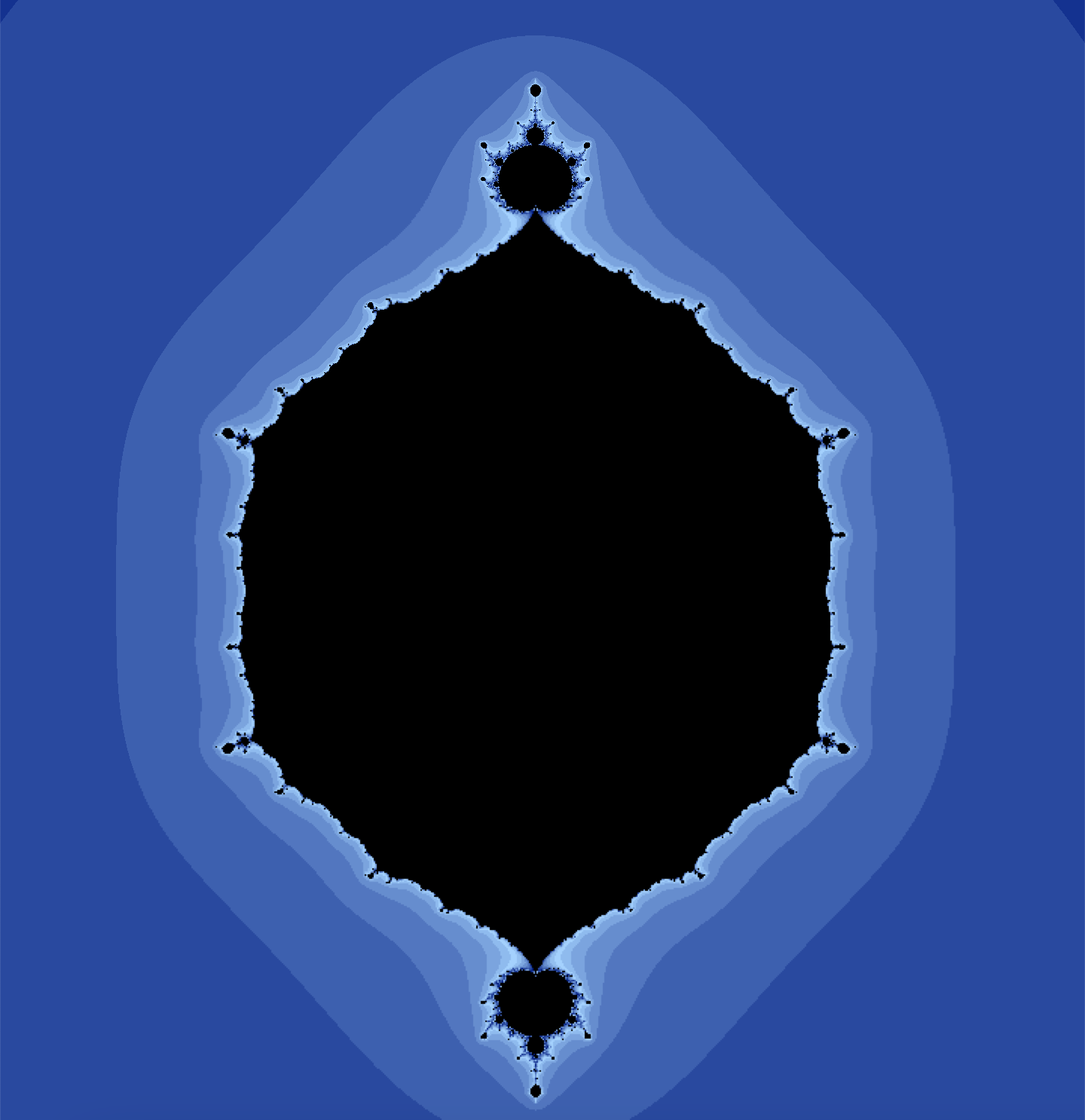}
  \includegraphics[width=0.45\textwidth]{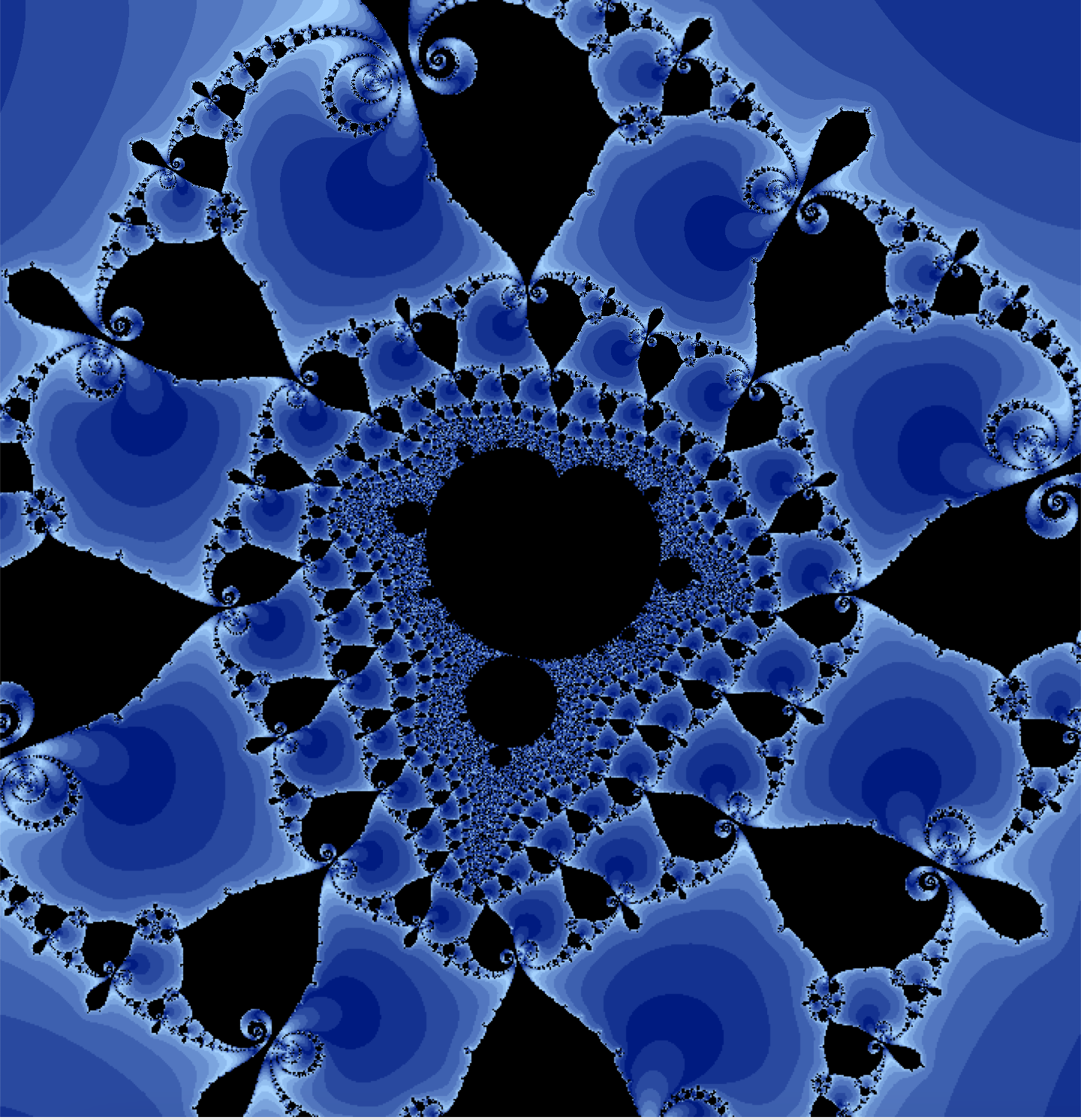}
  \caption{The main hyperbolic component $\mathcal{H}\subset \Per_1(0)$ and a magnification (zoom) near its boundary. The root of the most prominent Mandelbrot set is the geometrically finite polynomial for Figure~\ref{fig:wildcubicBasilica}.}
  \label{fig:cubicpara}
\end{figure}

\subsection*{Beyond geometric finiteness}
Many geometrically infinite rational maps have Basilica Julia sets; e.g., any quadratic polynomial of the form $z^2+e^{2\pi i \theta}z$ where $\theta$ is an irrational number of bounded type (see Figure~\ref{fig:mainMol}).
It is easy to see by zooming in at the critical value that its Julia set is not quasiconformally equivalent to Basilica Julia sets of geometrically finite maps. From renormalization theory, we expect the tail of the continued fraction expansion of the rotation number $\theta$ to play a role in identifying the quasiconformal classes.
Similarly, any polynomial on the boundary of the main hyperbolic component of $\Per_1(0)$ has a Basilica Julia set.
\begin{question}
    Can one classify the Basilica Julia sets on the boundary of the main hyperbolic component of the Mandelbrot set or of $\Per_1(0)$?
\end{question}

\subsection*{Basilicas in moduli space}
Basilicas also appear naturally in various bifurcation loci for rational maps. The following conjecture is equivalent to the Molecule conjecture in \cite{DLS20} (see Figure~\ref{fig:mainMol}).
\begin{conj}
    Let $\Mol$ be the main molecule in the quadratic family $z^2+cz$. Then $[\Mol] = [J(Q)]$.
\end{conj}

\subsection*{Basilica as a welding?}
In \cite{McM25}, the welding of the Minkowski ? function is used as the universal object for the Cauliflower class $[J(z^2+\frac14)]$. The ? function can be defined purely in number theoretic terms. It will be interesting to know whether there is a function that plays an analogous role for the universality class $[J(z^2-\frac34)]$.

\begin{figure}[ht]
\captionsetup{width=0.96\linewidth}
  \centering
  \includegraphics[width=0.45\textwidth]{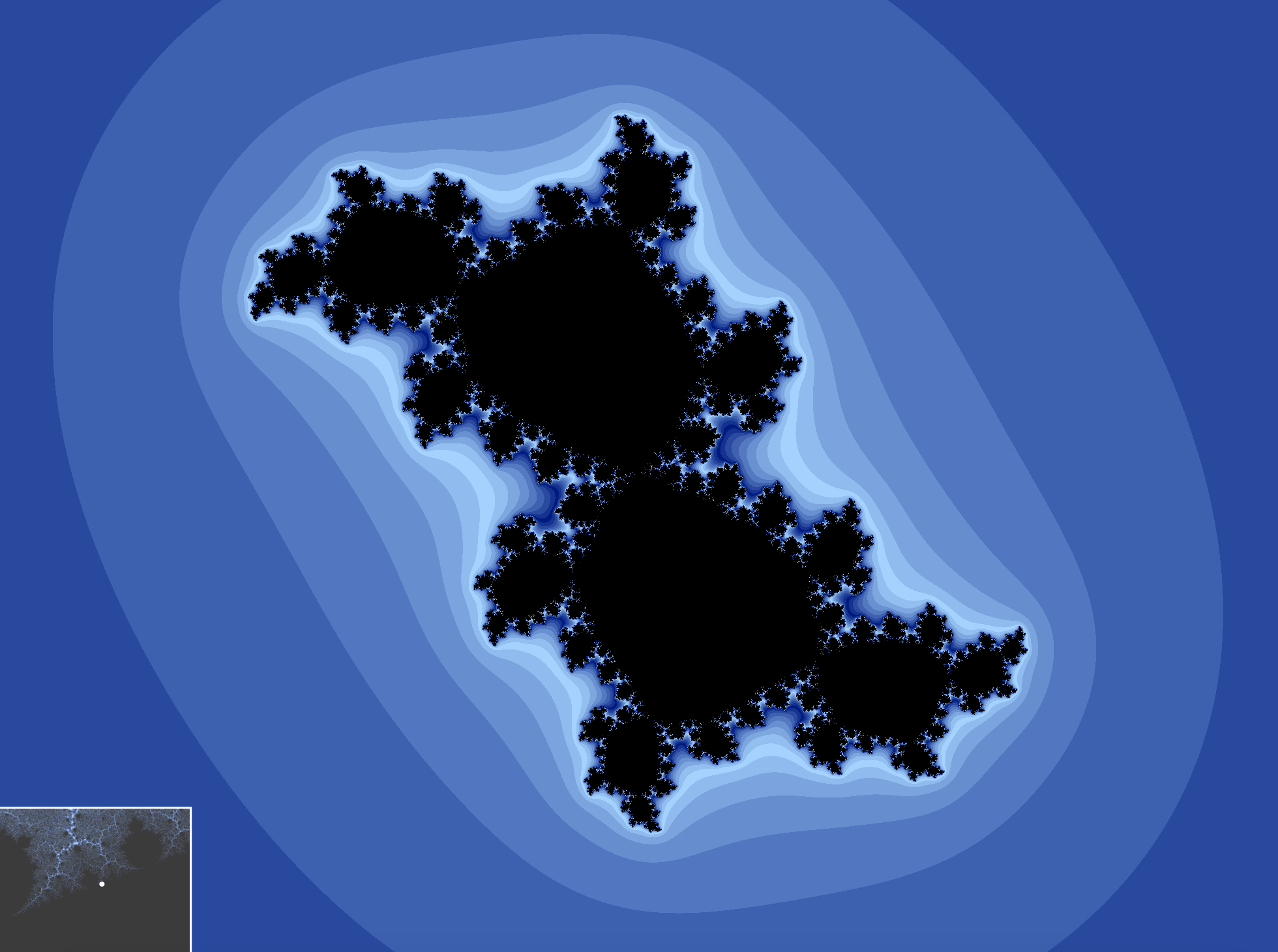}
  \includegraphics[width=0.45\textwidth]{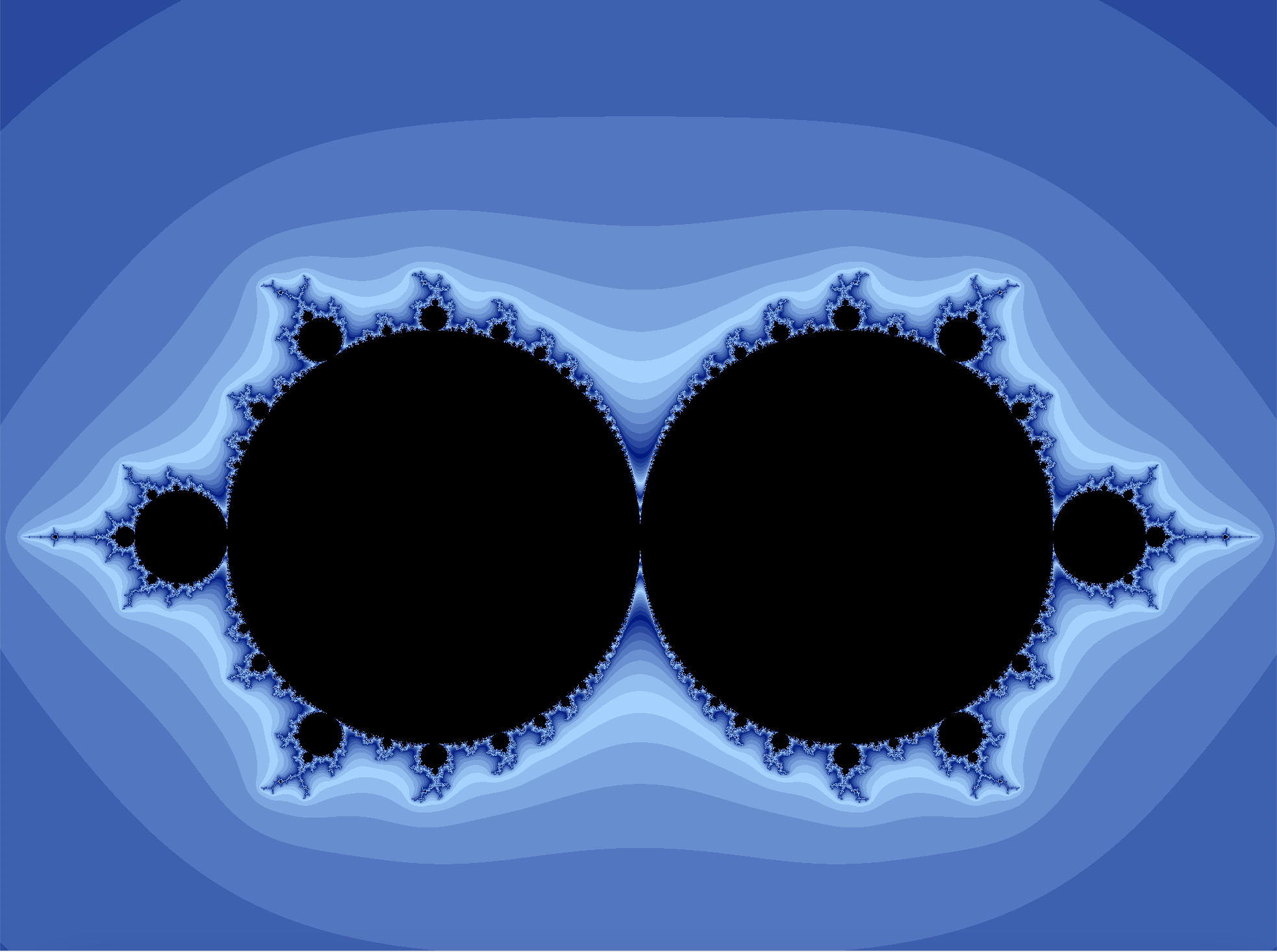}
  \caption{A Basilica Julia set of a Siegel polynomial on the left. The bifurcation locus of the quadratic family $z^2+cz$ on the right. The main molecule is the closure of the union of the hyperbolic components obtained via finite sequences of bifurcations from the main hyperbolic components containing $z^2$ and $z^2+2z$.}
  \label{fig:mainMol}
\end{figure}

\subsection{Rigidity vs Universality}\label{subsec:rvu}
Motivated by our universality results, we expect that there are two different regimes for quasi-symmetries between limit sets or Julia~sets.

Rigidity of quasi-symmetries of a Sierpi\'nski carpet are extensively studied in \cite{BKM09, BM13, Mer14, BLM16}.
Let $K_i, i \in\{1,2\}$, be either the carpet Julia set of a \pcf rational map or the carpet limit set of a convex co-compact Kleinian group. If $K_1$ and $K_2$ are quasi-symmetrically equivalent, then they are in fact M\"obius equivalent: see \cite[Theorem 1.1]{Mer14} when both $K_i$ are limits sets, and \cite[Corollary 1.3, Theorem 1.4]{BLM16} when one of the $K_i$ is a Julia set.
M\"obius equivalence implies that the corresponding dynamics must be {\em commensurable} (see \cite[Theorem B]{LP97}). 
In particular, no carpet Julia set is quasiconformally equivalent to a limit set (see \cite[Corollary 1.3]{BLM16}), and any quasi-symmetry between carpet sets necessarily arises from the underlying dynamics.

On the other hand, when the Julia sets or limit sets are Jordan curves, quasi-symmetries are very rich.
This is well known for quasi-circles since the work of Ahlfors and Bers. Recently, McMullen \cite{McM25} showed that the welding of the Minkowski $?-$map, or equivalently, the cauliflower Julia set $J(z^2+\frac14)$, is universal (in an appropriate sense stated more generally in the following conjecture).
Our main results are inspired, in part, by this, and give first examples of universality phenomenon beyond Jordan curves.

Our theorems and methods motivate the following conjecture for general conformal dimension $1$ limit sets.
\begin{conj}[Quasiconformal Universality]\label{conj:confdim1univer}
    Let $K_i, i \in\{1,2\}$, be the limit set of a hyperbolic conformal dynamical system (such as the limit set of a convex co-compact Kleinian group or the Julia set of a hyperbolic rational map).
    Suppose that $K_1, K_2$ have conformal dimension $1$ and that they are homeomorphic. Then $K_1$ is quasi-symmetric to $K_2$.

    More generally, let $K_i, i \in\{1,2\}$, be the limit set of a geometrically finite conformal dynamical system (such as the limit set of a geometrically finite Kleinian group or the Julia set of a geometric finite rational map).
    Suppose that $K_1, K_2$ have conformal dimension $1$ and that they are homeomorphic via some type-preserving map. Then $K_1$ is quasi-symmetric to $K_2$.
\end{conj}
\begin{rmk}
    We list some known cases of Conjecture~\ref{conj:confdim1univer}.
    \begin{itemize}[leftmargin=8mm]
        \item For a Jordan curve limit set $\Lambda$ of a geometrically finite conformal dynamical system, it is well-known that if the induced dynamics is hyperbolic on the two ideal boundaries, then $[\Lambda] = [\mathbb{S}^1]$.
        It is recently shown in \cite{McM25} that if the dynamics is hyperbolic on one ideal boundary but has a parabolic on the other, then $[\Lambda] = [J(z^2+\frac14)]$. For more general cases, see the discussion in \S~\ref{subsubsec:JC}.
        \item If the limit sets of two convex co-compact Kleinian groups corresponding to books of I-bundles $3$-manifolds are homeomorphic, then they are quasiconformally homeomorphic (see \cite{CM17}).
        \item If two gasket limit sets are homeomorphic, then they are quasiconformally homeomorphic (see \cite{LZ23} and \S~\ref{subsubsec:gasketJL} for more discussion).
        \item The Branner--Fagella homeomorphism between limbs of the Mandelbrot set also induces various quasiconformal homeomorphisms between Julia sets of different quadratic polynomials (see \cite{BF99}).
    \end{itemize}
\end{rmk}


\subsection{Methods, proof ingredients, and discussions}\label{methods_subsec}
    \subsubsection{Fragmented dynamics from Kleinian groups, rational maps, and Schwarz reflections}\label{subsubsec:fragdyn}
    All the fractals considered in this paper arise from dynamics, but the underlying dynamical systems are not conjugate. Our strategy to prove quasiconformal equivalence between these fractals is to exploit the dynamics nonetheless. This is done by associating a \emph{fragmented dynamical system} to the original one; more precisely, we cook up a piecewise conformal Markov (typically discontinuous) map on the limit/Julia set of a Kleinian group/rational map/Schwarz reflection. It turns out that even though the original dynamical systems are markedly different, the fragmented dynamical systems associated with them can be conjugate on the limit sets, and in fact this can be done quasi-symmetrically. This gives a unified approach to address the dynamical incompatibility between various unrelated conformal dynamical systems and to construct quasiconformal maps between their limit sets. Examples include the following. 
    \begin{enumerate}[leftmargin=8mm]
        \item Julia sets of rational maps and limit sets of Kleinian groups (Theorem~\ref{thm:qcclassfn-ltsets}, Theorem~\ref{thm:davidhi}),
        \item limit sets of non-isomorphic Kleinian groups (Theorem~\ref{thm:qcclassfn-ltsets}),
        \item Julia sets of rational maps of different degrees and/or different dynamical behavior (Theorem~\ref{thm-basilica-stdpolymodel}, Theorem~\ref{thm:cubicclass}), 
        \item Julia sets of (global) rational maps and limit sets of (semi-global) Schwarz reflections or B-involutions (Theorem~\ref{thm:schwarz}), etc.
    \end{enumerate}
    This is similar in spirit to the example of a quasi-symmetric map $\phi: \mathbb{S}^1 \cong \R/\Z \longrightarrow \mathbb{S}^1$ that sends the set of dyadic rationals $\{\frac {p}{2^n}\}$ onto the set of $d-$adic rationals $\{\frac{q}{d^m}\}$ (see Proposition~\ref{prop:2tod}). Special examples of fragmented dynamical systems associated with Fuchsian groups, historically used for the study of geodesic flow on surfaces \cite{BS79}, have been recently used to combine (anti-)polynomials with Fuchsian genus zero groups/reflection groups \cite{LLMM23,LMM22,LMM24,MM1,LLM24,MM2,BLLM,MV25}, and to construct (non-quasi-symmetric) homeomorphisms between certain limit and Julia sets \cite{LLMM23b,LLM22}. See also \cite{McM98, MU03} where conformal Markov maps are constructed from rational maps and Kleinian groups to compute Hausdorff dimensions of Julia/limit sets, and see \cite{LZ23, LN24} where fragmented dynamical systems are obtained via renormalization theory on circle packings and used to prove quasiconformal uniformization of gasket Julia sets.

\subsubsection{Basilica Bowen-Series maps}\label{basi_bs_subsubsec}

For a Kleinian group $G$ with a Basilica limit set, the associated fragmented dynamical system on $\Lambda(G)$ has its pieces in $G$. Prominent examples of such piecewise M{\"o}bius maps are given by \emph{Bowen-Series maps}, which are defined on the limit set $\mathbb{S}^1$ of Fuchsian groups \cite{BS79}. We define analogs of Bowen-Series maps on Basilica limit sets, and call them \emph{Basilica Bowen-Series maps}. On the one hand, the restriction of a Basilica Bowen-Series map to the limit set $\Lambda(G)$ yields a `real $1-$dimensional' Markov dynamical system, and on the other hand, their conformal extensions to `puzzle pieces' (i.e., suitable pinched neighborhoods of the Markov pieces, see \S~\ref{piecing_puz_subsec}) guide the construction of fragmented dynamical systems on Basilica Julia sets of rational~maps.

\subsubsection{Piecing together puzzles from two worlds}\label{piecing_puz_subsec}
The domains of definition of the fragmented dynamical systems mentioned above naturally flow from a number of precursors in the literature.
\begin{enumerate}[leftmargin=8mm]
    \item\label{yoc_puz} In standard polynomial dynamics, such domains are given by Yoccoz puzzle pieces (cut out by rays and equipotentials).
    \item\label{geod_puz} For Kleinian reflection groups \cite{LLM22} and certain Kleinian punctured sphere groups \cite[\S~3, \S~7.3]{MM1}, domains of conformal extensions of Nielsen maps or Bowen-Series maps (viewed as fragmented dynamical systems) are cut out by hyperbolic geodesics in components of the domain of discontinuity. Note that due to the presence of additional symmetry, such domains become round disks for reflection groups.
\end{enumerate}
The techniques of the present paper require a hybrid construction for domains of  conformal extensions involving the various types of `puzzles' mentioned above. 
\begin{enumerate}[leftmargin=8mm]
\item For a Basilica Bowen-Series map, the domains of conformal extensions are cut out by hyperbolic geodesics, generalizing type~\eqref{geod_puz} above (see Figure~\ref{fig:PinchedPolygon}). 
\item For fragmented dynamical systems associated with Basilica Julia sets, the domains of conformal extensions are cut out by rays and equipotentials in the basin of infinity and analogs of hyperbolic geodesics in bounded Fatou components, thus bringing elements from item~\eqref{yoc_puz} and item~\eqref{geod_puz} together (see Figure~\ref{fig:BasPuz}).
\item For fragmented dynamical systems associated with Basilica limit sets of matings between groups and polynomials (e.g., Schwarz reflections), the domains of conformal extensions are cut out by rays and equipotentials on the `polynomial side' of the limit set  and by genuine hyperbolic geodesics on the `group side' (see Figure~\ref{fig:AltDecSvsR}). This construction literally pieces together puzzles from the rational dynamics world and the Kleinian group world  in a common dynamical plane. (See \cite[\S 4.2]{LLMM23} for a precursor to such hybrid puzzle constructions in the context of Schwarz reflections.)
\end{enumerate}

\subsubsection{Jordan curve Julia set and limit set}\label{subsubsec:JC}
    For a Jordan curve limit set $\Lambda$ of a geometrically finite conformal dynamical system, a point $x \in \Lambda$ has four possible types: $(H,H), (H,P), (P,H)$ and $(P,P)$, depending on whether the induced dynamics at the corresponding points on the two sides is hyperbolic or parabolic.
    Let $S \subseteq \{(H,H), (H,P), (P,H), (P,P)\}$, and we say $\Lambda$ has type $S$ if each point of $\Lambda$ has type in $S$ and for each type $s\in S$, there exists a point $x\in \Lambda$ of type $s$. It is easy to see that if $\Lambda_1, \Lambda_2$ have the same type, then there exists a type-preserving homeomorphism between them.

    We can realize all possible types of $\Lambda$ via topologically expanding piecewise M\"obius maps (see \S~\ref{sec:XC}), and obtain exactly five different quasiconformal types that Conjecture~\ref{conj:confdim1univer} predicts (see Figure~\ref{fig:JC} and \cite[Figure~1.1, Figure 11.9]{LMMN}):
    \begin{itemize}[leftmargin=8mm]
        \item Quasicircle: $\{(H,H)\}$, $\{(H, H), (P,P)\}$;
        \item Cauliflower: $\{(H,H), (H,P)\}$, $\{(H,H), (P,H)\}$;
        \item Double Parabolic Cauliflower: $\{(H,H), (H,P), (P,P)\}$, $\{(H,H), (P,H), (P, P)\}$;
        \item Pine tree: $\{(H,H), (H,P), (P, H)\}$;
        \item Double Parabolic Pine tree: $\{(H,H), (H,P), (P, H), (P, P)\}$.
    \end{itemize}
    The methods developed in this paper should pave the way to prove the conjecture in this setting. This will be investigated in a future work.
    



    \begin{figure}[ht]
\captionsetup{width=0.96\linewidth}
  \centering
  \includegraphics[width=0.3\textwidth]{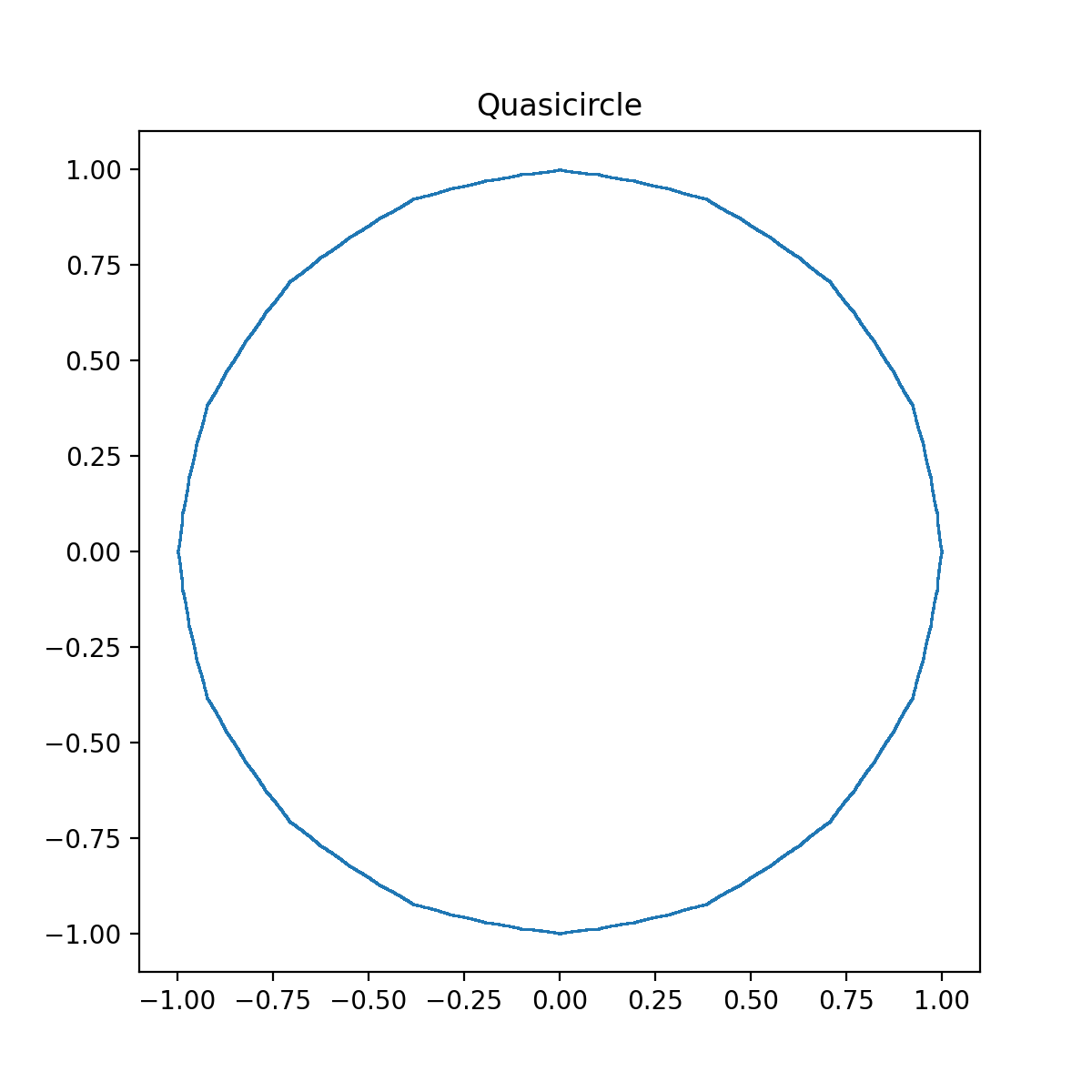}
  \includegraphics[width=0.3\textwidth]{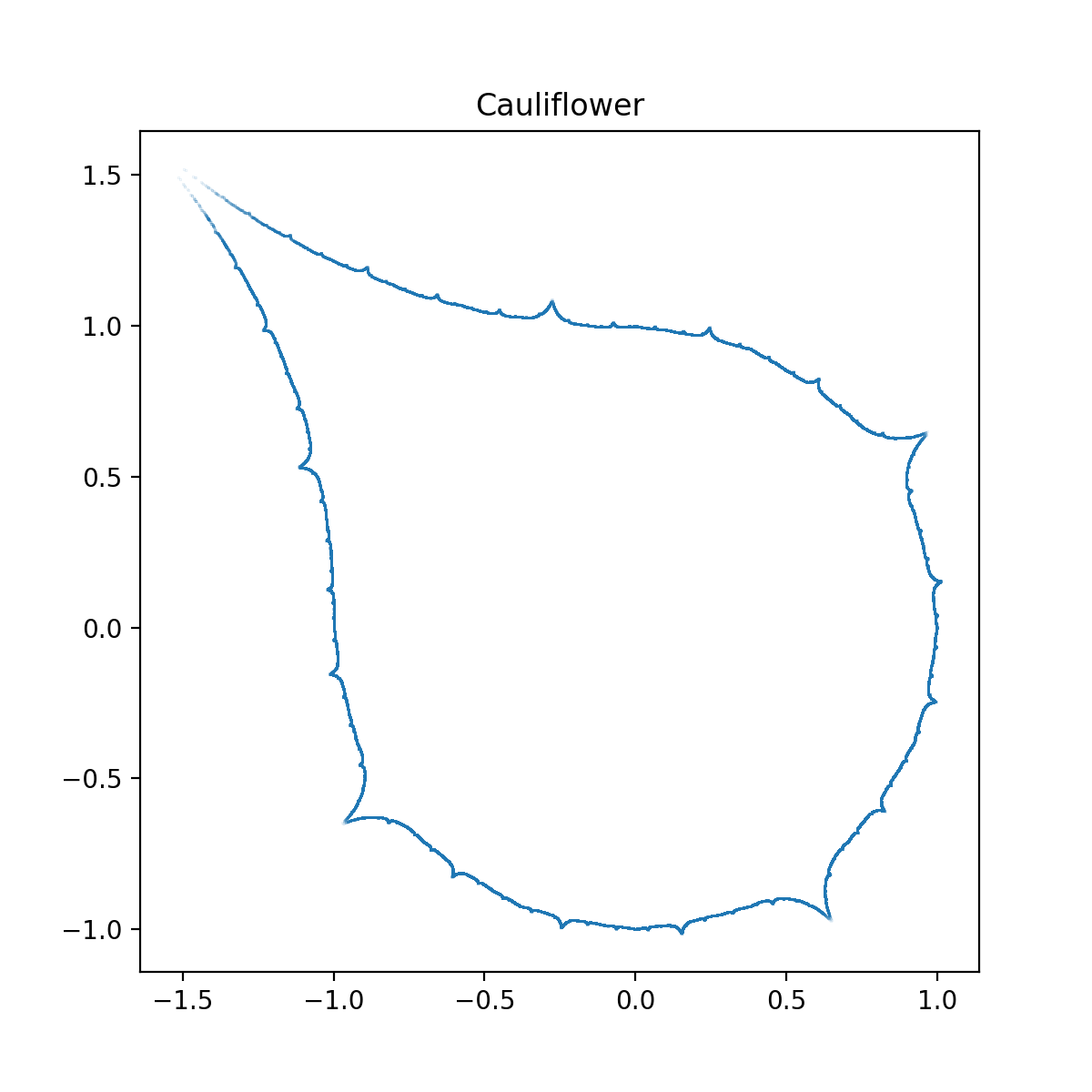}
  \includegraphics[width=0.3\textwidth]{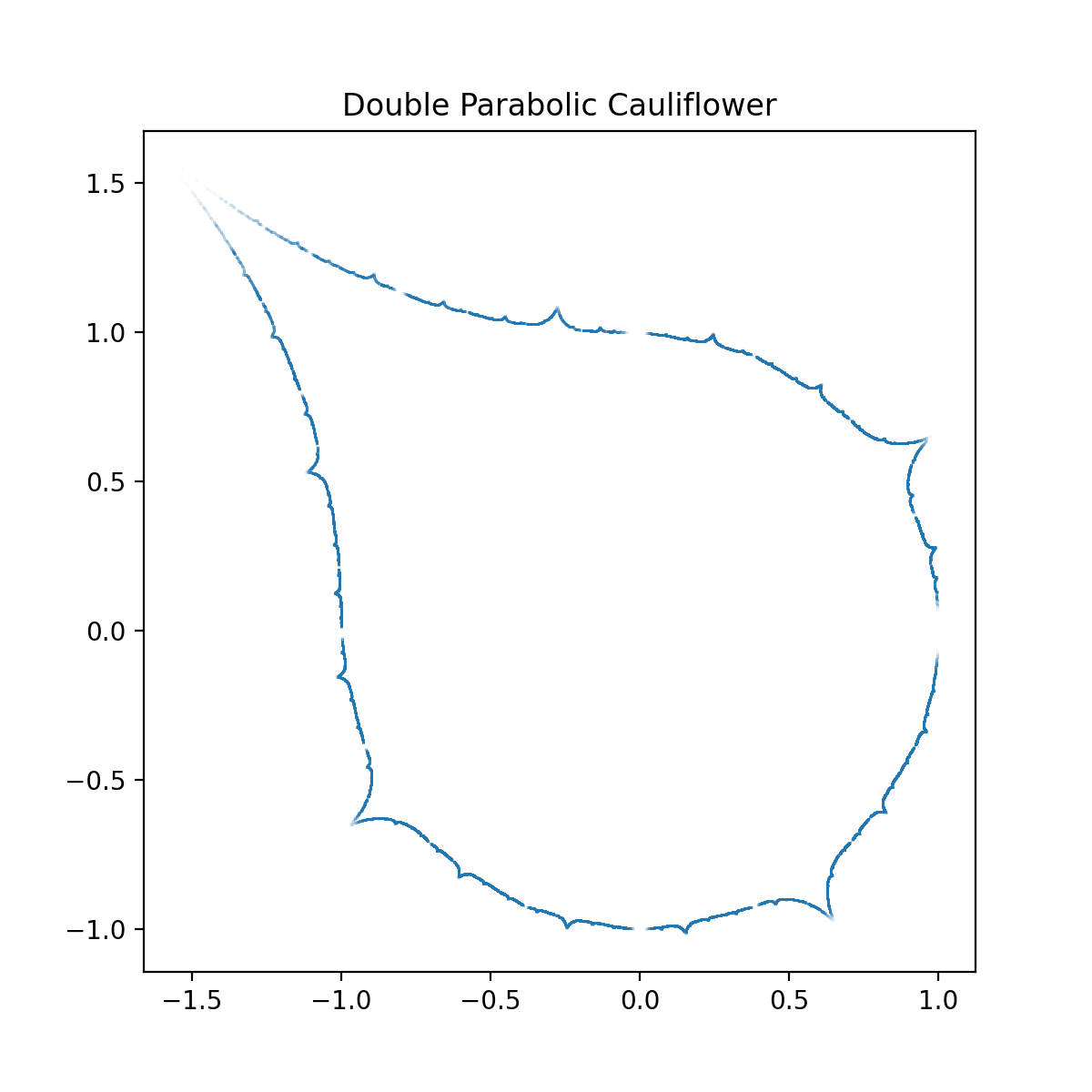}
  \includegraphics[width=0.3\textwidth]{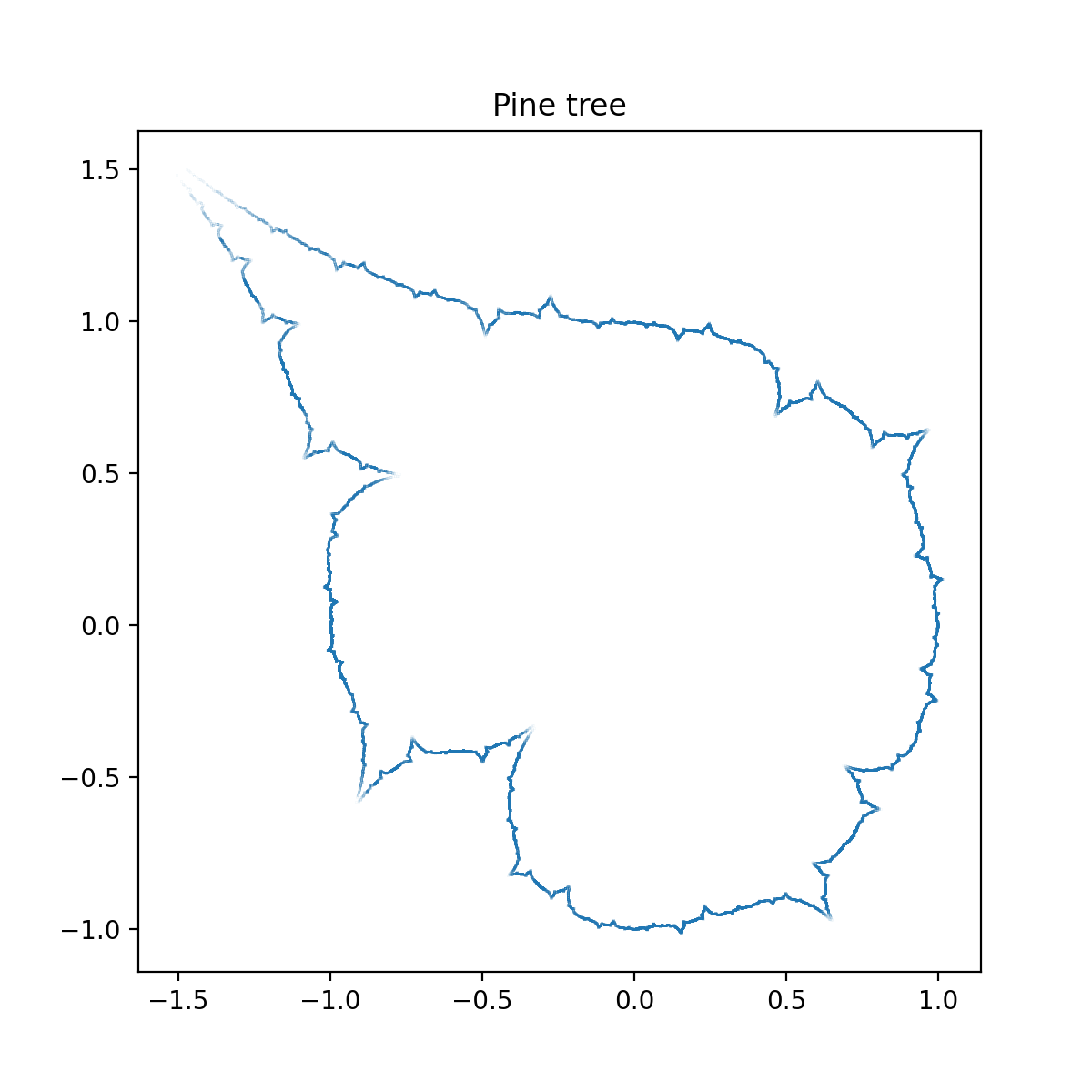}
  \includegraphics[width=0.3\textwidth]{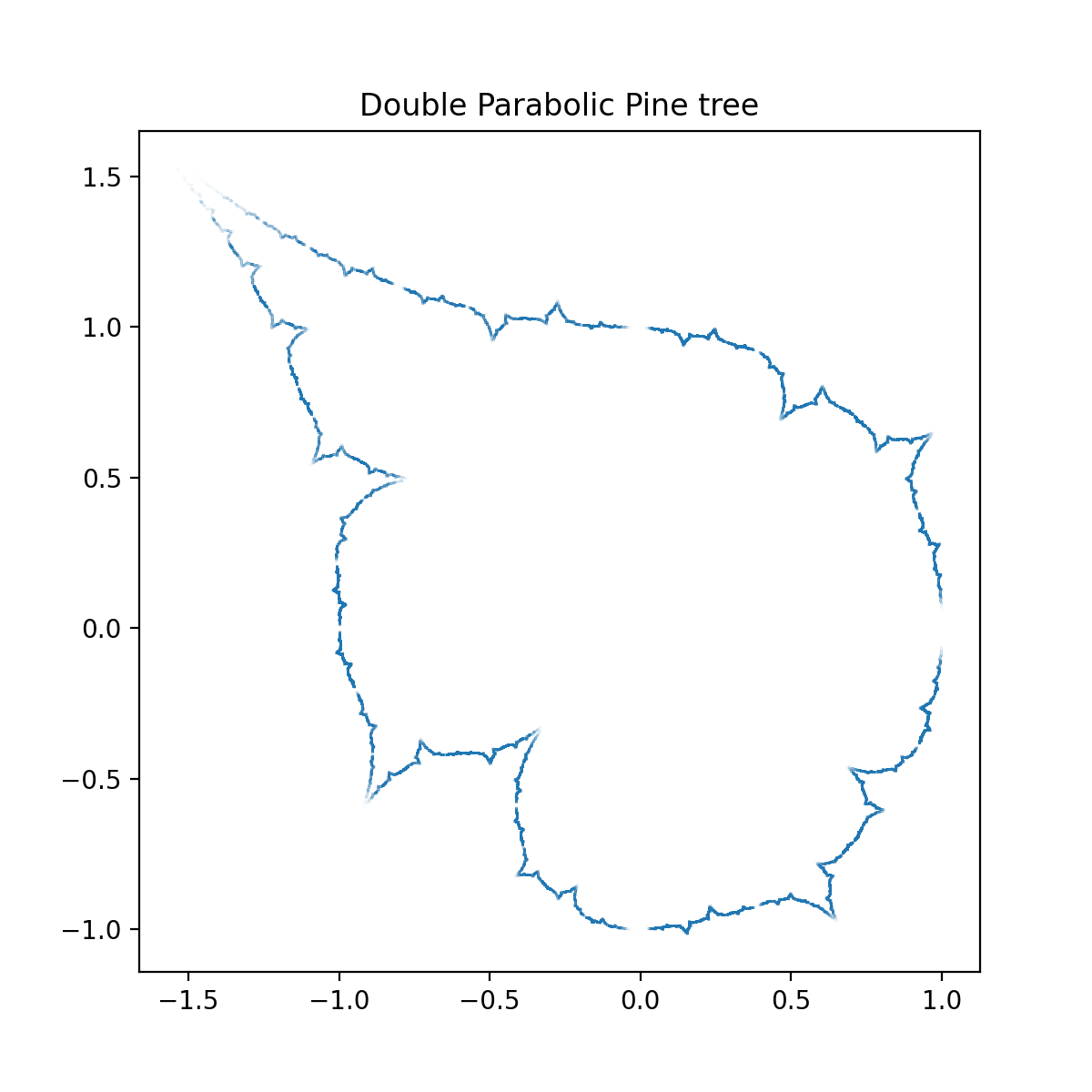}
  \caption{Five different quasiconformal types of Jordan curve limit sets. The existence of inward/outward cusps allows one to distinguish the quasicircle, the cauliflower and the pine tree. Note that as we zoom in at a $(P,P)$ point of the double parabolic cauliflower or the double parabolic pine tree, the limit set converges to a line in the Hausdorff topology. On the other hand, at any non-cusp point of the cauliflower or the pine tree, we can choose a specific sequence of zoom-ins so that the limit set converges to an arc that contains cusps. This shows that the cauliflower or the pine tree are quasiconformally different from their double parabolic counterparts.}
  \label{fig:JC}
\end{figure}

\subsubsection{Gasket Julia set and limit set}\label{subsubsec:gasketJL}
{\em Gaskets} or {\em circle packings} are another important type of fractal sets with conformal dimension $1$. Here, a {\em gasket} is a closed subset of $\widehat\C$ so that
\begin{enumerate}
    \item each complementary component is a Jordan domain;
    \item any two complementary components touch at most at $1$ point;
    \item no three complementary components share a boundary point; and
    \item the {\em nerve}, obtained by assigning a vertex to each complementary component and an edge if two components touch, is connected.
\end{enumerate}
Let us discuss some recent results for gasket Julia sets and limit sets in the framework of the current paper.

If two geometrically finite gasket limit sets are homeomorphic, then the homeomorphism is automatically type-preserving as a rank one parabolic corresponds precisely to the contact point between two components of the domain of discontinuity (see \cite[Proposition B.2]{LZ23} or \cite[Lemma~11.2]{BO22}). It follows from \cite[Theorem I]{LZ23} and \cite{McM90} that homeomorphic geometrically finite gasket limit sets are quasiconformally equivalent.
This proves  Conjecture~\ref{conj:confdim1univer} for gasket limit sets.

To compare gasket limit sets with Julia sets, in light of Conjecture~\ref{conj:confdim1univer}, the main obstruction to quasiconformal equivalence for gaskets arises from the {\em type-preserving condition}. This condition on  local dynamics actually imposes global constraints on the nerve of the gasket. 
The global combinatorial constraint is exploited in \cite{LZ25} to show that gasket Julia sets are quasiconformally nonequivalent to gasket limit sets in various settings.  
On the other hand, it is proved in \cite{LN24} that there are no local obstructions to  quasiconformal equivalence between gasket Julia sets and limit sets, again furnishing supporting evidence for Conjecture~\ref{conj:confdim1univer}.

\subsubsection{Applications to mating}

Let us also mention a feature of Basilica Bowen-Series maps (see \S~\ref{basi_bs_subsubsec}) that will play a central role in a sequel to this paper \cite{LMM26}. Although Basilica Bowen-Series maps are defined for Bers boundary groups, they give rise to (possibly discontinuous) piecewise M{\"o}bius Markov maps of the circle, called \emph{pullback Bowen-Series maps} (see \S~\ref{pullback_bs_subsec}). However, pullback Bowen-Series maps can often be turned into covering maps of the circle using acceleration techniques (defined and elaborated upon in \cite{LMM26}), and such coverings largely remember the grand orbit structure of the associated Fuchsian groups. This has applications to the problem of mating polynomial dynamics with Fuchsian groups (which was hitherto only achieved for Fuchsian genus zero groups, see \cite{MM1,LLM24,MM2,LM25b}). In summary, a detour to the boundary of the Teichm{\"u}ller space proves fruitful for mating a Fuchsian closed surface group with polynomials.

\subsection{Outline of the paper}
In \S~\ref{sec:XC}, we introduce the notion of \emph{fragmented dynamical systems}, or more formally $(X, \mathscr{C})-$maps. These are piecewise homeomorphisms on a compact metric space $X$ with  pieces belonging to a given class $\mathscr{C}$. The definition is motivated to an extent by $(G,X)$ structures on a manifold.
In \S~\ref{sec:cmm}, we study $(\mathbb{S}^1, \Conf)-$maps (piecewise conformal maps on $\mathbb{S}^1$) as one of the simplest instances of $(X, \mathscr{C})-$maps, and provide criteria for two such maps to be quasisymmetrically conjugate. 

In \S~\ref{sec-basilica} and \S~\ref{section:BasBS}, we construct further special classes of $(X, \mathscr{C})-$maps: $(J, \mathfrak{Q})-$maps associated to the model map $Q(z) = z^2-\frac34$, and $(\Lambda, G)-$maps associated to a geometrically finite Bers boundary group $G$. These maps arise as fragmented dynamics of rational maps and Kleinian groups, as discussed in \S~\ref{subsubsec:fragdyn}.

The main theorems (Theorem~\ref{thm:qcclassfn-ltsets}, Theorem~\ref{thm-basilica-stdpolymodel} and Theorem~\ref{thm:infKleinianRational}) are proved in \S~\ref{sec:qc_conj}, by applying quasiconformal surgeries to the corresponding fragmented dynamics.

To extend the above techniques to other universality classes of Basilicas, we note two key subtleties:
\begin{enumerate}[leftmargin=8mm]
    \item the presence of non-contact parabolic points, and
    \item the possibility that the dynamics may send contact to non-contact points.
\end{enumerate}
To address these two issues, we introduce suitable modifications by associating appropriately modified fragmented dynamics to the rational map or Kleinian group in question. With such modifications, we prove Theorem~\ref{thm:davidhi}, Theorem~\ref{thm:schwarz} and Theorem~\ref{thm:cubicclass} in \S~\ref{sec:davidHierarchy}, Appendix~\ref{sec:cubic_poly} and Appendix~\ref{sec:schwarz}.

\subsection*{Acknowledgement} 
We thank Curt McMullen for many insightful discussions.

\section{$(X, \mathscr{C})-$maps}\label{sec:XC}
In this section, we introduce a general class of 
maps defined piecewise  on a compact metrizable space $X$. Our goal is to formulate the definition in as broad a setting as possible. In \S\ref{sec:cmm}, \S\ref{sec-basilica}, and \S\ref{section:BasBS}, we will examine in detail three specific families of such maps. 

Definition~\ref{XCstrs} below is partly motivated by the notion of $(G,X)$ structures occurring in geometry:
 $X$ is a manifold, and $G$ is a group of diffeomorphisms of $X$. A manifold $M$ has a $(G,X)$ structure if it admits an open cover $\{U_\alpha\}$, where each $U_\alpha$ is homeomorphic to an open subset of $X$ and transition functions come from~$G$.

\begin{defn}\label{XCstrs}
    Let $X$ be a compact metrizable space. A finite collection $\mathcal{A} = \{A_1,..., A_m\}$ of subsets of $X$ is called a {\em partition} if
    \begin{itemize}
        \item $A_i$ is  closed and connected;
        \item $A_i \cap A_j$ is a finite set for $i \neq j$;
        \item $\bigcup A_i = X$.
    \end{itemize}
    We define the set of {\em break-points} as 
    $$
    \mathrm{bk}(\mathcal{A}):= \{x \in X: x \in A_i\cap A_j, i \neq j\}.
    $$
    Note that by construction, $\mathrm{bk}(\mathcal{A})$ is a finite set.

    Let $\mathscr{C}$ be a collection of homeomorphisms $f: A \longrightarrow B$, where $A, B$ belong to a family of closed and connected subsets of $X$.
    A finite collection of maps $\mathcal{F} = \{f_1,..,. f_m\}$ is called an {\em $(X, \mathscr{C})-$map} if there exists a partition $\mathcal{A} =  \{A_1,..., A_m\}$ of $X$ so that for each $i$, the map $f_i: A_i \longrightarrow B_i = f(A_i)$ is in $\mathscr{C}$.
\end{defn}

The second condition in Definition~\ref{XCstrs} 
implicitly indicates that we are  interested in the case where $X$ has topological dimension one. This condition could be relaxed if one was interested in dynamics in higher dimensions.

\begin{remark}
    We do not require the collection $\mathcal{F}$ to induce a continuous map on $X$. In other words, if $x \in A_i \cap A_{j} \subseteq \mathrm{bk}(\mathcal{A})$, then $f_i(x)$ may be different from $f_j(x)$.
    It is convenient to regard $\mathcal{F}$ as a multi-valued map. More precisely, if $x \in \bigcap_{i\in \mathcal{I}} A_i$ for some subset $\mathcal{I} \subseteq \{1,..., m\}$, then $\mathcal{F}(x) = \bigcup_{\mathcal{I}}\{f_i(x)\}$ can be multi-valued.
\end{remark}

\begin{example}\label{exm:xcmaps}
We give some important and well-known examples of $(X, \mathscr{C})-$maps.
    \begin{enumerate}
        \item Let $X = I = [0,1]$ and let
        $$
        \PL=\{(f:I_1 \longrightarrow I_2): f(x) = ax+b\}.
        $$
        Then $(I, \PL)-$maps are piecewise linear maps on $[0,1]$.
        \item Let $X = \mathbb{S}^1 = \R /\Z$, and let
        $$
        \PPSL_2(\Z)=\{(f:I \longrightarrow J): f(x) = (ax+b)/(cx+d), ad-bc = 1, a,b, c, d \in \Z\}.
        $$
        Then $(\mathbb{S}^1, \PPSL_2(\Z))-$maps are  piecewise integral projective maps.
        \item Let $X = \mathbb{S}^1 = \R /\Z$, and let
        $$
        \PDyad=\{(f:I \longrightarrow J): f(x) = 2^ax+ b, a\in \Z, b \in \Z[1/2]\}.
        $$
        Then $(\mathbb{S}^1, \PDyad)-$maps are dyadic piecewise linear maps. These are related to the Thompson's group $T$.
        \item Let $X = J(F)$ be the Julia set of a rational map $F$, 
        and let 
        $$
        \mathfrak{F}=\{(f:A \longrightarrow B): f(x) = F^{-n} \circ F^m(x)\},
        $$
        where $F^{-n} \circ F^m$ is a homeomorphism on $A$ given by the composition of some iterate $F^m$ and the inverse branch of some iterate $F^n$.
        We call such maps $(J, \mathfrak{F})-$maps. They  will be discussed in detail in \S~\ref{sec-basilica}. Analogous constructions for Schwarz reflection maps will be treated in Appendix~\ref{sec:schwarz}.
        \item Let $X = \Lambda(G)$ be the limit set of a Kleinian group, and let 
        $$
        \mathscr{C}=\{(f:A \longrightarrow B): f(x) = gx, \, \textrm{for\, some\, } g \in G\}.
        $$
        We call such maps $(\Lambda, G)-$maps. They will be discussed in detail in \S~\ref{section:BasBS}.
    \end{enumerate}
\end{example}

\subsection{Markov property}\label{subsec:markovp}
\begin{defn}\label{markov_map_def}
    Let $\mathcal{F} = \{f_1,..,. f_m\}$ be an $(X, \mathscr{C})-$map. We say that $\mathcal{F}$ is {\em Markov} if for each $i$, there exists an index set $\mathcal{J}_i \subseteq \{1,..., m\}$ so that
    $$
    f_i(A_i) = \bigcup_{j \in \mathcal{J}_i} A_j, \text{ and } f_i(A_i \cap \mathrm{bk}(\mathcal{A})) \subseteq \mathrm{bk}(\mathcal{A}).
    $$
\end{defn}

A partition $\mathcal{B}$ is a {\em refinement} of $\mathcal{A}$ if for each $B \in \mathcal{B}$, there exists $A \in \mathcal{A}$ so that $B \subseteq A$.
Given a Markov $(X, \mathscr{C})-$map $\mathcal{F} = \{f_1,..,. f_m\}$, one can define the pull-back partition 
$$
\mathcal{A}^1 = \mathcal{F}^*(\mathcal{A}) = \{A_{ij} :=f_i^{-1}(A_j): f_i(A_i) \cap A_j \neq \emptyset\}.
$$
We call $\mathcal{A}^0 = \mathcal{A}$ (resp.\ $\mathcal{A}^1$) the {\em level $0$} (resp.\ {\em level $1$}) Markov partitions.
By construction, 
\begin{itemize}
    \item(inclusion) $A_{ij} \subseteq A_i$, i.e., $\mathcal{A}^1$ is a refinement of $\mathcal{A}^0$; and
    \item(dynamics) $f_i: A_{ij} \longrightarrow A_j$ is a homeomorphism.
\end{itemize}
We define the level $n$ Markov partitions inductively, and denote them by $\mathcal{A}^n = \{A_{w}\}$, where $w$ is a word of length $n$ consisting of letters in $\{1,..., m\}$. We will call a word $w$ of length $n$ {\em admissible} if $A_w \in \mathcal{A}^n$.

It is convenient to write a Markov $(X, \mathscr{C})-$map as
$$
\mathcal{F} : \mathcal{A}^1 \longrightarrow \mathcal{A}^0.
$$
Note that if the collection of maps $\mathscr{C}$ is closed under restrictions, then $\mathcal{F}$ is an $(X, \mathscr{C})-$map with respect to the partition $\mathcal{A}^n$ for any $n$.

\subsection{Combinatorial and topological conjugacy}
Let $\mathcal{F} : \mathcal{A}^1 \longrightarrow \mathcal{A}^0$ be a Markov $(X, \mathscr{C})-$map.
One can record the combinatorics of the Markov map by an $m \times m$ matrix, called the {\em Markov matrix} $M = M_\mathcal{F}$. More precisely, 
$$
M_{ij} = \begin{cases}
    1 & \text{ if }A_i \subseteq f_j(A_j)\\
    0 & \text{ otherwise.}
\end{cases}.
$$
Note that the Markov matrix requires a choice of an ordering of the elements in the partition $\mathcal{A}^0$.

\begin{defn}[Topologically expanding]\label{def:top_expand}
    Let $\mathcal{F}$ be a Markov $(X, \mathscr{C})-$map with respect to the partition $\mathcal{A} = \{A_1,..., A_m\}$. We say that it is {\em topologically expanding} if for each $x \in X$, the collection
    $\{\Int\left(\bigcup_{x\in A_w \in \mathcal{A}^n} A_w\right): n \in \N\}$ forms a local basis of open neighborhood at $x$.
    
    Note that this implies that 
    \begin{enumerate}[leftmargin=8mm]
        \item\label{topexp:cond1} $\bigcap_{w_n} A_{w_n}$ is a singleton set for each nested infinite sequence of level $n$ Markov pieces;
        \item\label{topexp:cond2} for every open set $U \subseteq X$, there exists $n$ so that $\mathcal{F}^n (U)$ is not contained in $A_j$ for any $j\in \{1,..., m\}$.
    \end{enumerate} 
     We also remark that Condition~\eqref{topexp:cond1} is equivalent to topological expansion under the following assumption: 
    \medskip
    
      \noindent  There exists $N$ such that for each $n \in \N$, any point $x\in X$ is contained in at most $N$ level $n$ Markov pieces; in other words, more than $N$ Markov pieces of any given level have trivial intersection.
\end{defn}

\begin{defn}[Topological and Combinatorial conjugacy]\label{defn:topcom}
    Let $\mathcal{F}$ and  $\cG$ be two Markov $(X, \mathscr{C})-$maps with respect to partitions $\mathcal{A} = \{A_1,..., A_m\}$ and $\mathcal{B} = \{B_1,..., B_m\}$ respectively. 
    We say that they are {\em topologically conjugate} if there exists a homeomorphism
    $$
    H: X \longrightarrow X, \text{ so that}
    $$
    \begin{equation}\label{eqn:conj}
        H\circ f_i(x) = g_i \circ H(x) \text{ for all $x \in A_i$ and $i \in \{1,..., m\}$}.
    \end{equation}
    Abusing notation, we write the equation~\eqref{eqn:conj} as simply
    $$
    H\circ \mathcal{F} = \cG \circ H.    $$

    We say they are {\em combinatorially conjugate} if there exists a homeomorphism
    $$
    H: X \longrightarrow X, \text{ so that}
    $$
    \begin{itemize}
        \item $H(A_i) = B_i$ for all $i = 1,..., m$;
        \item for each break-point $x \in A_i \cap \mathrm{bk}(\mathcal{A})$,
        $$
            H \circ f_i(x) = g_i \circ H(x); \text{ and }
        $$
        \item $M_\mathcal{F} = M_\cG$.
    \end{itemize} 
\end{defn}

By the standard pull-back argument, we have the following.
\begin{prop}\label{prop-combconjimpliestopconj}
    Let $\mathcal{F}: \mathcal{A}^1 \longrightarrow \mathcal{A}^0$ and  $\cG: \mathcal{B}^1 \longrightarrow \mathcal{B}^0$ be two Markov $(X, \mathscr{C})-$maps.
    Suppose that they are topologically expanding and combinatorially conjugate. Then they are topologically conjugate.
\end{prop}
\begin{proof}
    We give the standard pull-back argument here for completeness. Since the two Markov matrices are the same, we have 
    $$
    H \circ f_i(A_i) = H(\bigcup_{j \in \mathcal{J}_i}A_j) = \bigcup_{j \in \mathcal{J}_i} H(A_j) = \bigcup_{j \in \mathcal{J}_i} B_j = g_i(B_i).
    $$
    Thus, for each $i \in \{1,..., m\}$, we define the homeomorphism $H^1: A_i \longrightarrow B_i$ by $H^1 = g_i^{-1} \circ H \circ f_i$.
    
    We verify that $H^1: X \longrightarrow X$ is a well-defined homeomorphism.
    To see this, let $x \in A_i \cap A_j$.
    Note that $H \circ f_i(x) = g_i \circ H(x)$ and $H\circ f_j(x) = g_j\circ H(x)$.
    We conclude that
    $$
    g_i^{-1} \circ H \circ f_i(x) = H(x) = g_j^{-1} \circ H \circ f_j(x).
    $$
    Therefore, $H^1$ is well-defined. By the pasting lemma, $H^1$ is a homeomorphism.
    Note that by construction, we have
    $$
    H \circ \mathcal{F} = \cG \circ H^1.
    $$
    Therefore, $H^1(A_w) = B_w$ for any admissible word of length $2$.

    Now inductively, we can construct a sequence of homeomorphisms $H^n: X \longrightarrow X$ so that
    $$
    H^{n-1} \circ \mathcal{F} = \cG \circ H^{n}.
    $$
    Note that by construction, we conclude that $H^n(A_w) = B_w$ for each admissible word of length $\leq n+1$.

    We first show $H^n$ converges.
    Let $x \in X$, and let $U_n = \Int\left(\bigcup_{x\in A_w \in \mathcal{A}^n} A_w\right)$ and $V_n = \Int\left(\bigcup_{x\in A_w \in \mathcal{A}^n} H^n(A_w)\right)$. Since the maps are topologically expanding, the nested intersection $\bigcap V_n$ is a singleton set. Thus the sequence $H^n(x)$ converges to the unique point in $\bigcap V_n$. We denote this point by $H^\infty(x)$. This defines a map $H^\infty:X\to X$ that is the pointwise limit of $H^n$.

    We now show that $H^\infty$ is continuous.
    Let $x,\ U_n,\ V_n$ be as above.
    By construction, we have $H^\infty(U_n) = V_n$.
    Since the maps are topologically expanding, $U_n$ and $V_n$ form local bases of open neighborhoods of $x$ and $H^\infty(x)$, respectively. We conclude that $H^\infty$ is continuous at $x$.
    Therefore, $H^\infty$ is continuous.
    
    By a symmetric argument, one can construct a continuous pointwise limit of the homeomorphisms $(H^n)^{-1}$, and show that it is the inverse of $H^\infty$. Thus, $H^\infty$ is the desired topological conjugacy.
\end{proof}

\subsection{Conformal and quasiconformal maps}\label{subsubsec:conf}
We now restrict ourselves to the case where $X$ is a closed subset of $\widehat{\C}$. This will be the standing assumption for the remainder of the paper.

We say that a map $f: A \longrightarrow B$ between two closed and connected subsets of $X$ is {\em conformal} (or {\em M\"obius} or {\em quasiconformal}) if there exist open neighborhoods $N_A$ and $N_B$ of $A, B$ in $\widehat{\C}$ (respectively) so that $f$ extends to a conformal map (or a M\"obius map or a quasiconformal map) between $N_A$ and $N_B$.
We denote the collection of such maps by 
\begin{align*}
    \Conf&:= \{(f:A \longrightarrow B): f \text{ is conformal}\};\\
    \Mob&:= \{(f:A \longrightarrow B): f \text{ is M\"obius}\};\\
    \QC&:= \{(f:A \longrightarrow B): f \text{ is quasiconformal}\}.
\end{align*}
Note that the $(J, \mathfrak{F})-$maps and $(\Lambda, G)-$maps of Example~\ref{exm:xcmaps} are natural examples of $(X, \Conf)-$maps. 

In our setting,  Markov partitions typically arise as the intersections of \emph{puzzle pieces} with $X$. These puzzle pieces are `pinched complex neighborhoods' of Markov partition pieces; they come from dynamics, and satisfy natural nesting properties (see Figure~\ref{fig:BasPuz}). We introduce the relevant definitions below.

Let $X$ be a closed set in $\widehat{\C}$. A subset $\widehat{X}$ of $\widehat{\C}$ is said to be a \emph{pinched neighborhood} of $X$, pinched at finitely many points $x_1,\cdots,x_n\in\partial X$, if $X\subset\widehat{X}$, each point of $X-\{x_1,\cdots,x_n\}$ is an interior point of $\widehat{X}$, but no $x_i$ is an interior point of $\widehat{X}$.

\begin{defn}[Puzzle structure]
    Let $X \subseteq \widehat{\C}$ be closed, and let $\mathcal{F}: \mathcal{A}^1 \longrightarrow \mathcal{A}^0$ be a Markov $(X,\Conf)-$map.
    We say that it admits a {\em puzzle structure} if there exist closed pinched neighborhoods $P_i$ of $A_i$, pinched at the points in $A_i\cap\mathrm{bk}(\mathcal{A})$
    (where $\mathrm{bk}(\mathcal{A})$ is the set of break-points), so that for each level $1$ Markov piece $A_{ij} \subseteq A_i$, we have
    $$
        f_i: \Int P_{ij} = f_i^{-1}(\Int P_j) \longrightarrow \Int P_j
    $$
    is conformal, and $\Int P_{ij} \subseteq \Int P_{i}$.
\end{defn}

\begin{remark}
    Conformal Markov maps have been introduced and studied in other settings in \cite{McM98, MU03, LMMN}. Quasiconformal Markov maps on the circle were investigated in \cite{LN24a}, and their analogs for circle packings were studied in \cite{LZ23, LN24}.
\end{remark}

\section{Conformal Markov maps on $\mathbb{S}^1$}\label{sec:cmm}
In this section, following the notation of \S~\ref{sec:XC}, we study $(\mathbb{S}^1, \Conf)-$maps satisfying the Markov property.
We will call such maps {\em conformal Markov maps} on $\mathbb{S}^1$.
The main conclusion is stated in Proposition~\ref{prop:qsmarkovmap} and Proposition~\ref{prop:qsmarkovmapFiniteCircle}.

Let $\mathcal{F} = \{f_1,..., f_m\}$ be a conformal Markov map on $\mathbb{S}^1$ with respect to a partition $\mathcal{A}= \{A_1,\ldots, A_m\}$, where we assume that $A_1,\ldots, A_m$ are ordered and oriented counterclockwise on the circle.
Thus, we have the following.
\begin{itemize}
    \item $A_i$ is a closed interval;
    \item $\Int A_i \cap \Int A_j = \emptyset$;
    \item $\bigcup A_i = \mathbb{S}^1$;
    \item for each $i$, there exists $j_i$ and $l_i$ so that 
    $$
    f_i: A_i \longrightarrow A_{j_i} \cup A_{j_i+1} \cup \ldots \cup A_{j_i+l_i}
    $$
    is a homeomorphism;
    \item $f_i$ extends to a conformal map on a neighborhood of $A_i$.
\end{itemize}
Recall that we do not require the collection $\mathcal{F}$ to induce a continuous global map on $\mathbb{S}^1$, and we regard $\mathcal{F}$ as a potentially multi-valued map.

In the setting of conformal Markov map on $\mathbb{S}^1$, we also use the following notation throughout the paper.
$$
    \mathcal{F}(x^+) = \lim_{t \to x^+} \mathcal{F}(t), \text{ and } \mathcal{F}(x^-) = \lim_{t \to x^-} \mathcal{F}(t).
$$
Similarly, we denote $\mathcal{F}^n(x^\pm) := \lim_{t \to x^\pm} \mathcal{F}^n(t)$.
With this notation, two conformal Markov maps $\mathcal{F}, \cG$ are topologically conjugate (see Definition~\ref{defn:topcom}) if there exists a homeomorphism $h: \mathbb{S}^1 \longrightarrow \mathbb{S}^1$ so that
    $$
    \mathcal{F}(h(x)^\pm) = h(\cG(x^\pm).
    $$

\subsection*{Break points and Lyapunov exponent}
Let $\mathcal{F}$ be a topologically expanding conformal Markov map with respect to the partition $\mathcal{A}$. 
Let $x \in A_i \cap A_{i+1}$ be a break-point for the dynamics of $\mathcal{F}$.
Denote by $x_n^\pm$ the right and left orbit of $x$, i.e., 
$$
x_n^\pm = \mathcal{F}^n(x^\pm).
$$
Since the partition is finite, we note that both right and left orbits are pre-periodic. Note that they may be eventually mapped to different periodic orbits.
Let us denote the right and left {\em Lyapunov exponent} by $\lambda(x^+)$ and $\lambda(x^-)$ respectively, i.e.
$$
\lambda(x^\pm) = \lim_{n\to\infty} \frac{1}{n} \log \lim_{\epsilon \to 0} \frac{|\mathcal{F}^n(x\pm\epsilon)- x_n^\pm|}{\epsilon}.
$$
Since $\mathcal{F}$ is topologically expanding, we have $\lambda(x^\pm) \geq 0$.
We say that $x^+$ (respectively, $x^-$) is {\em parabolic} if $\lambda(x^+) = 0$ (resp. $\lambda(x^-) = 0$). 
Similarly, we say that $x^+$ (respectively, $x^-$) is {\em hyperbolic} if $\lambda(x^+) > 0$ (resp. $\lambda(x^-) > 0$).

Since $x_n^\pm$ are eventually periodic, $x^\pm$ is parabolic or hyperbolic if and only if the periodic tail of the sequence $x_n^\pm$ is a parabolic or hyperbolic cycle.
If $x^\pm$ is parabolic, we define its {\em parabolic multiplicity}, denoted by $N(x^\pm)$, as the parabolic multiplicity of the periodic cycle it is mapped to (cf. \cite[\S 10]{Mil06}).

We say that the break-point $x$ is {\em symmetric} if 
\begin{itemize}
    \item $\lambda(x^+) = \lambda(x^-)$,
    \item in addition, if  $x^\pm$ are parabolic, then 
    $N(x^+) = N(x^-)$.
\end{itemize} 
Note that in the symmetric case, we will call the break-point $x$ \emph{symmetrically parabolic/hyperbolic}.

We now recall the notion of David homeomorphisms.
\begin{defn}\label{david_def}
	An orientation-preserving homeomorphism $H: U\to V$ between domains in the Riemann sphere $\widehat{\C}$ is called a \emph{David homeomorphism} if it lies in the Sobolev class $W^{1,1}_{\textrm{loc}}(U)$ and there exist constants $C,\alpha,\varepsilon_0>0$ with
	\begin{align}\label{david_cond}
		\sigma(\{z\in U: |\mu_H(z)|\geq 1-\varepsilon\}) \leq Ce^{-\alpha/\varepsilon}, \quad \varepsilon\leq \varepsilon_0.
	\end{align}
\end{defn}
Here $\sigma$ is the spherical measure, and
$\mu_H= \frac{\partial H/ \partial\overline{z}}{\partial H/\partial z}$
is the Beltrami coefficient of $H$. We refer the reader to  \cite[Chapter~20]{AIM09}, \cite[\S 2]{LMMN} for background on David homeomorphisms; including the David integrability theorem, removability properties of David maps, and David extension of circle homeomorphisms.

In \cite[Theorem 4.9]{LMMN}, the existence of quasi-symmetric or David conjugacies between appropriate Markov maps was established under the additional assumption that the Markov maps induce a continuous map on $\mathbb{S}^1$. 
See also \cite[Theorem 4.1]{LN24a} for a generalization.
We now show that the same result holds even if one drops the assumption of continuity on the whole circle.

\begin{prop}\label{prop:qsmarkovmap}
    Let $\mathcal{F}$ and $\cG$ be two topologically expanding conformal Markov maps on $\mathbb{S}^1$ with puzzle structures.
    Suppose that they are combinatorially conjugate, and suppose that all break-points for $\mathcal{F}$ and $\cG$ are symmetric.
    \begin{enumerate}[leftmargin=8mm]
        \item\label{HH} If the combinatorial conjugacy is type-preserving, then $\mathcal{F}$ and $\cG$ are quasi-symmetrically conjugate.
        \item\label{HP} If the combinatorial conjugacy from $\cF$ to $\cG$ does not send parabolic break-points of $\cF$ to hyperbolic break-points of $\cG$, then $\mathcal{F}$ and $\cG$ are David conjugate. More precisely, there is a homeomorphism $H: \mathbb{S}^1 \longrightarrow \mathbb{S}^1$ that conjugates $\cF$ to $\cG$, and extends continuously to a David homeomorphism of $\D$.
    \end{enumerate}
\end{prop}

Let us briefly discuss the proof of \cite[Theorem 4.9]{LMMN} and \cite[Theorem 4.1]{LN24a} (c.f. \cite[Theorem 1.4]{McM25}).

The proof relies on blowing up two adjacent interval $S = I \cup J$ using the dynamics. 
We remark that the presence of parabolic points does not allow us to blow up the arc $S$ quasisymmetrically to an arc of large diameter, comparable to $1$.

To overcome this obstacle, we first blow up $S$ via some iterate $m$ so that the interval $S'=\mathcal{F}^m(S)$ is close to a break point $a \in \mathrm{bk}(\mathcal{A}^0)$. This procedure is referred to as the \textit{(quasi)conformal elevator} and can be done quasisymmetrically due to our puzzle assumption, using the Koebe distortion theorem (see \cite[Lemma 4.5]{LN24a}).

There is a fundamental dichotomy of the configuration of $S'=\mathcal{F}^m(S)$ with respect to the break-point $a$: either $S'$ lies only on one side of $a$, or $S'$ contains the point $a$.

In the first case, we apply the local dynamics at $a$ to blow up $\mathcal{F}^m(I)$ to an interval of definite size. Depending on whether $a^\pm$ is hyperbolic or parabolic, this gives us different estimates on the diameters of the adjacent intervals.

In the second case, the break-point $a$ divides the interval $S'$ into two intervals, and we apply the above estimate on each interval.

We remark that in the proof, we only apply $\mathcal{F}$ iteratively on intervals that are disjoint from the break points of $\mathcal{A}^0$. Such iterations are still continuous in our setting. Thus the proof of Proposition~\ref{prop:qsmarkovmap} is almost identical to \cite[Theorem 4.9]{LMMN} and \cite[Theorem 4.1]{LN24a}, so we only provide a sketch of the proof.

\begin{proof}
    By Proposition~\ref{prop-combconjimpliestopconj}, $\mathcal{F}$ and $\cG$ are topologically conjugate. Denote the conjugacy by $h$.
    
    Let $I,J\subset \mathbb \mathbb{S}^1$ be adjacent closed arcs each of which has length $t\in (0,1/2)$. We will show that 
\begin{align}\label{distortion:qs}
    \diam h(I) \simeq \diam h(J) \quad \textrm{under condition \eqref{HH}};
\end{align} 
and that
\begin{align}\label{distortion:david}
    \max \left\{\frac{\diam h(I)}{\diam h(J)}, \frac{\diam h(J)}{\diam h(I)} \right\}\lesssim \log(1/t)\quad \textrm{under condition \eqref{HP}}.
\end{align}

    The above estimates, when combined with the Beurling-Ahlfors extension theorem \cite{BA56}, which provides a quasiconformal extension, or with \cite[Theorem 3.1]{Zak08} (cf. \cite[Theorem 3]{CCH96}), which provides a David extension,  complete the proof of the theorem.

    Let $S= I \cup J$. Then there exists an admissible word $w=(j_0,\dots,j_{l-1})$ for the Markov map $\mathcal{F}$ with the property that $S\subset A_w$ but $S$ is not contained in any {\em child} of $A_w$, i.e., any level $|w|+1$ Markov piece that is contained in $A_w$. We say that $w$ \textit{is the word associated to} $S$.
    Let $t\in \{0,\dots,l-1\}$ be such that $t$ is the minimum of the levels of the end-points of $A_w$. We set $v=(j_0,\dots,j_{t-1})$ and $u=(j_{t},\dots,j_{l-1})$. We call $(v,u)$ the \textit{canonical splitting} of the admissible word $w$ and we write $w=(v,u)$. Note that $v$ can be the empty word if $t=0$, but $u$ is never the empty word (see \cite[\S3.3]{LN24a}).

    First we apply $\mathcal{F}^m$ to $S$ for some iterate $m$ depending on the following dichotomy, which follows from \cite[Lemma 4.4 and Lemma 4.5]{LN24a}.
    \begin{enumerate}[label=\normalfont ($A$-\roman*)]
    \item\label{lemma:newelevator:1} The interval $S$ contains a level $|w| + r + 1$ Markov interval, where $r$ is the number of level $0$ Markov intervals.
    Then there exist consecutive level $|w| + r + 1$ Markov intervals $I_1,\dots,I_p$, with
    \begin{align*}
        I_{i_0} \subset S \subset A_w= \bigcup_{i=1}^{p} I_i \quad \textrm{for some $i_0\in \{2,\dots,p-1\}$. }
    \end{align*}
    Let $w=(v,u)$ be the canonical splitting of $w$, and set $m=|v|$ and $n=|u|$. Then $A_u$ has a point $a\in \mathrm{bk}(\mathcal{A}^0)$ as an end-point and 
     $$
     \mathcal{F}^{m}|_{A_w}\colon A_w\to A_u
     $$ 
     is  $\eta$-quasisymmetric, where $\eta(t) \simeq t$.
   
\medskip

    \item\label{lemma:newelevator:2} The interval $S$ does not contain any level $|w| + r + 1$ Markov interval.
    Then there exist non-overlapping intervals $S^-,S^+$ such that $S=S^-\cup S^+$. Set $m = |w|$. Then $\mathcal{F}^{m}$ maps the common end-point of $S^-$ and $S^+$ to a point $a\in \mathrm{bk}(\mathcal{A}^0)$.
    Moreover, if $w^{\pm}$ is the word associated to $S^{\pm}$, then the canonical splitting is of the form $w^{\pm}=(w, u^{\pm})$, and the first alternative \ref{lemma:newelevator:1} is applicable to each of $S^{\pm}$. Finally, $\mathcal{F}^{m}|_S$ is quasisymmetric.
    \end{enumerate}

    Let $S' = \mathcal{F}^m(S)$ and $I' = \mathcal{F}^m(I), J' = \mathcal{F}^m(J)$, where $m = |v|$ in alternative \ref{lemma:newelevator:1}, and $m = |w|$ in alternative \ref{lemma:newelevator:2}. Let $b = h(a)$, where $h$ is the topological conjugacy.

    Suppose that alternative \ref{lemma:newelevator:1} holds. Without loss of generality, we assume that $S' \subseteq [a, w] \in \mathcal{A}^0$, and $a$ is closer to $I'$ than $J'$ in $[a, w]$. Let $s \in \N$ be the smallest so that $J'$ is not contained in the level $s$ Markov interval with $a$ as a left end-point.
    We will blow up the interval $I'$ and $J'$ further by $\mathcal{F}^s$ (see \cite[\S 4.3.1]{LN24a} for more details). Note that since $w$ is the word associated to $S$, the interval $\mathcal{F}^s(S')$ has definite size.

    \begin{itemize}[leftmargin=8mm]
    \item Suppose that $\lambda(a^+) > 0$ (or $\lambda(b^+) > 0$). Then by \cite[Lemma 4.7]{LN24a}, 
    \begin{align}\label{distortion:ai:hyperbolic}
    \diam I'\simeq \diam J' \simeq \lambda(a^+)^{-s} \quad (\text{ or }\diam h(I')\simeq \diam h(J') \simeq \lambda(b^+)^{-s}).
    \end{align}
    \item Suppose that $\lambda(a^+) = 0$ (or $\lambda(b^+) = 0$). Let $l \in \N$ be the smallest so that $I'$ does not intersect the level $s+l$ Markov interval with $a$ as a left end-point. Then by \cite[Lemma 4.7]{LN24a},
    \begin{equation}\label{distortion:ai:parabolic}
    \begin{aligned}
    \diam I' \simeq s^{-\alpha -1} \min\{l,s\}  \quad &\textrm{and} \quad
    \diam J'  \simeq s^{-\alpha-1} \min\{1,s\};\\
    (\text{or }\diam h(I') \simeq s^{-\beta -1} \min\{l,s\}  \quad &\textrm{and} \quad
    \diam h(J')  \simeq s^{-\beta-1} \min\{1,s\}). 
    \end{aligned}
    \end{equation}
    Here, $\alpha=1/N(a^+)$ and $\beta = 1/N(b^+)$. 
    \end{itemize}

    \textbf{Case {H}$\to${H}:} Suppose that $a^+$ and $b^+$ are hyperbolic. Then $$ \diam h(I')\simeq \lambda(b^+)^{-s}\simeq  \diam h(J').$$
    Therefore, we conclude that \eqref{distortion:qs} holds as $\cG^m$ is quasisymmetric on $h(S) = h(I\cup J)$.

    \textbf{Case {H}$\to${P}:} Suppose that $a^+$ is hyperbolic and $b^+$ is parabolic. By \cite[Lemma 4.5]{LN24a}, there exists a constant $\gamma$ so that $\diam I' \gtrsim (\diam I)^{\gamma} \simeq t^\gamma$. Combined with \eqref{distortion:ai:hyperbolic}, we conclude that $\log(1/t)\gtrsim s$.
    By \eqref{distortion:ai:parabolic} and the fact that $\mathcal{F}^m$ is $\eta$-quasisymmetric, we conclude that 
    $$ \max \left\{\frac{\diam h(I)}{\diam h(J)}, \frac{\diam h(J)}{\diam h(I)} \right\}\lesssim s \lesssim \log(1/t).$$
    Therefore,  we conclude that \eqref{distortion:david} holds.

    \textbf{Case {P}$\to${P}:} Suppose that $a^+$ and $b^+$ are parabolic. Let $\alpha=1/N(a^+)$ and $\beta=1/N(b^+)$. By \eqref{distortion:ai:parabolic}, we conclude that 
    \begin{align*} \frac{\diam h(J')}{\diam J'} \simeq s^{\alpha- \beta}, \quad \textrm{and} \quad
    \frac{\diam h(I')}{\diam I'} \simeq s^{\alpha-\beta}.
    \end{align*}
    Altogether, since $\diam I'\simeq \diam J'$, we have $\diam h(I')\simeq \diam h(J')$. The fact that $\cG^{m}|_{h(I\cup J)}$ is quasisymmetric implies the relation~\eqref{distortion:qs}.

    Suppose that alternative \ref{lemma:newelevator:2} holds. Without loss of generality assume that $J'\subset [a,w] \in \mathcal{A}^0$ and $I'=I_1'\cup I_2'$, where $I_2'\subset [a,w]$ and $I_1'\subset [w,a]$. Note that $I_1'$ and $I_2' \cup J'$ satisfy the alternative \ref{lemma:newelevator:1}. Thus, we can blow up the intervals $I_1'$ and $I_2' \cup J'$ using the local dynamics at $a^-$ and $a^+$ respectively to estimate their diameters. The estimates are essentially the same as in the case of alternative \ref{lemma:newelevator:1}, and we refer the reader to \cite[\S 4.3.2]{LN24a} for detailed computations.
\end{proof}
\begin{remark}
    We remark that the conclusion of Proposition~\ref{prop:qsmarkovmap} still holds under a much weaker assumption on the regularity of the map, stated as conditions (M1), (M2), (M3), and (M3*) in \cite{LN24a}. However, we do not need this more general version in the present paper.
\end{remark}

We illustrate an application of the previous result with the following example.
\begin{prop}\label{prop:2tod}
    Let $d, e \geq 2$. There exists a quasi-symmetric map $\phi: \mathbb{S}^1 \cong \R/\Z \longrightarrow \mathbb{S}^1$ that sends the set of $d-$adic rationals $\{\frac {p}{d^n}\}$ onto the set of $e-$adic rationals $\{\frac{q}{e^m}\}$
\end{prop}
\begin{proof}
    Let us prove the case $d = 2$ and $e = 5$. The proof for the general case is similar.
    Let $A_1 = [0, \frac14], A_2 = [\frac14,\frac12], A_3 = [\frac12,\frac58], A_4 = [\frac58,\frac34], A_5=[\frac34,1]$. Let $F_i$, $i\in\{1,\cdots, 5\}$, be the linear map with derivative in $\Z[\frac12]$ so that each $F_i:A_i \longrightarrow [0,1]$ is a homeomorphism. Note that $\mathcal{F}$ induces a degree $5$ topological expanding covering map of $\mathbb{S}^1$, and every break-point is symmetrically hyperbolic. Therefore, by Proposition~\ref{prop:qsmarkovmap} (or by \cite[Theorem 4.9]{LMMN} as $\mathcal{F}$ is continuous), we conclude that $\mathcal{F}$ is quasi-symmetrically conjugate to $t \mapsto 5t$. Denote this conjugacy by $\phi$. Post-composing $\phi$ with some rotations, we assume that $\phi(0) = 0$. By construction, each iterated preimage of $0$ is a dyadic rational, and each dyadic rational is eventually mapped to $0$ under $\mathcal{F}$. Hence, $\phi(\{\frac{p}{2^n}\}) = \{\frac {q}{5^m}\}$.
\end{proof}

\subsection*{Conformal Markov maps on a finite union of circles}
Let $X = \bigcup_{s\in \mathcal{S}} \mathbb{S}^1_s$ be a finite union of circles.
Let $\Phi: \mathcal{S} \longrightarrow \mathcal{S}$ be a map on the index set.
Let $\mathcal{A}= \{\mathcal{A}_s\}_{s\in\mathcal{S}}$ be a partition of $X$, where $\mathcal{A}_s$ is a partition of the circle $\mathbb{S}^1_s$.
Let $\mathcal{F} = \{\mathcal{F}_s\}_{s\in\mathcal{S}}$ be a conformal Markov map on $X$ with respect to the partition $\mathcal{A}$ so that for each $A_{s, i} \in \mathcal{A}_s$, $f_{s,i} \in \mathcal{F}_s$, there exist $j_{i}$ and $l_{i}$ so that 
$$
f_{s,i}: A_{s,i} \longrightarrow A_{\Phi(s), j_i} \cup A_{\Phi(s), j_i+1} \cup ... \cup A_{\Phi(s), l_i}
$$
is a homeomorphism that extends to a conformal map on a neighborhood of $A_{s, i}$.
We call such maps {\em conformal Markov maps on $\bigcup_{s\in \mathcal{S}} \mathbb{S}^1_s$ (modeled by $\Phi$)}.

The notion of Lyapunov exponent and symmetrically hyperbolic/parabolic points generalize naturally to a conformal Markov map $\mathcal{F}$ on a finite union of circles $\bigcup_{s\in \mathcal{I}} \mathbb{S}^1_s$.
The proof of Proposition~\ref{prop:qsmarkovmap} can be easily adapted for this setting yielding the following result.

\begin{prop}\label{prop:qsmarkovmapFiniteCircle}
    Let $\mathcal{F}$ and $\cG$ be two topologically expanding conformal Markov maps on $\bigcup_{s\in \mathcal{S}} \mathbb{S}^1_s$ with puzzle structures.
    Suppose that they are combinatorially conjugate, and suppose that all break-points for $\mathcal{F}$ and $\cG$ are symmetric.
    \begin{enumerate}[leftmargin=8mm]
        \item\label{HHFC} If the combinatorial conjugacy is type-preserving, then $\mathcal{F}$ and $\cG$ are quasi-symmetrically conjugate.
        \item\label{HPFC} If the combinatorial conjugacy from $\cF$ to $\cG$ does not send parabolic break-points of $\cF$ to hyperbolic break-points of $\cG$, then $\mathcal{F}$ and $\cG$ are David conjugate. More precisely, there is a homeomorphism $H: \bigcup_{s\in \mathcal{S}} \mathbb{S}^1_s \longrightarrow \bigcup_{s\in \mathcal{S}} \mathbb{S}^1_s$  that conjugates $\mathcal{F}$ to $\cG$, and extends continuously as a David homeomorphism of $\bigcup_{s\in \mathcal{S}} \D_s$.
    \end{enumerate}
\end{prop}

\section{Basilica $\mathfrak{Q}-$maps}\label{sec-basilica}
Let $Q(z) = z^2-\frac{3}{4}$ be the \emph{fat Basilica} quadratic polynomial, with Julia set $J = J(Q)$. 
The map $Q$ has precisely two periodic bounded Fatou components $U_0, U_0'$, and they form a $2-$cycle (the immediate basin of the parabolic fixed point $-\frac12$). We enumerate them so that $U_0$ contains the critical point $0$, and call $U_0$ the {\em central component}. The unbounded Fatou component of $Q$ (i.e., the basin of infinity) is denoted by $U_\infty$. 

Let $\mathfrak{Q}$ be the collection of local conformal maps of the form $Q^{-n} \circ Q^m$. In this section, we follow the notation of \S~\ref{sec:XC}, and give an explicit construction of a special class of $(J, \mathfrak{Q})-$maps. The construction is both versatile and admits precise geometric control.

We will begin with a construction of puzzles around the Julia set.
The puzzle pieces are constructed by cutting the Julia set using
\begin{itemize}[leftmargin=8mm]
    \item `external puzzle pieces' bounded by external rays and equipotentials; and
    \item `internal puzzle pieces' in bounded Fatou components.
\end{itemize}
Such puzzles are termed \emph{Basilica puzzles} (see Definition~\ref{basi_puzz_def}).
We will then associate a piecewise conformal Markov map, called a {\em Basilica $\mathfrak{Q}-$map}, to a Basilica puzzle (see Definition~\ref{defn:basiqmap}). The action of a Basilica $\mathfrak{Q}-$map on the ideal boundaries of suitable Fatou components of $Q$ induces a piecewise conformal Markov map on a disjoint union of circles. The proof of Theorem~\ref{thm:qcclassfn-ltsets} necessitates an analysis of certain regularity properties of these circle maps, and this is carried out in \S~\ref{ideal_bdry_map_reg_subsec}. It turns out that the circle maps induced by the bounded Fatou components already come with the desired regularity. However, we need to (potentially) modify a Basilica $\mathfrak{Q}-$map (without affecting its topological dynamics on the Julia set) to ensure that the circle map induced by the unbounded Fatou component $U_\infty$ also has the right regularity (see Theorem~\ref{thm:symmetrichyperbolic}).

\subsection{Basilica puzzle}\label{bas_puz_subsec}
In this section, we first construct the internal puzzle pieces in bounded Fatou components. We then explain how to construct Basilica puzzles by gluing corresponding external puzzle pieces.
\subsubsection{Internal angles}
Recall that $U_0$ is the bounded critical Fatou component of $Q$.
The first return map  $Q^2 : U_0 \longrightarrow U_0$ induces a topologically expanding degree $2$ covering on $\partial U_0$.
Thus, there exists a unique topological conjugacy $\phi = \phi_U:  \mathbb{S}^1 \cong \R/\Z \longrightarrow \partial U_0$ so that $Q^2 \circ \phi(t) = \phi(\sigma_2(t))$, where $\sigma_2(t) = 2t$ is the angle doubling map.
This conjugacy $\phi$ gives  combinatorial coordinates for points on $\partial U_0$, which we call the {\em internal angle} in $U_0$.

Let $U$ be a bounded Fatou component.
Then there exists a smallest $l\geq 0$ so that $Q^l(U) = U_0$. Note that $Q^l:U \longrightarrow U_0$ is conformal, and induces a homeomorphism $Q^l:\partial U \longrightarrow \partial U_0$.
We define 
$$
\phi_U = (Q^l|_{\partial U})^{-1} \circ \phi: \mathbb{S}^1 \longrightarrow \partial U.
$$
In this way, we define {\em internal angles} for any bounded Fatou component $U$.

We will refer to $\phi_U(t)$ as the point on $\partial U$ with internal angle $t$.
Note that $\phi_U(t)$ lies on the boundary of precisely two distinct bounded Fatou components (and hence is a contact point of $J(Q)$ of valence $2$) if and only if $t$ is a {\em dyadic rational}, i.e., $t = \frac{p}{2^n}$ for some $p, n \in \N$.

We use $I_U[s,t]$ to denote the curve on $\partial U$ between points at internal angles $s$ and $t$. We also set $I_U(s,t)=I_U[s,t]\,-\, \{I_U(s),I_U(t)\}$. 
\begin{defn}[Dyadic intervals]\label{defn:dyaint}
    Let $[s, t] \subseteq \mathbb{S}^1\cong \R/\Z$ be an interval. We say it is {\em dyadic} if
    \begin{itemize}[leftmargin=8mm]
        \item both $s, t$ are dyadic; and 
        \item any dyadic rational in $(s,t)$ has denominator strictly greater than the denominators of $s$ and $t$ (where the dyadic rationals are expressed in their reduced form).
    \end{itemize}

    We call $I_U[s,t]) \subseteq \partial U$ a {\em dyadic} segment of $\partial U$ if $[s,t] \subseteq \mathbb{S}^1$ is dyadic. 
\end{defn}
We remark that dyadic intervals are precisely the intervals of the form $[\frac{p}{2^n}, \frac{p+1}{2^n}]$. In particular, the angle $0$ does not lie in any dyadic interval.

\begin{lem}\label{lem:mapprimitive}
    Let $I_U[s,t]$ be dyadic. Then there exists $m$ so that $Q^m: I_U(s,t) \longrightarrow I_{U_0}(0,1)$ extends to a conformal map in a neighborhood of $I_U(s,t)$.
\end{lem}
\begin{proof}
    Since $s, t$ are dyadic, there exists a smallest $k$ so that $\sigma_2^k(s) = \sigma_2^k(t) = 0$. Since $[s,t]$ is dyadic, $\sigma_2^k(x) \neq 0$ for any point $x \in (s, t)$.
    Therefore, $\sigma_2^k$ is a homeomorphism between $(s,t)$ and $(0,1)$.

    Now let $l$ be the smallest non-negative integer so that $Q^l(U) = U_0$. Let $m = k+l$. Then $Q^m: I_U(s,t) \longrightarrow I_{U_0}(0,1)$ is a homeomorphism. Since the critical values of $Q^m$ are disjoint from the Julia set, we can choose a simply connected open set $\Omega$ containing $I_{U_0}(0,1)$ which avoids the critical values of $Q^m$. The lemma follows by pulling back this open set $\Omega$ by $Q^m$.
\end{proof}

\begin{cor}
    Let $I_U[s,t], I_{U'}[s',t']$ be dyadic. Then there exist $n, m$ so that $Q^{-n} \circ Q^m: I_U(s,t) \longrightarrow I_{U'}(s',t')$ extends to a conformal map in a neighborhood of $I_U(s,t)$, where $Q^{-n}$ is an appropriate inverse branch of $Q^n$.
\end{cor}

\subsubsection{Internal puzzle}\label{subsubsec:internalpuzzle}
In this subsection, we will first canonically assign an internal puzzle piece $P_U[s,t]$ to any dyadic segment $I_U(s,t)$.
We use it to define internal puzzles.

We start with $U_0$.
Recall that $Q^2:U_0 \longrightarrow U_0$ is conformally conjugate to the parabolic quadratic Blaschke product $B: \D \longrightarrow \D$ (normalized to have its critical point at $0$ and parabolic fixed point at $1$).
By Schwarz reflection, $B$ extends to the rational map $B(z)=\frac{3z^2+1}{z^2+3}$ on $\widehat\C$.
Let $\PC(B) \subseteq \widehat\C$ be the post-critical set; i.e., the closure of the orbits of the critical values of $B$.
Fix an $\epsilon > 0$, and let $N$ be an $\epsilon-$neighborhood of $\mathbb{S}^1-\{1\}$ with respect to the hyperbolic metric on $\widehat\C - \PC(B)$.
Let $n \geq 0$ and let $N'$ be any component of $B^{-n}(N)$.
Since $B$ is expanding with respect to the hyperbolic metric, we conclude that $N' \subseteq N$.

\begin{defn}[Internal puzzle piece]\label{defn:internalpuzzle}
    We define the {\em internal puzzle piece} 
    $$
    P_{U_0}[0,1] \subseteq \overline{U_0}
    $$ 
    associated with $I_{U_0}(0,1)$ as $\overline{\Phi^{-1}(N \cap \D)}$, where $\Phi: U_0 \longrightarrow \D$ is the conformal conjugacy between $Q^2$ and the Blaschke product $B$.

    Let $I_U(s,t)$ be dyadic. We define the corresponding {\em internal puzzle piece} 
    $$
    P_{U}[s,t] \subseteq \overline{U}
    $$ 
    as the pull-back of $P_{U_0}[0,1]$ under a conformal restriction of $Q^m$ that carries $I_U(s,t)$ to $I_{U_0}(0,1)$. More precisely, $P_{U}[s,t]$ is the closure of the component of $(Q^m)^{-1}(\Int (P_{U_0}[0,1]))$ that has $I_U(s,t)$ on its boundary.
\end{defn}

Note that the parabolic fixed point $-\frac12$ necessarily lies on the boundary of two internal puzzle pieces contained in $\overline{U_0}$.
We also remark that the inner boundaries of these internal puzzle pieces play the role of the union of external rays and equipotentials for standard puzzle pieces for polynomials.
In particular, we have the following nested property.
\begin{lem}
    Let $I_U[s,t], I_U[s',t']$ be dyadic. Suppose $I_U(s,t) \subseteq I_U(s',t')$. Then $P_U[s,t] \subseteq P_{U}[s',t']$.
\end{lem}
\begin{proof}
    Let $l$ be the smallest non-negative integer so that $Q^l(U) = U_0$. Then $Q^l(P_U[s,t]) = P_{U_0}[s,t]$ and $Q^l(P_U[s',t']) = P_{U_0}[s',t']$. By conjugating the first return map $Q^2$ on $U$ to the parabolic Blaschke product, we see that $P_{U_0}[s,t] \subseteq P_{U_0}[s',t']$. Therefore, $P_U[s,t] \subseteq P_{U}[s',t']$ as $Q^l:\overline{U}\to \overline{U_0}$ is a homeomorphism.
\end{proof}

We define internal puzzles as collections of internal puzzle pieces that satisfy certain connectedness conditions as follows. 
\begin{defn}[Internal puzzle]\label{defn:intpuz}
Let $U_0, U_1,\ldots, U_r$ be a collection of bounded Fatou components which include the central component $U_0$ such that $\bigcup \overline{U_i}$ is connected.
Let $I_{U_i}[s_{i,j}, s_{i,j+1}]$ be a collection of partitions of $\partial U_i$ so that each piece is dyadic.
Let $P_{U_i}[s_{i,j}, s_{i,j+1}]$ be the corresponding internal puzzle pieces.
We call any finite collection of internal puzzle pieces that arise in this way an {\em internal puzzle} (see Figure~\ref{fig:BasPuz}), and denote it by 
$$
(\{U_i\}, P_{U_i}[s_{i,j}, s_{i,j+1}]).
$$
\end{defn}

 \begin{figure}[htp]
 \captionsetup{width=0.96\linewidth}
    \centering
    \includegraphics[width=0.9\textwidth]{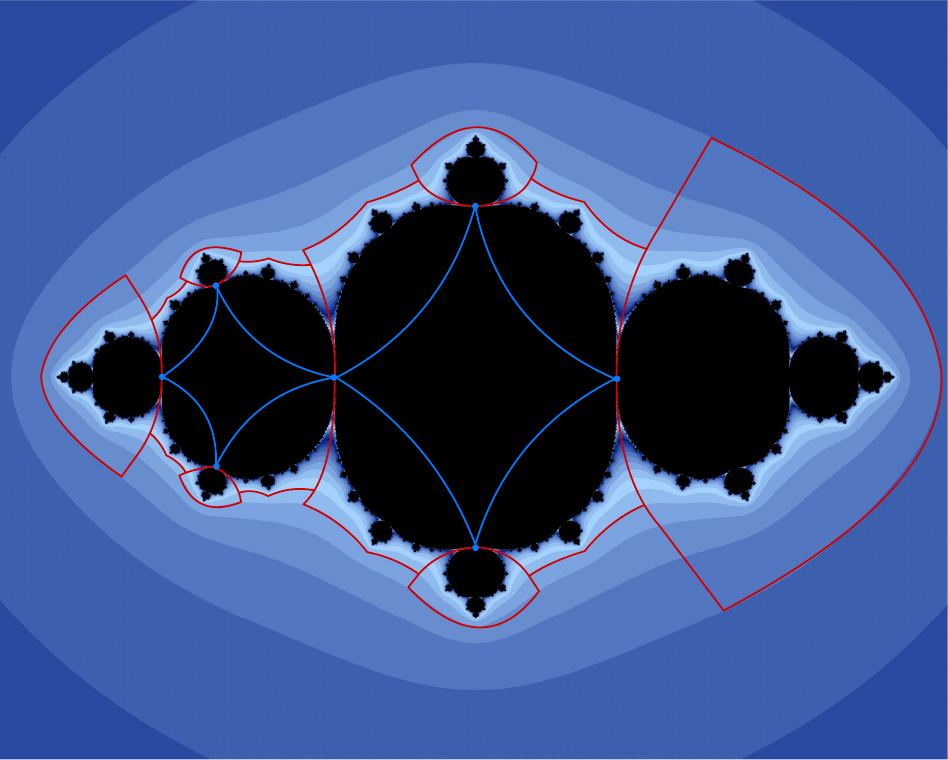}
    \caption{An illustration of the Basilica puzzle. The blue arcs correspond to the inner boundary of the internal puzzle pieces. The external equipotentials are selected according to Lemma~\ref{lem:conformalBlowUp}. For example, the equipotential for the right most limb type puzzle piece has level $\frac{1}{2}$. The red curves are composed of parts of equipotential lines and external rays. The Basilica puzzle pieces are the closed bounded regions cut off by the union of the blue and red curves.}
    \label{fig:BasPuz}
\end{figure}

\subsubsection{From internal puzzle to Basilica puzzle}\label{subsubsec:indbp}
In this subsection, we define and construct a Basilica puzzle from an internal puzzle.

Recall that $U_\infty$ is the unbounded Fatou component of $Q$. Then, there exists a conformal map $\phi_\infty:\D\to U_\infty$ conjugating $z^2$ to $Q$. Since $\partial U_\infty=J(Q)$ is locally connected, the Carath\'eodory theorem implies that $\phi_\infty$ extends to a continuous semi-conjugacy $\phi_\infty:\mathbb{S}^1\to J(Q)$ between $z^2$ and $Q$. Via the Riemann uniformization $\phi_\infty$, the \emph{ideal boundary} $\partial^I U_\infty$ of $U_\infty$ (i.e., the set of prime ends of $U_\infty$) is identified with $\mathbb{S}^1$, the set of external angles. The image of the radial line in $\D$ at angle $t$ under the conformal map $\phi_\infty$ is called the \emph{external ray} at angle $t$, and is denoted by $R_\infty(t)$.
We call $\phi_\infty(t)$ (the landing point of $R_t$) the point on $J(Q)$ with external angle $t$. We also use $I_\infty[s,t]$ to denote the interval on $\partial^I U_\infty$ between $s$ and $t$.

Let $(\{U_i\}, P_{U_i}[s_{i,j}, s_{i,j+1}])$ be an internal puzzle.
Let $\{t_k\}$ be the collection of external angles landing at $\phi_{U_i}(s_{i,j})$ for $i, j$. We call the collection $\{t_k\}$ the external angles corresponding to $(\{U_i\}, P_{U_i}[s_{i,j}, s_{i,j+1}])$.
Since $s_{i,j}$ is dyadic, there are exactly two external rays of $Q$ landing at $\phi_{U_i}(s_{i,j})$.

The external rays $R_\infty(\frac{1}{3})$ and $R_\infty(\frac{2}{3})$ land at the parabolic fixed point $-\frac12$ of $Q$. The angles $\frac13, \frac23$ belong to the collection $\{t_k\}$.
Denote by $W_0$ the component of $\C - \left(\overline{R_\infty(\frac{1}{3})} \cup \overline{R_\infty(\frac{2}{3})}\right)$ that contains the critical point $0$.

\begin{lem}\label{lem:conformalBlowUp}
Let $(\{U_i\}, P_{U_i}[s_{i,j}, s_{i,j+1}])$ be an internal puzzle with corresponding external angles $\{t_k\}$.
Let $R_\infty(t_k)$ and $R_\infty(t_{k+1})$ be two adjacent external rays.

Suppose that $R_\infty(t_k), R_\infty(t_{k+1})$ land at the same point. Let $W$ be the component of $\C - \left(\overline{R_\infty(t_k)} \cup \overline{R_\infty(t_{k+1})}\right)$ that is disjoint from $U_0,\ldots, U_r$. Then there exists a unique integer $m \geq 0$ so that 
$$
Q^m : W \longrightarrow W_0
$$
is conformal.

Suppose that $R_\infty(t_k), R_\infty(t_{k+1})$ land at two different points on $\partial U_i$. Let $W$ be the component of $\C - \left(\overline{R_\infty(t_k)} \cup \overline{R_\infty(t_{k+1})} \cup \overline{U_i}\right)$ that is disjoint from $U_0,\ldots, U_r$.
Then there exists a unique integer $m \geq 0$ so that 
$$
Q^m : W \longrightarrow W_0 - \overline{U_0}
$$
is conformal.
\end{lem}
\begin{remark}
    As every contact point of $J(Q)$ eventually maps to the parabolic fixed point $-\frac12$, and since the external rays $R_\infty(\frac13), R_\infty(\frac23)$ landing at $-\frac12$ participate in the collection $\{R_\infty(t_k)\}$, it follows that the two possibilities treated in Lemma~\ref{lem:conformalBlowUp} are exhaustive.
\end{remark}
\begin{proof}
    Suppose that $R_\infty(t_k)$ and $R_\infty(t_{k+1})$ land at the same point. Then the interval $(t_k, t_{k+1}) \subseteq \mathbb{S}^1$ contains all the angles of external rays in $W$. Then there exists $m$ so that $\sigma_2^m: (t_k, t_{k+1}) \longrightarrow (\frac{2}{3}, \frac{1}{3})$ is a homeomorphism. Thus, the corresponding map $Q^m: W \longrightarrow W_0$ is conformal.

    Suppose that $R_\infty(t_k)$ and $R_\infty(t_{k+1})$ land at two different points on $\partial U_i$. Then the landing points of $R_\infty(t_k)$ and $R_\infty(t_{k+1})$ bound a dyadic segment $I_{U_i}[s, t] \subseteq \partial U_i$. By Lemma~\ref{lem:mapprimitive}, there exists $m$ so that $Q^m:I_{U_i}(s,t) \longrightarrow I_{U_0}(0,1)$ is a homeomorphism. It easily follows from this observation that $Q^m : W \longrightarrow W_0 - \overline{U_0}$ is conformal.
\end{proof}

We now define a canonical way to extend an internal puzzle to the exterior of the filled Julia set.

Let $E_k$ be the part of the equipotential at level $\frac{1}{2^m}$ between the two adjacent external rays $R_\infty(t_k)$ and $R_\infty(t_{k+1})$, where $m$ is the unique integer in Lemma~\ref{lem:conformalBlowUp}.
We will call $E_k$ the {\em canonical equipotential} between $R_\infty(t_k)$ and $R_\infty(t_{k+1})$.
We call the union
$$
\Gamma:= \bigcup_k \overline{R_\infty(t_k)} \cup \bigcup_k E_k \cup \bigcup_{i,j} (\partial P_{U_i}[s_{i,j}, s_{i,j+1}]) \cap U_i)
$$
the {\em canonical extension graph} for $(\{U_i\}, P_{U_i}[s_{i,j}, s_{i,j+1}])$.
Note that $\Gamma$ cuts the Riemann sphere $\widehat\C$ into finitely many pieces.
(See Figure~\ref{fig:BasPuz}, where the external rays $R_\infty(t_k)$ and $R_\infty(t_{k+1})$ are only drawn up to equipotential level $\frac{1}{2^m}$).

\begin{defn}[Basilica puzzle]\label{basi_puzz_def}
Let $(\{U_i\}, P_{U_i}[s_{i,j}, s_{i,j+1}])$ be an internal puzzle with corresponding external angles $\{t_k\}$, and canonical extension graph $\Gamma$.

We call the closure of a complementary component of $\Gamma$ that intersects the Julia set a {\em Basilica puzzle piece}. 
We call the collection of all Basilica puzzle pieces a {\em Basilica puzzle} (induced by $(\{U_i\}, P_{U_i}[s_{i,j}, s_{i,j+1}])$), and denote it by $\{P_i\}_{i\in \mathcal{I}}$, where $\mathcal{I}$ is some index set.

Let $P_i$ be a Basilica puzzle piece. It is said to be of {\em limb type} or  called a \emph{limb puzzle piece} if $\Int P_i$ does not intersect any bounded Fatou components in the list $U_0,\ldots, U_r$. It is said to be of {\em torso type} or called a \emph{torso puzzle piece} otherwise. We say that it is of torso type for $U_j$ if it intersects $U_j$.
\end{defn}
\begin{rmk}
1) We think of the union $\bigcup_{i=0}^r U_i$ as the `torso' of the filled Julia set. The other bonded Fatou components are thought of as `limbs'. Thus, the nomenclature torso/limb type tells us whether or not a puzzle piece intersects the torso of the filled Julia set.

   2)  If a Basilica puzzle piece $P$ is of limb type, then $\partial P \cap J(Q)$ consists of a single point. Otherwise, $P$ is of torso type for some bounded Fatou component $U_j$, and $\partial P \cap J(Q)$ consists of exactly two points on $\partial U_j$.
\end{rmk}

By our choice of equipotentials, we have the following dynamical compatibility.
\begin{lem}[Canonical maps between puzzle pieces]\label{lem:TypePreservingPuzzleMap}
    Let $\mathfrak{P}=\{P_i\}_{i\in \mathcal{I}}$, $\mathfrak{S}=\{S_j\}_{j\in \mathcal{J}}$ be two Basilica puzzles.
    Suppose that $P_i$ and $S_j$ are of the same type.
    Then there exist unique $m, n \geq 0$ so that $F_{i,j} = Q^{-n} \circ Q^m : \Int S_j \longrightarrow \Int P_i$ is conformal.
\end{lem}
\begin{remark}
    We will call $F_{i,j}$ the {\em canonical map} between $P_i, S_j$.
\end{remark}
\begin{proof}
    Suppose that $S$ is a limb puzzle piece. Then $S$ is bounded by two external rays $R_{\infty}(s), R_{\infty}(t)$ with the canonical equipotential. Let $P_{limb}$ be the closure of the component bounded by $R_{\infty}(1/3), R_{\infty}(2/3)$ and the equipotential of level $1$ that contains $U_0$. Then by Lemma~\ref{lem:conformalBlowUp}, there exists $m$ so that $Q^m: \Int(S) \longrightarrow \Int(P_{limb})$ is conformal. Similarly, there exists $n$ so that $Q^n: \Int(P) \longrightarrow \Int(P_{limb})$ is conformal, and the lemma follows in this case.

    Suppose that $S$ is a torso puzzle piece. Then $\partial S$ intersect the Julia set at two points, which bounds a dyadic segment $I_{U_i}[s,t]$ for some bounded Fatou component $U_i$. Let $P_{imm}:= \left(P_{limb} - U_0\right) \cup P_{U_0}[0,1]$. Then by Lemma~\ref{lem:conformalBlowUp} and the definition of the internal puzzle (Definition~\ref{defn:internalpuzzle}), there exists $m$ so that $Q^m: \Int(S) \longrightarrow \Int(P_{imm})$ is conformal. Similarly, there exists $n$ so that $Q^n: \Int(P) \longrightarrow \Int(P_{imm})$ is conformal, and the lemma follows in this case as well.
\end{proof}

\subsection{Nested puzzles and conformal Markov map}
In this subsection, we introduce special types of refinements of Basilica puzzles and naturally extract conformal Markov maps associated with such Basilica puzzles.

\begin{defn}[(Markov) refinement of Basilica puzzles]\label{defn:refineBp}
Let $\mathfrak{P} = \{P_i\}_{i\in \mathcal{I}}$ and $\mathfrak{S} = \{S_j\}_{j\in \mathcal{J}}$ be two Basilica puzzles. 
We say that $\mathfrak{S}$ is a {\em refinement} of $\mathfrak{P}$ if for each $i \in \mathcal{I}$, there exists a subset $\mathcal{J}_i \subseteq \mathcal{J}$ so that
    $$
    P_i \cap J(Q) = \bigcup_{j \in \mathcal{J}_i} S_{j} \cap J(Q).
    $$
We say that a map $\Psi: \mathcal{J} \longrightarrow \mathcal{I}$ (between the index sets) is {\em type preserving} if $S_j$ and $P_{\Psi(j)}$ are of the same type for all $j \in \mathcal{J}$. 
We call a refinement $\mathfrak{S}$ of $\mathfrak{P}$ together with a type preserving map $\Psi:\mathcal{J} \longrightarrow \mathcal{I}$ a {\em Markov refinement} of $\mathfrak{P}$.
\end{defn}

\begin{remark}
By construction, if two Basilica puzzle pieces satisfy $S\cap J(Q) \subseteq P\cap J(Q)$, then $S \subseteq P$.
    Therefore, if $\mathfrak{S}$ is a refinement of $\mathfrak{P}$, then each puzzle piece $S_j, j\in \mathcal{J}$, is contained in $P_i$ for some $i = i(j) \in \mathcal{I}$.
    
We also remark that it is convenient to use the following notation to record a sub-class of Markov refinements.

Let $\mathfrak{P} =\mathfrak{P}^0$, and $  \mathfrak{P}^1$ be two Basilica puzzles with index sets $\mathcal{I}$ and $\mathcal{I}^1 \subseteq \mathcal{I} \times \mathcal{I}$ (called the set of {\em admissible} words of length $2$) for $\mathfrak{P}^0$ and $\mathfrak{P}^1$, respectively, so that 
for each $(ij) \in \mathcal{I}^1$, we have that
\begin{itemize}
    \item (nested) $P_{ij} \subseteq P_i$;
    \item (type-preserving) $P_{ij}$ and $P_j$ are of the same type.
\end{itemize}
Then $\mathfrak{P}^1$ is a Markov refinement of $\mathfrak{P}^0$; indeed, in this case, the type preserving map $\Psi: \mathcal{I}^1 \longrightarrow\mathcal{I}$ sends $(ij)$ to $j$.

\end{remark}
    The nomenclature `Markov refinement' and `admissible word' above are justified by the following result that allows one to associate a piecewise conformal Markov map $\cF$ to the pair $(\mathfrak{P}^1,\mathfrak{P}^0)$ such that $\mathfrak{P}^0,\mathfrak{P}^1$ become level $0$ and level $1$ puzzles for $\cF$ (cf. \S~\ref{subsec:markovp}).
\begin{prop}\label{prop:MarkovMapRefine}
Let $\mathfrak{P} = \{P_i\}_{i\in \mathcal{I}}$, and 
let $\mathfrak{P}^1=\{P_{ij}\}_{(ij)\in \mathcal{I}^1}$ be a Markov refinement of $\mathfrak{P}$.
Then for each $(ij) \in \mathcal{I}^1$, there exists an induced homeomorphism $F_{ij}$ where
    $$
    F_{ij} = Q^{-n}\circ Q^m: P_{ij} \longrightarrow P_{j}
    $$
for some $n = n(i,j), m = m(i,j)$ so that $F_{ij}$ is conformal in the interior.
\end{prop}
\begin{proof}
 This is an immediate consequence of Lemma~\ref{lem:TypePreservingPuzzleMap}.
\end{proof}

\begin{defn}[Basilica $\mathfrak{Q}-$map]\label{defn:basiqmap}
    Let $\mathfrak{P} = \{P_i\}_{i\in \mathcal{I}}$ and let $\mathfrak{P}^1=\{P_{ij}\}_{(ij)\in \mathcal{I}^1}$ be a Markov refinement of $\mathfrak{P}$.
    We call the collection of maps $\mathcal{F} = \{F_{ij}: (ij) \in \mathcal{I}^1\}$ in Proposition~\ref{prop:MarkovMapRefine} a {\em Basilica $\mathfrak{Q}-$map}, and write it as 
    $$
    \mathcal{F}: \mathfrak{P}^1 \rightarrow \mathfrak{P}.
    $$
\end{defn}

We will call $\mathfrak{P}$ and $\mathfrak{P}^1$ the level $0$ and level $1$ puzzle, respectively.
One can define the level $n$ puzzle by pulling $\mathfrak{P}$ back by $\cF^n$, and we denote it by $\mathfrak{P}^n = \{P_w\}_{w \in \mathcal{I}^n}$, where $\mathcal{I}^n$ consists of {\em admissible words} of length $n+1$.
Note that 
$$
\mathcal{F}^n: \mathfrak{P}^n \rightarrow \mathfrak{P}^0 = \mathfrak{P}
$$
is also a Basilica $\mathfrak{Q}-$map.
\begin{rmk}
    We remark that  a Basilica $\mathfrak{Q}-$map $\mathcal{F}$ is a Markov $(J, \mathfrak{Q})-$map  in the sense of \S~\ref{sec:XC} that admits a puzzle structure.

    Recall that a Basilica $\mathfrak{Q}-$map is topologically expanding if it is topologically expanding as a Markov $(J, \mathfrak{Q})-$map. It is easy to verify that this happens if and only if every nested intersection $\bigcap_w P_{w}$ is a singleton set.
\end{rmk}

In the next subsection, we will alter Markov refinements of Basilica puzzles and the associated Basilica $\mathfrak{Q}-$maps appropriately to guarantee that they are $C^1$ at periodic break-points. To streamline the exposition, we formalize the alteration steps below. The first step (refinement) only changes the puzzles, not the maps. On the other hand, the second step (modification) does not alter the level $0$ puzzle, but only changes the level $1$ puzzle and the associated Basilica $\mathfrak{Q}-$map.

\begin{defn}\label{def-modfn}
  Let  $
    \mathcal{F}: \mathfrak{P}^1 \rightarrow \mathfrak{P} 
    $ and 
    $
    \widehat{\mathcal{F}}: \widehat{\mathfrak{P}}^{1} \rightarrow \widehat{\mathfrak{P}}
    $
    be Basilica $\mathfrak{Q}-$maps.  We say that 
    $
    \widehat{\mathcal{F}}$ is a \emph{refinement} of $\mathcal{F}$ if
    \begin{itemize}
        \item $\widehat{\mathfrak{P}}^{1}$ and $\widehat{\mathfrak{P}}$ are refinements of $\mathfrak{P}^1$ and $\mathfrak{P}$ respectively; and
        \item for each puzzle piece $P \in \widehat{\mathfrak{P}}^1$, we have $\widehat{\mathcal{F}}|_P = \mathcal{F}|_P$.
    \end{itemize}
    We say that $\widehat{\mathcal{F}}$ is a \emph{modification} of $\mathcal{F}$ if
    \begin{itemize}
        \item $\widehat{\mathfrak{P}} = \mathfrak{P}$; and
        \item $\widehat{\mathcal{F}}$ is topologically conjugate to $\mathcal{F}$ on the Julia set.
    \end{itemize}
    We say that $\widehat{\mathcal{F}}$ is a \emph{modified refinement} of $\mathcal{F}$ if
    \begin{itemize}
        \item $\widehat{\mathfrak{P}}^{1}$ and $\widehat{\mathfrak{P}}$ are refinements of $\mathfrak{P}^1$ and $\mathfrak{P}$ respectively; and
        \item $\widehat{\mathcal{F}}$ is topologically conjugate to $\mathcal{F}$ on the Julia set.
    \end{itemize}
\end{defn}
\begin{remark}
    We remark that a modified refinement $\widehat{\mathcal{F}}$ can be obtained from $\mathcal{F}$ via a refinement and a modification. Also note that the above notion generalizes naturally to any $(X, \mathscr{C})$ Markov maps.
\end{remark}

\subsection{Induced Markov map on ideal boundaries}\label{ideal_bdry_map_reg_subsec}
In this subsection, we study induced Markov maps on the ideal boundaries of Fatou components, and prove the key theorem for Basilica $\mathfrak{Q}-$maps (Theorem~\ref{thm:symmetrichyperbolic}).

\begin{lem}
Let $\mathcal{F}: \mathfrak{P}^1 \rightarrow \mathfrak{P}$ be a Basilica $\mathfrak{Q}-$map.
Let $U_0,\ldots, U_r$ be the chain of bounded Fatou components associated to $\mathfrak{P}$, and let $U_\infty$ be the unbounded Fatou component.
Then $\mathcal{F}$ induces a conformal Markov map 
$$
\bigcup_{s\in \{0,\ldots, r, \infty\}} \mathcal{F}_s : \bigcup_{s\in \{0,\ldots, r, \infty\}}  \partial^I U_s \longrightarrow \bigcup_{s\in \{0,\ldots, r, \infty\}}  \partial^I U_s
$$ 
with puzzle structure, where $\partial^I U_s \cong \mathbb{S}^1$ is the ideal boundary of $U_s$.

Moreover, if $\mathcal{F}$ is topologically expanding, then $\bigcup_s \mathcal{F}_s$ is topological expanding.
\end{lem}
\begin{remark}
    We remark that here the map is defined on a finite disjoint union of circles, which we can regard as some (disconnected) closed subset of $\widehat{\C}$.
    The notion of conformality, topological expansion, puzzle structures (introduced in \S~\ref{sec:XC}) generalize naturally to maps on a finite union of circles.
\end{remark}
\begin{proof}
    Let $\Phi_s: \D \longrightarrow U_s$ be the uniformization map. Then for each Basilica puzzle piece $P_i \in \mathfrak{P}$ or $P_{ij} \in \mathfrak{P}^1$, we can pull back $P_i \cap U_s$ or $P_{ij}\cap U_s$ by the uniformization map $\Phi_s$. Thus, $\mathfrak{P}$ and $\mathfrak{P}^1$ induce two partitions of $\bigcup_s \partial^I U_s$, which we denote by $\mathcal{A}_s$ and $\mathcal{A}_s^1$. The Basilica $\mathfrak{Q}-$map $\mathcal{F}$ is conjugated via $\Phi_s$ to a Markov map on $\bigcup_s \partial^I U_s$.
    By Schwarz reflection, the induced map can be extended to a conformal map on a neighborhood of each interval in $\mathcal{A}^1_s$. The fact that it has a puzzle structure follows from the fact that $\mathfrak{P}^1$ is a refinement of $\mathfrak{P}$ and Proposition~\ref{prop:MarkovMapRefine}. The moreover part follows from the fact that if $\bigcap_w P_w$ is a singleton, then the corresponding nested intersection of the Markov pieces in $\mathcal{A}_s^n$ is a singleton set.
\end{proof}

It is easy to see that the induced Markov map on the union of circles decomposes into two parts:
\begin{align}
    \mathcal{F}_{\bdd}:=\bigcup_{s\in \{0,\ldots, r\}} \mathcal{F}_s &: \bigcup_{s\in \{0,\ldots, r\}}  \partial^I U_s \longrightarrow \bigcup_{s\in \{0,\ldots, r\}}  \partial^I U_s, \text{ and } \label{f_bdd}\\
    \mathcal{F}_{\infty} &:\partial^I U_\infty\longrightarrow \partial^I U_\infty. \label{f_infty}
\end{align}
We study the break-points of these two maps in \S~\ref{sss:bdd} and \S~\ref{sss:infinity}.

\subsubsection{Bounded Fatou component}\label{sss:bdd}
\begin{prop}\label{prop:bdd_fatou_symm_para}
    Suppose that $\mathcal{F}$ is topologically expanding.
    Then all break-points of the induced Markov map $\mathcal{F}_{\bdd}$ in \eqref{f_bdd} are symmetrically parabolic.
\end{prop}
\begin{proof}
    Let $a$ be a break-point of $\mathcal{F}_{\bdd}$.
    We will show that $\lambda(a^\pm) = 0$ and $N(a^\pm) = 2$.
    Suppose that the tail of the orbit $a_n^+$ consists of a periodic orbit. For simplicity of  presentation, we assume that this periodic orbit is a fixed point, denoted by $x$. Let $I^+ \subseteq \bigcup_{s\in \{0,\ldots, r\}}  \partial^I U_s$ be the Markov interval having $x$ on its boundary and lying on the right of $x$, and let $S$ be the corresponding level $1$ Basilica puzzle piece for $I^+$. Let $p \in \partial S$ be the point associated to $x$.
    Note that $\mathcal{F}|_{S} = Q^{-m} \circ Q^n$ for some $m, n \geq 0$. Since $p$ is fixed and $\mathcal{F}$ is topologically expanding, $n > m$ and $\mathcal{F}|_{S} = Q^{-m} \circ Q^{n-m} \circ Q^m$, where $Q^{-m}$ is the inverse of $Q^m|_{S}$ when restricting on $Q^m(S)$. Moreover, $Q^m(p)$ is the parabolic fixed point of $Q$. Thus $\mathcal{F}|_{S}$ is conjugate to $Q^{n-m}$ near $p$. Therefore, the induced map $\mathcal{F}_{\bdd}$, restricted to $I^+$, has a parabolic fixed point at $x$ with multiplicity $2$.
    Therefore, $\lambda(a^+) = 0$ and $N(a^+) = 2$. The same is true for $a_n^-$. We conclude that all break-points are symmetrically parabolic.
\end{proof}

\subsubsection{Unbounded Fatou component}\label{sss:infinity}
The following result is one of the key technical ingredients in the proof of our main quasiconformal equivalence theorem (Theorem~\ref{thm:qcclassfn-ltsets}).

\begin{theorem}\label{thm:symmetrichyperbolic}
    Suppose that $\mathcal{F}$ is topologically expanding.
    Then there exists a modified refinement $\widetilde{\mathcal{F}}: \widetilde{\mathfrak{P}}^1 \rightarrow \widetilde{\mathfrak{P}}$ of $\mathcal{F}: \mathfrak{P}^{1} \rightarrow \mathfrak{P}$ (in the sense of Definition~\ref{def-modfn}) whose induced Markov map $\widetilde{\mathcal{F}}_\infty$ is symmetrically hyperbolic.
    Note that in particular, $\widetilde{\mathcal{F}}$ is topologically conjugate to $\mathcal{F}$ on the Julia set.
\end{theorem}

For the rest of this section, we will prove the above theorem.
We fix a topologically expanding Basilica $\mathfrak{Q}-$map $\mathcal{F}: \mathfrak{P}^1 \rightarrow \mathfrak{P}$.
Let 
$$
\mathcal{F}_\infty: \mathbb{S}^1 \cong \partial^I U_\infty \longrightarrow \mathbb{S}^1 \cong \partial^I U_\infty
$$ 
be the induced Markov map on $\partial^I U_\infty$.
Each puzzle piece $P_i \in \mathfrak{P}$ (or $P_{ij} \in \mathfrak{P}^1$) induces a closed interval, denoted by $A_i$ (or $A_{ij}$), on $\mathbb{S}^1 \cong \partial^I U_\infty$.
Denote this corresponding partition of $\mathbb{S}^1$ for $\mathfrak{P}$ by $\mathcal{A} = \mathcal{A}_\infty = \{A_i\}_{i\in \mathcal{I}}$ and for $\mathfrak{P}^1$ by $\mathcal{A}^1 = \mathcal{A}^1_\infty = \{A_{ij}\}_{(ij) \in \mathcal{I}^1}$.
We say that $A_i$ (or $A_{ij}$) is of limb type or torso type if the corresponding puzzle piece is so.
Similarly, we denote the corresponding induced partition for $\mathfrak{P}^n$ by $\mathcal{A}^n$.

We remark that for any $A_{ij} \in \mathcal{A}^1$ with the corresponding puzzle piece $P_{ij}$, we have that $\mathcal{F}|_{P_{ij}} = F_{ij} = Q^{-n}\circ Q^m$ for some $m,n \geq 0$.
We denote the difference by
$$
N(A_{ij}) = N(P_{ij}) = m-n.
$$
Then the derivative 
$$
D\mathcal{F}_\infty|_{A_{ij}} = 2^{N(A_{ij})}.
$$ 

Since the induced map $\mathcal{F}_\infty$ maps break-points to break-points, every break-point is pre-periodic, and periodic break-points are contained in  $\bigcup_{i\in\mathcal{I}} \partial A_i$.
We first show that the boundary points of a limb type Markov piece $A \in \mathcal{A}$ are never periodic.
\begin{lem}\label{lem:strictpre}
    Suppose $A = [x,y] \in \mathcal{A}$ is of limb type. Then the right orbit $x_n^+$ of $x$ and the left orbit $y_n^-$ of $y$ are strictly pre-periodic.
\end{lem}
\begin{proof}
    Suppose that the right orbit of $x$ is periodic, with period $p$. 
    Let $B \in \mathcal{A}^p$ be the partition piece having $x$ on its boundary and lying on the right of $x$. Then $\mathcal{F}_\infty^p|_{B}$ fixes $x$.
    Since $\mathcal{F}$ is topologically expanding, $B \subsetneq A$.
    Note that by Proposition~\ref{prop:MarkovMapRefine} and construction, $\mathcal{F}_\infty$ maps an interval of $\mathcal{A}^1$ to an interval of $\mathcal{A}$.
    Therefore, $\mathcal{F}_\infty^p: B \longrightarrow A$ is a homeomorphism. In particular, $B$ is of limb type. But this is not possible as the only limb type Markov piece of $\bigcup_n \mathcal{A}^n$ whose boundary contains $x$ is $A$.
    The same argument also applies to the left orbit of $y$, and the lemma follows.
\end{proof}

We now turn to the torso puzzle pieces.

\begin{lem}\label{lem:2k}
   Suppose $A_i = [x,y] \in \mathcal{A}$ is of torso type, and the right orbit $x_n^+$ is periodic with period $q$. Then $D\mathcal{F}_\infty^q(x^+) = 2^{2k}$ for some $k \in \N$.
\end{lem}
\begin{proof}
    Let $P_w \in \mathfrak{P}^q$ and $P_i \in \mathfrak{P}$ be the level $q$ and level $0$ puzzle pieces having $x$ on their boundary and lying on the right of $x$.
    Then 
    $$
    F_w = \mathcal{F}^q|_{P_w} = Q^{-n} \circ Q^m: P_w \longrightarrow P_i.
    $$
    Assume that both $P_w, P_i$ are of torso type for a bounded Fatou component $U$, and $P_w \cap U = P_U[a,b], P_i \cap U = P_U[a,c]$. Let $k \in \N$ so that $2^k = |[a,c]|/|[a,b]|$. Then $m-n = 2k$, where the factor $2$ comes from the fact that the first return map of $U_0$ has period $2$ (i.e., the angular width of an arc on $\partial U_0$ is doubled under $Q^2$). Therefore, $D\mathcal{F}_\infty^q(x^+) = 2^{2k}$.
\end{proof}

\begin{lem}\label{lem:fac2}
    Let $P$ be a Basilica puzzle piece of torso type. Let $S_1, S_2$ be two Basilica puzzle pieces of torso type for $U$ that are both mapped to $P$ under $\mathcal{F}$. Denote $S_{i} \cap U = P_{U}[s_i, t_i]$. 
    Suppose that $|[s_1, t_1]| = 2^k |[s_2, t_2]|$. Then
    $$
    N(S_{2}) - N(S_{1}) = 2k.
    $$
\end{lem}
\begin{proof}
    Let $F_{i} = Q^{-{n_i}}\circ Q^{m_i}: S_{i} \longrightarrow P$ be the canonical map. Then $N(S_i) = m_i - n_i$.
    By construction of the canonical map, $Q^{m_i}(I_{U}(s_i, t_i) = I_{U_0}(0,1)$, and $n_1 = n_2$. If $|[s_1, t_1]| = 2^k |[s_2, t_2]|$, then $m_2=m_1+2k$. Note there is a factor of $2$ for $k$ as the first return map of $U_0$ has period $2$.
    Therefore, $N(S_2) - N(S_1) = 2k$.
\end{proof}

We need the following combinatorial lemma.
\begin{lem}\label{lem:primitivedyadicdecomp}
    Let $A = [s,t]$ be a dyadic interval. Suppose that it is decomposed into $M\geq 2$ dyadic intervals $A = B_1 \cup B_2 \cup \ldots \cup B_M$.
    Then 
    $$
    |B_1|/|A| \in \{\frac{1}{2^{j}}: j = 1,\ldots, M-1\}.
    $$
    Moreover, for any $a \in \{\frac{1}{2^{j}}: j = 1,\ldots, M-1\}$, there is a decomposition $A = B_1 \cup B_2 \cup \ldots \cup B_M$ with each $B_i$ dyadic and $|B_1|/|A| = a$.
\end{lem}
\begin{proof}
    The proof is by induction (see Figure~\ref{fig:dyaint}).
    Since $A$ and each $B_i$ are dyadic, we conclude that the mid point $(s+t)/2$ must be a boundary point of some $B_i$.
    
    The base case when $M=2$ follows, as in this case, $B_1 = [s, (s+t)/2]$ and $B_2 = [(s+t)/2, t]$.

    For the induction step, let $A' = [s, (s+t)/2]$. Then $A'$ is dyadic and is decomposed into $A' = B_1 \cup \ldots \cup B_{M'}$ for some $M' < M$. If $M' = 1$, then $|B_1|/|A| = \frac{1}{2}$.
    Otherwise, by the induction hypothesis, we conclude that $|B_1|/|A'| \in \{\frac{1}{2^{j}}: j = 1,\ldots, M'-1\}$, and any number in this list is realizable by some decomposition. It is also easy to see that one can realize any number $M' \in \{1,\ldots, M-1\}$. Therefore, the lemma follows.
\end{proof}

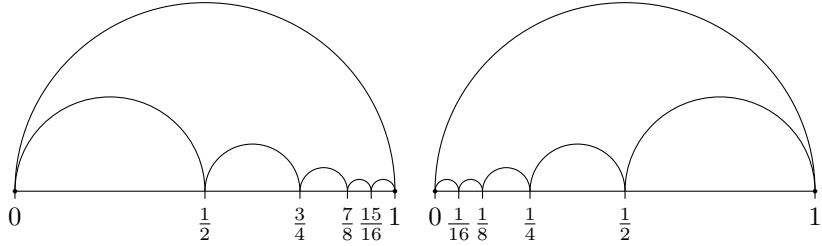
\begin{figure}[htp]
\captionsetup{width=0.96\linewidth}
\centering
\begin{tikzpicture}[x=5cm,y=5cm]
  \draw (0,0) -- (1,0);
  \fill (0,0) circle (0.006);
  \fill (1,0) circle (0.006);
  \foreach \x/\lab in {
      0/0,
      0.5/{\tfrac12},
      0.75/{\tfrac34},
      0.875/{\tfrac78},
      0.9375/{\tfrac{15}{16}},
      1/1
  }{
    \draw (\x,0) -- (\x,-0.02);
    \node[below] at (\x,-0.02) {$\lab$};
  }
  \draw (0,0) arc [start angle=180, end angle=0, radius=0.5];
  \draw (0,0) arc [start angle=180, end angle=0, radius=0.25];
  \draw (0.5,0) arc [start angle=180, end angle=0, radius=0.125];
  \draw (0.75,0) arc [start angle=180, end angle=0, radius=0.0625];
  \draw (0.875,0) arc [start angle=180, end angle=0, radius=0.0625/2];
  \draw (0.9375,0) arc [start angle=180, end angle=0, radius=0.0625/2];  
\end{tikzpicture}
\begin{tikzpicture}[x=5cm,y=5cm]
  \draw (0,0) -- (1,0);

  \fill (0,0) circle (0.006);
  \fill (1,0) circle (0.006);

  \foreach \x/\lab in {
      0/0,
      0.0625/{\tfrac{1}{16}},
      0.125/{\tfrac18},
      0.25/{\tfrac14},
      0.5/{\tfrac12},
      1/1
  }{
    \draw (\x,0) -- (\x,-0.02);
    \node[below] at (\x,-0.02) {$\lab$};
  }

  \draw (0,0) arc [start angle=180, end angle=0, radius=0.5];

  \draw (0,0) arc [start angle=180, end angle=0, radius=0.0625/2];
  
  \draw (0.0625,0) arc [start angle=180, end angle=0, radius=0.0625/2];

  \draw (0.125,0) arc [start angle=180, end angle=0, radius=0.0625];

  \draw (0.25,0) arc [start angle=180, end angle=0, radius=0.125];

  \draw (0.5,0) arc [start angle=180, end angle=0, radius=0.25]; 
\end{tikzpicture}
\caption{Under suitable iteration of the multiplication by $2$ map, the dyadic interval $A$ is mapped homeomorphically to $[0,1]$. The left and right figures give the two extreme configurations of $B_1,\ldots, B_5$ for the case $M = 5$: $|B_1|/|A| = \frac12$ on the left and $|B_1|/|A| = \frac1{2^{5-1}}$ on the right.}
\label{fig:dyaint}
\end{figure}

\begin{figure}[htp]
\captionsetup{width=0.96\linewidth}
    \centering
     \includegraphics[width=0.9\textwidth]{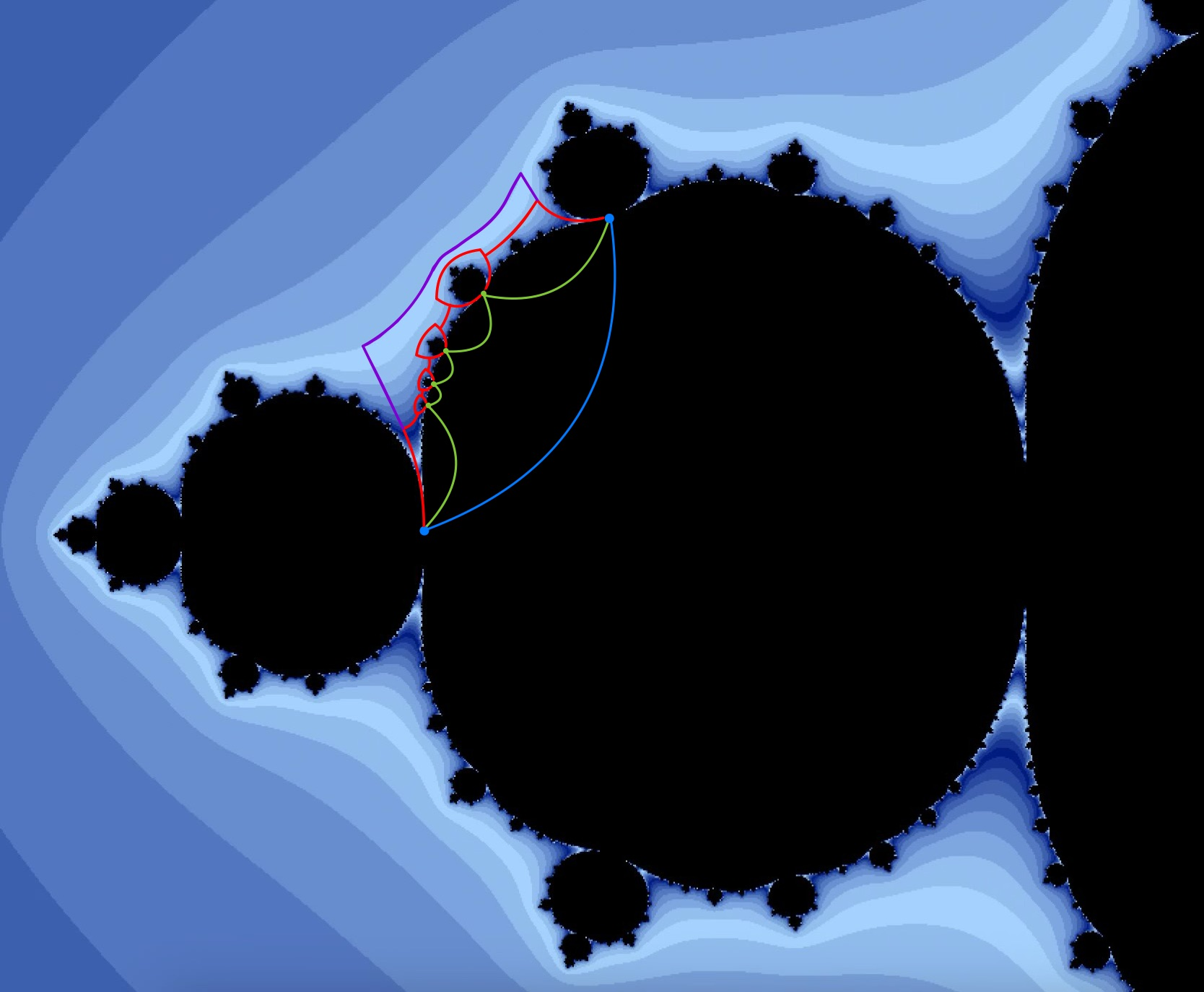}
    \includegraphics[width=0.9\textwidth]{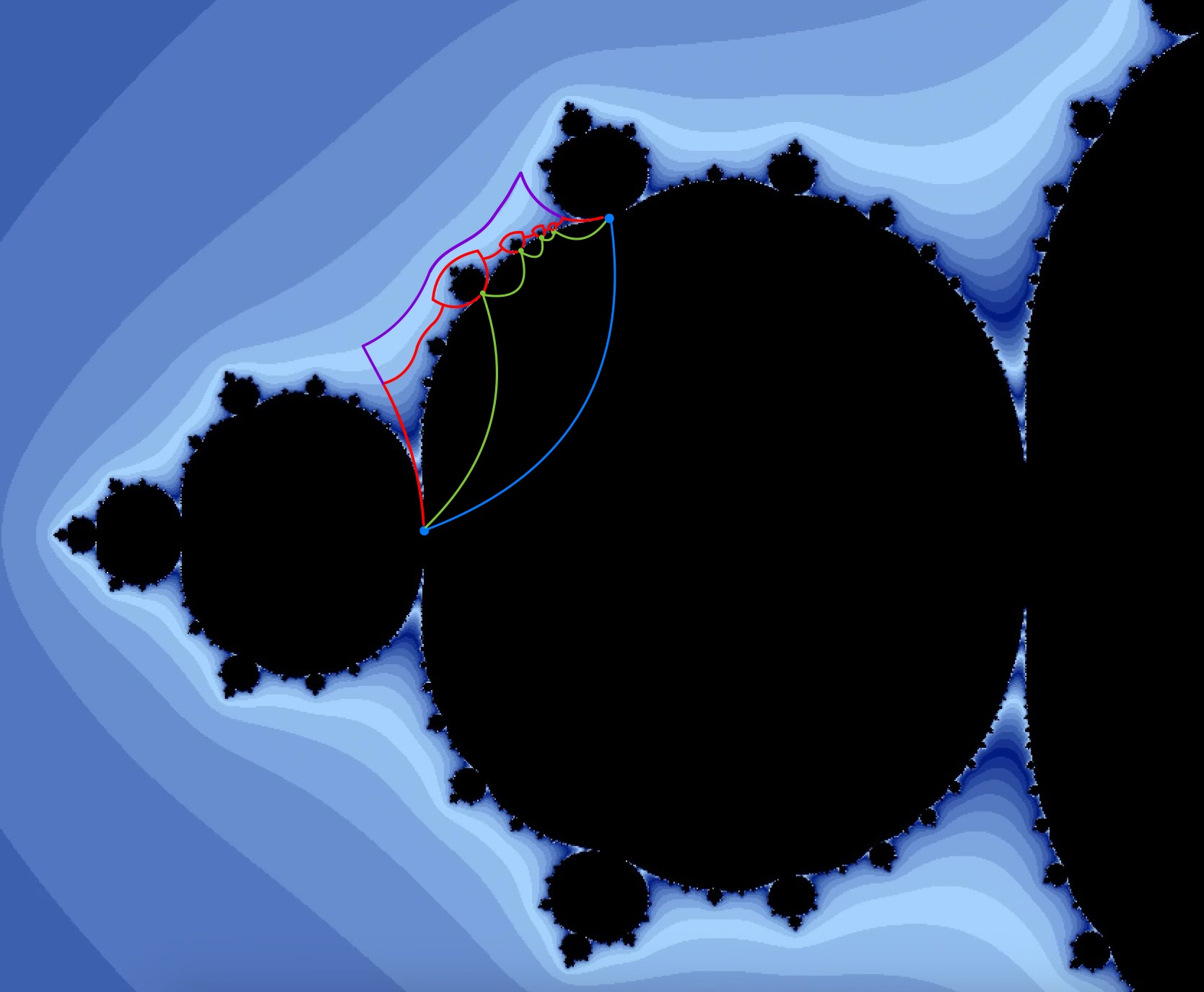}
    \caption{An illustration of the modification of a torso type Basilica puzzle piece $P$ carried out in Lemma~\ref{lem:localrefine}, with $\mathfrak{S}$ on the top and $\mathfrak{S}'$ on the bottom. It contains $9$ puzzle pieces of $\mathfrak{S}$, where $5$ of them are of torso type. The modification is similar to the two extreme cases in Figure~\ref{fig:dyaint}, and gives that $N(S_1') - N(S_1) = 2\times 3$.}
    \label{fig:mod}
\end{figure}
Recall that the derivative on a puzzle piece $A_{ij} \in \mathcal{A}^1$ is $D\mathcal{F}_\infty|_{A_{ij}} = 2^{N(A_{ij})}$. 
The following lemma allows us to modify the derivative.
\begin{lem}\label{lem:localrefine}
    Let $P\in \mathfrak{P}$ be a torso type Basilica puzzle piece for the Fatou component $U$ with $P\cap J(Q) = \left( S_1\cup \ldots\cup S_K\right) \cap J(Q), S_j \in \mathfrak{P}^1$, with index chosen so that $S_1$ shares a boundary point with $P$.
    Suppose that $M$ of the pieces $S_1,\ldots, S_K$ are of torso type for $U$.
    Then there exists
    \begin{itemize}[leftmargin=8mm]
        \item an integer $L$, and
        \item for any integer $l$ in $[0, M-1]$, there exists a refinement of $P$ by $S_1',\ldots, S_K'$ 
    \end{itemize} 
    so that the following hold. 
    \begin{itemize}[leftmargin=8mm]
        \item $N(S_1) \in [L, L+2(M-2)]$.
        \item If $\widetilde{\mathfrak{P}}^1$ is the puzzle obtained by replacing $S_j\in\mathfrak{P}^1$ with $S_j'$, and if $\mathcal{F}': \widetilde{\mathfrak{P}}^1\longrightarrow \mathfrak{P}$ is defined by 
        \begin{align*}
           & \mathcal{F}'|_S =\mathcal{F}|_S,\quad \textrm{if}\quad S \notin \{S_1',\ldots, S_K'\},\quad \textrm{and}\\
         & \mathcal{F}': S_j' \longrightarrow \mathcal{F}(S_j)\quad \textrm{is the canonical map of Lemma~\ref{lem:TypePreservingPuzzleMap};}  
        \end{align*} 
        then
        \begin{enumerate}[leftmargin=8mm]
            \item\label{lem:refine:item1} $\mathcal{F}': \widetilde{\mathfrak{P}}^1\longrightarrow \mathfrak{P}$ is topologically conjugate to $\mathcal{F}: \mathfrak{P}^1\longrightarrow \mathfrak{P}$ on $J(Q)$; i.e., $\mathcal{F}'$ is a modification of $\mathcal{F}$ in the sense of Definition~\ref{def-modfn}, and 
        \item\label{lem:refine:item3} $N(S_1') = L+2l$.
        \end{enumerate}
    \end{itemize}
\end{lem}
\begin{proof}
    Since $P$ is of torso type, we can associate to it a dyadic segment $I_{U}[s,t]$. Let $I_{U}[s, s_1]$ be the dyadic segment for $S_1$.
    By Lemma~\ref{lem:primitivedyadicdecomp}, 
    $$
    \frac{|s_1-s|}{|t-s|} \in \left\{\frac{1}{2^{j}}: j = 1,\ldots, M-1\right\},
    $$ 
    and we can construct a decomposition of $I_{U}[s,t]$ into $M$ dyadic segments so that the ratio $|s_1'-s|/|t-s|$ realizes any number in $\{\frac{1}{2^{j}}: j = 1,\ldots, M-1\}$, where $I_{U}[s,s_1']$ is the first segment in the decomposition (see Figure~\ref{fig:mod}).
    Note that any such decomposition gives a collection of $M$ torso type Basilica puzzle pieces $\{S_1',\cdots, S_K'\}$ for $U$.
    If $|s_1'-s|/|s_1-s| = 2^k$, then $N(S_1') = N(S_1) + 2k$ by Lemma~\ref{lem:fac2}. Therefore, there exists $L$ so that $N(S_1) \in [L, L+2(M-2)]$ and by choosing different decompositions of $I_{U}[s,t]$, $N(S_1')$ realizes any number $L+2l$, for $l \in [0, M-1]$.
    It is not hard to extend this collection of $M$ torso type Basilica puzzle pieces for $U$ to $\{S_1',\ldots, S_K'\}$ so that $\mathcal{F}'$ is combinatorially conjugate to $\mathcal{F}$. Thus, by Proposition~\ref{prop-combconjimpliestopconj}, Property \ref{lem:refine:item1} is satisfied. The lemma now follows.
\end{proof}

Let $P \in \mathfrak{P}^n$. We say that it is {\em boundary periodic} if at least one point $x\in \partial P \cap J(Q)$ is periodic under $\cF\vert_P$. We say that it is {\em one-sided boundary periodic} if exactly one point $x\in \partial P \cap J(Q)$ is periodic.
Note that by Lemma~\ref{lem:strictpre}, if $P$ is boundary periodic, then it is of torso type.
Finally, $P$ is said to be \emph{boundary periodic associated to $U$} if it is of torso type for $U$.

\begin{lem}\label{lem:onesideperiodic}
    There exists $n$ such that the following holds.
    If $S \in \mathfrak{P}^n$ is boundary periodic with period $q$ associated to some bounded Fatou component $U$, then 
    \begin{itemize}
        \item $S$ is one-sided boundary periodic; and
        \item $\mathcal{F}^q(S) \in \mathfrak{P}^{n-q}$ contains at least one torso type Basilica puzzle piece of $\mathfrak{P}^n$ (for $U$) that is not boundary periodic.
    \end{itemize}
\end{lem}
\begin{proof}
    Let $n$ be an integer that is strictly bigger than the least common multiple of the periods of the boundary points of level $0$ puzzle pieces.
    Since each boundary point of a level $n$ puzzle piece is eventually mapped to a boundary point of a level $0$ puzzle piece, we conclude that periodic points in $\partial S \cap J(Q)$ are boundary points of some level $0$ puzzle pieces.
    Suppose that both points in $\partial S \cap J(Q)$ are periodic. Let $L$ be their least common multiple. Then $L \leq n$.
    Hence, $\mathcal{F}^L: S \longrightarrow \mathcal{F}^L(S)$ is a homeomorphism and fixes the two boundary points. Therefore, $\mathcal{F}^L(S) = S$, which is a contradiction to the fact that $\mathcal{F}$ is topologically expanding.

    For the second claim, we note that since $\mathcal{F}$ is topologically expanding, $\mathcal{F}^q(S) \in \mathfrak{P}^{n-q}$ must contain another torso type Basilica puzzle piece of $P \in \mathfrak{P}^n$ for $U$. 
    If $P$ is not boundary periodic, then we are done. Otherwise, we apply the following pull-back argument.
    Let $S' \in \mathfrak{P}^{n+q}$ be contained in $S$, sharing the periodic boundary point. Then $\mathcal{F}^{q}(S') = S$ must contain some non boundary periodic Basilica puzzle piece in $\mathfrak{P}^{n+q}$ that is mapped to $P$ under $\mathcal{F}^q$. Therefore, by enlarging the constant $n$ to $2n$, the lemma follows.
\end{proof}

\begin{lem}\label{lem:LargeM}
    For any $M>1$, there exists a refinement $\widehat{\mathcal{F}}: \widehat{\mathfrak{P}}^{1} \rightarrow \widehat{\mathfrak{P}}$ of $\mathcal{F}: \mathfrak{P}^{1} \rightarrow \mathfrak{P}$ (in the sense of Definition~\ref{def-modfn}) so that the following holds.
    \begin{enumerate}[leftmargin=8mm]
        \item Each boundary periodic Basilica puzzle piece $P \in\widehat{\mathfrak{P}}$ is one-sided boundary periodic.
        \item For each boundary periodic Basilica puzzle piece $P \in \widehat{\mathfrak{P}}$ associated to some bounded Fatou component $U$, there are at least $M$ Basilica puzzle pieces $S \in \widehat{\mathfrak{P}}^1$ of torso type for $U$ with $S \subseteq P$.
    \end{enumerate}
\end{lem}
\begin{proof}
    Let $M > 1$ be given. 
    We construct $\widehat{\mathfrak{P}}$ and $\widehat{\mathfrak{P}}^1$ via a sequence of refinements.

    First, by Lemma~\ref{lem:onesideperiodic}, we let $n$ be sufficiently large so that every boundary periodic Basilica puzzle piece $P \in \mathfrak{P}^n$ is one-sided boundary periodic. We consider $\mathcal{F}: \mathfrak{P}^{n+1} \longrightarrow\mathfrak{P}^n$, which is topologically conjugate to $\mathcal{F}: \mathfrak{P}^1 \longrightarrow \mathfrak{P}$ on the Julia set.

    Let $(\mathfrak{P}^n)' \subset \mathfrak{P}^n$ be the collection of one-sided boundary periodic Basilica puzzle pieces.
    Let $P \in \mathfrak{P}^n - (\mathfrak{P}^n)'$. Let $N \in \N$. We denote by $\mathfrak{P}^{n}_N(P)$ the collection of Basilica puzzle pieces in $\mathfrak{P}^{n+N}$ that are contained in $P$.
    We define
    $$
    \widehat{\mathfrak{P}}^{n}_N = (\mathfrak{P}^n)' \cup \bigcup_{P \in \mathfrak{P}^n - (\mathfrak{P}^n)'} \mathfrak{P}^{n}_N(P).
    $$
    Since a one-sided boundary periodic Basilica puzzle piece in $\mathfrak{P}^{n+1}$ is mapped by $\mathcal{F}$ to a one-sided boundary periodic puzzle piece in $\mathfrak{P}^{n}$, we conclude that
    $$
    \mathcal{F}: \widehat{\mathfrak{P}}^{n+1}_N \longrightarrow \widehat{\mathfrak{P}}^{n}_N
    $$
    is a refinement of $\mathcal{F}: \mathfrak{P}^1 \longrightarrow \mathfrak{P}$. 
    
    Let $P \in (\mathfrak{P}^n)' \subseteq \widehat{\mathfrak{P}}^{n}_N$ be boundary periodic of period $q$ associated to some bounded Fatou component $U$. 
    We claim that the number of Basilica puzzle pieces of $\widehat{\mathfrak{P}}^{n+1}_N$ of torso type for $U$ that are contained in $P$ go to infinity as $N \to \infty$.
    To see this, by Lemma~\ref{lem:onesideperiodic}, $\mathcal{F}^q(P)$ must contain at least one Basilica puzzle piece $S$ in $\mathfrak{P}^n - (\mathfrak{P}^n)'$ of torso type for $U$. Since $S \notin (\mathfrak{P}^n)'$, the number of Basilica puzzle pieces of $\widehat{\mathfrak{P}}^{n}_N$ of torso type for $U$ contained in $S$ is the same as the number of the corresponding ones in $\widehat{\mathfrak{P}}^{n+N}$, and the latter goes to infinity as $N \to \infty$. By pulling back under $\mathcal{F}^q$, we conclude the claim.

    The lemma now follows by choosing $\widehat{\mathfrak{P}} = \widehat{\mathfrak{P}}^{n}_N$ and $\widehat{\mathfrak{P}}^1 = \widehat{\mathfrak{P}}^{n+1}_N$ for some sufficiently large $N$.
\end{proof}

\begin{proof}[Proof of Theorem~\ref{thm:symmetrichyperbolic}]
    Let $N \geq \max N(S)$ be an even integer where the maximum is taken over all boundary periodic $S \in \mathfrak{P}^1$.
    Let $M \geq 2N$. Let $\widehat{\mathcal{F}}: \widehat{\mathfrak{P}}^1 \longrightarrow \widehat{\mathfrak{P}}$ be the Basilica $\mathfrak{Q}-$map in Lemma~\ref{lem:LargeM} with constant $M$.
    By Lemma~\ref{lem:localrefine}, we can locally modify the refinement of each boundary periodic Basilica puzzle piece of $\widehat{\mathfrak{P}}$, and obtain $\widetilde{\mathcal{F}}: \widetilde{\mathfrak{P}}^1\longrightarrow \widetilde{\mathfrak{P}} = \widehat{\mathfrak{P}}$ so that it is topologically conjugate to $\widehat{\mathcal{F}}: \widehat{\mathfrak{P}}^1 \longrightarrow \widehat{\mathfrak{P}}$ on the Julia set, and if the right/left orbit $x_n^\pm$ of a break-point $x$ is periodic with period $q^\pm$ under the induced map $\widetilde{\mathcal{F}}_\infty$, then we have 
    $$
    D\widetilde{\mathcal{F}}_\infty^{q^\pm}(x^\pm) = 2^{Nq^\pm}.
    $$
    This is possible by Lemma~\ref{lem:2k}, as the derivative is always an even power of $2$, and for each $j\in\{0,1..., q^\pm-1\}$, we can choose the derivative $D\widetilde{\mathcal{F}}_\infty(\widetilde{\mathcal{F}}_\infty^j(x^\pm)) \in \{2^{N-1}, 2^{N}, 2^{N+1}\}$ by Lemma~\ref{lem:localrefine} (Property~\ref{lem:refine:item3}).
    This implies that the Lyapunov exponent of the induced map $\widetilde{\mathcal{F}}_\infty$ at any break-point is $N > 0$, and the theorem follows.
\end{proof}

\section{Basilica Bowen-Series maps}\label{section:BasBS}
In this section, we explicate a natural construction of a Markov $(\Lambda, G)-$map for any geometrically finite Bers boundary group $G$. We call such a map a {\em Basilica Bowen-Series map}, and study some basic properties of the map.

\subsection{Geometrically finite Bers boundary groups}\label{geom_fin_bers_bdry_subsec}
Let $\Sigma$ be a two-dimensional orbifold with {\em signature} $(g, n; \nu_1,\ldots, \nu_n)$, where $g$ is the genus, $n$ is the number of cusps and torsion points, and $\nu_i \in \Z_{\geq 2}\cup \{\infty\}$ is the order of the torsion at the $i$th torsion point.
It is {\em hyperbolic}, {\em parabolic} or {\em elliptic} if the orbifold Euler characteristic 
$$\chi(\Sigma) = 2-2g-n +\sum_{i=1}^n\frac{1}{\nu_i}$$ 
is negative, zero or positive, respectively.
Note that $\Sigma$ is compact if and only if $\nu_1,\ldots, \nu_n \neq \infty$.
Abusing terminology, we will simply call $\Sigma$ a surface.

Let $\Sigma$ be a hyperbolic surface.
Let $\mathcal{C} = \{\alpha_1, \ldots, \alpha_{k_0}\}$ be a non-empty collection of disjoint, pairwise non-homotopic simple closed curves on $\Sigma$.
We will call such a collection of simple closed curves a {\em pinching multicurve}.
We say that $\mathcal{C}$ is {\em strictly separating} if every curve in $\mathcal{C}$ is separating.

By pinching the simple closed curves in the collection $\mathcal{C}$ to nodes, we degenerate the original surface into finitely many compact (connected) surfaces $\overline{\Sigma_1},\ldots, \overline{\Sigma_{j_0}}$ glued along  nodal singularities (see \cite[\S6]{Mas70}). We denote by $\Sigma_i$ the \emph{non-compact} surfaces corresponding to the components of $\Sigma - \mathcal{C}$, and by $\overline{\Sigma_i}$ the union of ${\Sigma_i}$ with all the nodal singularities on it. 
Note that $\overline{\Sigma_i}$ may not be compact, as $\Sigma_i$ may contain cusps of the original surface $\Sigma$.
Thus, the cusps of ${\Sigma_i}$ correspond
to the union of  the nodal singularities of 
$\overline{\Sigma_i}$ and the cusps of $\Sigma$ on $\Sigma_i$.
We will write the degenerate surface as
$$
(\Sigma, \mathcal{C}) = \Sigma_1+\Sigma_2+\ldots+\Sigma_{j_0}.
$$
Note that $\mathcal{C}$ is strictly separating if and only if $j_0= k_0+1$.

We will call $\Sigma_i$ a {\em nodal component} of $(\Sigma, \mathcal{C})$.
We denote the signature of the nodal component $\Sigma_i$ by $(g_i, n_i; \nu_{i,1},\ldots, \nu_{i, n_i})$.
By construction, the nodal components $\Sigma_i$ are in one-to-one correspondence with the connected components of $\Sigma - \bigcup_i \alpha_i$.
Thus, each component $\Sigma_i$ is non-compact. Further, 
 $\Sigma_i$ is either hyperbolic or parabolic. In the latter case, $\Sigma_i$ must have signature $(0, 3; 2, 2, \infty)$.
Thus $\Sigma_i$ is parabolic if and only if there is a curve $\alpha_i\in\mathcal{C}$ enclosing exactly two orbifold points of order $2$.
The parabolic nodal component $\Sigma_i$ corresponds precisely to the region bounded by this curve.

To describe the topological structure of the limit set, we want to introduce a notion of adjacency between various hyperbolic nodal components.
Let $\Sigma_{i_1}$ and $\Sigma_{i_2}$ be two (potentially same) hyperbolic nodal components.
We say that a simple closed curve $\alpha_i$ is a {\em connector} between $\Sigma_{i_1}$ and $\Sigma_{i_2}$ if 
\begin{itemize}
    \item either the two sides of $\alpha_i$ in $\Sigma$ correspond to $\Sigma_{i_1}$ and $\Sigma_{i_2}$ (where $i_1$ may be equal to $i_2$); or
    \item $\Sigma_{i_1} = \Sigma_{i_2}$ and the two sides of $\alpha_i$ in $\Sigma$ correspond to $\Sigma_{i_1}(= \Sigma_{i_2})$ and a parabolic nodal component.
\end{itemize}

The following structural result follows from \cite[Theorem 5, 6]{Mas70}.
\begin{prop}\label{prop:structgfbersboundary}
    Let $\Sigma$ be a surface with $\chi(\Sigma) < 0$ and let $\mathcal{C}$ be a pinching multicurve, with $(\Sigma, \mathcal{C}) = \Sigma_1+\Sigma_2+\ldots+\Sigma_{l_0} + \Sigma_{l_0+1} + \ldots + \Sigma_{j_0}$, where $\Sigma_1, \ldots, \Sigma_{l_0}$ are hyperbolic and $\Sigma_{l_0+1}, \ldots, \Sigma_{j_0}$ are parabolic (thus have signature $(0, 3; 2, 2, \infty)$).
    
    Let $\sigma \in \Teich(\Sigma)$ and $\sigma_i \in \Teich(\Sigma_i),\ i\in\{1,\ldots, l_0\}$, with the corresponding Riemann surfaces $S_\sigma$ and $S_{\sigma_i},\ i\in\{1,\ldots, l_0\}$.
    Then there exists a geometrically finite Kleinian group $G$ on the Bers boundary $\partial \Ber_\sigma(\Sigma)$ such that the following hold.
    \begin{itemize}[leftmargin=8mm]
        \item There is a $G-$invariant component $\Delta_\infty$ of $\Omega(G)$ so that $\Delta_\infty/G$ is conformally equivalent to $S_\sigma$.
        \item For each $\alpha_i \in \mathcal{C}$, there exists an arc $\widetilde{\alpha}_i \subseteq \Delta_\infty$, whose stabilizer in $G$ is an infinite cyclic group $<g_i>$ generated by an accidental parabolic element $g_i$. Moreover, the projection of $\widetilde{\alpha}_i$ to $S_\sigma$ is the geodesic representative of the curve $\alpha_i$ and the two ends of the arc $\widetilde{\alpha}_i$ accumulate to the (unique parabolic) fixed point of $g_i$.
        \item The quotient $(\Omega(G) - \Delta_\infty)/ G$ is conformally equivalent to the disjoint union $S_{\sigma_1} \sqcup \ldots \sqcup S_{\sigma_l}$.
        \item Each component  $\Delta$ of $\Omega(G) - \Delta_\infty$ is a Jordan domain, and $\Delta/ G_\Delta$ is conformally equivalent to $S_{\sigma_i}$ for some $i \in\{ 1,\ldots, l_0\}$.
        (Here, $G_\Delta$ denotes the stabilizer of $\Delta$ in $G$.)
        \item For any pair of distinct components $\Delta, \Delta'$ of $\Omega(G) - \Delta_\infty$, either they have disjoint closures, or their closures intersect at exactly one point.
        In the second case, if $\Delta, \Delta'$ correspond to $\Sigma_i$ and $\Sigma_{i'}$ respectively, then $\partial \Delta \cap \partial \Delta'$ is a parabolic fixed point associated to the curve $\alpha_m$ that is a connector between $\Sigma_{i}$ and $\Sigma_{i'}$
    \end{itemize}
    Moreover, any geometrically finite Bers boundary group $G \in \partial \Ber_\sigma(\Sigma)$ is obtained in this way.
\end{prop}

\subsection{Core polygon for non co-compact Fuchsian groups}\label{ss:idealcorpolygon}
In this subsection, we associate a special ideal polygon to each non co-compact Fuchsian group. This polygon is later used in \S~\ref{subsec:pinchedcore} for the construction of a pinched polygon for any geometrically finite Bers boundary group.

Let $\Sigma$ be a non-compact hyperbolic surface with signature $(g, n; \nu_1,\ldots, \nu_n)$.
Let us arrange so that $\nu_1 = \nu_2 = \ldots = \nu_k = \infty$, and $\nu_{k+1}, \ldots, \nu_{n} < \infty$.
Since $\Sigma$ is non-compact, we have that $k \geq 1$.

Let $G$ be a Fuchsian group with the given signature. This means that the quotient $S = \D / G$ is a hyperbolic surface homeomorphic to $\Sigma$.
We define a {\em core polygon} $\cP$ for $G$ as follows (see Figure~\ref{fig:cuspBM}). For $x,y\in\mathbb{S}^1$, we denote the hyperbolic geodesic in $\D$ connecting $x, y$ by $\arc{x,y}$.
 \begin{figure}[htp]
 \captionsetup{width=0.96\linewidth}
    \centering
    \includegraphics[width=0.98\textwidth]{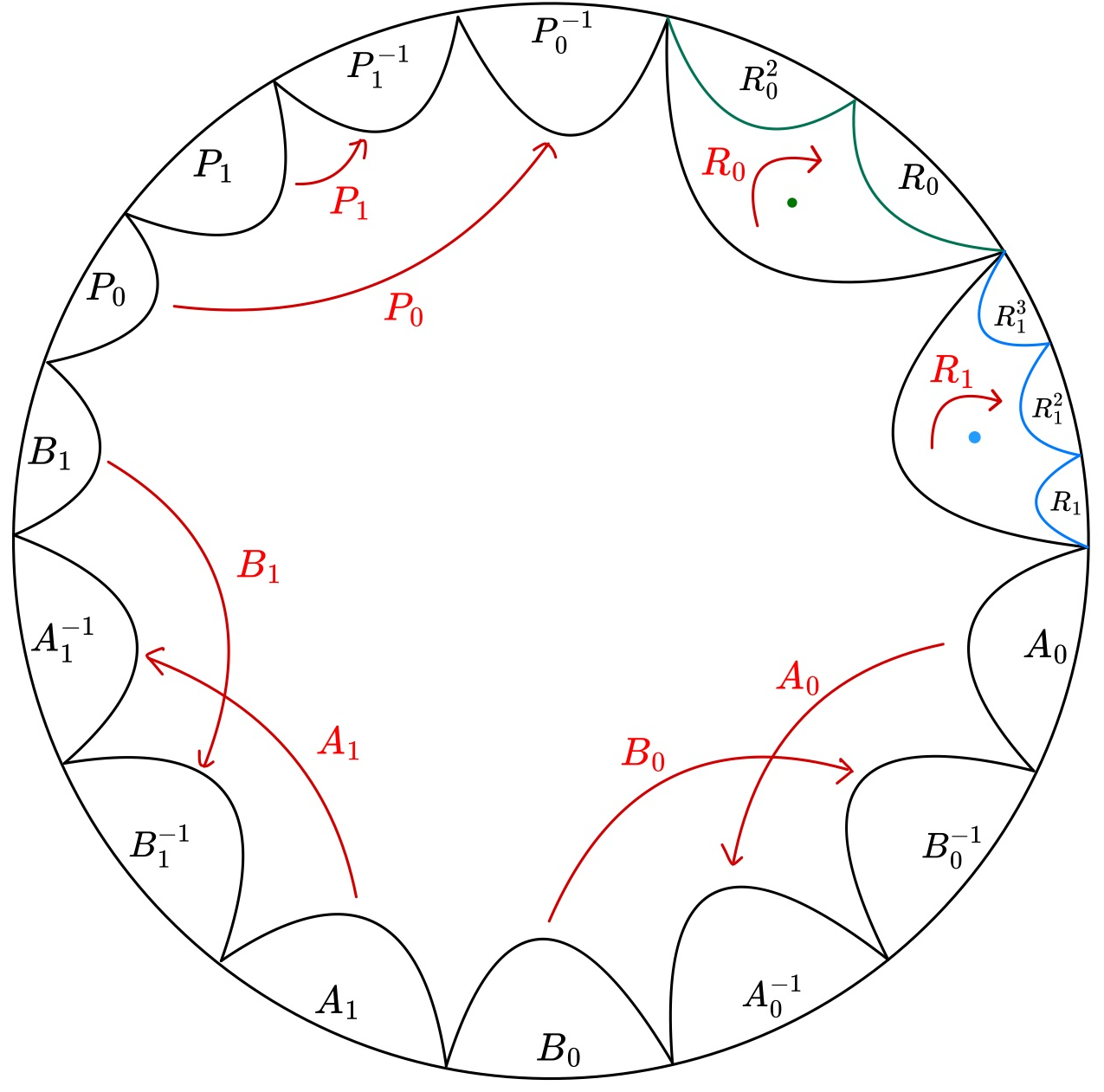}
    \caption{The Bowen-Series map for a Fuchsian group of type $(2, 5; \infty, \infty, \infty, 3,4)$.}
    \label{fig:cuspBM}
\end{figure}

\begin{itemize}[leftmargin=8mm]
    \item $\cP$ is an ideal hyperbolic $N-$gon in $\D$, where $N = 4g+2(k-1)+\sum_{i=k+1}^n(\nu_i-1)$.
    \item Denote the vertices of $\cP$ by $a_0,\ldots, a_{N-1} \in \mathbb{S}^1$. Then each edge $\gamma_i,\ i\in\{0,\ldots, N-1\}$, is a hyperbolic geodesic in $\D$ with end-points $a_i, a_{i+1}$ (with indices mod $N$).
    \item ($g$ handles) There exist $2g$ loxodromic elements $A_i, B_i \in G,\  i\in\{0,\ldots, g-1\}$, such that for $i\in\{0,\cdots,g-1\}$, 
    $$
    A_i(\arc{a_{4i},a_{4i+1}}) = \arc{a_{4i+3}, a_{4i+2}}, \text{ and } B_i(\arc{a_{4i+3}, a_{4i+4}}) = \arc{a_{4i+2}, a_{4i+1}}.
    $$
    \item ($k$ punctures) There exist $k-1$ elements $P_0, \ldots, P_{k-2} \in G$ such that
    $$
    P_i(\arc{a_{4g+i}, a_{4g+i+1}}) = \arc{a_{4g+2(k-1)-i}, a_{4g+2(k-1)-i-1}}.
    $$
    Moreover, if $k\geq 4$, then $P_1,\ldots, P_{k-3}$ are loxodromic and $P_{k-2}$ is parabolic. The element $P_0$ is parabolic if $(g,n) = (0, k)$, i.e., if the surface is a punctured sphere, and is loxodromic otherwise. For $i\in\{0,\ldots, k-3\}$, the group element $P^{-1}_{i+1} \circ P_i$ is parabolic. (The cases $k\in\{1,2,3\}$ can be described similarly.)
    \item ($n-k$ orbifold points) There exist $n-k$ elliptic elements $R_0,\ldots, R_{n-k-1} \in G$ such that for $i\in\{0,\ldots, n-k-1\}$, we have $R_i^{\nu_{k+1+i}} = \text{id}$, and
    $$
    R_i(\arc{a_{i,j}, a_{i,j+1}}) = \arc{a_{i,j+1}, a_{i,j+2}},
    $$
    where $j\in\{0,\ldots, \nu_{k+1+i}-1\}\ \left(\textrm{mod}\ \nu_{k+1+i}\right)$. Here,
    \begin{align}
      a_{i,j} = 
      \begin{cases}
          a_{4g+2(k-1)+\sum_{s=k+1}^{k+i}(\nu_s-1)+j}\quad \mathrm{for}\ i\geq 1,\\
          a_{4g+2(k-1)+j}\hspace{2.4cm} \mathrm{for}\ i=0.
      \end{cases}  
      \label{a_i_j_formula}
    \end{align}
\end{itemize}

Note that if $G$ is torsion-free, then $\cP$ is a fundamental polygon for the action of $G$ on $\D$.

The vertex set induces a partition $\mathcal{A} = \{[a_i,a_{i+1}]:i\in\{0,\ldots, N-1\}\}$  of $\mathbb{S}^1$.
We define the {\em Bowen-Series Markov map} associated to the core polygon $\cP$ as a Markov map with respect to the partition $\mathcal{A}$:
$$
\mathcal{F} = \mathcal{F}_{G, \cP} = \{F_i:=\mathcal{F}|_{[a_i, a_{i+1}]}: i\in\{0,\ldots, N-1\}\},
$$ 
where $\mathcal{F}|_{[a_i, a_{i+1}]}$ is defined as follows (see Figure~\ref{fig:cuspBM}, cf. \cite{BS79}).

\begin{itemize}
    \item For $i\in\{0,\ldots, g-1\}$, 
    \begin{align*}
        \mathcal{F}|_{[a_{4i},a_{4i+1}]} &= A_i,\\
        \mathcal{F}|_{[a_{4i+1},a_{4i+2}]} &= B_i^{-1},\\
        \mathcal{F}|_{[a_{4i+2},a_{4i+3}]} &= A_i^{-1},\\
        \mathcal{F}|_{[a_{4i+3},a_{4i+4}]} &= B_i.
        \end{align*}
    \item For $i\in\{0,\ldots, k-1\}$,
    \begin{align*}
        \mathcal{F}|_{[a_{4g+i},a_{4g+i+1}]} &= P_i,\\
        \mathcal{F}|_{[a_{4g+2(k-1)-i-1},a_{4g+2(k-1)-i}]} &= P_i^{-1}.
    \end{align*}
    \item For $i\in\{0,\ldots, n-k-1\}$, $j\in\{0,1,\ldots, \nu_{k+1+i}-2\}$,
    $$
    \mathcal{F}|_{[a_{i,j}, a_{i,j+1}]} = R_i^{\nu_{k+1+i}-1-j} 
    $$
    where $a_{i,j}$ is defined by Formula~\eqref{a_i_j_formula}.
\end{itemize}
We remark that $\mathcal{F}$ is a conformal Markov map on $\mathbb{S}^1$ in the sense of \S~\ref{sec:cmm}. It induces a continuous map on the circle if and only if $\Sigma$ has signature $(0,n;\infty,\cdots,\infty)$ or $(0,n;\infty,\cdots,\infty,2)$; i.e., the surface is a punctured sphere possibly with an order $2$ orbifold point (cf. \cite[\S 3]{MM1}).

It follows from the construction that we have the following.
\begin{lem}\label{lem:symFs}
    The Bowen-Series Markov map $\mathcal{F} = \mathcal{F}_{G, \cP}$ is topologically expanding, conformal, and admits a natural puzzle structure: for each Markov interval $[a_i, a_{i+1}] \in \mathcal{A}$, the corresponding level 0 puzzle $P_i$ is the disk bounded by the circle that meets $\mathbb{S}^1$ orthogonally at $a_i, a_{i+1}$. 
    
    Moreover, every break-point is symmetrically parabolic.
    More precisely, let $a$ be a periodic break-point. Then the stabilizer of $a$ in $G$ is generated by a parabolic element $g \in G$ so that $g$ is repelling on the circular arc $[a, a+\epsilon]$ and attracting on $[a-\epsilon, a]$.
    Further, if $q^\pm$ is the period of the right/left orbit $a_n^\pm$, respectively, then $\mathcal{F}^{q^\pm}|_{[a, a\pm \epsilon]} = g^{\pm 1}$ for all sufficiently small $\epsilon$.
\end{lem}

\subsection{Pinched core polygon and puzzles}\label{subsec:pinchedcore}
In this subsection, we give a natural puzzle structure for the limit set of any geometrically finite Bers boundary group.

As in \S~\ref{geom_fin_bers_bdry_subsec}, let $\Sigma$ be a surface with $\chi(\Sigma) < 0$ and let $\mathcal{C}$ be a pinching multicurve, with $(\Sigma, \mathcal{C}) = \Sigma_1+\Sigma_2+ \ldots + \Sigma_{j_0}$.
For simplicity of  presentation, let us first assume that all nodal components are hyperbolic. We will discuss necessary modifications to handle parabolic nodal components at the end.

Let $G$ be a geometrically finite Bers boundary group associated to $(\Sigma, \mathcal{C})$.
Let $\Delta_1$ be a component of $\Omega(G) - \Delta_\infty$, where $\Delta_\infty$ is the $G-$invariant component.
Then $\Delta_1$ correspond to some component $\Sigma_i$ which is hyperbolic. Without loss of generality, we assume that $\Delta_1$ corresponds to $\Sigma_1$.

Since $\Sigma_1$ is non-compact, we can construct a core polygon $\cP_1 \subseteq \Delta_1$ following \S~\ref{ss:idealcorpolygon}.
Let $G_{\Delta_1} \subseteq G$ be the stabilizer of $\Delta_1$.
Note that the $G_{\Delta_1}-$orbit of the ideal vertices of $\cP_1$ are in one-to-one correspondence with the cusps of $\Sigma_1$.
Thus, if $j_0>1$, there exists at least one component $\Delta_2 \left(\neq\Delta_1\right)$  of $\Omega(G) - \Delta_\infty$ associated to some surface in $\{\Sigma_2, \ldots, \Sigma_{j_0}\}$ that is adjacent to $\Delta_1$ at some vertex of $\cP_1$. Without loss of generality, we assume that $\Delta_2$ corresponds to the surface $\Sigma_2$.
Let $\cP_2 \subseteq \Delta_2$ be a core polygon.
Inductively, we can construct 
\begin{itemize}[leftmargin=8mm]
    \item a tree of components $\Delta_1,\ldots, \Delta_{j_0}$ corresponding to $S_1,\ldots, S_{j_0}$ (respectively) such that $\bigcup_{i=1}^{j_0} \overline{\Delta_i}$ is connected, and
    \item a collection of core polygons $\cP_1,\ldots,\cP_{j_0}$ so that $\cP_i \subseteq \Delta_i$ and $\bigcup_{i=1}^{j_0} \overline{\cP_i}$ is connected.
\end{itemize}
We will call this collection $\{\cP_1,\ldots, \cP_{j_0}\}$ a {\em pinched core polygon} for $G$ (see Figure~\ref{fig:PinchedPolygon}).
We remark that the pinched core polygon plays a similar role as the internal puzzle in Definition~\ref{defn:intpuz}.

 \begin{figure}[htp]
 \captionsetup{width=0.96\linewidth}
    \centering
    \includegraphics[width=0.9\textwidth]{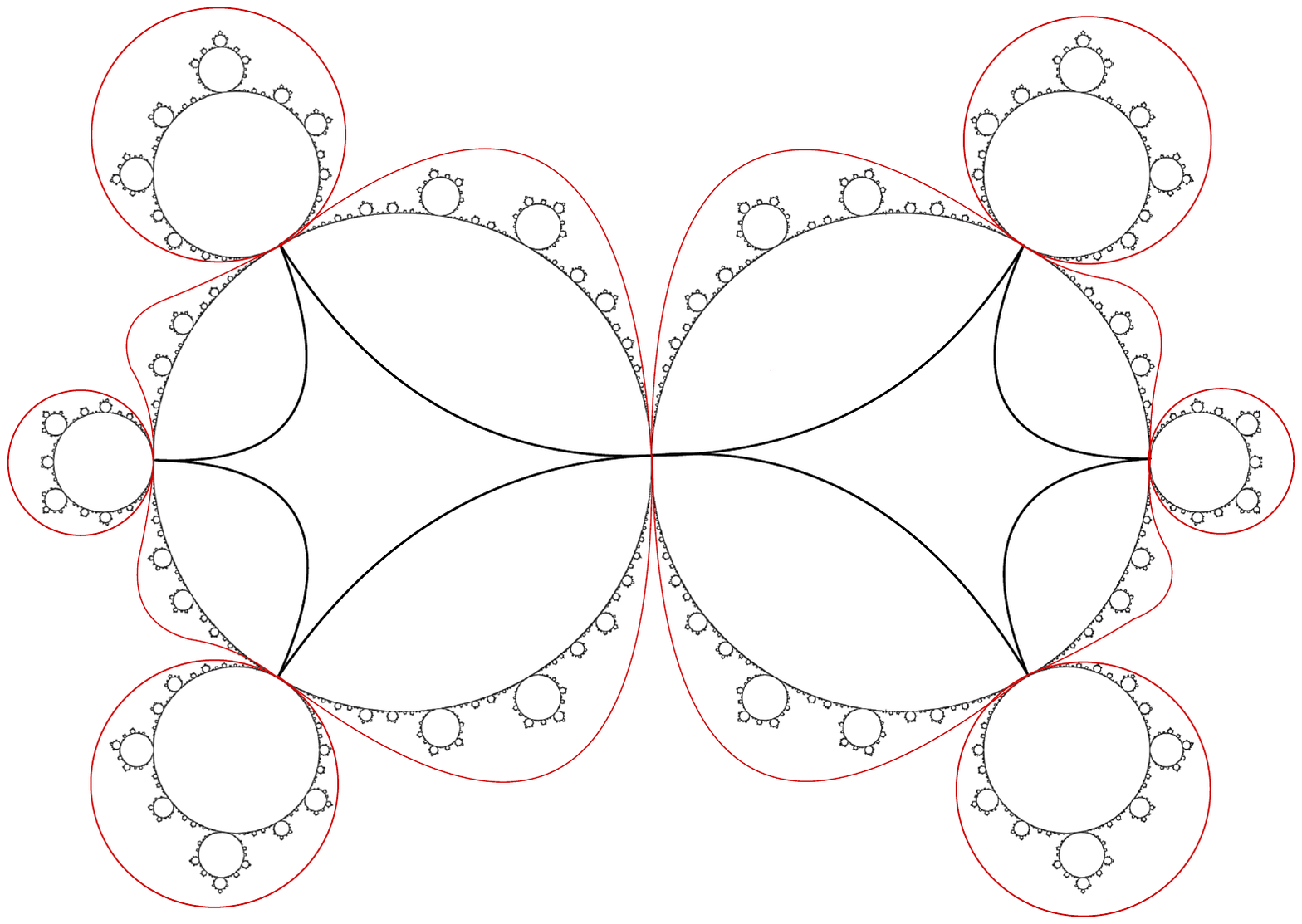}
    \caption{The puzzle structure for a genus $2$ Bers boundary group. The union of the two black ideal quadrilaterals $\overline{\cP_1} \cup \overline{\cP_2}$ is the pinched core polygon for $G$. The associated ideal polygon $\cP_\infty \subseteq \Delta_\infty$ is bounded by the union of red geodesic edges. The closure of connected components of $\widehat{\C} - \left(\overline{\cP_1}\cup \overline{\cP_2} \cup \overline{\cP_\infty}\right)$ form the level $0$ puzzle pieces. A limb type puzzle piece is bounded by a red edge, while a torso type puzzle piece is bounded by the union of one black edge and one red edge.}
    \label{fig:PinchedPolygon}
\end{figure}

We remark that if $G$ is torsion free, then one can realize this collection of ideal polygons as the limit of some fundamental polygon for the Fuchsian model $G_0$ of $\Sigma$. Specifically, pinching the corresponding multicurve $\mathcal{C}$ on $\Sigma$ yields a sequence of quasi-Fuchsian deformations of $G_0$ that degenerates to $G$.

Recall that we denote the ideal boundary of $\Delta_\infty$ by $\partial^I \Delta_\infty$.
Each vertex $x$ of $\cP_i$ corresponds to a cusp of $\Sigma_i$ which is either a nodal singularity of $\overline{\Sigma_i}$ or a cusp of $\Sigma_i$ coming from the original surface $\Sigma$.
In the former case, $x$ corresponds to $2$ ideal boundary points for $\Delta_\infty$, and there is a component of $\Omega(G) - \Delta_\infty$ attached to $\Delta_i$ at~$x$.
In the latter case, $x$ corresponds to a unique ideal boundary point for $\Delta_\infty$.
We consider the ideal polygon $\cP_\infty \subseteq \Delta_\infty$ whose vertices are the ideal boundary points (for $\Delta_\infty$) associated to the vertices of $\bigcup_{i=1}^{j_0}\cP_i$. (In Figure~\ref{fig:PinchedPolygon}, the polygon $\cP_\infty$ in $\Delta_\infty$ is shown in red. Here, $\Sigma$ is a compact surface of genus $2$, and the multi-curve $\mathcal{C}$ has a single component; hence, pinching $\mathcal{C}$ produces two genus $1$ surfaces $\Sigma_1$ and $\Sigma_2$ each of which has a unique node and no other cusp. The sides of the polygon $\cP_\infty$ are lifts of the curve $\mathcal{C}$. In particular, the second possibility discussed above does not arise in this example.)

The complement $\widehat\C - \left(\bigcup_{i=1}^{j_0}\overline{\cP_i} \cup \overline{\cP_\infty}\right)$ consists of finitely many Jordan domains, whose closures are denoted by $P_1,\ldots, P_r$. We will refer to these as level 0 puzzle pieces, and denote the collection of such puzzle pieces by
$$
\mathfrak{P}^0 := \{P_1, \ldots, P_r\}.
$$

Note that there are two types of puzzle pieces.
We say that a puzzle piece $P_i$ is of {\em limb type} if $\Int P_i \cap \bigcup_{i=1}^{j_0}\Delta_i = \emptyset$, and call it {\em torso type} otherwise. We remark that the current setup is analogous to that of Basilica puzzles, see Definition~\ref{basi_puzz_def}.
If $P_i$ is of limb type, then it is bounded by some geodesic $\widetilde{\alpha}_i \subseteq \Delta_\infty$ together with its accumulation point, which is the fixed point of an accidental parabolic element in $G$.
If $P_i$ is of torso type, then it is bounded by two geodesics and their accumulation points, with one geodesic in $\Delta_\infty$, and the other in $\Delta_s$ for some $s \in \{1,\ldots, j_0\}$ (see Figure~\ref{fig:PinchedPolygon}).

\subsubsection{Modification for parabolic nodal components}\label{subsec:parabolicnodal}
Suppose that 
$$
(\Sigma, \mathcal{C}) = \Sigma_1+\Sigma_2+ \ldots + \Sigma_{l_0} + \Sigma_{l_0+1} + \ldots + \Sigma_{j_0},
$$
where $\Sigma_i$ is hyperbolic for $i \leq l_0$ and parabolic for $i \geq l_0+1$.
Note that each parabolic nodal component is attached to a unique hyperbolic nodal component via a nodal singularity.
These parabolic nodal components do not correspond to any components in the domain of discontinuity.
Thus, the same construction gives  a connected pinched core polygon
$$
\overline{\cP_1} \cup \overline{\cP_2} \cup \ldots \cup \overline{\cP_{l_0}},
$$
where $\cP_i \subseteq \Delta_i$ is the core polygon associated to $\Sigma_i$.
The level $0$ puzzle pieces $\mathfrak{P}^0 = \{P_1,\ldots, P_r\}$ are constructed exactly the same as the closure of components $\widehat{\C} - \left( \overline{\cP_1} \cup \overline{\cP_2} \cup \ldots \cup \overline{\cP_{l_0}} \cup \overline{\cP_\infty}\right)$.

\subsection{The construction of Basilica Bowen-Series maps}\label{subsec:constructionBBS}
In this section, we explain how the natural puzzle constructed in \S~\ref{subsec:pinchedcore} induces a Markov $(G,\Lambda(G))-$map.

Let $\mathfrak{P}^0 = \{P_1,\ldots, P_r\}$ be the level $0$ puzzle associated to a pinched core polygon $\{\cP_1,\ldots, \cP_{j_0}\}$.
For each $\cP_i$, we choose a subset $\mathcal{PV}_i$ of the vertex set of $\cP_i$ so that it contains exactly one element in the $G_{\Delta_i}-$orbit of each vertex of $\cP_i$, where $G_{\Delta_i}$ is the stabilizer of $\Delta_i$.
We call $\mathcal{PV}_i$ the {\em representative vertex set}.

We are going to define a collection of maps
$$
\mathcal{F} = \{f_i: i\in\{1,\ldots, r\}\}
$$
where $f_i : P_i \longrightarrow \widehat{\C}$.

Suppose that $P_i$ is of limb type. Let $a$ be the unique point in $\partial P_i \cap \Lambda(G)$, and $\Delta \subseteq P_i$ be the component of $\Omega(G)-\Delta_\infty$ with $a \in \partial \Delta$.
Then $\Delta$ corresponds to some hyperbolic nodal component, say $\Sigma_s$.
Let $\Delta_s$ be the component of $\Omega(G)$ associated to $\Sigma_s$ that contains $\cP_s$ in the pinched core polygon.
Let $g \in G$ be an element so that $g(\Delta) = \Delta_s$ and $g(a) \in \mathcal{PV}_s$.
We define the map $f_i: P_i \longrightarrow \widehat{\C}$ as $f_i(x) = gx$. 
We remark that the choice of the element $g$ is not unique: any two choices differ by post-composition with an element in the stabilizer of $g(a)$, which is isomorphic to $\Z$. Note that this group is generated by the (primitive) accidental parabolic element with its fixed point at $g(a)$.

Suppose that $P_i$ is torso type. Let $s \in \{1,\ldots, j_0\}$ be the unique index such that $\Int P_i \cap \Delta_s \neq \emptyset$.
Note that $\cP_s$ is the corresponding core polygon for $\Sigma_s$ as defined in \S~\ref{ss:idealcorpolygon}, and hence $P_i\cap \partial \Delta_s$ is some Markov piece $[a,b]$ for the Bowen-Series Markov map on $\Delta_s$ associated to $\cP_s$ as defined in \S~\ref{ss:idealcorpolygon}. This Bowen-Series map is defined by some $h \in G_{\Delta_s} \subseteq G$ on $[a,b]$. We define the map $f_i: P_i \longrightarrow \widehat{\C}$ as $f_i(x) = hx$.

\begin{defn}
    Let $G$ be a geometrically finite Bers boundary group, with pinched core polygon $\{\cP_1, \ldots, \cP_{j_0}\}$ and representative vertex set $\mathcal{PV}_i$.
    Let $\mathfrak{P}^0 = \{P_1,\ldots, P_r\}$ be the corresponding level $0$ puzzle.
    We call $\mathcal{F} = \{f_1,\ldots, f_r\}$ constructed above a \emph{Basilica Bowen-Series map} for $G$.
\end{defn}

\begin{remark}
As we will see in the next section, Basilica Bowen-Series maps are closely related to Basilica $\mathfrak{Q}-$maps (see Definition~\ref{defn:basiqmap}).
\end{remark}
\begin{remark}
    The construction of a Basilica Bowen-Series map is not canonical. There are several steps where choices need to be made.
    \begin{enumerate}
        \item The pinched core polygon $\{\cP_1,\ldots, \cP_{j_0}\}$ requires a choice. Note that the level $0$ puzzle depends only on such a choice.
        \item The subset $\mathcal{PV}_i$ of the vertex set of $\cP_i$ requires choices of representatives from the $G_{\Delta_i}-$orbit of the vertex set.
        \item Finally, the map on a limb type puzzle piece requires a choice of the element in the stabilizer of the corresponding point.
    \end{enumerate}
\end{remark}

\subsubsection{Markov structure}
\begin{prop}\label{prop:MarkovProp}
    A Basilica Bowen-Series map $\mathcal{F} = \{f_1,\ldots, f_r\}$ induces a Markov $(\Lambda,G)-$map on the limit set $\Lambda(G)$.
    More precisely, for each $i$, there exists a subset $I_i \subseteq \{1,\ldots, r\}$ so that
    $$
    f_i(P_i \cap \Lambda(G)) = \bigcup_{j\in I_i} P_j \cap \Lambda(G).
    $$
    Further, $\bigcup_{j\in I_i} P_j \subseteq f_i(P_i)$.
\end{prop}
\begin{proof}
    The fact that $\mathcal{F}$ is a $(\Lambda,G)-$map follows directly from the construction.
    It remains to verify that it has the Markov property.
    
    Suppose $P_i$ is of limb type. Let $a$ be the unique point in $\partial P_i \cap \Lambda(G)$. By construction, $f_i(P_i \cap \Lambda(G))$ is the closure of one of the two connected components of $\Lambda(G) - \{f_i(a)\}$. Since $f_i(a)$ is a vertex for $\cP_s$ of some $s$, we have a decomposition $\{1,\ldots, r\} = A \sqcup B$ so that $\bigcup_{j \in A} P_j \cap \Lambda(G)$ and $\bigcup_{j \in B} P_j \cap \Lambda(G)$ correspond to the closure of the two connected components of $\Lambda(G) - \{f_i(a)\}$. Therefore, $f_i(P_i \cap \Lambda(G)) = \bigcup_{j\in I_i} P_i \cap \Lambda(G)$ for some index set $I_i \subset \{1,\ldots, r\}$.
    
    Suppose $P_i$ is of torso type. Let $a, b$ be the two points in $\partial P_i \cap \Lambda(G)$. Note that $\Lambda(G) - \{f_i(a), f_i(b)\}$ contains at most $4$ components. The number of components depends on whether $f_i(a)$ (or $f_i(b)$) corresponds to a cusp of $\Sigma$ or a nodal singularity on a nodal component.
    By construction, $f_i(P_i \cap \Lambda(G))$ is the closure of a connected component of $\Lambda(G) - \{f_i(a), f_i(b)\}$. Since both $f_i(a), f_i(b)$ are vertices of $\cP_s$ for some $s$, the closure of each connected component of $\Lambda(G) - \{f_i(a), f_i(b)\}$ is the intersection of $\Lambda(G)$ with the union of some level $0$ puzzle pieces.
    Therefore, $f_i(P_i \cap \Lambda(G)) = \bigcup_{j\in I_i} P_j \cap \Lambda(G)$ for some index set $I_i \subset \{1,\ldots, r\}$.
    
    The moreover part follows from the facts that $\partial P_i \cap \Delta_i$ is a hyperbolic geodesic of $\Delta_i$, and that each $h\in G_{\Delta_i}$ is a hyperbolic isometry of $\Delta_i$.
\end{proof}

By Proposition~\ref{prop:MarkovProp}, we inductively define level $n$ puzzle pieces 
$$
\mathfrak{P}^n = \{P_w\}_{w\in \mathcal{I}^n}
$$ 
by pulling back the level $n-1$ puzzle piece.
The following lemma follows from Proposition~\ref{prop:MarkovProp} by induction.
\begin{lem}
    Let $\mathcal{F} = \{f_1,\ldots, f_r\}$ be a Basilica Bowen-Series map.
    For each $P_{w} \in \mathfrak{P}^n$, there exists $P_{w'} \in \mathfrak{P}^{n-1}$ so that $P_{w} \subseteq P_{w'}$.

    Suppose that $P_{w}\in \mathfrak{P}^n$ with $P_{w} \subseteq P_s \in \mathfrak{P}^0$. There exists $P_{w''} \in \mathfrak{P}^{n-1}$ so that
    $$
    f_s: P_{w} \longrightarrow P_{w''}
    $$
    is a homeomorphism and is conformal in its interior.
\end{lem}

\subsubsection{Topological expansion}
\begin{lem}\label{lem:tp}
    The Basilica Bowen-Series map $\mathcal{F}$ is topologically expanding. More precisely, let $(P_{w})_w$ be a nested infinite sequence of level $n$ puzzle pieces. Then 
    $$
    \bigcap_w P_{w}
    $$
    is a singleton set.
\end{lem}
\begin{proof} 
    Let $a$ be a contact point of $\Lambda(G)$; i.e.,  the point of intersection of the boundaries of two connected components of $\Omega(G) - \Delta_\infty$. Then there exists some $n \geq 0$ and a level $n$ puzzle piece $P_{w} \in \mathfrak{P}^n$ so that $a \in \partial P_{w}$. Indeed, suppose $a$ is on the boundary of $\Delta_i$ for some $i = 1,\ldots, J$. Then by construction, $a$ is eventually mapped to some vertex of $\cP_i$. Otherwise, one can easily verify that $a$ is eventually mapped to $\partial \Delta_i$ for some $i$.

    Since the set of all such contact points are dense on $\Lambda(G)$, we conclude that each intersection $\bigcap_{n = 0}^\infty P_{w} \cap \Lambda(G)$ is a singleton set.
    On the other hand, since all puzzle pieces are bounded by hyperbolic geodesics in components of $\Omega(G)$ along with their accumulation points, the shrinking of $\bigcap_w P_{w} \cap \Lambda(G)$ to points implies that $\bigcap_w P_{w} \cap \Omega(G)$ is empty.
    The lemma follows.
\end{proof}

\subsection{The induced Markov map on $\partial^I\Delta_s$, and the pull-back Bowen-Series maps}\label{pullback_bs_subsec}

A Basilica Bowen-Series map $\cF$ associated with $(\Sigma,\mathcal{C})$ induces piecewise conformal, Markov maps on the ideal boundaries of the components $\Delta_s$ of $\Omega(G)$. In this section, we now study certain regularity properties of these Markov maps. For simplicity of exposition, we assume that none of the nodal components $\Sigma_i$ are parabolic; i.e., $l_0=j_0$. The results can be extended to the general case using the discussion in \S~\ref{subsec:parabolicnodal}.

Let $s\in \{1,\ldots, j_0\} \cup \{\infty\}$.
Let $\Phi_s: \D \longrightarrow \Delta_s$ be a Riemann uniformization.
Let $\mathcal{F}|_{\Delta_s}$ be the restriction of the Basilica Bowen-Series map on $\Delta_s$ (recall that strictly speaking, $\cF$ is only defined on a subset of $\Delta_s$), and let $\mathcal{F}_s = \{f_{1,s},\ldots, f_{r_s, s}\}$ be the conjugate of $\mathcal{F}|_{\Delta_s}$ via $\Phi_s$.
By the Schwarz reflection principle, $\mathcal{F}_s$ extends to a conformal Markov map on $\mathbb{S}^1 = \partial \D$ (see \S~\ref{sec:cmm}).

Note that $\mathbb{S}^1$ is identified with the ideal boundary $\partial^I \Delta_s$. If $s \neq \infty$, then $\Delta_s$ is a Jordan domain, and hence $\partial^I \Delta_s$ can be identified with $\partial \Delta_s$ in this case.
We will call $\mathcal{F}_s = \{f_{1,s}, \ldots, f_{r_s,s}\}$ on $\mathbb{S}^1$ the induced Markov map on $\partial^I\Delta_s \cong \mathbb{S}^1$.

We give the induced Markov map a special name if $s = \infty$, and call $\mathcal{F}_\infty$ on $\partial^I\Delta_\infty$ the {\em pull-back Bowen-Series map}. This map will play an important role in our mating results in a sequel paper.

\begin{prop}\label{prop:basilicabssymmetric}
    Let $\mathcal{F}$ be a Basilica Bowen-Series map for a geometrically finite Bers boundary group.
    Let $s\in \{1,\ldots, j_0\} \cup \{\infty\}$ and let $\mathcal{F}_s = \{f_{1,s}, \ldots, f_{r_s,s}\}$ be the induced Markov map on $\partial^I\Delta_s \cong \mathbb{S}^1$.
    Then $\mathcal{F}_s$ is topologically expanding, conformal, and admits a puzzle structure. 
    
    Moreover, all break-points are symmetric. More precisely, we have the following.
    \begin{itemize}
        \item If $s \in \{1,\ldots, j_0\}$, then any break-point is symmetrically parabolic.
        \item If $s = \infty$, then a break-point is symmetrically parabolic if it corresponds to a cusp of the surface $\Sigma$; and it is symmetrically hyperbolic otherwise.
    \end{itemize}
    In particular, if $\Sigma$ is compact, then any break-point for $\mathcal{F}_\infty$ is symmetrically hyperbolic.
\end{prop}

\begin{proof}
    By Lemma~\ref{lem:tp}, $\mathcal{F}$ is topologically expanding. Thus, for each $s \in \{1,\ldots, j_0\} \cup \{\infty\}$, the induced Markov map $\mathcal{F}_s$ is topologically expanding as well.
    
    Consider the intersection of the level $0$ puzzle pieces in $\mathfrak{P}^0$ with $\Delta_s$, and pull them back under the Riemann uniformization $\Phi_s:\D\to\Delta_s$ to obtain the sets $\widetilde{P}_{1,s},\ldots, \widetilde{P}_{r_s,s}$. Let $r: \widehat{\C} \longrightarrow \widehat{\C}$ be the reflection in the unit circle $\mathbb{S}^1$, and $P_{i,s}:= \overline{\widetilde{P}_{i,s} \cup r(\widetilde{P}_{i,s}})$. Then $f_{i,s}$ extends to a conformal map on $\Int (P_{i,s})$ by the Schwarz reflection principle. Therefore, $\mathcal{F}_s$ is conformal.
    
    The collection $\{P_{1,s}, \ldots, P_{r_s, s}\}$ forms the level $0$ puzzle pieces for $\mathcal{F}_s$. By Proposition~\ref{prop:MarkovProp}, the preimage of $P_{i,s}$ is contained in some level $0$ puzzle pieces. Therefore, $\mathcal{F}_s$ admits a puzzle structure.

    For the last part, note that if $s \in \{1,\ldots, j_0\}$, then $\mathcal{F}_s$ is the Bowen-Series Markov map associated to the core polygon $\Phi_s^{-1}(\cP_s)$. Thus by Lemma~\ref{lem:symFs}, every break-point is symmetrically parabolic.

    Let us now consider the case $s = \infty$. Let $a$ be a periodic break-point, and let $q^\pm$ be the period of the right/left orbit $a_n^\pm$ under the induced map $\cF_\infty$. Further, let $I_{a^\pm, n} \subseteq \mathbb{S}^1$ be the level $n$ Markov intervals of $\mathcal{F}_\infty$ associated with the right/left orbit $a_n^\pm$. Note that $\mathcal{F}^{q^\pm}_\infty(I_{a^\pm, q^\pm}) = I_{a^\pm, 0}$.
    
    Note that $\Phi_\infty(a) \in \Lambda(G)$ is a vertex of the pinched core polygon. Let $P_{a^\pm, n} \in \mathfrak{P}^{n}$ be the level $n$ puzzle pieces associated to $I_{a^\pm, n}$.
    Let $t^\pm \in \{1,\ldots, j_0\}$ be such that $\Phi_\infty(a) \in \partial \Delta_{t^\pm}$ and $P_{a^\pm, 0} \cap \Delta_{t^\pm} \neq \emptyset$.
    Let $I^{t^\pm}_{a^\pm, n} \subseteq \partial^I\Delta_{t^\pm}$ be the corresponding Markov intervals for $\mathcal{F}_{t^\pm}$.
    By Lemma~\ref{lem:symFs}, the restriction $\mathcal{F}^{q^\pm}_{t^\pm} |_{I^{t^\pm}_{a^\pm, q^\pm}}$ is repelling on the corresponding sides of $\Phi_{t^\pm}^{-1}(\Phi_\infty(a))$, and generates the stabilizer of $\Phi_{t^\pm}^{-1}(\Phi_\infty(a))$ in the Fuchsian group $\Phi_{t^\pm}^{-1}\ G_{\Delta_{t^\pm}}\ \Phi_{t^\pm}$, where $G_{\Delta_{t^\pm}}$ is the stabilizer subgroup of $\Delta_{t^\pm}$ in $G$.
    Therefore, each of the restrictions $\mathcal{F}^{q^+}|_{P_{a^+, q^+}}$, and $\mathcal{F}^{q^-}|_{P_{a^-, q^-}}$ is repelling at $\Phi_\infty(a)$, and generates the subgroup
    $$
    H= \{g \in G_{\Delta_{t^\pm}}: g(\Phi_\infty(a)) = \Phi_\infty(a)\}\subseteq G.
    $$ 
    In particular, these M{\"o}bius maps are primitive elements in $G$.
    Note that $H$ is simply the stabilizer of $\Phi_\infty(a)$ (in $G$) if there are no parabolic nodal components. Otherwise, the stabilizer of $\Phi_\infty(a)$ may be isomorphic to $H \times \Z/2\Z$, where the factor $\Z/2\Z$ comes from the group element that interchanges two components of $\Omega(G) - \Delta_\infty$ that are attached to $\Phi_\infty(a)$.
    Therefore, there exists $\gamma \in \Aut(\D)$ so that
    $$
  \mathcal{F}^{q^\pm}_\infty|_{I_{a^\pm, q^\pm}} =
  \begin{cases}
       \gamma^{\pm 1}\quad \textrm{if $\Phi_\infty(a)$ corresponds to a cusp of the surface $\Sigma$,}\\
       \gamma\qquad \textrm{otherwise}.
  \end{cases}
  $$
  In the former case, $a$ is symmetrically parabolic, and in the latter case, $a$ is a symmetrically hyperbolic periodic point for $\cF_\infty$. The proposition follows.
\end{proof}

\section{Quasiconformal model for limit set and Julia set}\label{sec:qc_conj}
In this section, we prove our mains results regarding quasiconformal geometry of limit and Julia sets (Theorem~\ref{thm:qcclassfn-ltsets}, Theorem~\ref{thm-basilica-stdpolymodel} and Theorem~\ref{thm:infKleinianRational}).

In \S~\ref{subsec:topconj}, we first show that for any Basilica Bowen-Series map $\mathcal{F}_G$ of a Bers boundary group $G$ (or a rational map with fat Basilica Julia set), we can construct a Basilica $\mathfrak{Q}$-map $\widetilde{\mathcal{F}}$ that is topologically conjugate to $\mathcal{F}_G$ on the Julia set.
In \S~\ref{subsuec:quasicon}, we show that we can modify $\widetilde{\mathcal{F}}$ via Theorem~\ref{thm:symmetrichyperbolic} and obtain a new Basilica $\mathfrak{Q}$-map $\mathcal{F}$ which is quasiconformally conjugate to $\mathcal{F}_G$.
In \S~\ref{subsec:QuasiInv}, we introduce quasiconformal invariants that distinguish various Basilica limit sets and Julia sets.
Finally, we assemble the ingredients
to prove the main theorems in \S~\ref{proof_three_main_thm_subsec}.

\subsection{Topological conjugacy}\label{subsec:topconj}
Recall that in \S~\ref{sec-basilica}, we introduced Basilica $\mathfrak{Q}-$maps associated with the fat Basilica polynomial $Q(z) = z^2-\frac{3}{4}$,   and in \S~\ref{section:BasBS}, we defined analogous objects in the Kleinian group world, termed as Basilica Bowen-Series maps.
In the following, we construct Basilica $\mathfrak{Q}-$maps that serve as topological models both for Basilica Bowen–Series maps and for any polynomial whose Julia set is a fat Basilica.

\subsubsection{Topological model for compact Basilica Bowen-Series maps}\label{basi_bs_model_subsec}

We employ the notation of \S~\ref{section:BasBS}, and assume that none of the nodal components $\Sigma_i$ are parabolic. As explained in \S~\ref{subsec:parabolicnodal}, this is not a loss of generality as far as the action of the associated Basilica Bowen-Series map is concerned.

\begin{prop}\label{prop:topmodelbowenseries}
    Let $\Sigma$ be a hyperbolic compact surface, and let $G$ be a geometrically finite Bers boundary group associated to $(\Sigma, \mathcal{C})$. Let $\mathcal{F}_G$ be a Basilica Bowen-Series map.
    Then there exists a Basilica $\mathfrak{Q}-$map $\mathcal{F}: \mathfrak{P}^1\longrightarrow \mathfrak{P}$ whose restriction to $J(Q)$ is topologically conjugate to the restriction of $\mathcal{F}_G$ on  $\Lambda(G)$.
\end{prop}
\begin{proof}
    Following the notation of \S~\ref{subsec:pinchedcore}, let $\{\cP_1,\ldots, \cP_{j_0}\}$ be a pinched core polygon for $G$, and let $\Delta_1,\ldots, \Delta_{j_0}$ be the corresponding components in the domain of discontinuity with $\cP_i \subseteq \Delta_i$. Let $\cP_\infty$ be the corresponding ideal polygon in $\Delta_\infty$, where $\Delta_\infty$ represents the $G-$invariant component.
    Let $\mathcal{F}_G$ be a Basilica Bowen-Series map associated to this pinched core polygon.
    Recall that the closure of the complementary components of $\overline{\cP_\infty} \cup \bigcup_i\overline{\cP_i}$ are level $0$ puzzle pieces, which are denoted by $\{P_1,\ldots, P_r\}$. We denote the level $1$ puzzle pieces by $\{S_1,\ldots, S_s\}$, so that for each $S_i$, there exists some $P_{\sigma(i)} \in \{P_1,\ldots, P_r\}$ with 
    $$
    \mathcal{F}_G: \Int(S_i) \longrightarrow \Int(P_{\sigma(i)})
    $$
    a conformal isomorphism.
    Note that the complementary components of the union of the level $1$ puzzle pieces give a collection of ideal polygons in $\Omega(G)$, and we denote them by $\{\mathcal{S}_\infty, \mathcal{S}_1,\ldots, \mathcal{S}_{j_0}, \mathcal{S}_{j_0+1},\ldots \mathcal{S}_{j_1}\}$.
    As $\Sigma$ is compact, every vertex of $\{\cP_1,\ldots, \cP_{j_0}\}$ is a touching point of two components of $\Omega(G) - \Delta_\infty$ (as all parabolics of $G$ are accidental).

    We can construct Basilica puzzles $\mathfrak{P}$ and $\mathfrak{P}^1$ on $J(Q)$ that are homeomorphic to the level $0$ and level $1$ puzzles for $\mathcal{F}_G$ as follows.
    Start with the vertex $v_1$ of  $\cP_1 \subseteq \Delta_1$. Let $U_0$ be the Fatou component of $Q$ containing the critical point (cf. \S\ref{sec-basilica}). We define the corresponding point $x_1$ as the fixed point of $Q$ on $\partial U_0$.
    
    Inductively, let $v_1, \ldots, v_n$ be a collection of vertices that form a connected set when all open edges of $\cP_1,\ldots, \cP_{j_0}$ are added.
    Suppose that their corresponding vertices $x_1,\ldots, x_n$ are constructed with the following property: if $x_i, x_j \in \partial U_m$ is a pair such that the arc $(x_i, x_j) \subseteq \partial U_m$ is disjoint from $\{x_1,\ldots, x_n\}$, then $[x_i, x_j]$ is dyadic.
    
    Let $v_{n+1}$ be some vertex of the pinched core polygon for $G$ adjacent to a vertex in $\{v_1,\cdots,v_n\}$. 
    Suppose that $v_{n+1} \in \partial \Delta_m$. Let $v_{i_1}, v_{i_2} \in \{v_1, \ldots, v_n\} \cap \partial \Delta_m$ be the closest vertices to $v_{n+1}$ in counterclockwise and clockwise order, respectively. Note that $v_{i_1}$ could be the same as $v_{i_2}$.
    We can choose $x_{n+1} \in \partial U_m$ so that the arcs connecting $x_{i_1}, x_{n+1}$ and $x_{i_2}, x_{n+1}$ are dyadic. This is possible as the arc between $x_{i_1}$ and $x_{i_2}$ is dyadic.
    
    

    This set of points induces a Basilica puzzle $\mathfrak{P}$. The construction for $\mathfrak{P}^1$ is similar.
    Note that $\mathcal{F}_G$ induces a map on the index set of $\mathfrak{P}^1$ to the index set of $\mathfrak{P}$, and we obtain a Basilica $\mathfrak{Q}-$map $\mathcal{F}:\mathfrak{P}^1 \longrightarrow\mathfrak{P}$.
    Since the Basilica Bowen-Series map is topologically expanding by Lemma~\ref{lem:tp}, we conclude that $\mathcal{F}_G$ and $\mathcal{F}$ are topologically conjugate between the limit set and the Julia set by Proposition~\ref{prop-combconjimpliestopconj}.
\end{proof}

\subsubsection{Topological model for polynomials with fat Basilica Julia set}\label{subsubsec:topmodelpoly}
Recall that the Julia set $J(g)$ of a rational map $g$ is a {\em Basilica} if it is homeomorphic to $J(Q)$, where $Q(z) = z^2-\frac{3}{4}$. After possibly conjugating $g$ with a M{\"o}bius map, we can assume that the unique non-Jordan Fatou component of $g$ is unbounded.
The Julia set $J(g)$ is a {\em fat Basilica} if it is a Basilica and
\begin{enumerate}
    \item\label{fatbasi:2} each bounded Fatou component is a quasi-disk; and
    \item\label{fatbasi:3} if any two bounded Fatou components touch, they touch tangentially.
\end{enumerate}
Here we say that two disjoint Jordan domain $U,V\subset \widehat\C$ are \textit{tangent to each other at a point} $z_0\in \partial U\cap \partial V$ if there exists $\theta_0\in [0,2\pi)$ such that for each $\varepsilon>0$ there exists $\delta>0$ with the property that $\{z_0+ re^{i(\theta_0+\theta)}:  r\in (0,\delta), \,\, |\theta|<\pi/2-\varepsilon\}\subset U$ and $\{z_0+ re^{i(\theta_0+\theta)}:  r\in (-\delta,0), \,\, |\theta|<\pi/2-\varepsilon\}\subset V$.

It is easy to see that
\begin{lem}
    Let $g$ be a rational map whose Julia set is a Basilica. Then the unbounded Fatou component $U_\infty$ is invariant under $g$.
\end{lem}

\begin{lem}\label{lem:noJuliaCrit}(c.f. \cite[Lemma 3.2]{LZ25})
    Let $g$ be a rational map whose Julia set is a fat Basilica. Then $J(g)$ does not contain any critical points.
\end{lem}
\begin{proof}
    It is easy to see that a non-contact point $x$ of $J(g)$ cannot be a critical point as the map is locally $\deg(x):1$, and the Julia set is fully invariant under $g$. It is also easy to see that the boundary of a bounded Fatou component cannot contain a critical point by conditions~\eqref{fatbasi:2} and \eqref{fatbasi:3}.
\end{proof}

By Lemma~\ref{lem:noJuliaCrit}, we conclude that there exists a hyperbolic post-critically finite polynomial $g_0$ which is topologically conjugate to $g$ on the Julia sets (cf. \cite[Corollary~9.7]{LMMN}).
Let $T_0$ be the Hubbard tree of the $g_0$. Then the dynamics of $g_0: T_0 \longrightarrow T_0$ is simplicial, i.e., $g_0$ sends an edge of $T_0$ to an edge of $T_0$.

Recall that a contact point is the intersection point of two distinct bounded Fatou components.
\begin{lem}\label{lem:cutpnteventuallypara}(c.f. \cite[Lemma 3.4]{LZ25})
    Let $g$ be a rational map whose Julia set is a fat Basilica. Then every contact point is eventually mapped to a unique fixed point, which is parabolic with multiplicity $3$.
\end{lem}
\begin{proof}
    Since there are no critical points on the Julia set, contact points are pre-periodic.
    By condition~\eqref{fatbasi:3}, every periodic contact point is parabolic with multiplicity $3$.

    Now suppose that there are two periodic contact points $a, b$. Let $g_0 : T_0 \longrightarrow T_0$ be the corresponding simplicial map on the Hubbard tree. After passing to some iterate, we have that $a, b$ correspond to two fixed points of $g_0^k$ on $T_0$. Since $g_0$ is simplicial on $T_0$, the arc connecting $a, b$ is fixed under $g_0^k$. Therefore, there exists a chain of bounded Fatou components $V_1, V_2, ..., V_s$ which are fixed under $g^k$. Denote $v_0 = a \in \partial V_1$ and $v_s = b \in \partial V_s$ and $v_i = \partial V_i \cap \partial V_{i+1}$. Then $v_i$ is parabolic with multiplicity $3$. This means that $V_1$ is the basin of attraction for both $v_0$ and $v_1$, which is a contradiction. The lemma follows.
\end{proof}

\begin{prop}\label{prop:topmodelJulia}
    Let $g$ be a rational map whose Julia set is a fat Basilica.
    Then there exists a Basilica $\mathfrak{Q}-$map $\mathcal{F}: \mathfrak{P}^1\longrightarrow \mathfrak{P}$ which is topologically conjugate to $g$ on $J(Q)$.
\end{prop}
\begin{proof}
    Let $(V_i, X_i)_{i\in \mathcal{I}}$ be a finite collection so that 
    \begin{itemize}
        \item each $V_i$ is a bounded Fatou component of $g$;
        \item $\overline{\bigcup_{i\in \mathcal{I}} V_i}$ is connected, invariant under $g$ and contains the critical values;
        \item $X_i \subseteq \partial V_i$ consists of finitely many contact points on $\partial V_i$;
        \item $|X_i|\geq 2$ for each $i$;
        \item $\partial V_{i} \cap \partial V_{j} \subseteq X_i \cap X_j$ for $i \neq j$; and
        \item $\bigcup X_i$ is invariant under $g$.
    \end{itemize}
    Let $(W_j, Y_j)_{j \in \mathcal{J}}$ be the pull-back of $(V_i, X_i)_{i\in \mathcal{I}}$ under $g$.
    Then by construction, each connected component of $J(g) - \bigcup_j Y_j$ is mapped homeomorphically under $g$ to a connected component of $J(g) - \bigcup_i X_i$.
    Now similar to the proof of Proposition~\ref{prop:topmodelbowenseries}, we can construct $\mathfrak{P}$ and $\mathfrak{P}^1$ according to the data $(V_i, X_i)$ and $(W_j, Y_j)$, and a Basilica $\mathfrak{Q}-$map $\mathcal{F}: \mathfrak{P}^1 \longrightarrow \mathfrak{P}$ which is topologically conjugate to $g$ on the Julia sets.
\end{proof}

\subsection{Quasiconformal conjugacy}\label{subsuec:quasicon}
In Proposition~\ref{prop:topmodelbowenseries} and Proposition~\ref{prop:topmodelJulia}, we constructed Basilica $\mathfrak{Q}-$maps that model Basilica Bowen-Series maps associated with compact surfaces and polynomials with fat Basilica Julia sets, respectively.
In this subsection, we first apply Theorem~\ref{thm:symmetrichyperbolic} to obtain modified refinements of these Basilica $\mathfrak{Q}-$maps, and then employ quasiconformal surgery to show that the resulting topological conjugacies are quasi-symmetric.

\subsubsection{Quasiconformal and David removability}
We need the following lemma for quasiconformal and David removability of Julia sets. 

\begin{lem}\label{lem:eccentricity}
\noindent\begin{enumerate}[leftmargin=8mm]
    \item (Quasiconformal removability) Let $J\subset \widehat\C$ be the Julia set of a geometrically finite polynomial. Suppose that $J$ is connected. Let $\psi\colon \widehat \C\to \widehat\C$ be a homeomorphism.
    If $\psi$ is quasiconformal on $\widehat\C- J$, then $\psi$ is quasiconformal on~$\widehat\C$.
    \item (David removability) Let $J\subset \widehat\C$ be the Julia set of a hyperbolic polynomial $p$. Suppose that $J$ is connected. Let $\psi\colon \widehat \C\to \widehat\C$ be a homeomorphism.
    If $\psi$ is a David map on $\widehat\C- J$, then $\psi$ is a David map on $\widehat\C$.
    \end{enumerate}
\end{lem}
\begin{proof}
   The first part follows from the fact that the connected Julia set of a geometrically finite polynomial is conformally removable (see \cite[Theorem~9.2]{LMMN}). For the second part, we note that the basin of infinity of $p$ is a John domain, and hence $J(p)$ is $W^{1,1}-$removable (cf. \cite{Jon95}, \cite[Theorem~4]{JS00}). According to \cite[Lemma~2.9]{LMMN}, such sets are removable for David maps.
\end{proof}

\subsubsection{Quasiconformal surgery}
\begin{prop}\label{prop:quasisurg}
    Let $\Sigma$ be a hyperbolic compact surface, and let $G$ be a geometrically finite Bers boundary group associated to $(\Sigma, \mathcal{C})$. Let $\mathcal{F}_G$ be a Basilica Bowen-Series map. Then there exists a Basilica $\mathfrak{Q}-$map $\mathcal{F}: \mathfrak{P}^1\longrightarrow \mathfrak{P}$ which is quasiconformally conjugate to $\mathcal{F}_G$, restricted to the respective Julia and limit sets. More precisely, this means that there exists a quasiconformal map $H: \widehat{\C} \longrightarrow \widehat{\C}$ so that
    $$
    H \circ \mathcal{F}(z) = \mathcal{F}_G \circ H(z)
    $$
    for all $z \in J(Q)$.

    Similarly, let $g$ be a rational map whose Julia set is a fat Basilica. Suppose that the unbounded Fatou component $U_\infty$ is attracting or super-attracting.
    Then there exists a Basilica $\mathfrak{Q}-$map $\mathcal{F}: \mathfrak{P}^1\longrightarrow \mathfrak{P}$ which is quasiconformally conjugate to $g$, restricted to the respective Julia sets.
\end{prop}
\begin{proof}
    We present the proof for Basilica Bowen-Series maps. The result for rational maps with fat Basilica Julia set can be proved similarly using Proposition~\ref{prop:topmodelJulia}.

    Let $\mathcal{F}_G$ be a Basilica Bowen-Series map. 
    Let $\Delta_s, s\in \{1,..., j_0\} \cup \{\infty\}$ be the components of $\Omega(G)$ associated to the pinched core polygon, and let $\Delta_s$, $ s\in \{j_0+1,..., j_1\}$ be the additional components of $\Omega(G)$ associated to the pull-back of the pinched core polygon via $\mathcal{F}_G$. They are characterized by the following.
    \begin{itemize}
        \item $\Delta_s, s\in \{1,..., j_0\} \cup \{\infty\}$, are the components of $\Omega(G)$ that intersect at least two level $0$ puzzle pieces;
        \item $\Delta_s, s\in \{1,..., j_0, j_0+1,..., j_1\} \cup \{\infty\}$, are the components of $\Omega(G)$ that intersect at least two level $1$ puzzle pieces.
    \end{itemize}
    Denote the induced Markov maps on the ideal boundaries by
    \begin{align*}
    (\mathcal{F}_G)_{\bdd} : \bigcup_{s=1}^{j_1} \partial^I \Delta_s &\longrightarrow \bigcup_{s=1}^{j_1} \partial^I \Delta_s, \text{ and }\\
    (\mathcal{F}_G)_\infty: \partial^I \Delta_\infty &\longrightarrow \partial^I \Delta_\infty.
    \end{align*}
    Note that by Proposition~\ref{prop:basilicabssymmetric}, all break-points of $(\mathcal{F}_G)_s$ are symmetrically parabolic if $s \neq \infty$ and are symmetrically hyperbolic if $s = \infty$.

    By Proposition~\ref{prop:topmodelbowenseries}, there exists a Basilica $\mathfrak{Q}-$map $\widetilde{\mathcal{F}}: \widetilde{\mathfrak{P}}^1\longrightarrow \widetilde{\mathfrak{P}}$ whose restriction to $J(Q)$ is topologically conjugate to $\mathcal{F}_G\vert_{\Lambda(G)}$.
    Note that since Basilica Bowen-Series maps are topologically expanding by Lemma~\ref{lem:tp}, the map $\widetilde{\mathcal{F}}$ is also topologically expanding.
    By Theorem~\ref{thm:symmetrichyperbolic}, there exists a Basilica $\mathfrak{Q}-$map $\mathcal{F}: \mathfrak{P}^1\longrightarrow \mathfrak{P}$ such that
    \begin{itemize}
        \item $\mathcal{F}\vert_{J(Q)}$ is topologically conjugate to $\widetilde{\mathcal{F}}\vert_{J(Q)}$ (and hence to $\mathcal{F}_G\vert_{\Lambda(G)}$); and 
        \item the induced Markov map $\mathcal{F}_\infty$ is symmetrically hyperbolic.
    \end{itemize}
    By Proposition~\ref{prop:bdd_fatou_symm_para}, all break-points of the induced Markov map 
    $$
    \mathcal{F}_{\bdd}: \bigcup_{s=1}^{j_1} \partial^I U_s \longrightarrow \bigcup_{s=1}^{j_1} \partial^I U_s
    $$
    are symmetrically parabolic, where $U_s$ are the Fatou components of $Q$ corresponding to $\Delta_s$.

    Denote the topological conjugacy between $\mathcal{F}$ and $\mathcal{F}_G$ by $H$.
    Note that for each $s \in \{1,..., j_0\} \cup \{\infty\}$, $(\mathcal{F}_G)_s$ is topologically conjugate to $\mathcal{F}_s$. By Proposition~\ref{prop:qsmarkovmapFiniteCircle}, they are quasi-symmetrically conjugate. Therefore, there exists $K>0$ so that we can extend $H|_{\partial U_s}$ $K-$quasiconformally to $H|_{U_s}: U_s \longrightarrow \Delta_s$.

    Let $U$ be a bounded Fatou component that is contained in some level $1$ puzzle piece. Then there exists a minimal integer $n$ such that $\mathcal{F}^n(U) = U_s$ for some $s\in \{1,..., j_1\}$. Note that $H\circ \mathcal{F}^n = \mathcal{F}^n_G \circ H$ on $\partial U$. As $\mathcal{F}_G$ is conformal in $U$, we can extend $H|_{\partial U}$ to $\overline{U}$ as a homeomorphism by pulling back $H|_{\mathcal{F}^n(U)}$.

    The above construction yields a homeomorphism $H$ of $\widehat{\C}$ which is $K-$quasiconformal on $\widehat{\C}- J(Q)$, and conjugates  $\mathcal{F}\vert_{J(Q)}$ to $\mathcal{F}_G\vert_{\Lambda(G)}$.
    By Lemma~\ref{lem:eccentricity}, the homeomorphism $H$ is quasiconformal on $\widehat{\C}$.
\end{proof}

\subsection{Quasiconformal invariant}\label{subsec:QuasiInv}
In this subsection, we introduce certain combinatorial invariants that allow us to quasiconformally distinguish Basilica limit sets or Julia sets.

\subsubsection{Purely accidental parabolic and persistently parabolic components}\label{pure_acci_para_subsec}
Let $G$ be a geometrically finite Bers boundary group associated to $(\Sigma, \mathcal{C})$.
Let $\Delta$ be a component of $\Omega(G) - \Delta_\infty$ associated to a nodal component $\Sigma_i$.
We say that $\Delta$ is {\em purely accidental parabolic} if $\Sigma_i$ contains no cusps of $\Sigma$, and {\em persistently parabolic} otherwise.

\begin{prop}\label{prop:qsinv}
    Let $G_i, i\in\{1,2\}$, be geometrically finite Bers boundary groups. Let $\Delta_i$ be a component of $\Omega(G_i) - \Delta_\infty(G_i)$.
    Suppose that $\Delta_1$ is purely accidental parabolic and $\Delta_2$ is persistently parabolic.
    Then there is no quasiconformal map $\Phi: \widehat{\C} \longrightarrow \widehat{\C}$ so that $\Phi(\Lambda(G_1)) = \Lambda(G_2)$ and $\Phi(\Delta_1)=\Delta_2$.

    In other words, quasiconformal maps send purely accidental parabolic components to purely accidental parabolic ones, and persistently parabolic components to persistently parabolic ones.
\end{prop}
\begin{proof}
    The proof is an adaptation of \cite[Theorem 1.8]{McM25}.
    By way of contradiction, suppose that such a quasiconformal map $\Phi$ exists.

    For each $z \in \Lambda(G_i)$, we let $K(z)$ be the set of all possible limits (in the Hausdorff topology on closed subsets of $\widehat{\C}$) of blowups $\frac{1}{r}(\Lambda(G_i)-z)$ as $r\to 0$. Note that each set $A \in K(z)$ is a closed set containing $0$.
    
    Let $a \in \partial \Delta_2$ be a parabolic fixed point of $g \in G_2$ so that $a$ is not a contact point. Such a point exists as $\Delta_2$ is persistently parabolic.
    By construction, each $A\in K(a)$ is a line, as locally, the limit set $\Lambda(G_2)$ is contained in the region bounded by two tangent parabolas (see Figure~\ref{fig:persistantBasilica}).
    
    Let $b = \Phi^{-1}(a)$. Let $P^n$ be a level $n$ puzzle piece that contains $b$.
    Since $b$ is not a contact point and $\Delta_1$ is purely accidental parabolic, $b\in \Int(P^n)$.
    For each level $0$ puzzle $P_i$, we can choose $X_i \subseteq P_i$ so that $\Int(P_i) - X_i$ is an annulus and the orbit of $b$ passes through $\bigcup \Int(X_i)$ infinitely often.
    Indeed, by topological expansion of $\cF_{G_1}$, we can choose $X_i$ large enough so that the only non-escaping points in $\bigcup \left(P_i-X_i\right)$ are the boundary points $\bigcup \partial P_i\cap \Lambda(G_1)$; i.e., all other points in $\bigcup \left(P_i-X_i\right)$ escape these annuli under iterates of $\cF_{G_1}$.
    Let $n_k$ denote the sequence so that $\mathcal{F}_{G_1}^{n_k}: P^{n_k} \longrightarrow P_{i(n_k)}$ sends $b$ into $X_{i(n_k)}$.
    There exists an $\epsilon_0>0$ such that locally near $\mathcal{F}_{G_1}^{n_k}(b)$, the limit set is not contained in a cone formed by two straight lines meeting (at $\mathcal{F}_{G_1}^{n_k}(b)$) at an angle $\epsilon\in(0,\epsilon_0)$.
    Applying the Koebe distortion theorem to $(\mathcal{F}_{G_1}^{n_k})^{-1}$ on $X_{i(n_k)}$, we see that $(\mathcal{F}_{G_1}^{n_k})^{-1}$ only mildly distorts the limit set and the two lines. 
    This means that a limit $B \in K(b) $ contains copies of limbs of the Basilica, and is thus not a topological arc.

    By compactness of quasiconformal maps, for each $B \in K(b)$, there is some $A \in K(a)$ related to $B$ by a quasiconformal homeomorphism of $\C$ fixing $0$. But this is a contradiction as $A$ and $B$ are not homeomorphic.
\end{proof}

A similar proof also gives the following.
\begin{prop}\label{prop:inequivJulia}
\noindent\begin{enumerate}[leftmargin=8mm]
    \item\label{prop:inequivJulia:item1} Let $\Sigma$ be a hyperbolic non-compact surface, and let $G$ be a geometrically finite Bers boundary group associated to $(\Sigma, \mathcal{C})$.
    Then $\Lambda(G)$ is not quasiconformally homeomorphic to $J(Q)$.
    \item \label{prop:inequivJulia:item2} Let $g$ be a rational map with fat Basilica Julia set. Suppose that the unbounded Fatou component $U_\infty(g)$ is neither attracting nor super-attracting.
    Then $J(g)$ is not quasiconformally homeomorphic to $J(Q)$. Moreover, $J(g)$ is not quasiconformally homeomorphic to $\Lambda(G)$ for any geometrically finite Bers boundary group.
\end{enumerate}
\end{prop}
\begin{proof}
1)    Let $U$ be a bounded Fatou component of $Q$, and let $a \in \partial U$ be a non-contact point. Then $d(Q^n(a), \R) \asymp 1$ for infinitely many $n$. Since the post-critical set is contained in $\R$, we can blow up a neighborhood of $a$ to a definite size with bounded distortion for infinitely many $n$. Then the same proof as in Proposition~\ref{prop:qsinv} demonstrates the non-existence of a global quasiconformal homeomorphism that maps $U$ to a persistently parabolic component $\Delta$ of $\Omega(G)-\Delta_\infty$.
    
    Since $\Sigma$ is non-compact, at least one component of $\Omega(G) - \Delta_\infty$ is persistently parabolic.
    Therefore, $\Lambda(G)$ is not quasiconformally homeomorphic to $J(Q)$.

2)  A non-contact point of $J(g)$ either lies on the boundary of a unique bounded Fatou component of $g$, or does not lie on the boundary of any bounded Fatou component. A non-contact point of $J(g)$ of the latter kind is called a \emph{tip point}.  

    By Lemma~\ref{lem:noJuliaCrit}, the Julia set $J(g)$ contains no critical points.
    Since $U_\infty(g)$ is neither attracting nor super-attracting, $U_{\infty}(g)$ must be parabolic. The corresponding parabolic fixed point $a$ cannot be on the boundary of any bounded Fatou component by conditions~\eqref{fatbasi:2} and \eqref{fatbasi:3}.
    Therefore, $a$ must be a tip point of $J(g)$.
    Let $A \in K(a)$ be an accumulation point of the blowups of the Julia set near $a$ in the Hausdorff topology. Then $A$ is a ray.
    
    On the other hand, let $b \in J(Q)$ be a tip point. By the same argument as in the previous case,  there exist some $B \in K(b)$ which is not topologically a ray.
    Therefore, $J(g)$ is not quasiconformally equivalent to $J(Q)$.
    
    The same argument applies to any tip point of the limit set $\Lambda(G)$ of a geometrically finite Bers boundary group $G$. Indeed, a tip point $b$ of $\Lambda(G)$ is not a parabolic fixed point for any element of $G$, and hence by choosing an appropriate blow-up of $\Lambda(G)$ near $b$, we obtain some $B \in K(b)$ which is not topologically a ray. 
    Thus, $J(g)$ is not quasiconformally homeomorphic to $\Lambda(G)$. 
\end{proof}

\subsubsection{Bi-colored contact tree for Basilica limit set}\label{subsubsec:bicolct}
Let $G$ be a Bers boundary group, and let $\Lambda(G)$ be its limit set.
Let $\mathcal{T}$ be the contact graph for the components of $\Omega(G) - \Delta_\infty$, i.e., the vertices of $\mathcal{T}$ correspond to components of $\Omega(G) - \Delta_\infty$, and two vertices are connected via an edge if the two corresponding components touch.
Note that $\mathcal{T}$ is a tree with infinite valence at every vertex.
The planar structure of $\Lambda(G)$ gives a {\em ribbon structure} on $\mathcal{T}$, i.e., a cyclic ordering on the edges adjacent to each vertex.

We color each vertex by black and white depending on whether the corresponding component is purely accidental parabolic or not.
We call the infinite tree $\mathcal{T}$ together with the partition of its vertex set $\mathcal{V} = \mathcal{B} \sqcup \mathcal{W}$ the {\em bi-colored contact tree} for $G$.
Two bi-colored contact trees are equivalent if there is a homeomorphism that preserves the coloring and the ribbon structure.

By Proposition~\ref{prop:qsinv}, the bi-colored contact tree is a quasiconformal invariant, i.e., two quasiconformally equivalent Basilica limit sets have equivalent bi-colored contact trees.
It is easy to see that for two Basilica limit sets, there is a type-preserving homeomorphism between them if and only if their bi-colored contact trees are equivalent.
Therefore, Conjecture~\ref{conj:confdim1univer} in this setting is equivalent to the following.
\begin{conj}\label{conj:comqcinv}
    The bi-colored contact tree is a complete quasiconformal invariant for limit sets of geometrically finite Bers boundary groups, i.e., two Basilica limit sets are quasiconformally equivalent if and only if their bi-colored contact trees are isomorphic.
\end{conj}
With this formulation, Theorem~\ref{thm:qcclassfn-ltsets}, or more precisely, Proposition~\ref{prop:quasisurg} says that two Basilica limit sets with all black vertices are quasiconformally equivalent.

One can easily construct infinitely many nonequivalent bi-colored contact trees as follows.
Let $\Sigma$ be a genus $g$ surface with $2$ punctures. Let $\mathcal{C}$ be a strictly separating multi-curve so that
$$
\Sigma = \Sigma_1+...+\Sigma_g
$$
where each $\Sigma_i$ is a twice punctured torus so that they form a chain of nodal surfaces.
The two punctures of $\Sigma$ are contained in $\Sigma_1$ and $\Sigma_g$, and for $i\in\{2,..., g-1\}$, the nodal surfaces $\Sigma_i$ is connected to $\Sigma_{i-1}$ and $\Sigma_{i+1}$ via the two nodal singularities on $\Sigma_i$.
Therefore, we have a covering map from the bi-colored contact tree $\cT$ associated with $G$ to a bi-colored chain-tree $\widehat{\cT}$ consisting of $g$ vertices $v_1-v_2-v_3-\cdots-v_g$, where $v_1, v_g$ are white and $v_2,\cdots, v_{g-1}$ are black. We consider $\widehat{\cT}$ as the tree underlying a graph of spaces. An essential path is an edge-path in the graph of spaces  that corresponds to a word in normal form. 
Note that any essential path (possibly with backtracking) in the chain-tree $\widehat{\cT}$ can be lifted to a path in $\cT$ without backtracking.
Suppose that $g\geq 4$ is even. Then by construction, the distance between two white vertices of $\cT$ (with respect to the edge metric) is either an even integer, or an odd integer $\geq g-1$. Further, both possibilities are realized.
This (more precisely, the minimum odd distance $g-1$) is clearly an invariant of the equivalence class of bi-colored contact trees. Therefore, we have the following.
\begin{lem}\label{lem:infinKlein}
    There are infinitely many quasiconformally nonequivalent Basilica limit sets, i.e., there are infinitely many Kleinian points in $\mathfrak{B}$.
\end{lem}

\subsubsection{Bi-colored contact tree for Basilica Julia set}\label{subsubsec:bicolctrat}
A complete description of Basilica Julia sets up to type-preserving homeomorphisms (as in Conjecture~\ref{conj:confdim1univer}) is more involved. Instead of giving a complete quasiconformal invariant, here we consider the following rough invariants, which already allow us to see that there are infinitely many quasiconformally nonequivalent Basilica Julia sets.

Let $g$ be a rational map with Basilica Julia set,  and let $\mathcal{T}$ be its contact tree.
We color the vertices black or white depending on whether the corresponding component is a quasidisk or not.
It is easy to construct infinitely many non-equivalent bi-colored contact trees as follows.
Let $g_0$ be a critically fixed polynomial of degree $2d+1$ so that each critical point has local degree $3$, and the corresponding invariant Fatou components $U_1,..., U_{d}$ form a chain of length $d$.
Here for each $i\in\{2\cdots, d-1\}$, the component $U_i$ touches $U_{i-1}$ and $U_{i+1}$ at two repelling fixed points. Note that there are two repelling fixed points on $\partial U_1$ (and $\partial U_d$), but only one of them is a contact point.
Let $g$ be a parabolic polynomial obtained from $g_0$ by a simple pinching, so that each $U_i$, $i\in\{2,..., d-1\}$, remains superattracting, while $U_1, U_d$ become parabolic with orbits in $U_1, U_d$ converging to the non-contact fixed points on these components.
Suppose that $d\geq 4$ is even. Then by construction, as explained in \S~\ref{subsubsec:bicolct}, the distance (in the edge metric) between two white vertices of $\cT$ is either an even integer, or an odd integer $\geq d-1$.
Therefore,  we have the following.
\begin{lem}\label{lem:infinrat}
    There are infinitely many quasiconformally nonequivalent Basilica Julia sets, i.e., there are infinitely many rational points in $\mathfrak{B}$.
\end{lem}

\subsection{Proof of Theorem~\ref{thm:qcclassfn-ltsets}, Theorem~\ref{thm-basilica-stdpolymodel} and Theorem~\ref{thm:infKleinianRational}}\label{proof_three_main_thm_subsec}
We now assemble the ingredients and prove three of our main theorems.
\begin{proof}[Proof of Theorem~\ref{thm:qcclassfn-ltsets}]
    Suppose that $\Omega(G)$ consists of two components. Then $G$ is quasi-Fuchsian and $\Lambda(G)$ is a quasicircle.
    
    Otherwise, $G$ is a geometrically finite Bers boundary group. By Proposition~\ref{prop:structgfbersboundary} and \S~\ref{subsec:constructionBBS}, we can construct a Basilica Bowen-Series map $\mathcal{F}_G$ for $G$. 
    
    Suppose that $\Sigma=\Delta_\infty/G$ is compact.
    By Proposition~\ref{prop:quasisurg}, the limit set $\Lambda(G)$ is quasiconformally equivalent to $J(Q)$.

    Suppose that $\Sigma$ is non-compact. Suppose for contradiction that $\Lambda(G)$ is quasiconformally equivalent to $J(g)$ for some rational map $g$. Since $\Lambda(G)$ is a fat Basilica, $J(g)$ is a fat Basilica as well. 
    
    If $U_\infty$ is attracting or super-attracting, then by Proposition~\ref{prop:quasisurg}, $J(g)$ is quasiconformally equivalent to $J(Q)$. Thus $\Lambda(G)$ is quasiconformally equivalent to $J(Q)$. This is a contradiction to Proposition~\ref{prop:inequivJulia} \eqref{prop:inequivJulia:item1}.

    If $U_\infty$ is neither attracting nor super-attracting, then by Proposition~\ref{prop:inequivJulia} \eqref{prop:inequivJulia:item2}, $J(g)$ is not quasiconformally equivalent to $\Lambda(G)$, which is a contradiction.
    Therefore, $\Lambda(G)$ is not quasiconformally equivalent to any Julia set.
\end{proof}

\begin{proof}[Proof of Theorem~\ref{thm-basilica-stdpolymodel}]
    The `if' direction follows from Proposition~\ref{prop:quasisurg}, and the `only if' direction follows from Proposition~\ref{prop:inequivJulia} \eqref{prop:inequivJulia:item2}.
\end{proof}

\begin{proof}[Proof of Theorem~\ref{thm:infKleinianRational}]
    The theorem follows from Lemma~\ref{lem:infinKlein} and Lemma~\ref{lem:infinrat}.
\end{proof}

\section{David Hierarchy}\label{sec:davidHierarchy}
In this section, we establish Theorem~\ref{thm:davidhi} by showing that the Julia set of the postcritically finite quadratic polynomial $z^2-1$ sits atop the David hierarchy; i.e., $J(z^2-1)$ is the archbasilica in the world of basilicas arising from geometrically finite conformal dynamical systems.

\subsection{Generalized Basilica $\mathfrak{Q}_{\pc}-$maps}\label{subsec:geneBQmap}
Let us first discuss the subtleties that prevent us from using Basilica $\mathfrak{Q}-$maps to model Basilica Bowen-Series maps for non-compact surfaces, and then explain how to modify the construction to address these issues. The 
Basilica Bowen-Series maps for non-compact surfaces
will be referred to as \emph{non-compact Basilica Bowen-Series maps.}

\subsubsection{Persistent parabolics}
Let $G$ be a geometrically finite Bers boundary group associated with $(\Sigma, \mathcal{C})$. Recall that the unique non-Jordan component of the domain of discontinuity $\Omega(G)$ is denoted by $\Delta_\infty$.
Recall also the construction of the Basilica Bowen-Series map $\mathcal{F}_G$ from \S~\ref{section:BasBS}.
If the surface $\Sigma$ is non-compact, the intersection of the boundary of a puzzle piece $P$ with the limit set $\Lambda(G)$ may not be a contact point (see \S~\ref{subsec:pinchedcore}). Specifically, the representation $\rho:\pi_1(\Sigma)\to G$ sends the homotopy classes of curves freely homotopic to punctures of $\Sigma$ to parabolic elements of $G$, and the points of $\partial P\cap\Lambda(G)$ that are not contact points of $\Lambda(G)$ correspond to such \emph{persistently parabolic} elements of the group $G$.

On the other hand, each Basilica puzzle piece for a Basilica $\mathfrak{Q}-$map intersects the Julia set $J(Q)$ at contact points.
Thus, to model non-compact Basilica Bowen-Series maps as $(J, \mathfrak{Q})-$maps, we need to extend the definition of puzzle pieces.

\subsubsection{Generalized dyadic interval}\label{gen_dyad_int_subsec}
We start with a generalization of dyadic intervals given in Definition~\ref{defn:dyaint}.
\begin{defn}[Generalized dyadic intervals]\label{defn:gendyadic}
    Let $I = [s,t] \subseteq \mathbb{S}^1 \cong \R/\Z$ be an interval. We say that it is {\em generalized dyadic} if  there exist an integer $m$ and an interval $J$ in the collection $\{(0,1), (0, \frac23), (\frac13, \frac23), (\frac13, 1)\}$ so that  $\sigma_2^m$ is a homeomorphism between $\Int{I}=(s,t)$ and $J$. 
\end{defn}
Note that the intervals $[0,\frac13]$ and $[\frac23,1]$ are generalized dyadic as well, as their interiors are mapped under $\sigma_2$ to $(0, \frac23)$ and $(\frac13, 1)$ respectively.
Thus, any closed interval in $[0,1]$ with end-points in $\{0, \frac13, \frac23, 1\}$ is generalized dyadic.
More generally, it follows from the definition that we have the following characterization.
\begin{lem}
    An interval $I$ is generalized dyadic if and only if there exist non-negative integers $p, n$ such that $\partial I \subseteq \{\frac{p}{2^n}, \frac{p+\frac13}{2^n}, \frac{p+\frac23}{2^n}, \frac{p+1}{2^n}\}$.
\end{lem}

We color a point on $\mathbb{S}^1$ \emph{red} if it is eventually mapped to $0$, and \emph{blue} if it is eventually mapped to the cycle $\frac13\leftrightarrow\frac23$.
A boundary point of a generalized dyadic interval $I$ is thus colored red or blue. We will refer to them as {\em dyadic boundary points}, or {\em non-dyadic boundary points}, respectively. These end-points will correspond to the contact points or cusps (respectively) on the boundaries of the components of $\Omega(G)-\Delta_\infty$.
With notation as in Definition~\ref{defn:gendyadic}, we call the image $J$ the {\em type} of the generalized dyadic interval $I$.
By definition, dyadic intervals are generalized dyadic of type $[0,1]$. For two generalized dyadic intervals $I_1, I_2$ of the same type, there exist $m,n$ so that $\sigma_2^{-n} \circ \sigma_2^m$ is a homeomorphism between $I_1$ and $I_2$.
\begin{remark}\label{remark:compatible}
   Note that the four types of generalized dyadic intervals correspond to the four different choices of colorings of the end-points of $I$.
\end{remark}

Generalized dyadic intervals share some good properties of dyadic intervals. One important feature is the following decomposition property.

\begin{lem}\label{lem:decomp}
    Let $I=[s,t]$ be a generalized dyadic interval. Then the following hold.
    \begin{itemize}[leftmargin=8mm]
        \item There exists a decomposition $I = [s,x] \cup [x, t]$ into generalized dyadic intervals such that $x$ has color red.
        \item There exists a decomposition $I = [s,x] \cup [x, t]$ into generalized dyadic intervals such that $x$ has color blue, provided that at least one of $s, t$ has color red.
    \end{itemize}
\end{lem}
\begin{proof}
    We blowup the interval $I$ using $\sigma_2^m$ to one of the intervals in the list 
    $$
    \bigg\{[0,1], [0, \frac23], [\frac13, \frac23], [\frac13, 1]\bigg\}.
    $$
    We can explicitly verify the statement for each one of the four intervals; for instance, one has the decompositions $[0,\frac23]=[0,\frac12]\cup[\frac12,\frac23]$ and $[\frac13,\frac23]=[\frac13,\frac12]\cup[\frac12,\frac23]$ such that the intermediate point has color red.
\end{proof}

By applying Lemma~\ref{lem:decomp} inductively, we can construct a generalized dyadic decomposition realizing a given color sequence.
\begin{cor}\label{cor:alternating}
    Let $(c_i)_{i=1}^N$ be a sequence such that each $c_i\in\{blue, red\}$, and at least one $c_i$ is red.
    Then there exists a decomposition of 
    $$
    \mathbb{S}^1 = [s_1,s_2]\cup...\cup [s_N, s_1],
    $$
    so that each interval $[s_i, s_{i+1}]$ is generalized dyadic, and $s_i$ has color $c_i$.
\end{cor}

\subsubsection{Basilica $\mathfrak{Q}_{\pc}-$maps}
Let $Q_{\pc}(z) = z^2-1$ be the quadratic \pcf polynomial with Basilica Julia set.
Let $\mathfrak{Q}_{\pc}$ be the collection of conformal maps of the form $Q_{\pc}^{-n} \circ Q_{\pc}^m$.
The construction of Basilica $\mathfrak{Q}-$maps in \S~\ref{sec-basilica} extends directly to produce Basilica $\mathfrak{Q}_{\pc}-$maps.

Let $U_0$ be the Fatou component of $Q_{\pc}$ containing $0$. Note that $Q_{\pc}^2\vert_{\overline{U_0}}$ is conformally conjugate to $z^2\vert_{\overline{\D}}$. As $z^2$ is expanding on $\mathbb{S}^1$, we can define internal puzzle pieces corresponding to the arcs $I_{U_0}[0,1]$, $I_{U_0}[0,\frac23]$, $I_{U_0}[\frac13,1]$, $I_{U_0}[\frac13,\frac23]\subset \partial U_0$ as in \S~\ref{subsubsec:internalpuzzle}.
With these internal puzzle pieces at our disposal, we define {\em generalized dyadic segments} $I_{U}[s,t]$ and {\em generalized internal puzzle pieces} $P_U[s,t]$ for any bounded Fatou component $U$ in the exact same way as in Definition~\ref{defn:internalpuzzle}.
By the above observations, for any two internal puzzle pieces $P_{U_i}[s_i,t_i],\ i\in\{1,2\}$, if $[s_i, t_i]$ are of the same type, then there exist $m,n$ so that $Q_{\pc}^{-m} \circ Q_{\pc}^n : P_{U_1}[s_1,t_1] \longrightarrow P_{U_2}[s_2,t_2]$ is a homeomorphism.

The definition of \emph{generalized internal puzzle}
$$
\left(\{U_i\}, P_{U_i}[s_{i,j}, s_{i,j+1}]\right); i \in\{0,\cdots, r\},\ j\in\{1,\cdots, n_i\},
$$
is the same as Definition~\ref{defn:intpuz} after replacing dyadic segments with generalized dyadic segments. 
It induces a \emph{generalized Basilica puzzle} $(P_i)_{i\in \mathcal{I}}$ as in Definition~\ref{basi_puzz_def} and \S~\ref{subsubsec:indbp}.

Let $P_i$ be a generalized Basilica puzzle piece. Then it is either of limb type or torso type.
If it is of torso type, then $P_i$ intersects some bounded Fatou component $U$ in a generalized internal puzzle piece $P_U[s,t]$ for some generalized dyadic interval $[s,t]$.
Note that if $s$ is not dyadic, then the corresponding point in $\partial P_i \cap J(Q_{\pc})$ is not a contact point. Observe that there are $5$ types of puzzle pieces in this setting: one limb type, and four different torso types depending on the colors of the end-points of the corresponding generalized dyadic interval (cf. Remark~\ref{remark:compatible}).
Given two puzzle pieces $P_i, P_j$ of the same type, there is a canonical map $F = Q_{\pc}^{-m}\circ Q_{\pc}^n$ which is a homeomorphism from $P_i$ to $P_j$.
With this convention of types, the definition of Markov refinements in Definition~\ref{defn:refineBp} also carries over naturally to generalized Basilica puzzles.

\begin{defn}[c.f. Definition~\ref{defn:basiqmap}]
    Let $\mathfrak{P} = \{P_i\}_{i\in \mathcal{I}}$ be a generalized Basilica puzzle and let $\mathfrak{P}^1=\{P_{ij}\}_{(ij)\in \mathcal{I}^1}$ be a Markov refinement of $\mathfrak{P}$.
    We call the collection of canonical maps $\mathcal{F} = \{F_{ij}: P_{ij} \longrightarrow P_j\}_{(ij) \in \mathcal{I}^1}$ a {\em generalized Basilica $\mathfrak{Q}_{\pc}-$map}, and write it as 
    $$
    \mathcal{F}: \mathfrak{P}^1 \rightarrow \mathfrak{P}.
    $$
\end{defn}
Since the map $Q_{\pc}$ has no parabolic fixed point, a similar argument gives the following analog of Theorem~\ref{thm:symmetrichyperbolic} and Proposition~\ref{prop:bdd_fatou_symm_para} (where the modification uses Lemma~\ref{lem:decomp} and Corollary~\ref{cor:alternating}).
\begin{theorem}\label{thm:symmetrichyperbolicgeneralizedpcf}
    Let $\mathcal{F}: \mathfrak{P}^1 \rightarrow \mathfrak{P}$ be a generalized Basilica $\mathfrak{Q}_{\pc}-$map.
    Suppose that $\mathcal{F}$ is topologically expanding. 
     Then there exists a modified refinement $\widetilde{\mathcal{F}}: \widetilde{\mathfrak{P}}^1 \rightarrow \widetilde{\mathfrak{P}}$ which is in particular topologically conjugate to $\mathcal{F}: \mathfrak{P}^1 \rightarrow \mathfrak{P}$ on the Julia set, and 
    \begin{itemize}[leftmargin=8mm]
        \item the induced Markov map $\widetilde{\mathcal{F}}_\infty$ is symmetrically hyperbolic;
        \item the induced Markov map $\widetilde{\mathcal{F}}_{\bdd}$ on $\displaystyle\bigcup_{s\in \{0,\ldots, r\}}  \partial^I U_s$ is symmetrically hyperbolic.
    \end{itemize}
\end{theorem}

\subsection{Model for general Basilica Bowen-Series maps}
We now model a (possibly non-compact) Basilica Bowen-Series map $\mathcal{F}_G$ associated with a geometrically finite Bers boundary group $G$ by an appropriate generalized Basilica $\mathfrak{Q}_{\pc}-$map. This will yield a David homeomorphism carrying $J(Q_{\pc})$ onto $\Lambda(G)$, as promised in Theorem~\ref{thm:davidhi}.

\begin{prop}\label{prop:quasisurggenpcf}
    Let $\Sigma$ be a hyperbolic surface, and let $G$ be a geometrically finite Bers boundary group associated to $(\Sigma, \mathcal{C})$. Let $\mathcal{F}_G$ be a Basilica Bowen-Series map. Then there exists a generalized Basilica $\mathfrak{Q}_{\pc}-$map $\mathcal{F}: \mathfrak{P}^1\longrightarrow \mathfrak{P}$ which is David conjugate to $\mathcal{F}_G$ on $J(Q_{\pc})$. More precisely, this means that there exists a David map $H: \widehat{\C} \longrightarrow \widehat{\C}$ so that
    $$
    H \circ \mathcal{F}(z) = \mathcal{F}_G \circ H(z)
    $$
    for all $z \in J(Q_{\pc})$.
\end{prop}
\begin{proof}
    As in the proof of Proposition~\ref{prop:quasisurg}, let $\Delta_s, s\in \{1,..., j_0\} \cup \{\infty\}$ be the components of $\Omega(G)$ associated to the pinched core polygon, and let $\Delta_s$, $ s\in \{j_0+1,..., j_1\}$ be the additional components of $\Omega(G)$ associated to the pull-back of the pinched core polygon via $\mathcal{F}_G$. They are characterized by the following.
    \begin{itemize}
        \item $\Delta_s, s\in \{1,..., j_0\} \cup \{\infty\}$ are the components of $\Omega(G)$ that intersect at least two level $0$ puzzle pieces;
        \item $\Delta_s, s\in \{1,..., j_0, j_0+1,..., j_1\} \cup \{\infty\}$ are the components of $\Omega(G)$ that intersect at least two level $1$ puzzle pieces.
    \end{itemize}
    Denote the induced Markov maps by
    \begin{align*}
    (\mathcal{F}_G)_{\bdd} : \bigcup_{s=1}^{j_1} \partial^I \Delta_s &\longrightarrow \bigcup_{s=1}^{j_1} \partial^I \Delta_s, \text{ and }\\
    (\mathcal{F}_G)_\infty: \partial^I \Delta_\infty &\longrightarrow \partial^I \Delta_\infty.
    \end{align*}
    (See \S~\ref{pullback_bs_subsec} for the definition of the induced Markov maps $(\cF_G)_s$.)
    Note that by Proposition~\ref{prop:basilicabssymmetric}, 
    \begin{itemize}
        \item for $s \neq \infty$, any break-point of $(\mathcal{F}_G)_s$ is symmetrically parabolic;
        \item for $s = \infty$, a break-point is symmetrically parabolic if it corresponds to a cusp of the surface $\Sigma$; and it is symmetrically hyperbolic otherwise.
    \end{itemize}

    By a similar construction as in Proposition~\ref{prop:topmodelbowenseries} we obtain a generalized Basilica $\mathfrak{Q}_{\pc}-$map $\widetilde{\mathcal{F}}: \widetilde{\mathfrak{P}}^1 \rightarrow \widetilde{\mathfrak{P}}$ which is topologically conjugate on $J(Q_{\pc})$ to the restriction of $\mathcal{F}_G$ on $\Lambda(G)$.

    Thus, by Theorem~\ref{thm:symmetrichyperbolicgeneralizedpcf}, there exists a generalized Basilica $\mathfrak{Q}_{\pc}-$map $\mathcal{F}: \mathfrak{P}^1\longrightarrow \mathfrak{P}$ such that
    \begin{itemize}[leftmargin=8mm]
        \item $\mathcal{F}\vert_{J(Q_{\pc})}$ is topologically conjugate to $\widetilde{\mathcal{F}}\vert_{J(Q_{\pc})}$ (and hence $\mathcal{F}_G\vert_{\Lambda(G)}$); 
        \item the induced Markov map $\mathcal{F}_\infty$ is symmetrically hyperbolic; and 
        \item the induced Markov map $\mathcal{F}_{\bdd}$ on $\bigcup_{s\in \{0,\ldots, r\}}  \partial^I U_s$ is symmetrically hyperbolic.
    \end{itemize}

    Let $H:J(Q_{\pc})\to\Lambda(G)$ be a homeomorphism conjugating $\mathcal{F}$ to $\mathcal{F}_G$.
    Note that for each $s \in \{1,..., j_1\} \cup \{\infty\}$, the map $H$ induces a circle homeomorphism that conjugates $\mathcal{F}_s$ to $(\mathcal{F}_G)_s$. Since this conjugacy does not send parabolic type to hyperbolic type, by Proposition~\ref{prop:qsmarkovmapFiniteCircle}, it extends to a David homeomorphism of the disk. Composing this David map with the Riemann maps of $U_s$ and $\Delta_s$, one obtains a David extension $H|_{U_s}: U_s \longrightarrow \Delta_s$ of $H|_{\partial U_s}:\partial U_s\to\partial \Delta_s$. We note that the composition of the David map of the disk with the Riemann maps of $U_s, \Delta_s$ is again a David homeomorphism because $U_s$ is a John domain (cf. \cite[Proposition~2.5]{LMMN}).

    Using the dynamics, we lift the maps $H|_{U_s}$ to all bounded Fatou components of $Q_{\pc}$, thus obtaining a homeomorphism $H: \widehat{\C} \longrightarrow \widehat{\C}$.  By \cite[Proposition~2.5]{LMMN}, the map $H$ is a David homeomorphism on the Fatou set.
    Hence, Lemma~\ref{lem:eccentricity} implies that $H$ is a global David homeomorphism.
\end{proof}

\subsection{Model for Basilica Julia sets}\label{basilica_julia_david_subsec}
Now suppose that $g$ is a geometrically finite rational map with a Basilica Julia set. Then the unique non-Jordan disk Fatou component $U_\infty$ is totally invariant under $g$.
Hence, there exists a $\pcf$ polynomial $g_0$ which is topologically conjugate to $g$ on the Julia sets (cf. \cite[Corollary~9.7]{LMMN}).

\begin{lem}\label{david_pcf_to_gf_lem}
There exists a David homeomorphism $H:\widehat{\C}\to\widehat{\C}$ with $H(J(g_0))=H(J(g))$.    
\end{lem}
\begin{proof}
If $g$ is a polynomial, then the existence of the desired David homeomorphism follows from the proof of \cite[Theorem~9.2]{LMMN}. Now let $g\vert_{U_\infty}$ be conformally conjugate to a degree $d$ parabolic Blaschke product $B$. Then, in addition to the surgery carried out in \cite[Theorem~9.2]{LMMN}, one needs to replace the action of $g_0$ on its basin of infinity (which is conformally modeled by $z^d\vert_{\D}$) with the Blaschke product $B$. This step can also be done using David surgery as 
\begin{enumerate}[leftmargin=8mm]
    \item the circle homeomorphism conjugating $z^d$ to $B$ extends to a David homeomorphism of the disk (cf. \cite[Theorem~1.5]{LN24a}), and
    \item the basin of infinity of a postcritically finite polynomial is a John domain.
\end{enumerate}
Together, this provides the required global David map carrying $J(g_0)$ onto $J(g)$.
\end{proof}

We continue to work with the above postcritically finite polynomial $g_0$.
Note that unlike the polynomial $z^2-1$, the map $g_0$ may contain critical points on the Julia set. However, any Julia critical point of $g_0$ is necessarily strictly pre-periodic.
In this setting, there is a new mechanism of creating contact points: a contact point of $J(g_0)$ can be a strictly pre-periodic simple critical point (or one of its iterated preimages) that eventually maps to a non-contact Julia point (see Figure~\ref{fig:cubicBasilica}).
On the other hand, a (generalized) Basilica $\mathfrak{Q}_{\pc}-$map preserves the contact points. Thus, it is in general not possible to construct a generalized Basilica $\mathfrak{Q}_{\pc}-$map that is conjugate to $g_0$ on the Julia sets. We surmount this difficulty in two steps.
\begin{itemize}[leftmargin=8mm]
    \item We first associate a topologically expanding $(J(g_0), \mathfrak{G}_0)-$map to $g_0$, where $\mathfrak{G}_0$ is the collection of local conformal maps of the form $g_0^{-m}\circ g_0^n$. This is done in such a way that the boundary of any internal puzzle piece associated with the $(J(g_0), \mathfrak{G}_0)-$map consists only of contact points of $J(g_0)$.
    \item We then construct a conjugacy between an appropriately chosen Basilica $\mathfrak{Q}_{\pc}-$map and the above $(J(g_0), \mathfrak{G}_0)-$map.
\end{itemize}  

The following lemma is important for the fist step mentioned above.

\begin{lem}\label{lem:modifiedcontact}
    Let $\sigma_d:\mathbb{S}^1 \cong \R/\Z \longrightarrow \mathbb{S}^1$ be the map $\sigma_d(t) = dt$.
    Let $X$ be a finite set that is invariant under $\sigma_d$, i.e., $\sigma_d(X) \subseteq X$.
    Let $X'\subsetneq X$ be a proper invariant subset. Let $N$ be a sufficiently large common multiple of the periods of the periodic points in $X$.
    
    Then there exists $Y^0, Y^1 \subseteq \mathbb{S}^1$ so that $X-X'\subseteq Y^0 \subseteq Y^1$ with $X'\cap Y^0=\emptyset$ and the following holds.
    The sets $Y^0, Y^1$ determine a Markov map 
    $$
    \mathcal{F}: \mathcal{A}^1 \longrightarrow \mathcal{A}^0,
    $$
    where $\mathcal{A}^i$ consists of the closure of components of $\mathbb{S}^1 - Y^i$ so that for each $I\in \mathcal{A}^1$, there exists $J \in \mathcal{A}^0$ satisfying the following.
    \begin{enumerate}
        \item\label{lem:modifiedcontact:item1} there exists $n, m$ so that
        $$
        F_I:=\sigma_d^{-n}\circ \sigma_d^{m}: I \longrightarrow J
        $$
        is a homeomorphism;
        \item\label{lem:modifiedcontact:item2} $\mathcal{F}$ is topologically expanding; and 
        \item\label{lem:modifiedcontact:item3} each break-point is symmetric.
    \end{enumerate}
\end{lem}
Note that since $X'$ is invariant, $X' \cap Y^1 = \emptyset$. 
\begin{proof}
    Let $k$ be sufficiently large, and choose $X-X' \subseteq Y^0 \subseteq \sigma_d^{-k}(X) - X'$ which induces the partition $\mathcal{A}^0$ so that
    \begin{itemize}
        \item for $a \in Y^0 - X$, $\sigma_d(a) \in Y^0$;
        \item for $A \in \mathcal{A}^0$, $|\partial A \cap (X-X')| \leq 1$;
        \item if $\partial A = \{a, b\} \cap (X-X') \neq \emptyset$, say $a \in X- X'$, then $\sigma_d^n(b) = a$ for some $1\leq n \leq k$;
        \item if $\partial A \cap (X-X') = \emptyset$, then $\sigma_d$ is injective on $A$.
    \end{itemize}

    Let $A \in \mathcal{A}^0$. 
    Case (1): Suppose that $\partial A \cap (X-X') = \emptyset$. Then $\sigma_d: A \longrightarrow \mathbb{S}^1$ is injective and $\sigma_d(\partial A) \subseteq Y^0$. We define $Y^1 \cap A = (\sigma_d:A \to \sigma_d(A))^{-1}(Y^0)$ as the pull back of $Y^0$ by $\sigma_d$.

    Case (2): Suppose that $\partial A \cap (X-X') \neq \emptyset$. Say $A= [a,b]$ and $a \in X-X'$. Let $n_a$ be the pre-period of $a$ (i.e., $n_a$ is the smallest integer such that $\sigma_d^{n_a}(a)$ is periodic) and let $N \gg k$.
    Let us lift the dynamics of $\sigma_d:\mathbb{S}^1 \longrightarrow \mathbb{S}^1$ to the universal cover $\widetilde{\sigma_d}: \R \longrightarrow\R$.
    Since $\sigma_d^j(b) = a$ for some $j\in\{1,\cdots,k\}$, and $N \gg k$, we have that $\widetilde{\sigma_d}^{n_a+N}[a,b]$ is an interval of length dividing $d^{n_a}$. Consider a lift of the inverse branch $\widetilde{\sigma}_d^{-n_a}$ so that $\widetilde{\sigma}_d^{-n_a} \circ \widetilde{\sigma}_d^{n_a+N}$ fixes $a$.
    Note that the length of the interval $\widetilde{\sigma}_d^{-n_a} \circ \widetilde{\sigma}_d^{n_a+N}([a,b])$ is an integer, and that the projection $\sigma_d^{-n_a}\circ \sigma_d^{n_a+N}$ sends $a, b$ to $a$. Note that the image of $[a,b]$ may wrap around the circle multiple times.
    We define $Y^1 \cap A$ as the pull back of $Y^0$ by $\sigma_d^{-n_a} \circ \sigma_d^{n_a+N}$.

    This gives the construction of the Markov map $\mathcal{F}: \mathcal{A}^1 \longrightarrow\mathcal{A}^0$, where the map is either $\sigma_d$ or $\sigma_d^{-n} \circ \sigma_d^{n+N}$ for some $n$ depending on whether the interval $I \in \mathcal{A}^1$ is constructed in the first or the second case. 
    By construction, each periodic break-point is fixed and has left and right derivative $d^N$, so that each break-point is symmetric.
    The map $\mathcal{F}$ is topologically expanding as the derivative is at least $d > 1$.
\end{proof}

We will now construct, as in Proposition~\ref{prop:topmodelJulia}, a Basilica $\mathfrak{Q}_{\pc}-$map that topologically models a $(J(g_0),\mathfrak{G}_0)-$map associated to the postcritically finite polynomial $g_0$ with a Basilica Julia set. This will give rise to a quasiconformal conjugacy between their actions on the Julia sets.

We carry out the construction in a slightly more general setting.

\begin{prop}\label{prop:davidJulia}
    Let $\widetilde{g}$ be a subhyperbolic rational map (i.e., every Fatou critical point converges to an attracting cycle, and every Julia critical point is strictly pre-periodic) with a Basilica Julia set.
    Then there exist a topologically expanding $(J(\widetilde{g}), \widetilde{\mathfrak{G}})-$map $\cG$ and a Basilica $\mathfrak{Q}_{\pc}-$map $\mathcal{F}: \mathfrak{P}^1\longrightarrow \mathfrak{P}$ so that $\mathcal{F}$ is quasiconformally conjugate to $\cG$ on $J(Q_{\pc})$.
    
\end{prop}
\begin{proof}
    Let $(V_i, X_i)_{i\in \mathcal{I}}$ be a finite collection so that 
    \begin{itemize}
        \item each $V_i$ is a bounded Fatou component of $\widetilde{g}$;
        \item $\overline{\bigcup_{i\in \mathcal{I}} V_i}$ is connected, invariant under $\widetilde{g}$ and contains the critical points;
        \item $X_i \subseteq \partial V_i$ consists of finitely many points on $\partial V_i$, and $X_i' \subseteq X_i$ consists of non-contact points in $X_i$;
        \item $2 \leq |X_i|$;
        \item $\partial V_{i} \cap \partial V_{j} \subseteq X_i \cap X_j$ for $i \neq j$; and
        \item $\bigcup X_i$ and $\bigcup X_i'$ are both invariant under $\widetilde{g}$.
    \end{itemize}
    Note that the map $\widetilde{g}: \bigcup \partial^I V_i \longrightarrow \bigcup \partial^I V_i$ is a topologically expanding covering on a finite union of circles.
    Thus by Lemma~\ref{lem:modifiedcontact}, we obtain $(V_i, Y^0_i)_{i\in \mathcal{I}}$ and $(V_i, Y^1_i)_{i\in \mathcal{I}}$ and the corresponding Markov map $\cG_{\mathrm{bdd}}: \bigcup \partial^I V_i \longrightarrow \bigcup \partial^I V_i$.
    By Lemma~\ref{lem:modifiedcontact} \eqref{lem:modifiedcontact:item3}, the map is symmetrically hyperbolic.
   Now, following the construction of Basilica $\mathfrak{Q}-$maps in \S~\ref{sec-basilica}, we cook up a topologically expanding $(J(\widetilde{g}), \widetilde{\mathfrak{G}})-$map $\cG$ from the map $\cG_{\mathrm{bdd}}$. 

    Finally, similar to the proof of Proposition~\ref{prop:topmodelbowenseries}, we can construct $\mathfrak{P}$ and $\mathfrak{P}^1$ according to the data $(V_i, Y^0_i)_{i\in \mathcal{I}}$ and $(V_i, Y^1_i)_{i\in \mathcal{I}}$, and a Basilica $\mathfrak{Q}_{\pc}-$map $\mathcal{F}: \mathfrak{P}^1 \longrightarrow \mathfrak{P}$ which is topologically conjugate 
    to $\cG$ on the Julia set. Possibly after passing to a modified refinement (Definition~\ref{def-modfn}), we may assume that each break-point of $\mathcal{F}$ is symmetrically hyperbolic (cf. Theorem~\ref{thm:symmetrichyperbolic}). Hence, the desired quasiconformal homeomorphism of $\widehat{\C}$ that conjugates $\cF\vert_{J(Q_{\pc})}$ to $\cG\vert_{J(\widetilde{g})}$ can be obtained by applying the arguments of Proposition~\ref{prop:quasisurg}.
\end{proof}

\subsection{Proof of Theorem~\ref{thm:davidhi}}
If $[K]$ is Kleinian, then by Proposition~\ref{prop:quasisurggenpcf}, we have $[J(Q_{\pc})] \succeq [K]$.

Now let $[K]$ be rational; i.e., $[K]=[J(g)]$ for some geometrically finite rational map. Then the required David homeomorphism is given by the composition of the David homeomorphism of Lemma~\ref{david_pcf_to_gf_lem} (carrying $J(g_0)$ onto $J(g)$) and the quasiconformal homeomorphism of Proposition~\ref{prop:davidJulia} (choosing $\widetilde{g}=g_0$, in which case the quasiconformal map carries $J(Q_{\pc})$ onto $J(g_0)$). Hence, $[J(Q_{\pc})]~\succeq~[K]$. 

Finally, if $[K]=[J(g)]$, and if $g$ has no parabolic cycle, then $g$ is subhyperbolic. In this case, the equality $[J(Q_{\pc})]=[K]$ is the content of Proposition~\ref{prop:davidJulia}.
\qed

\begin{remark}\label{david_hierarchy_rem}
The construction of a David map carrying $J(Q_{\pc})$ onto a Basilica Julia/limit set crucially uses the John property of the basin of infinity and $W^{1,1}-$removability of the Julia set of $Q_{\pc}$. These properties are not enjoyed by Basilica Julia/limit sets containing parabolics. In order to prove that $\succeq$ is a partial order (i.e.\ the possibility that a parabolic Basilica can be turned into a `more parabolic' one using a David homeomorphism), one needs to overcome this issue.
\end{remark}

\appendix

\section{Cubic polynomials}\label{sec:cubic_poly}
In this section, we explain how to modify the construction of fragmented dynamical systems for the purpose of Theorem~\ref{thm:cubicclass}. 

Recall that $\Per_1(0) := \{f_a(z)=z^3+\frac{3a}{2}z^2:a\in \C\}$. Let $\mathcal{H} \subseteq \Per_1(0)$ be the {\em main hyperbolic component}, i.e., the hyperbolic component that contains $z^3$.
The boundary $\partial \mathcal{H}$ is a Jordan curve (see \cite{Roe07}), Any geometrically finite polynomial $f_a \in \partial \mathcal{H}$ has a Basilica Julia set and is of one of the following two types (see Figure~\ref{fig:cubicBasilica}):
\begin{itemize}[leftmargin=8mm]
    \item (parabolic type) the free critical point $-a$ is in a parabolic basin, in which case $a \in \partial \mathcal{H}$ corresponds to a cusp;
    \item (Julia type) the free critical point $-a$ is on the Julia set $J(f_a)$.
\end{itemize}

\begin{proof}[Proof of Theorem~\ref{thm:cubicclass}]
If $f_a$ is of Julia type, then $f_a$ has no parabolic cycle. Thus by Proposition~\ref{prop:davidJulia}, we have $[J(f_a)] = [J(Q_{\pc})]$.

To model maps of parabolic type, we consider $F(z) = z^3+2iz^2$.
Note that $F$ has a parabolic fixed point at $-i$ (see Figure~\ref{fig:cubicBasilica} where the figure is rotated by $90$ degrees).
Let $U_0, U_1$ be the super-attracting and parabolic Fatou components of $F$. Note that both $U_0$ and $U_1$ are fixed, and any bounded Fatou component is eventually mapped to either $U_0$ or $U_1$.
For each $i\in\{0, 1\}$, the map $F|_{\partial U_i}$ is topologically conjugate to $\sigma_2: \mathbb{S}^1 \longrightarrow\mathbb{S}^1$.
Thus, we can define the internal puzzle pieces for $U_0$ and $U_1$ using dyadic intervals just as the internal puzzles for $Q(z)=z^2-\frac34$, and pull them back by $F$ to obtain internal puzzle pieces for each bounded Fatou component of $F$.
Thus, the construction of Basilica $\mathfrak{Q}$-maps can be carried out verbatim in the current setting, giving rise to Basilica $\mathfrak{F}$-maps (recall from Example~\ref{exm:xcmaps}(4) that such maps are piecewise conformal, where the pieces are of the form $F^{-n}\circ F^m$).
Similar to Proposition~\ref{prop:topmodelJulia}, we can manufacture a Basilica $\mathfrak{F}$-map $\widetilde{\mathcal{F}}$ topologically conjugate to $f_a$ on $J(F)$.
By the proof of Proposition~\ref{prop:quasisurg}, after possibly passing to a modified refinement of $\widetilde{\mathcal{F}}$, we obtain a Basilica $\mathfrak{F}$-map $\mathcal{F}$ which is quasisymmetrically conjugate to $f_a $ on $J(F)$. This proves that $[J(f_a)] = [J(F)]$.
\end{proof}

\section{Schwarz reflections}\label{sec:schwarz}
We now describe the construction of a fragmented dynamical system from a Schwarz reflection map leading to the proof of Theorem~\ref{thm:schwarz}.

\subsection{The family $\Sigma^*$}
We begin by recalling the space of Schwarz reflections that realize matings of $\bar z^d$ with the so-called \emph{necklace reflection groups} (see \cite{LMM22} for details). 
Consider the family of degree $d+1$ rational maps:
$$ 
\Sigma_d^* := \left\{ f(z)= z+\frac{a_1}{z} + \cdots -\frac{1}{d z^d} : f\vert_{\widehat{\C}-\overline{\D}} \textrm{ is univalent}\right\}.
$$
For $f\in \Sigma_d^*$,  let $\Omega=f(\widehat{\C}-\overline{\D})$.
Then the Schwarz reflection map associated to $f$ is defined as
\begin{align*}
    & S\equiv S_f:\Omega \longrightarrow \widehat{\C}\\
    z &\mapsto f\circ \eta \circ (f|_{\widehat{\C}-\overline{\D}})^{-1}(z),
\end{align*}
where $\eta(z) = \frac{1}{\bar z}$ is the reflection along the unit circle. 
Abusing notation, we say that $S \in \Sigma^*_d$ as well.
Note that $S$ extends continuously as the identity map on $\partial \Omega$.

\subsubsection{Invariant partition of Schwarz dynamical plane}
The dynamical plane of a Schwarz reflection map $S$ arising from $\Sigma_d^*$ naturally splits into two invariant subsets.
The {\em basin of infinity} of $S$ is defined as 
$$
\mathcal{B}_\infty(S) = \{z \in \overline{\Omega}: S^n(z) \to \infty\},
$$
and the {\em limit set} of $S$ is 
$$
\Lambda(S) = \partial \mathcal{B}_\infty(S).
$$
The {\em droplet} is defined as $T(S) = \widehat{\C} - \Omega$, and the {\em desingularized droplet/fundamental tile} $T^0(S)$ is obtained from $T(S)$ by deleting the finitely many singular points (conformal cusps or double points) on $\partial T(S) = f(\mathbb{S}^1)$.
The complement of the (closure of the) basin of infinity forms the {\em tiling set}
$$
\mathcal{T}_\infty(S) = \widehat{\C} - \overline{\mathcal{B}_\infty(S)} = \{z \in \widehat{\C}: S^n(z) \in T^0 \text{ for some } n\}.
$$
Note that the tiling set has a level structure.

\subsubsection{Necklace reflection groups}
We say that a Kleinian reflection group $G$ is a {\em necklace reflection group} if it can be generated by reflections in (oriented) round circles $C_1,..., C_d$ which bound closed disks $D_1,..., D_d$ so that
\begin{enumerate}
    \item the disks $D_i$ have pairwise disjoint interiors;
    \item each circle $C_i$ is tangent to $C_{i+1}$ (with $i+1$ mod $d$); and
    \item the boundary of the unbounded component of $\widehat{\C} - \bigcup C_i$ intersects each $C_i$.
\end{enumerate}
(This is equivalent to saying that the nerve of the circle packing $\{C_i\}$ is $2-$connected and outerplanar; see Figure~\ref{fig:sigma}, cf. \cite{LLM22}.)
The number $d$ is called the {\em rank} of $G$.
The Bowen-Series map for $G$ is defined as
\begin{align*}
    &\mathcal{F}: \bigcup_{i=1}^d D_i \longrightarrow \widehat{\C}\\
    z &\mapsto r_i(z) \text{ if } z\in D_i,
\end{align*}
where $r_i$ is the reflection along the circle $C_i$.

We say that a Jordan domain $V$ is an {\em ideal $n-$gon} if $\partial V$ consists of $n$ non-singular analytic arcs so that two adjacent arcs intersect tangentially and form an outward cusp of $V$. We call the cusps on $\partial V$ the {\em ideal vertices} of $V$.

Let $G$ be a necklace reflection group.
We denote the set of touching points of the circles $C_i$ by $\mathfrak{s}$.
Let $\cP_\infty(G)$ be the unbounded component of $\displaystyle \widehat{\C} - \left(\bigcup \Int{D_i}\cup\mathfrak{s}\right)$. Note that $\cP_\infty(G)$ is a fundamental domain for the $G-$action on $\Delta_\infty$, where $\Delta_\infty$ is the unbounded component in the domain of discontinuity $\Omega(G)$ of $G$.
Let $T^0(G)$ be the union of bounded components of $\displaystyle \widehat{\C} - \left(\bigcup \Int{D_i}\cup\mathfrak{s}\right)$. 
Note that $\Int{T^0(G)}$ is a union of ideal polygons. We call $T^0(G)$ the {\em fundamental tile} of $G$. In fact, $T^0(G)$ is a fundamental domain for the $G-$action on
the {\em tiling set} $\mathcal{T}_\infty(G)$ for $G$, defined as
$$
\mathcal{T}_\infty(G) = \Omega(G) - \Delta_\infty.
$$
(See Figure~\ref{fig:sigma}.)

\begin{figure}
\captionsetup{width=0.96\linewidth}
    \centering
    \includegraphics[width=0.66\linewidth]{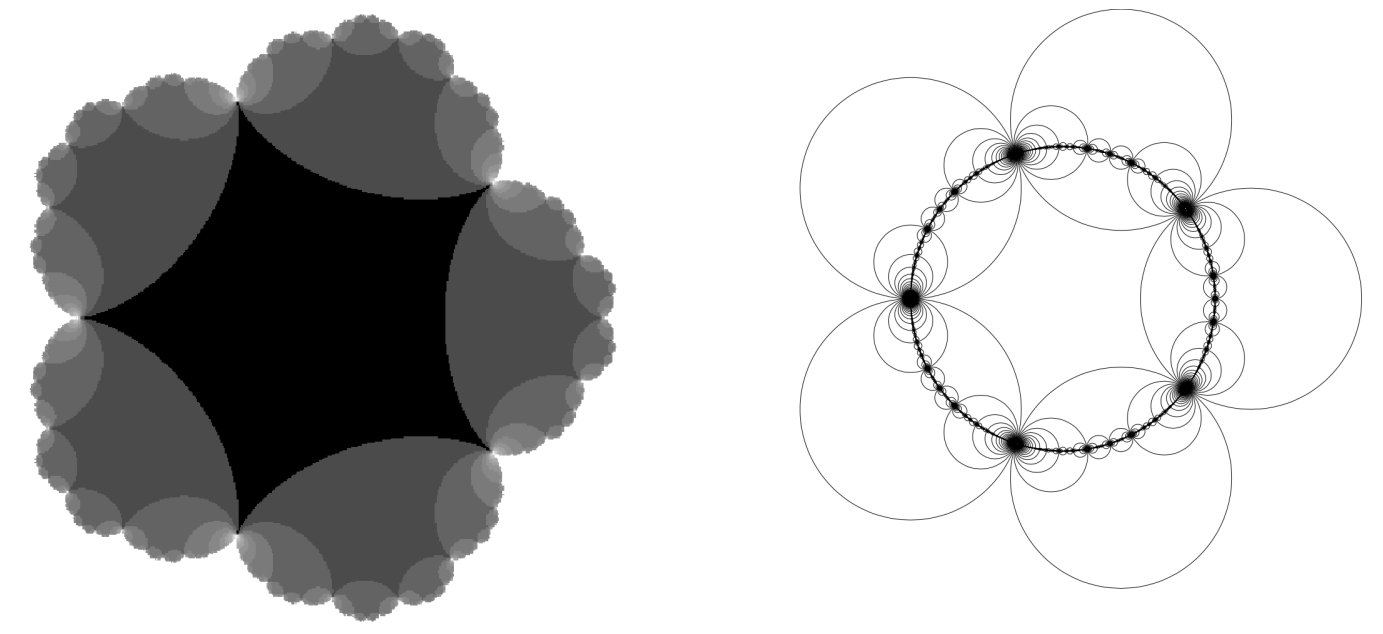}
    \includegraphics[width=0.8\linewidth]{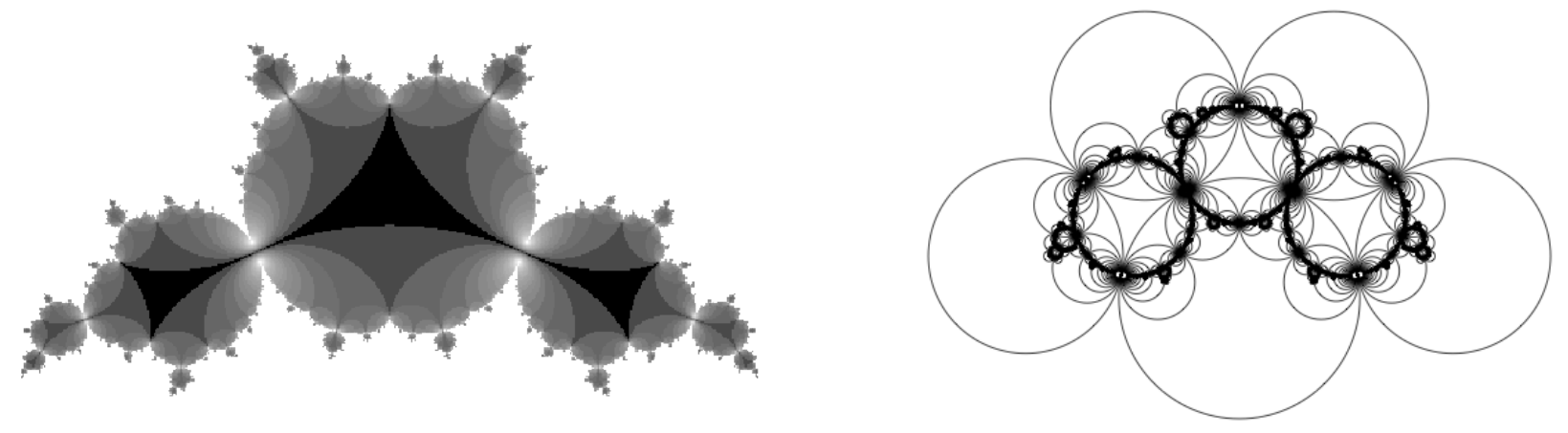}
    \caption{Displayed are the limit sets of two Schwarz reflections in $\Sigma_4^*$ (left) and those of the associated necklace reflection groups (right).}
    \label{fig:sigma}
\end{figure}

If $T^0(G)$ is connected, then $G$ is a quasi-Fuchsian reflection group and the limit set is a quasicircle.
Otherwise, $G$ is a Bers boundary reflection group whose limit set is a Basilica.

\subsubsection{Schwarz reflections as matings}
The following result states that the Schwarz reflections arising from $\Sigma_d^*$ are precisely the conformal matings between necklace reflection groups of rank $d+1$ and the anti-polynomial $\overline{z}^d$.

\begin{prop}\cite[Theorem~A]{LMM22}\label{prop:reflectionmating}
    Let $S \in \Sigma_d^*$. 
    Then there exists a necklace reflection group $G$ of rank $d+1$ so that $S$ is a mating of $\bar z^d$ with $G$.
    More precisely,
    \begin{enumerate}
        \item $S\vert_{\mathcal{B}_\infty(S)}$ is conformally conjugate to $\bar z^d\vert_{\D}$; and
        \item $S$ on $\overline{\mathcal{T}_\infty(S)}$ is conformally conjugate to $\mathcal{F}_G$ on $\overline{\mathcal{T}_\infty(G)}$.
    \end{enumerate}
    In particular, if $S \in \Int\Sigma_d^*$, the limit set is a Jordan curve. If $S \in \partial\Sigma_d^*$, the limit set is a Basilica.
    
    Conversely, for any necklace reflection group $G$ of rank $d+1$, there exists $S \in \Sigma_d^*$ so that $S$ is a mating of $\bar z^d$ with $G$.
\end{prop}

\subsubsection{Geometry of Schwarz tiling set, and infinitude of Schwarz Basilicas}
Let $S \in \partial \Sigma_d^*$ be a mating of $\bar z^d$ and the Bers boundary reflection group $G$.
A component $U$ of $\cT_\infty(S)$ is called {\em purely accidental parabolic} or {\em persistently parabolic} if the corresponding component $\Delta$ of $\Omega(G)$ is so (cf. \S~\ref{pure_acci_para_subsec}) Equivalently, a component $U$ is purely accidental parabolic if $\partial U$ does not intersect the set of iterated preimages of the conformal cusps of $\partial T(S)$.

Recall that two disjoint Jordan domains $U,V\subset \widehat\C$ are said to be \textit{tangent to each other at a point} $z_0\in \partial U\cap \partial V$ if there exists $\theta_0\in [0,2\pi)$ such that for each $\varepsilon>0$ there exists $\delta>0$ with the property that $\{z_0+ re^{i(\theta_0+\theta)}:  r\in (0,\delta), \,\, |\theta|<\pi/2-\varepsilon\}\subset U$ and $\{z_0+ re^{i(\theta_0+\theta)}:  r\in (-\delta,0), \,\, |\theta|<\pi/2-\varepsilon\}\subset V$.
Using quasiconformal surgery, we have the following proposition on the geometry of the boundary of bounded components of $\cT_\infty(S)=\widehat{\C} - \Lambda(S)$ (c.f. \cite[Theorem 1.7]{LN24a}).
\begin{prop}\label{prop:boundarySch}
    Let $S \in \partial \Sigma_d^*$. Let $U,V$ be two bounded components of $\cT_\infty(S)$. Then the following statements hold.
    \begin{enumerate}
    \item If $\partial U \cap \partial V \neq \emptyset$, then they touch tangentially. 
        \item The boundary $\partial U$ is a quasicircle if $U$ is purely accidental parabolic.
        \item The boundary $\partial U$ is quasiconformally equivalent to the cauliflower $J(z^2+\frac14)$ if $U$ is persistently parabolic.
    \end{enumerate}
\end{prop}
\begin{proof}[Sketch of proof]
We call a component of $\cT_\infty(S)$ \emph{principal} if it intersects $T(S)$, and \emph{non-principal} otherwise.
If $U$ is a non-principal component of $\cT_\infty(S)$, then some iterate of $S$ maps an open neighborhood of $\overline{U}$ conformally onto an open neighborhood of $\overline{U}_0$, where $U_0$ is some principal component. Further, every contact point of $\Lambda(S)$ maps to a double point on $\partial T(S)$ under some iterate of $S$ (see \cite[Proposition~48]{LMM22}). Thus, for the proof of the proposition, we may assume that $U, V$ are principal components of $\cT_\infty(S)$, and that $\partial U\cap\partial V$ is a double point on $\partial T(S)$. 
\smallskip

1. By our assumption, $\partial U$ and $\partial V$ intersect at a double point $p$ of $\partial T(S)$. More precisely, two local non-singular arcs of $\partial T(S)$ intersect at $p$ with order of contact $1$ (cf. \cite[Proposition~2.10]{LMM21}). In particular, these two non-singular arcs have a common tangent line at $p$. The two normal directions $\ell^\pm$ to this common tangent line are repelling directions for $S$.
A standard argument involving the local parabolic dynamics of $S$ near a double point (having order of contact $1$) shows that locally near $p$, the non-escaping set $K(S):=\widehat{\C}-\cT_\infty(S)$ is contained in two sectors of arbitrarily small angles containing $\ell^\pm$ (see Figure~\ref{fig:sigma}, cf. \cite[\S~3.2]{MR25}). Hence, each of the components $U, V$ contains a sector of angle arbitrarily close to $\pi$ at $p$; i.e., they touch tangentially. 
\smallskip

2. and 3. Note that for a principal component $U$ of $\cT_\infty(S)$, we have $S(\partial U)=\partial U$. The map $S$ induces two topologically expanding, $C^1$, conformal Markov maps $g_{\mathrm{in}}$ and $g_{\mathrm{out}}$ on the ideal boundaries $\partial ^I U\cong\mathbb{S}^1$ and $\partial^I (\widehat{\C}-\overline{U})\cong\mathbb{S}^1$, respectively. These maps admit puzzle structures, and they are topologically conjugate by a circle homeomorphism $h$ that is the welding homeomorphism for the Jordan curve $\partial U$. The periodic break-points of these ideal boundary maps $g_{\mathrm{in}}$, $g_{\mathrm{out}}$ come from the double points or conformal cusps (of $\partial T(S)$) on $\partial U$. We make the following additional observations.
\begin{itemize}[leftmargin=8mm]
    \item If a periodic break-point of $g_{\mathrm{in/out}}$ comes from a double point of $\partial T(S)$, then the local parabolic dynamics of $S$ near double points implies that this break-point is symmetrically parabolic for $g_{\mathrm{in/out}}$.
    \item The conformal cusps of $\partial T(S)$ are of type $(3,2)$ (cf. \cite[Proposition~2.10]{LMM21}). If a periodic break-point of $g_{\mathrm{in/out}}$ comes from a cusp of $\partial T(S)$, then the local parabolic dynamics of $S$ near conformal cusps of type $(3,2)$ implies that this break-point is symmetrically parabolic for $g_{\mathrm{in}}$, but symmetrically hyperbolic for $g_{\mathrm{out}}$ (cf. \cite[Proposition~A.3]{LMM24}, \cite[\S~4.2.1.]{LLMM23}).
\end{itemize}

If $U$ is purely accidental parabolic, then all singular points of $\partial T(S)$ on $\partial U$ are double points. Thus in this case, all the periodic break-points of $g_{\mathrm{in}}, g_{\mathrm{out}}$ are symmetrically parabolic. By Proposition~\ref{prop:qsmarkovmap}, the conjugacy $h$ is quasi-symmetric, and hence the welding curve $\partial U$ is a quasicircle. (One can also employ the arguments of \cite[Theorem~6.1]{CJY94}, combined with the fact that $U$ contains a sector of angle arbitrarily close to $\pi$ at each double point, to prove that $\partial U$ is a quasicircle).

On the other hand, if $U$ is persistently parabolic, then at least one singular point of $\partial T(S)$ on $\partial U$ is a cusp. Thus in this case, the conjugacy $h$ (from $g_{\mathrm{out}}$ to $g_{\mathrm{in}}$) carries some hyperbolic fixed point to a parabolic point, but never the other way round.    By \cite{McM25}, the welding curve $\partial U$ associated with $h$ is quasiconformally equivalent to $J(z^2+\frac14)$.
\end{proof}

We can construct a {\em bi-colored contact tree} for $S\in \partial \Sigma_d^*$: vertices correspond to bounded components of $\widehat{\C} - \Lambda(S)$, and there is an edge between two vertices if the corresponding components touch. The vertex is colored black if it is purely accidental parabolic, and white otherwise. By Proposition~\ref{prop:boundarySch}, this bi-colored contact tree is a quasiconformal invariant.
One can easily construct, as in \S~\ref{subsubsec:bicolct}, infinitely many Bers boundary reflection groups with pairwise non-isomorphic bi-colored contact trees. By Proposition~\ref{prop:reflectionmating}, we have infinitely many Schwarz reflections in $\Sigma^* = \bigcup_d\Sigma_d^*$ with pairwise non-isomorphic bi-colored contact trees. Thus, we have the following lemma. Recall that $\mathfrak{U}_{\partial \Sigma^*} \subseteq \mathfrak{B}$ is the set of quasiconformal classes of Basilicas arising as limit sets of Schwarz reflections in $\partial\Sigma^*$.
\begin{lem}\label{lem:infinite_schwarz_basilica}
    There are infinitely many quasiconformally nonequivalent Basilica limit sets for Schwarz reflections $S \in \partial \Sigma^*$, i.e., $\mathfrak{U}_{\partial \Sigma^*}$ is an infinite set.
\end{lem}

\begin{cor}\label{cor:schwarz_klein_disjoint}
    The intersection $\mathfrak{U}_{\partial \Sigma^*} \cap \mathfrak{U}_{\kle} = \emptyset$.
\end{cor}
\begin{proof}
    Let $S \in \partial \Sigma^*$.
    Note that there always exists a persistently parabolic component for $S$. By Proposition~\ref{prop:boundarySch}, this persistently parabolic component is not a quasicircle. Since each Jordan domain complementary component of a Basilica limit set of a geometrically finite Kleinian group is a quasidisk, the result follows.
\end{proof}

\subsection{Cuspidal Basilica}
In this subsection, we will analyze the intersection of $\mathfrak{U}_{\partial \Sigma^*} \cap \mathfrak{U}_{\kle}$.
Recall that 
\begin{equation}\label{eqn:defnR}
    R(z) = az^4+\frac{1-4a}{3}z^3+\frac{2+a}{3}
\end{equation}
is the degree $4$ polynomial, where $a=\frac{1}{12}+\frac{i\sqrt{2}}{24}$ is chosen so that the critical point $1-\frac{1}{4a}$ is mapped to the parabolic fixed point $1$.
Its Julia set is a cuspidal Basilica, see Figure~\ref{fig:CuspidalBasilica}.
In the following, we will show that $\mathfrak{U}_{\partial \Sigma^*} \cap \mathfrak{U}_{\rat} = [J(R)]$.

\subsubsection{Necessary condition for quasiconformal equivalence}
We begin with a necessary criterion for Basilica Julia sets to lie in the intersection $\mathfrak{U}_{\partial \Sigma^*} \cap \mathfrak{U}_{\rat}$.
\begin{lem}\label{lem:strucCuspidalB}
    Let $g$ be a geometrically finite rational map with Basilica Julia set satisfying the following properties:
    \begin{itemize}[leftmargin=8mm]
        \item the boundary of each Jordan domain Fatou component is either a quasicircle, or is quasiconformally equivalent to the cauliflower $J(z^2+\frac14)$; and 
        \item any two adjacent Jordan domain Fatou components touch tangentially.
    \end{itemize}
    If some Fatou component of $g$ is a quasi-disk, then each Jordan domain Fatou component of $g$ is a quasi-disk, i.e., $J(g)$ is a fat Basilica.
\end{lem}
\begin{proof}
    Let $U$ be a quasi-disk Fatou component. Note that by our assumption, its image is also a quasi-disk (as it cannot be a cauliflower). Therefore, after passing to an iterate, we may assume that it is fixed under $g$.
    Suppose by contradiction that there exists a non-quasi-disk Fatou component $V$. Then by the same reasoning and after passing to an iterate, we may assume that $V$ is fixed. Let $b$ be the parabolic fixed point on $\partial V$. 
    
   The rest of the proof is similar to that of Lemma~\ref{lem:cutpnteventuallypara}.
    Let $U_0 = U, U_1, ..., U_k = V$ be the chain of Fatou components connecting $U$ to $V$. Then $g$ fixes $U_i$ for all $i$. 
    Denote $v_i = \partial U_i \cap \partial U_{i+1}$, and $v_k = b$. Then $v_i$ is parabolic with multiplicity $3$. This means that $U_1$ is the basin of attraction for both $v_0$ and $v_1$, which is a contradiction. The lemma follows.
\end{proof}

\begin{cor}\label{cor:onlypossintersect}
    Let $S\in \partial \Sigma^*_d$. Suppose that $\Lambda(S)$ is quasiconformally equivalent to the Julia set of a geometrically finite rational map $g$. Then each component of $\cT_\infty(S)$ is persistently parabolic.
\end{cor}
\begin{proof}
    Let $g$ be a geometrically finite rational map whose Julia set is quasiconformally equivalent to $\Lambda(S)$.
    By Proposition~\ref{prop:boundarySch}, $g$ satisfies the assumptions of Lemma~\ref{lem:strucCuspidalB}.
    Observe that $S$ has at least one persistently parabolic component. (Indeed, a peripheral component of $T^0(S)$ has a conformal cusp on its boundary, and hence the corresponding principal component of $\cT_\infty(S)$ is persistently parabolic.) Hence, $g$ has at least one Jordan domain Fatou component that is not a quasidisk. Therefore, by Lemma~\ref{lem:strucCuspidalB}, no Jordan domain Fatou component is a quasidisk, and the corollary follows.
\end{proof}

\subsubsection{$R-$intervals}
    We will now define \emph{$R-$intervals} as analogs of dyadic intervals (see Definition~\ref{defn:dyaint}) and generalized dyadic intervals (see Definition~\ref{defn:gendyadic}) for the current set-up. The modifications are dictated by the simultaneous existence of cusps and contact points in the grand orbit of the parabolic fixed point of $R$. The notion of $R-$intervals and its classification in terms of colorings of its end-points (which record the cusps and contact points) are given in Definition~\ref{r_int_def} and Definition~\ref{defn:alternating}. Subsequently, we study appropriate decompositions of $R-$intervals and the  ratios of the lengths of the associated sub-intervals in Lemma~\ref{lem:Rintdecomp} (which is the analog of Lemma~\ref{lem:primitivedyadicdecomp} for the present setting).
    
    Let $U_0$ be  the unique fixed bounded Fatou component of $R$. Note that $R: \partial U_0 \longrightarrow \partial U_0$ is topologically conjugate to $\sigma_3$ on $\mathbb{S}^1 \cong \R/\Z$, where we choose the conjugacy so that the parabolic fixed point corresponds to the angle $0$. Further, for any bounded Fatou component $U$, there exists $l \geq 1$ so that $R^l$ maps $U$ conformally to $U_0$. 
     We obtain, via this conformal map, a parametrization of the boundary of a bounded Fatou component by internal angles.
    Note that if $U \neq U_0$, then 
    \begin{itemize}[leftmargin=8mm]
        \item the contact points on $\partial U$ correspond to 
        $$
        \bigg\{\frac{p}{3^n} \mod 1: p \equiv 1 \mod 3,\ n \geq 1\bigg\} \cup \{0\};\ \mathrm{and}
        $$
        \item the cusps of $\partial U$ correspond to 
        $$
        \bigg\{\frac{p}{3^n} \mod 1: p \equiv 2 \mod 3,\ n \geq 1\bigg\}.
        $$
    \end{itemize}
    On the other hand,
    \begin{itemize}[leftmargin=8mm]
        \item the contact points on $\partial U_0$ correspond to $\{\frac{p}{3^n} \mod 1: p \equiv 1 \mod 3,\ n \geq 1\}$;
        \item the cusps of $\partial U_0$ correspond to $\{\frac{p}{3^n} \mod 1: p \equiv 2 \mod 3,\ n \geq 1\} \cup \{0\}$.
    \end{itemize}

    We color a non-zero 3-adic rational  $t = \frac{p}{3^n}$ \emph{red} or \emph{blue} depending on whether $p\equiv 1 \mod 3$ or $p\equiv 2 \mod 3$.
    For simplicity, in the following, let us assume that $0$ is \emph{blue}. This corresponds to the case of the Fatou component $U_0$. The case when $0$ is red is similar, with routine modifications.
    
    
    \begin{defn}\label{r_int_def}
    An interval $I=[s,t] \subseteq \mathbb{S}^1 \cong \R/\Z$ is called an {\em $R-$interval} if there exists $n\geq 0$ so that $\sigma_3^n$ maps $(s,t)$ homeomorphically onto $(0,1)$ or $(0, \frac23)$ or $(\frac13, 1)$.
    We call the unique point $x_I\in(s,t)$ with $\sigma_3^n(x_I) = \frac12$ the {\em marked point} of the $R-$interval $I$.
    We say that an $R-$interval is of type A, B or C depending on whether it is mapped to $[0,1], [0,\frac23]$ or $[\frac13, 1]$.
    \end{defn}
    Note that $[0,\frac13],\ [\frac13,\frac23],\ [\frac23,1]$ are also $R-$intervals.
    \begin{remark}\label{remark:incompatible}
1) We remark that unlike generalized dyadic intervals introduced in Definition~\ref{defn:gendyadic}, the type of an $R-$interval is not determined purely by the colorings of its end-points (see Remark~\ref{remark:compatible}). For instance, both end-points of the type A interval $[0,1]$ are blue; but for the type A interval $[0,\frac13]$, the left end-point is blue while the right end-point is red.
\smallskip

\noindent 2) The marked point $x_I$ of an $R$-interval $I$ will be important in defining canonical maps between various puzzle pieces in \S~\ref{canonical_map_cusp_basi_subsec}.
    \end{remark}

Similar to Lemma~\ref{lem:decomp}, we have the following weaker decomposition property for an $R-$interval.
\begin{lem}\label{lem:decomp-schwarz}
    Let $I=[s,t]$ be an $R-$interval. 
    Given two distinct colors $\{c_1, c_2\} = $ $\{blue\ , red\}$,
     there exists a decomposition $I = [s,x_1] \cup [x_1, x_2] \cup [x_2, t]$ into $R-$intervals so that $x_i$ is of color $c_i$.
\end{lem}
\begin{proof}
    Let $m \geq 0$ so that $\sigma_3^m(I)$ is one of the intervals in the list
    $$[0,1], [0,\frac23], [\frac13, 1].
    $$
    For these intervals, we have the required decompositions:
    $$
    [0,1]=[0,\frac13]\cup [\frac13,\frac23]\cup [\frac23,1],\quad [0,1]=[0,\frac23]\cup[\frac23,\frac79]\cup[\frac79,1],
    $$
    $$
    [0,\frac23]=[0,\frac13]\cup [\frac13,\frac59]\cup [\frac59,\frac23],\quad [0,\frac23]=[0,\frac29]\cup[\frac29,\frac13]\cup[\frac13,\frac23],
    $$
    $$
    [\frac13,1]=[\frac13,\frac49]\cup [\frac49,\frac23]\cup [\frac23,1],\quad [\frac13,1]=[\frac13,\frac23]\cup[\frac23,\frac79]\cup[\frac79,1].
    $$
    As $\sigma_3^m$ is a color-preserving map for $3-$adic rational numbers in $\Int(I)$, we can pull back the above decompositions via $\sigma_3^m$ to produce the desired decomposition of $I$.
\end{proof}

By induction, we have the following.
\begin{cor}\label{cor:alternatingdecomposition}
    Let $(c_i)_{i=1}^{2N}$ be an alternating sequence of colors in blue and red (where indices are taken modulo $N$). Then there exists a decomposition
    $$
    \mathbb{S}^1 = [s_1,s_2]\cup[s_2, s_3]\cup ...\cup [s_{2N}, s_1],
    $$
    so that each interval $[s_i, s_{i+1}]$ is an $R-$interval, and $s_i$ has color $c_i$.
\end{cor}

\begin{defn}\label{defn:alternating}
We call an $R-$interval $I$ {\em alternating} if the two end-points of $I$ are of different colors.
We define the {\em subtype} of an alternating $R-$interval as follows.
\begin{itemize}
    \item $A1$, if $I$ is of type $A$, and the left end-point is red;
    \item $A2$, if $I$ is of type $A$, and the left end-point is blue;
    \item $B$, if $I$ is of type $B$, in which case the right end-point is necessarily blue (and hence the left end-point is necessarily red);
    \item $C$, if $I$ is of type $C$, in which case the left end-point is necessarily red.
\end{itemize}

Let $I = [s,t]$ be an alternating $R-$interval.
We say $I = J_1 \cup J_2 \cup \ldots \cup J_{2M+1}$ is an {\em alternating decomposition} if each $J_i$ is an alternating $R-$interval and they are ordered linearly.
We say an alternating decomposition $I = J_1 \cup J_2 \cup \ldots \cup J_{2M+1}$ is {\em admissible} if there exists $J_i$ so that the marked point $x_{J_i}$ of $J_i$ and the marked point $x_I$ of $I$ are equal, and $J_i$ is of the same type as $I$. Note that we allow them to have different subtypes.
\end{defn}

\noindent Let 
$$
\mathcal{N}^+(I, M):=\{ |J_1|/|I|: I = J_1 \cup \ldots \cup J_{2M+1} \text{ is an alternating decomposition} \}.
$$
Similarly, we define 
$$
\mathcal{N}^-(I, M):= \{ |J_{2M+1}|/|I|: I = J_1 \cup \ldots \cup J_{2M+1} \text{ is an alternating decomposition} \}.
$$
We also define 
\begin{equation*}
\begin{split}
&\mathcal{N}^+_o(I, M):= \{ |J_1|/|I|: I = J_1 \cup \ldots \cup J_{2M+1} \text{ is an admissible}\\ &\hspace{5.5cm} \text{alternating decomposition} \};
\end{split}
\end{equation*}
and 
\begin{equation*}
\begin{split}
&\mathcal{N}^-_o(I, M):= \{ |J_{2M+1}|/|I|: I = J_1 \cup \ldots \cup J_{2M+1} \text{ is an admissible}\\ &\hspace{6cm} \text{alternating decomposition} \}.
\end{split}
\end{equation*}

The following lemma shows the flexibility in constructing alternating decompositions and admissible alternating decompositions (c.f. Lemma~\ref{lem:primitivedyadicdecomp} for dyadic decomposition).
\begin{lem}\label{lem:Rintdecomp}
    Let $I = [s,t]$ be an alternating $R-$interval. Let $M \geq 2$.
    Then the following holds.
    \begin{itemize}
        \item If $I$ is Type $A1$, then 
        $$
            \mathcal{N}^\pm(I, M)  = \mathcal{N}^\pm_o(I, M)\sqcup \bigg\{\frac23, \frac{2}{3^{M+1}}\bigg\} = \bigg\{\frac{2}{3^j}: j\in\{1,\ldots, M+1\}\bigg\}.
        $$
        \item If $I$ is Type $A2$, then
        $$
            \mathcal{N}^\pm(I, M)  = \mathcal{N}^\pm_o(I, M) =  \bigg\{\frac{1}{3^{j}}: j\in\{1,\ldots, M\}\bigg\};
        $$
        \item If $I$ is Type $B$, then
        \begin{align*}
            \mathcal{N}^+(I, M) = \mathcal{N}^+_o(I, M)\sqcup \bigg\{\frac{1}{3^M}\bigg\} &= \bigg\{\frac{1}{3^{j}}: j\in\{1,\ldots, M\}\bigg\};\\
            \mathcal{N}^-(I, M) = \mathcal{N}^-_o(I, M)\sqcup \bigg\{\frac{1}{2}, \frac13\bigg\} &= \bigg\{\frac12\bigg\} \cup \bigg\{\frac{1}{3^{j}}: j\in\{1,\ldots, M\}\bigg\}.
        \end{align*}
        \item If $I$ is Type $C$, then
        \begin{align*}
            \mathcal{N}^+(I, M) = \mathcal{N}^+_o(I, M)\sqcup \{\{\frac{1}{2}, \frac13\bigg\} &= \bigg\{\frac12\bigg\} \cup \bigg\{\frac{1}{3^j}: j\in\{1,\ldots, M\}\bigg\};\\
            \mathcal{N}^-(I, M) = \mathcal{N}^-_o(I, M)\sqcup \bigg\{\frac{1}{3^M}\bigg\} &= \bigg\{\frac{1}{3^{j}}: j\in\{1,\ldots, M\}\bigg\}.
        \end{align*}
    \end{itemize}
    In particular, if $M\geq 2$, then admissible alternating decomposition exists.
\end{lem}
\begin{proof}
    We can first blow up $I$ by some iterate of $\sigma_3$ to one of the intervals in the list $\{[0,1], [0, \frac23], [\frac13,0]\}$.
    All possible alternating decompositions with $M = 1$ are listed in Figure~\ref{fig:altdecom}. The set $\mathcal{N}^\pm(I, M)$ can be computed simply by inductively decomposing one of the alternating intervals into alternating $R-$intervals.

    The computation for $\mathcal{N}^\pm_o(I, M)$ is similar, the smallest admissible alternating decompositions are listed in Figure~\ref{fig:admaltdecom}
\end{proof}

\begin{figure}
\captionsetup{width=0.96\linewidth}
    \centering
    \includegraphics[width=0.9\linewidth]{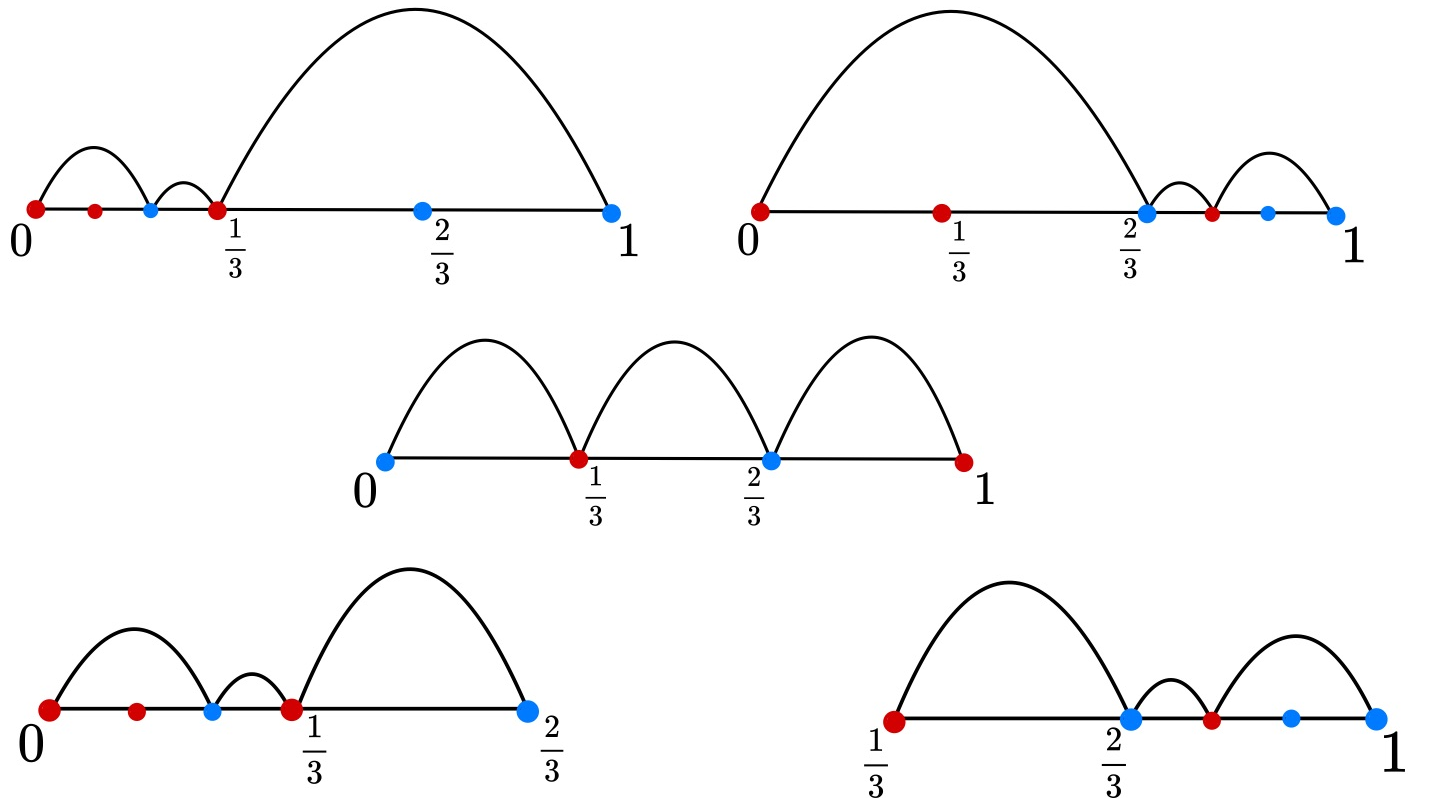}
    \caption{The list of all possible alternating decompositions with $M = 1$. The top two correspond to two different alternating decompositions for subtype $A1$. The middle one corresponds to subtype $A2$, and the bottom two correspond to type $B$ and type $C$.}
    \label{fig:altdecom}
\end{figure}

\begin{figure}
\captionsetup{width=0.96\linewidth}
    \centering
    \includegraphics[width=0.9\linewidth]{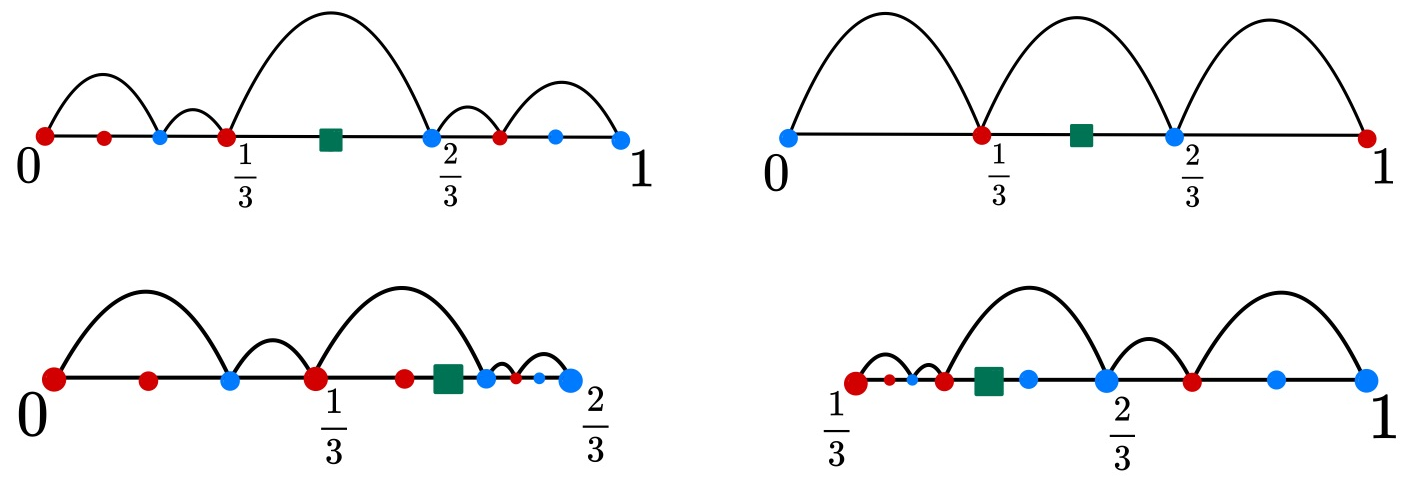}
    \caption{The list of smallest admissible alternating decompositions for the four different subtypes. The green vertex corresponds to the marked point $\frac12$.}
    \label{fig:admaltdecom}
\end{figure}

\subsubsection{Canonical maps}\label{canonical_map_cusp_basi_subsec}
Our next goal is to model color-preserving Markov maps. 
    By Remark~\ref{remark:incompatible}, we need to consider maps between different types of $R-$intervals. One can construct such maps by introducing an additional break-point at an iterated preimage of the fixed point $\frac12$ of $\sigma_3$ as follows.
    
    Let $q^+:[0,\frac23] \longrightarrow [0,1]$ be the map
    $$
    q^+(x) = \begin{cases}
        x\;\qquad \text{ if }x \in [0,\frac12],\\
        \sigma_3(x)\;\ \text{ if }x \in [\frac12,\frac23].
    \end{cases}
    $$
    Similarly, let $q^-:[\frac13,1] \longrightarrow [0,1]$ be the map
    $$
    q^-(x) = \begin{cases}
        \sigma_3(x)\;\ \text{ if }x \in [\frac13,\frac12],\\
        x\;\qquad \text{ if }x \in [\frac12,1].
    \end{cases}
    $$
    It follows from our construction that any two $R-$intervals are related by iterates of $\sigma_3$, $q^\pm$ and their inverses. We summarize this as follows.
    \begin{lem}\label{lem:correspondence}
        Let $I_1 = [s_1, t_1], I_2 = [s_2, t_2]$ be two $R-$intervals.
        \begin{itemize}
            \item If they are of the same type, then there exist $n, m$ so that $\sigma_3^{-n}\circ \sigma_3^m$ is a homeomorphism between $I_1$ and $I_2$.
            \item If $I_1$ is of type $B$ and $I_2$ is of type $A$, then there exist $n, m$ so that $\sigma_3^{-n}\circ q^+ \circ \sigma_3^m$ is a homeomorphism between $I_1$ and $I_2$.
            \item If $I_1$ is of type $C$ and $I_2$ is of type $A$, then there exist $n, m$ so that $\sigma_3^{-n}\circ q^- \circ \sigma_3^m$ is a homeomorphism between $I_1$ and $I_2$.
            \item If $I_1$ is of type $B$, and $I_2$ is of type $C$, then there exist $n, m$ so that $\sigma_3^{-n}\circ (q^-)^{-1}\circ q^+ \circ \sigma_3^m$ is a homeomorphism between $I_1$ and $I_2$.
        \end{itemize}
        We call the map above (or its inverse) the {\em canonical map} between two $R-$intervals.
        Note that the canonical map sends the marked point $x_{I_1}$ to the marked point $x_{I_2}$.
    \end{lem}
    Note that the canonical map has a potential break-point at the marked point. However, if $I_1, I_2$ are of the same type, the canonical map is smooth at the marked point. This observation motivates the definition of admissible decomposition in Definition~\ref{defn:alternating} above, and this is used in the proof of Theorem~\ref{thm:symmetricR} below.

\subsubsection{Basilica $\mathfrak{R}-$maps}
Let $R(z)$ be as defined in \eqref{eqn:defnR}. Let $\mathfrak{R}$ be the collection of conformal maps of the form $R^{-n} \circ R^m$.

Recall that the unique fixed bounded Fatou component of $R$ is denoted by $U_0$. 
To account for the marked point in each $R-$interval, we define the level $0$ internal puzzle piece corresponding to $I_{U_0}[0,1]$ as a suitable pinched neighborhood of $\partial U_0$ in $\overline{U_0}$ that is pinched at the two fixed points of $R$ on $\partial U_0$ (cf. \S~\ref{subsubsec:internalpuzzle}). 
Thus, $\Int{P_{U_0}[0,1]}$ is the union of two  topological disks touching at the parabolic fixed point at internal angle $0$ and the repelling fixed point at internal angle $\frac12$ (the latter corresponds to the marked point of the $R-$interval $[0,1]$). The internal $R-$puzzle pieces $P_{U_0}[0,\frac23]$ and $P_{U_0}[\frac13,1]$ are defined similarly, with pinchings at the corresponding marked points. By pulling back under iterates of $R$, we obtain internal $R-$puzzle pieces $P_U[s,t]$ for any bounded Fatou component $U$ and $R-$interval $[s,t]$ (cf. Definition~\ref{defn:internalpuzzle}).
We shall think of these marked points as additional break-points.


Let $P_i, P_j$ be two internal $R-$puzzle pieces.
Then by Lemma~\ref{lem:correspondence}, there is a canonical map $F$ which is a homeomorphism from $P_i$ to $P_j$.
Moreover, the canonical map is of the form $Q_{\pc}^{-m}\circ Q_{\pc}^n$ on each of the two connected components of $P_i$.

The definition of internal $R-$puzzle
$$
\left(\{U_i\}, P_{U_i}[s_{i,j}, s_{i,j+1}]\right); i \in\{0,\cdots, r\},\ j\in\{1,\cdots, n_i\}
$$
is the same as Definition~\ref{defn:intpuz} after replacing dyadic segments with $R-$segments and with the additional assumption that the decomposition of 
$$
\partial U_i = \bigcup_j I_{U_i}[s_{i,j}, s_{i,j+1}]
$$ 
is alternating (see Definition~\ref{defn:alternating}). 
It induces a Basilica $R-$puzzle $(P_i)_{i\in \mathcal{I}}$ as in Definition~\ref{basi_puzz_def} and \S~\ref{subsubsec:indbp}.

The definition of Markov refinements in Definition~\ref{defn:refineBp} generalizes to this setting with the additional assumption that the corresponding alternating decomposition of $I_{U_i}[s_{i,j}, s_{i,j+1}] \subseteq \partial U_i$ is admissible (see Definition~\ref{defn:alternating}).

Let $\mathcal{F}: \mathfrak{P}^1 \rightarrow \mathfrak{P}$ be a Basilica $\mathfrak{R}-$map.
Then we have the induced Markov maps
$$
\mathcal{F}_\infty: \partial^I U_\infty \longrightarrow \partial^I U_\infty \text{ and } \mathcal{F}_{\bdd}: \bigcup \partial^I U_i \longrightarrow \bigcup \partial^I U_i.
$$
Note that the break-points of the induced Markov map consist of the boundary points and the marked points of the corresponding $R-$intervals.
We remark that by the admissibility condition (see Definition~\ref{defn:alternating}), each marked point of an $R-$interval is eventually fixed. Therefore, by Lemma~\ref{lem:correspondence}, any fixed marked point is a smooth point as the type of the level $1$ and level $0$ puzzle puzzle pieces containing the fixed marked point are the same.

With this preparation, and using Corollary~\ref{cor:alternatingdecomposition} and Lemma~\ref{lem:Rintdecomp} instead of Lemma~\ref{lem:primitivedyadicdecomp}, we have the following counterpart of Theorem~\ref{thm:symmetrichyperbolic} and Proposition~\ref{prop:bdd_fatou_symm_para}.
\begin{theorem}\label{thm:symmetricR}
    Let $\mathcal{F}: \mathfrak{P}^1 \rightarrow \mathfrak{P}$ be a Basilica $\mathfrak{R}-$map.
    Suppose that $\mathcal{F}$ is topologically expanding. 
    Then there exists a modified refinement $\widetilde{\mathcal{F}}: \widetilde{\mathfrak{P}}^1 \rightarrow \widetilde{\mathfrak{P}}$ which is in particular topologically conjugate to $\mathcal{F}: \mathfrak{P}^1 \rightarrow \mathfrak{P}$ on the Julia set and 
    \begin{itemize}[leftmargin=8mm]
        \item the induced Markov map $\widetilde{\mathcal{F}}_\infty$ is symmetrically hyperbolic; and
        \item for the induced Markov map $\widetilde{\mathcal{F}}_{\bdd}$ on $\displaystyle\bigcup_{s\in \{0,\ldots, r\}}  \partial^I U_s$, 
        \begin{itemize}
            \item the break-points associated to the marked points of the corresponding $R-$intervals are symmetrically hyperbolic;
            \item the break-points associated to the boundary points of the corresponding $R-$intervals are symmetrically parabolic.
        \end{itemize}
    \end{itemize}
\end{theorem}

\subsubsection{Sufficient condition for quasiconformal equivalence}
\begin{prop}\label{prop:Rmodel}
     Let $S\in \partial \Sigma^*_d$. Suppose that each component of $\cT_\infty(S)$ is persistently parabolic. Then $\Lambda(S)$ is quasiconformally equivalent to $J(R)$.
\end{prop}
\begin{proof}
    The proof is similar to Theorem~\ref{thm:qcclassfn-ltsets}, Theorem~\ref{thm-basilica-stdpolymodel} and Proposition~\ref{prop:quasisurggenpcf}.
    We discuss the necessary modifications.

    Let $\Delta$ be a component of $\cT_\infty(S)$ intersecting the droplet $T^0(S)$ non-trivially (i.e., $\Delta$ is a principal component of $\cT_\infty(S)$ in the notation of Proposition~\ref{prop:boundarySch}). Let $T_\Delta = T^0(S) \cap \Delta$. Then $T_\Delta$ is an ideal polygon. Let us denote the ideal vertices of $T_\Delta$ by red or blue, depending on whether the vertex is a contact point of $\Lambda(S)$ or a cusp of $\partial \Delta$. 
    Let $T_\Delta^1 \subseteq \Delta$ be the pullback of $T_\Delta$ under $S$. Then, $T_\Delta^1$ is also an ideal polygon, and we color each ideal vertex of $T_\Delta^1$ red/blue depending on whether $S$ maps it to a red/blue ideal vertex of $T_\Delta$. Note that as $S$ has no critical points on $\Lambda(S)$, a red (respectively, blue) vertex of $T_\Delta^1$ is a contact point of $\Lambda(S)$ (respectively, a cusp of $\partial \Delta$). We will now modify the polygon $T_\Delta$ to ensure that its vertices have alternating colors. 
    
    To this end, let $V^0\subseteq V^1 \subseteq \partial \Delta$ be the set of ideal vertices of $T_\Delta$ and $T_\Delta^1$. Since $\Delta$ is persistently parabolic, both colors appear in $V^0$ and $V^1$.
    We choose a set $\widetilde{V}$ so that
    \begin{itemize}[leftmargin=8mm]
        \item $V^0 \subseteq \widetilde{V} \subseteq V^1$; and
        \item the coloring of the vertices in $\widetilde{V}$ is alternating with respect to the cyclic order on $\partial \Delta$.
    \end{itemize}
    Note that this is possible as $S$ sends the arc in $\partial \Delta$ bounded by two adjacent vertices of $V^0$ to its complement in $\partial \Delta$ (see Figure~\ref{fig:AltDecSvsR}).
    This alternating condition is crucial as we will apply Corollary~\ref{cor:alternatingdecomposition} and Lemma~\ref{lem:Rintdecomp} to construct the corresponding internal puzzles for $R$.
    
    Let $\widetilde{T}_\Delta$ be the convex hull of $\widetilde{V}$.
    Let $\widetilde{T}^1_\Delta$ be the pullback under $S^2$. Here, we use $S^2$ so that the map is orientation-preserving. Note that the coloring of the ideal vertices of $\widetilde{T}^1_\Delta$ is also alternating. By taking a higher iterate of $S^2$ if necessary, we may assume that each arc in $\partial \Delta$ bounded by two adjacent vertices of $\widetilde{V}$ is decomposed into at least $5$ arcs so that we can apply Lemma~\ref{lem:Rintdecomp}.

    \begin{figure}
\captionsetup{width=0.96\linewidth}
    \centering
    \includegraphics[width=0.4\linewidth]{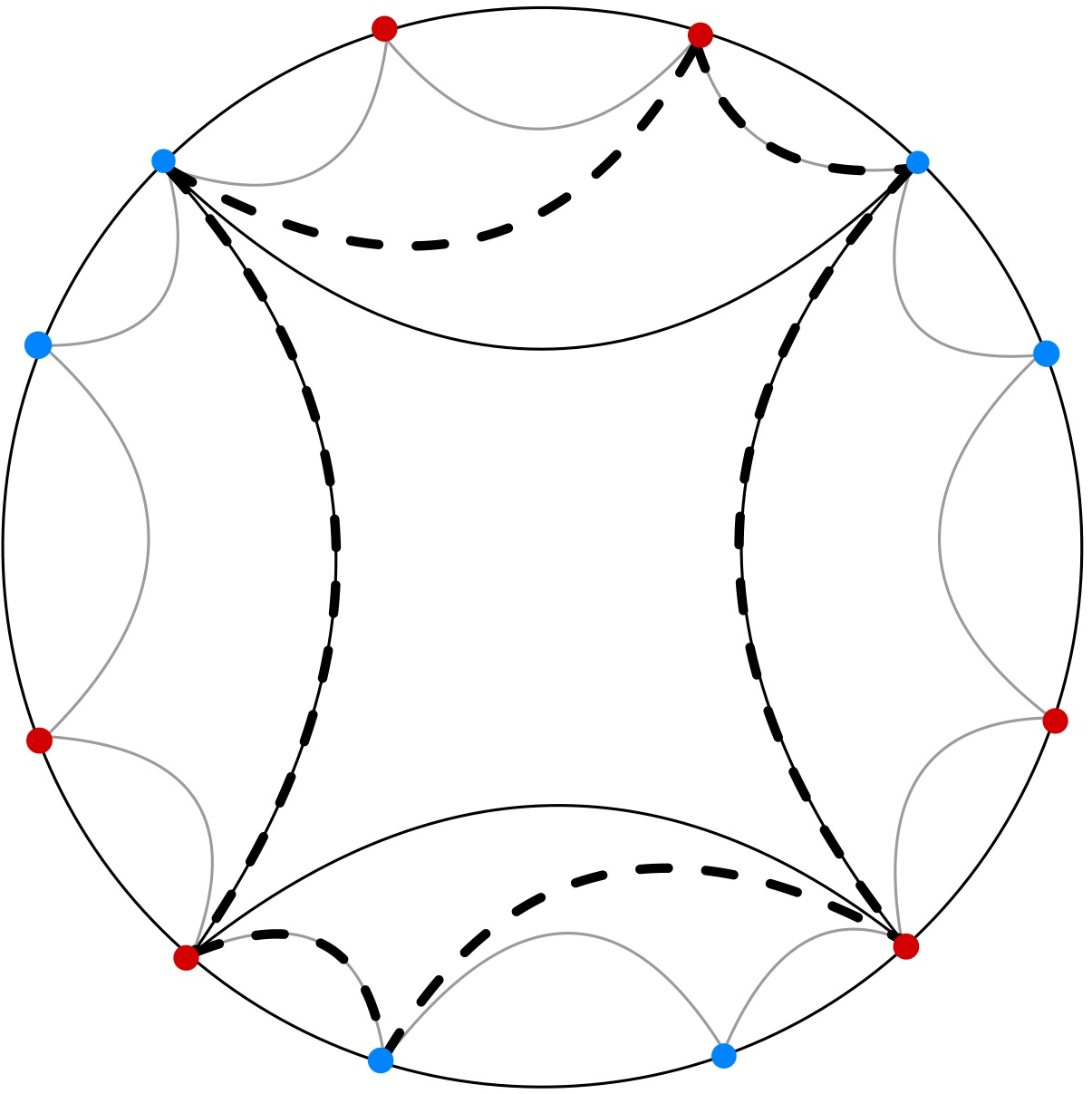}\quad
    \includegraphics[width=0.48\linewidth]{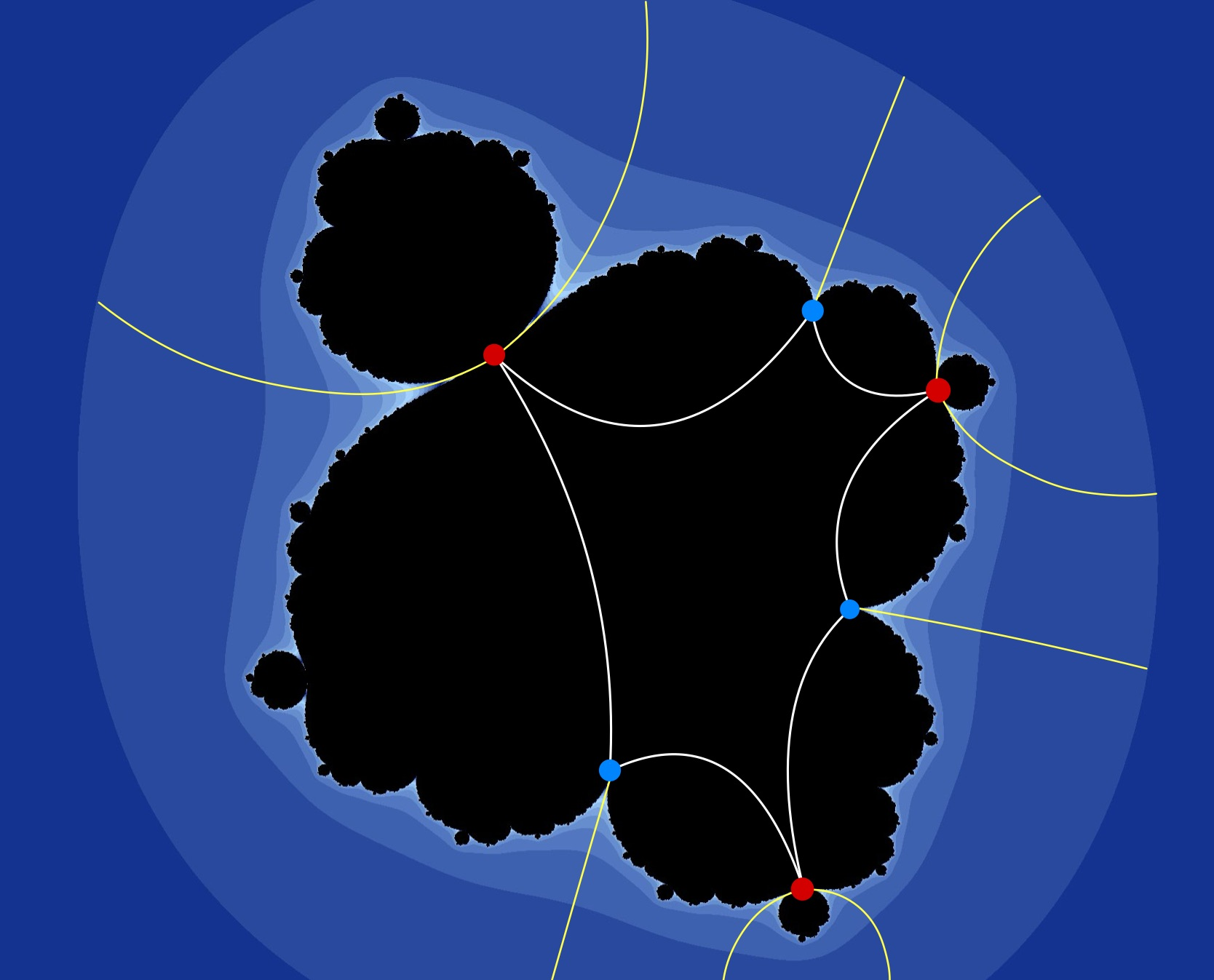}
    \caption{Left: Illustrated is a modification of $T_\Delta$ to $\widetilde{T}_\Delta$ ensuring that the vertices of the latter have alternating colors. Right: A schematic depiction of the corresponding $R-$puzzle pieces on the dynamical plane of $R$. For simplicity, we do not display the additional pinching at the marked points.}
    \label{fig:AltDecSvsR}
\end{figure}
    
    Let $\widetilde{T} = \bigcup_\Delta\widetilde{T}_\Delta$, where the union is over all principal components of $\cT_\infty(S)$. We call $\widetilde{T}$ a {\em pinched core polygon} for $S$.
    
    By a similar construction as in Proposition~\ref{prop:topmodelbowenseries} we obtain a  Basilica $\mathfrak{R}-$map $\widetilde{\mathcal{F}}: \widetilde{\mathfrak{P}}^1 \rightarrow \widetilde{\mathfrak{P}}$ such that $\widetilde{\cF}\vert_{J(R)}$ is topologically conjugate to $S\vert_{\Lambda(S)}$.
    Thus, by Theorem~\ref{thm:symmetricR}, there exists a Basilica $\mathfrak{R}-$map $\mathcal{F}: \mathfrak{P}^1\longrightarrow \mathfrak{P}$ such that
    \begin{itemize}[leftmargin=8mm]
        \item $\mathcal{F}$ is topologically conjugate to $\widetilde{\mathcal{F}}$ (and hence to $S$) on $J(R)$; 
        \item the induced Markov map $\mathcal{F}_\infty$ is symmetrically hyperbolic; and 
        \item for the induced Markov map $\mathcal{F}_{\bdd}$ on $\displaystyle\bigcup_{s\in \{0,\ldots, r\}}  \partial^I U_s$, 
        \begin{itemize}
            \item the break-points associated to the marked points of the corresponding $R-$intervals are symmetrically hyperbolic;
            \item the break-points associated to the boundary points of the corresponding $R-$intervals are symmetrically parabolic.
        \end{itemize}
    \end{itemize}

    Denote the topological conjugacy between $\mathcal{F}$ and $S$ by $H$.
    Note that for each each principal component $\Delta_s$ of $\cT_\infty$ as well as for the unbounded component $\Delta_\infty=\mathcal{B}_\infty(S)$, the restriction $S|_{\Delta_s}$ is topologically conjugate to $\mathcal{F}_s$. Since the conjugacy is type-preserving, Proposition~\ref{prop:qsmarkovmap} provides us with a quasiconformal extension $H|_{U_s}: U_s \longrightarrow \Delta_s$ of the conjugacy   $H|_{\partial U_s}$.

    By pulling back the maps $H|_{U_s}$ using dynamics, we obtain a homeomorphism $H: \widehat{\C} \longrightarrow \widehat{\C}$, which is quasiconformal outside of $J(R)$. 
    Since $J(R)$ is quasiconformally removable by Lemma~\ref{lem:eccentricity}, $H$ is a quasiconformal map of $\widehat{\C}$.
\end{proof}

\begin{proof}[Proof of Theorem~\ref{thm:schwarz}]
The fact that $\mathfrak{U}_{\sch}\subset\mathfrak{B}$ is infinite is the content of Lemma~\ref{lem:infinite_schwarz_basilica}. By Corollary~\ref{cor:schwarz_klein_disjoint}, the sets $\mathfrak{U}_{\sch}$ and $\mathfrak{U}_{\kle}$ are disjoint.

  Let us choose $S_0\in\Sigma_d^*$ ($d\geq 3$) such that $\partial T(S_0)$ has a unique double point (cf. \cite[Theorem~4.14]{LMM21}). Then each component of $T^0(S_0)$ has a conformal cusp on its boundary, and hence each component of $\cT_\infty(S_0)$ is persistently parabolic. By Proposition~\ref{prop:Rmodel}, the limit set $\Lambda(S_0)$ is quasiconformally equivalent to $J(R)$. Hence, $[J(R)]\in \mathfrak{U}_{\sch}\cap\mathfrak{U}_{\rat}$. Conversely, if $[\Lambda(S_1)]\in \mathfrak{U}_{\sch}\cap\mathfrak{U}_{\rat}$ for some $S_1\in\Sigma_d^*$, then by Corollary~\ref{cor:onlypossintersect}, each component of $\cT_\infty(S_1)$ is persistently parabolic. Once again, Proposition~\ref{prop:Rmodel} implies that $[\Lambda(S_1)]=[J(R)]$. Thus, $\mathfrak{U}_{\sch}\cap\mathfrak{U}_{\rat}= \{\left[J(R)\right]\}$.
\end{proof}


\begin{thebibliography}{BCM12}

\bibitem[AIM09]{AIM09}
	K.~Astala, T.~Iwaniec, and G.~Martin.
	\newblock {\em Elliptic partial differential equations and quasiconformal
		mappings in the plane}, volume 148 of {\em Princeton Mathematical Series}.
	\newblock Princeton Univ. Press, Princeton, NJ, 2009.

\bibitem[Ber60]{Ber60}
L.~Bers.
\newblock Simultaneous uniformization.
\newblock {\em Bull. Amer. Math. Soc.}, 66:94--97, 1960.

  \bibitem[BCM12]{BCM12}
J. F. Brock,  R. D. Canary, Y. N.  Minsky. 
\newblock The classification of Kleinian surface groups, II: The ending lamination conjecture. 
\newblock {\em Ann. of Math. (2) 176}, no. 1, 1--149, 2012. 

    \bibitem[BS79]{BS79}
R.~Bowen and C.~Series.
\newblock Markov maps associated with {F}uchsian groups.
\newblock {\em Inst. Hautes \'{E}tudes Sci. Publ. Math.}, 50:153--170, 1979.
	

\bibitem[BA56]{BA56}
    A. Beurling and L. Ahlfors.
    \newblock  The boundary correspondence under quasiconformal mappings.
    \newblock {\em Acta Math.}, 96: 125--142, 1956.

\bibitem[BF99]{BF99}
B.~Branner and N.~Fagella.
\newblock Homeomorphisms between limbs of the {M}andelbrot set.
\newblock {\em J. Geom. Anal.}, 9:327--390, 1999.

\bibitem[BF25]{BF25}
J.~Belk and B.~Forrest.
\newblock Quasi-symmetries of finitely ramified Julia sets.
\newblock {\em Math. Ann.}, 393(3--4):1683--1740, 2025.



\bibitem[BKM09]{BKM09}
M. Bonk, B. Kleiner, and S. Merenkov.
\newblock Rigidity of Schottky set.
\newblock {\em Amer. J. Math.}, 131(2): 409--443, 2009.

\bibitem[BLLM24]{BLLM}
S.~Bullett, L.~Lomonaco, M.~Lyubich, and S.~Mukherjee.
\newblock Mating parabolic rational maps with Hecke groups.
\newblock \url{https://arxiv.org/abs/2407.14780}, 2024.

\bibitem[BLM16]{BLM16}
M. Bonk, M. Lyubich, and S. Merenkov.
\newblock Quasi-symmetries of Sierpinski carpet Julia sets.
\newblock {\em Adv. Math.}, 301: 383--422, 2016.

\bibitem[BM13]{BM13}
M. Bonk and S. Merenkov.
\newblock Quasi-symmetric rigidity of square Sierpi\'nski carpets.
\newblock {\em Ann. of Math. (2)}, 177: 591--643, 2013.

\bibitem[BO22]{BO22}
Y.~Benoist and H.~Oh.
\newblock Geodesic planes in geometrically finite acylindrical manifolds.
\newblock {\em Ergodic Theory Dynam. Systems}, 42: 514--553, 2022.

    

\bibitem[CCH96]{CCH96}
J.~Chen, Z.~Chen, and C.~He,
\newblock Boundary correspondence under $\mu(z)$-homeomorphisms.
\newblock {\em Michigan Math. J.}, 43:211--220, 1996.
    
\bibitem[CJY94]{CJY94}
L.~Carleson, P.~W.~Jones, and J.-C.~Yoccoz.
\newblock Julia and John.
\newblock {\em Bol. Soc. Brasil. Mat. (N.S.)}, 25:1--30, 1994.

\bibitem[CM17]{CM17}
C.~Cashen and A.~Martin.
\newblock Quasi-isometries between groups with two-ended splittings.
\newblock {\em Math. Proc. Cambridge Philos. Soc.}, 162(2):249--291, 2017.

\bibitem[DLS20]{DLS20}
D.~Dudko, M.~Lyubich, and N.~Selinger.
\newblock Pacman renormalization and self-similarity of the Mandelbrot set near Siegel parameters.
\newblock {\em J. Amer. Math. Soc.}, 33:653--733, 2020.

\bibitem[EMM04]{EMM04}
D. Epstein, A. Marden, and V. Markovic.
\newblock Quasiconformal homeomorphisms and the convex hull boundary.
\newblock {\em Ann. of Math.}, 159(1): 205--336, 2004.


\bibitem[Jon95]{Jon95}
P.~Jones.
\newblock On removable sets for {S}obolev spaces in the plane.
\newblock In {\em Essays on Fourier Analysis in Honor of Elias M. Stein} (Princeton, NJ, 1991), 
pages 250--267. Princeton Math. Ser., 42. Princeton Univ. Press, Princeton, NJ, 1995.

\bibitem[JS00]{JS00}
P.~Jones and S.~Smirnov.
\newblock Removability theorems for {S}obolev functions and quasiconformal maps.
\newblock {\em Ark. Mat.}, 38(2):263--279, 2000.

\bibitem[Kah98]{Kah98}
J.~Kahn.
\newblock Holomorphic removability of Julia sets.
\newblock \url{https://arxiv.org/abs/math/9812164}, 1998.

\bibitem[LLM22]{LLM22}
R.~Lodge, Y.~Luo, and S.~Mukherjee.
\newblock Circle packings, kissing reflection groups and critically fixed anti-rational maps.
\newblock {\em Forum Math. Sigma}, 10, paper no. e3, 38 pp., 2022.

\bibitem[LLMM23a]{LLMM23}
S.-Y. Lee, M.~Lyubich, N.~G. Makarov, and S.~Mukherjee.
\newblock Dynamics of {S}chwarz reflections: the mating phenomena.
\newblock {\em Ann. Sci. {\'E}c. Norm. Sup{\'e}r. (4)}, 56:1825--1881, 2023.

\bibitem[LLMM23b]{LLMM23b}
R.~Lodge, M.~Lyubich, S.~Merenkov, and S.~Mukherjee.
\newblock{On dynamical gaskets generated by rational maps, Kleinian groups, and Schwarz reflections}.
\newblock {\em Conform. Geom. Dyn.}, 27: 1--54, 2023.

\bibitem[LLM24]{LLM24}
Y.~Luo, M.~Lyubich, and S.~Mukherjee.
\newblock Degenerate (anti-)polynomial-like maps, Schwarz reflections, and boundary involutions.
\newblock \url{https://arxiv.org/abs/2408.00204}, 2024.

\bibitem[LM18]{LM18}
M. Lyubich and S. Merenkov.
\newblock Quasi-symmetries of the Basilica and the Thompson group.
\newblock {\em Geom. Funct. Anal.} 28:727--754, 2018.

\bibitem[LMM21]{LMM21}
K.~Lazebnik, N.~G. Makarov, and S.~Mukherjee.
\newblock Univalent polynomials and {H}ubbard trees.
\newblock {\em Trans. Amer. Math. Soc.}, 374(7): 4839--4893, 2021.

\bibitem[LMM22]{LMM22}
K.~Lazebnik, N.~G. Makarov, and S.~Mukherjee.
\newblock Bers slices in families of univalent maps.
\newblock {\em Math. Z.}, 300:2771--2808, 2022.

\bibitem[LMM24]{LMM24}
M.~Lyubich, J.~Mazor, and S.~Mukherjee.
\newblock Antiholomorphic correspondences and mating I: realization theorems.
\newblock {\em Commun. Am. Math. Soc.}, 4:495--547, 2024.

\bibitem[LM25a]{LM25a}
M.~Lyubich and S.~Mukherjee.
\newblock Mirrors of conformal dynamics: Interplay between anti-rational maps, reflection groups, Schwarz reflections, and correspondences.
\newblock \url{https://arxiv.org/abs/2310.03316}, to appear in {\em Algebraic, Complex, and Arithmetic Dynamics}, edited by M.~Jonsson and L.~DeMarco, Simons Symposia, Springer, 2025.

\bibitem[LM26]{LM25b}
L.~Lomonaco and S.~Mukherjee.
\newblock Algebraic correspondences and Schwarz reflections: Where rational dynamics meets Kleinian groups.
\newblock \url{https://arxiv.org/abs/2511.08408}, to appear in the {\em Proceedings of the ICM 2026}.

\bibitem[LMM26]{LMM26}
Y.~Luo, M.~Mj, and S.~Mukherjee.
\newblock Accelerated Bowen-Series maps and matings.
\newblock In preparation, 2026.

\bibitem[LMMN25]{LMMN}
M.~Lyubich, S.~Merenkov, S.~Mukherjee, and D.~Ntalampekos.
\newblock David extension of circle homeomorphisms, welding, mating, and removability.
\newblock {\em Mem. Amer. Math. Soc.}, vol. 313, no. 1588, pp. v+110, 2025.

\bibitem[LN24a]{LN24a}
Y.~Luo and D.~Ntalampekos.
\newblock Piecewise quasiconformal dynamical systems of the unit circle.
\newblock \url{https://arxiv.org/abs/2411.14203}, 2024.

\bibitem[LN24b]{LN24}
Y.~Luo and D.~Ntalampekos.
\newblock Uniformization of gasket Julia sets.
\newblock \url{https://arxiv.org/abs/2411.17227}, 2024.

\bibitem[LP97]{LP97}
G.~Levin and F.~Przytycki.
\newblock When do two rational functions have the same {J}ulia set?
\newblock {\em Proc. Amer. Math. Soc.}, 125(7):2179–2190, 1997.

\bibitem[LZ23]{LZ23}
Y.~Luo and Y.~Zhang.
\newblock Circle packings, renormalizations and subdivision rules.
\newblock \url{https://arxiv.org/abs/2308.13151}, 2023.

\bibitem[LZ25]{LZ25}
Y.~Luo and Y.~Zhang.
\newblock On quasiconformal non-equivalence of gasket Julia sets and limit sets.
\newblock {\em Ergodic Theory Dynam. Systems}, 45(11):3465--3489, 2025.

\bibitem[Mar06]{Mar06}
V. Markovic. 
\newblock Quasisymmetric groups. 
\newblock {\em J. Amer. Math. Soc.}, 19, no. 3, 673--715, 2006.
     
\bibitem[Mas70]{Mas70}
B. Maskit.
\newblock On Boundaries of Teichmüller Spaces and on Kleinian Groups: II
\newblock {\em Ann. of Math.}, 91: 607-639, 1970.

\bibitem[McM84]{McM84}
C.~McMullen.
\newblock Simultaneous uniformization of Blaschke products.
\newblock Preprint, \url{https://people.math.harvard.edu/~ctm/papers/home/text/papers/simunif/simunif.pdf}, 1984.

\bibitem[McM90]{McM90}
C.~McMullen.
\newblock Iteration on Teichm\"uller space.
\newblock {\em Invent. Math.}, 99(2):425--454, 1990.

\bibitem[McM98]{McM98}
C.~McMullen.
\newblock Hausdorff dimension and conformal dynamics. {III}: Computation of Dimension.
\newblock {\em Amer. J. Math.}, 120(4):691--721, 1998.

\bibitem[McM26]{McM25}
C.~McMullen.
\newblock The question mark function, welding, and complex dynamics.
\newblock {\em Invent. Math.}, 2026, \url{https://doi.org/10.1007/s00222-026-01405-9}.

\bibitem[McS98]{McS98}
C.~McMullen and D.~Sullivan.
\newblock Quasiconformal homeomorphisms and dynamics. {III}. The
              {T}eichm{\"u}ller space of a holomorphic dynamical system.
\newblock {\em Adv. Math.}, 135:351–395, 1998.

\bibitem[Mer14]{Mer14}
S. Merenkov.
\newblock Local rigidity for hyperbolic groups with Sierpi\'nski carpet boundaries.
\newblock {\em Compos. Math.}, 150(11): 1928--1938, 2014.

\bibitem[Mil92]{Mil92}
J.~Milnor.
\newblock Remarks on iterated cubic maps.
\newblock {\em Experiment. Math.}, 1:5--24, 1992.

\bibitem[Mil06]{Mil06}
J.~Milnor.
\newblock {\em Dynamics in one complex variable}, volume 160 of {\em Annals of
  Mathematics Studies}.
\newblock Princeton University Press, Princeton, NJ, third edition, 2006.

\bibitem[Mj14]{Mj14}
M.~Mj
\newblock Ending laminations and Cannon-Thurston maps. With an appendix by Shubhabrata Das and Mj. 
\newblock \emph{Geom. Funct. Anal. 24}, no. 1, 297--321, 2014. 

\bibitem[MM23]{MM1}
M.~Mj and S.~Mukherjee.
\newblock Combining rational maps and {K}leinian groups via orbit equivalence.
\newblock \emph{Proc. Lond. Math. Soc (3)}, 126:1740--1809, 2023.


\bibitem[MM25]{MM2}
M.~Mj and S.~Mukherjee.
\newblock Matings, holomorphic correspondences, and a Bers slice. 
\newblock {\em J. {\'E}c. polytech. Math.}, 12:1445--1502, 2025.

\bibitem[MV25]{MV25}
S.~Mukherjee and S.~Viswanathan.
\newblock Correspondences on hyperelliptic surfaces, combination theorems, and Hurwitz spaces.
\newblock \url{https://arxiv.org/abs/2508.18711}, 2025.

\bibitem[MR25]{MR25}
S.~Mukherjee and Rashmita.
\newblock On topology and singularities of quadrature domains.
\newblock \url{https://arxiv.org/abs/2509.21468}, 2025.

\bibitem[MU03]{MU03}
R.~Mauldin and M.~Urbański.
\newblock \emph{Graph Directed Markov Systems: Geometry and Dynamics of Limit Sets}.
\newblock Cambridge Tracts in Mathematics 148. Cambridge University Press, 2003.

\bibitem[Nta24]{Nta24}
D.~Ntalampekos.
\newblock Metric definition of quasiconformality and exceptional sets.
\newblock {\em Math. Ann.}, 389(3): 3231--3253 2024.

\bibitem[Nta25]{Nta25}
D.~Ntalampekos.
\newblock Quasiconformal characterization of Schottky sets.
\newblock \url{https://arxiv.org/abs/2507.22658}, 2025.

\bibitem[Roe07]{Roe07}
P.~Roesch.
\newblock {\em Hyperbolic components of polynomials with a fixed critical point of maximal order}.
\newblock {\em Ann. Sci. \'{E}c. Norm. Sup\'er.} 40:901-949 , 2007.

\bibitem[Sul78]{Sul78}
D. Sullivan, 
\newblock On the ergodic theory at infinity of an arbitrary discrete group of hyperbolic motions, 
\newblock {\em Riemann surfaces and related topics: Proceedings of the 1978 Stony Brook Conference}, pp. 465–496, 1978.

\bibitem[Zak08]{Zak08}
S.~Zakeri.
\newblock On boundary homeomorphisms of trans-quasiconformal maps of the disk.
\newblock {\em Ann. Acad. Sci. Fenn. Math.}, 33:241--260, 2008.
    

\end{thebibliography}
\end{document}